\documentclass[12pt]{amsart}

\usepackage[T1]{fontenc}
\usepackage{fullpage}
\usepackage{amsmath,amssymb,mathtools,mathrsfs}
\usepackage[dvipsnames]{xcolor}
\usepackage{tikz-cd}
\usepackage{longtable}
\usepackage{caption}
\usepackage{booktabs}
\usepackage{placeins}
\usepackage{graphicx}
\usepackage{adjustbox}
\usepackage{microtype}
\usepackage[
  colorlinks=true,
  linkcolor=BrickRed,
  citecolor=Blue,
  urlcolor=Blue
]{hyperref}

\usepackage{algorithm}
\usepackage{algpseudocode}
\usepackage{needspace}
\usepackage{array}
\usepackage{multicol}
\usepackage{multirow}
\numberwithin{equation}{section}

\makeatletter
\newenvironment{breakablealgorithm}{%
  \par\addvspace{\intextsep}%
  \Needspace{6\baselineskip}%
  \begingroup
  \edef\BreakableAlgorithmItemPenalty{\the\@itempenalty}%
  \let\BreakableAlgorithmItem\@item
  \def\@item[##1]{%
    \BreakableAlgorithmItem[##1]%
    \@itempenalty\BreakableAlgorithmItemPenalty\relax
  }%
  \let\BreakableAlgorithmIf\If
  \renewcommand{\If}[1]{\BreakableAlgorithmIf{##1}\@itempenalty\@M}%
  \let\BreakableAlgorithmElse\Else
  \renewcommand{\Else}{\BreakableAlgorithmElse\@itempenalty\@M}%
  \let\BreakableAlgorithmForAll\ForAll
  \renewcommand{\ForAll}[1]{\BreakableAlgorithmForAll{##1}\@itempenalty\@M}%
  \let\BreakableAlgorithmWhile\While
  \renewcommand{\While}[1]{\BreakableAlgorithmWhile{##1}\@itempenalty\@M}%
  \let\BreakableAlgorithmEndIf\EndIf
  \renewcommand{\EndIf}{\@itempenalty\@M\BreakableAlgorithmEndIf}%
  \let\BreakableAlgorithmEndFor\EndFor
  \renewcommand{\EndFor}{\@itempenalty\@M\BreakableAlgorithmEndFor}%
  \let\BreakableAlgorithmEndWhile\EndWhile
  \renewcommand{\EndWhile}{\@itempenalty\@M\BreakableAlgorithmEndWhile}%
  \let\BreakableAlgorithmState\State
  \renewcommand{\State}{%
    \@ifnextchar\Return{\@itempenalty\@M\BreakableAlgorithmState}%
      {\BreakableAlgorithmState}%
  }%
  \renewcommand{\caption}[1]{%
    \refstepcounter{algorithm}%
    \def\@currentlabelname{##1}%
    \addcontentsline{loa}{algorithm}{\protect\numberline{\thealgorithm}##1}%
    \hrule height .8pt\kern 2pt\nobreak
    {\raggedright\noindent\textbf{Algorithm~\thealgorithm}\enspace ##1\par}%
    \nobreak\kern 2pt\hrule height .4pt\kern 2pt\nobreak
  }%
}{%
  \par\nobreak\kern 2pt\hrule height .8pt\relax
  \endgroup
  \par\addvspace{\intextsep}%
}
\makeatother

\theoremstyle{plain}
\newtheorem{thm}{Theorem}[section]
\newtheorem{lem}[thm]{Lemma}
\newtheorem{prop}[thm]{Proposition}
\newtheorem{cor}[thm]{Corollary}

\theoremstyle{definition}
\newtheorem{defn}[thm]{Definition}

\theoremstyle{remark}
\newtheorem{rmk}[thm]{Remark}

\newcommand{\CC}{\mathbb C}
\newcommand{\ZZ}{\mathbb Z}
\newcommand{\FF}{\mathbb F}
\newcommand{\PP}{\mathbb P}
\newcommand{\DD}{\mathbb D}
\newcommand{\calF}{\mathcal F}
\newcommand{\calH}{\mathcal H}
\newcommand{\calM}{\mathcal M}

\DeclareMathOperator{\Aut}{Aut}
\DeclareMathOperator{\GL}{GL}
\DeclareMathOperator{\PGL}{PGL}
\DeclareMathOperator{\SL}{SL}
\DeclareMathOperator{\PSL}{PSL}
\DeclareMathOperator{\Sym}{Sym}
\DeclareMathOperator{\diag}{diag}
\DeclareMathOperator{\rank}{rank}
\DeclareMathOperator{\tr}{tr}
\DeclareMathOperator{\Hom}{Hom}
\DeclareMathOperator{\Ann}{Ann}
\newcommand{\Auts}{\Aut^{\mathrm s}}
\DeclareMathOperator{\ord}{ord}

\newcommand{\PBasisLabel}[1]{\mathrlap{P_{#1}}\phantom{P_{99}}}
\newenvironment{SecEightAlignedBasis}[1]{%
  \begingroup\renewcommand{\arraystretch}{1.18}%
  \[\begin{array}{@{}#1@{}}%
}{\end{array}\]\endgroup}

\newenvironment{familydisplay}{%
  \begingroup
  \setlength{\arraycolsep}{1.35pt}%
  \renewcommand{\arraystretch}{0.96}%
  \setlength{\jot}{2pt}%
}{\endgroup}

\newenvironment{SecEightMatrixBlock}{%
  \par\smallskip\begingroup
  \setlength{\arraycolsep}{1.55pt}\renewcommand{\arraystretch}{0.96}%
  \noindent\centering\leavevmode
}{%
  \par\endgroup\vspace{-0.15em}}
\newcommand{\SecEightMatrix}[2]{%
  \mbox{$\displaystyle #1=#2$}%
  \penalty0\hskip1.25em plus 0.45em minus 0.20em\relax}
\newcommand{\SecEightDenseMatrix}[2]{%
  \par\smallskip\noindent\begin{center}%
  \begingroup\setlength{\arraycolsep}{1.0pt}\renewcommand{\arraystretch}{0.94}%
  \adjustbox{max width=.98\linewidth}{$\displaystyle #1=#2$}%
  \endgroup\end{center}\vspace{-0.20em}}
\newenvironment{SecEightBasisBlock}{%
  \par\smallskip\begingroup\noindent\centering\leavevmode
}{%
  \par\endgroup\vspace{-0.15em}}
\newcommand{\SecEightBasisLine}[1]{%
  \adjustbox{max width=.98\linewidth}{\(\displaystyle #1\)}%
  \penalty0\hskip1.45em plus 0.55em minus 0.25em\relax}
\newcommand{\smat}[1]{\left(\begin{smallmatrix}#1\end{smallmatrix}\right)}

\title{Classification of Automorphism Groups of Smooth Cubic Threefolds and Fourfolds}

\author[J. Fu]{Jie Fu}
\address{Tsinghua University, China}
\email{fu-j21@mails.tsinghua.edu.cn}

\author[S. Wang]{Shihao Wang}
\address{Tsinghua University, China}
\email{wangshihao25@mails.tsinghua.edu.cn}

\author[Z. Zheng]{Zhiwei Zheng}
\address{Tsinghua University, China}
\email{zhengzhiwei@mail.tsinghua.edu.cn}

\date{}

\begin{document}
\raggedbottom
\allowdisplaybreaks[4]
\bibliographystyle{amsalpha}

\begin{abstract}
We classify the automorphism groups of smooth cubic threefolds and fourfolds. We also show that there are $156$ (respectively, $40$) connected families of smooth cubic fourfolds (respectively, threefolds) with specified automorphism group action. For each family, the group action and defining equations are also given. The classification combines both representation theory (based on GAP) and lattice theory (based on SageMath).
\end{abstract} 

\maketitle
\setcounter{tocdepth}{1}
\tableofcontents

\medskip
\noindent\emph{Notation.}
\begin{enumerate}
\item For subgroups $G_1$ and $G_2$ of a group $G$, we write
$G_1\leq G_2$ for inclusion. The centralizer and normalizer of $G_1$
in $G$ are denoted by $C_G(G_1)$ and $N_G(G_1)$.
\item For elements or subgroups of $G$, the relation $\sim_G$ denotes
conjugacy in $G$. For subgroups $G_1,G_2\leq G$, we write
$G_1\preceq_G G_2$ if a conjugate of $G_1$ is contained in $G_2$.
We write $G_1\prec_G G_2$ if this containment is proper. Abstract
isomorphism is denoted by $\cong$.
\item We write $G_1:G_2=G_1\rtimes G_2$ for a semidirect product, where $G_1$ is the normal subgroup. We write $G_1.G_2$ for a group fitting into a short exact sequence:
\[1\longrightarrow G_1\longrightarrow G_1.G_2\longrightarrow G_2\longrightarrow 1. \]
\item For a finite group $G$, write
$G_{\mathrm{ab}}=G/[G,G]$ and
$\operatorname{Out}(G)=\operatorname{Aut}(G)/\operatorname{Inn}(G)$.
\item We use $D_{2n}$ for the dihedral group of order $2n$, $Q_8$ for the
quaternion group of order $8$, $\mathrm{QD}_{16}$ for the quasidihedral
group of order $16$. We use $\mathrm{Dic}_3$ and $\mathrm{Hol}(5)$ for the groups
with GAP IDs $[12,1]$ and $[20,3]$, respectively.

\item The group of $n$-th roots of unity in $\CC^\times$ is denoted by $\mu_n$.  We occasionally write $n$ for a cyclic group of order $n$.
\item We write $L_2(q)=\PSL_2(\FF_q)$ and $3^{1+4}=3^{1+4}_+$ for the extraspecial group of order $3^5$ and exponent $3$.
\item Put $\zeta_n=\exp(2\pi\sqrt{-1}/n)$ and $\omega=\zeta_3$. For
$a=(a_i)_{i=1}^r\in\ZZ^r$, we sometimes put
$\frac1n(a)=\diag(\zeta_n^{a_i})_{i=1}^r$.
\item For $\sigma\in S_r$, we identify $\sigma$ with
the matrix whose $(i,\sigma(i))$-entry is $1$, and write it in disjoint
cycles; a monomial matrix is written $D\sigma$.
\item We write $\CC[x_1,\ldots,x_n]_d$ for the vector space of homogeneous
polynomials of degree $d$ in $x_1,\ldots,x_n$.
\item We regard $x=(x_1,\ldots,x_n)$ as a row vector and use the left action
on forms
$(A\cdot F)(x)=F(xA)$.  We write $F\sim_{\mathrm{lin}}F'$ if
$F'=A\cdot F$ for some $A\in\GL(n,\CC)$.
\item The natural projection $\GL(n,\CC)\to\PGL(n,\CC)$ is denoted by
$\pi$.
\end{enumerate}
\section{Introduction}

For a smooth cubic fourfold $X\subset\PP^5_{\CC}$, a basic problem is to
determine its automorphism group $\Aut(X)$.  Its symplectic subgroup
$\Auts(X)$ consists of the automorphisms acting trivially on
$H^{3,1}(X)$.  Laza--Zheng classified the possible symplectic automorphism
groups and their coinvariant lattices \cite{laza2022automorphisms}.
Koike described the corresponding projective representations and
invariant cubic forms, giving $48$ connected symplectic families
\cite{KOIKE202512,KOIKE2026}.  Yang--Yu--Zhu classified finite groups
admitting faithful actions on smooth cubic fourfolds
\cite{yang2024automorphism}. However, it happens that certain faithful group action on a smooth cubic fourfold is never realized as an automorphism group action, see \cite[Theorem~1.2(5)(b)]{laza2022automorphisms} for an example. It is still open to classify all possible full automorphism groups $\Aut(X)$ with their action on $\PP^5$. 

This paper continues \cite{FWZ26}, \cite{FZ26}.  The first paper studies the
non-symplectic index $[\Aut(X):\Auts(X)]$, proves general restrictions on
its possible values, and determines the automorphism groups when
the symplectic coinvariant lattice has rank $19$.  The second paper
gives a lattice-theoretic classification of automorphism groups and
their actions on the middle cohomology lattice when the symplectic
coinvariant rank is at least $15$.

Our aim is to complete the classification, including the remaining
cases of rank less than $15$, and to describe the corresponding
families explicitly. If
$F\in\CC[x_1,\ldots,x_n]_3$ defines a smooth cubic hypersurface
$X_F\subset\PP^{n-1}$, with $n=5$ or $6$, we put
\[
\Aut(F)=\{A\in\GL(n,\CC)\mid F(xA)=F(x)\}
\]
and call it the linear automorphism group of $F$. Every automorphism of $X_F$ is projectively induced
\cite{Matsumura1963OnTA}, and there is an exact sequence
\[
1\longrightarrow\ \omega I_n\longrightarrow\Aut(F)
\overset{\pi}{\longrightarrow}\Aut(X_F)\longrightarrow1.
\]
The group $\Aut(X_F)$ together with its projective representation may not determine the family; see \cite[\S 8 (8.6)]{KOIKE202512}. Thus in the proof of classification, we will concentrate on $\Aut(F)$ instead of $\Aut(X_F)$.

For a finite subgroup $H\leq\GL(n,\CC)$, consider the family of smooth
$H$-invariant cubic forms modulo linear equivalence.  We call $H$ and
its associated family \emph{saturated} if a general such form $F$ satisfies
$\Aut(F)=H$.

Two saturated families given by $H_1,H_2\leq\GL(n,\CC)$ are
called \emph{equivalent} if $A^{-1}H_1A=H_2$ for some
$A\in\GL(n,\CC)$. Equivalently, a linear change of coordinates
identifies their spaces of invariant cubic forms.

\begin{thm}[=Theorem~\ref{thm:fourfold-classification}]
\label{thm:fourfold}
There are exactly $156$ equivalence classes of saturated families of smooth cubic fourfolds. They are listed in Table~\ref{tab:fourfold}. For each family, more precise description of the group and defining equation are given in \S\ref{subsubsection: description for fourfolds}.
\end{thm}

For each family, we determine the groups
$\Auts(X_F)$, $\Aut(X_F)$ and $\Aut(F)$ for a general member $F$.
Explicit representatives or references to known ones are given in \S \ref{sec:tables}. We also determine the incidence relations among the families, and find $21$ minimal ones (which is maximal in the sense of group action). They are listed in Table~\ref{tab:fourfold-action-maximal}.

For $\rank(S)\geq15$, with $S$ the symplectic coinvariant lattice,
\cite{FZ26} classifies the possible groups and their
actions on middle cohomology. We combine this with Koike's symplectic
representations and invariant cubics \cite{KOIKE202512,KOIKE2026} to
construct the corresponding strict lifts and cubic families. The lattice calculation used OSCAR \cite{OSCAR} and is inspired by Brandhorst--Hofmann
\cite{Brandhorst_Hofmann_2023}.

For $\rank(S)<15$, the group classifications and index restrictions of
\cite{laza2022automorphisms}, \cite{yang2024automorphism}, \cite{FWZ26} restrict the
possible groups. In Algorithm \ref{alg:liftable-abelian-enumeration}, we use the classified diagonal actions for liftable
abelian groups \cite{zheng2021liftableabelian}. For non-abelian groups, we
enumerate central extensions and six-dimensional representations extending
the specified symplectic action. The prime-order classification of
\cite[Theorem~3.8]{gonzalez2011automorphisms} gives further restrictions; see Algorithm \ref{alg:small-nonabelian-enumeration}.

We then determine the invariant cubic spaces, which both contain smooth forms and are saturated; see Algorithms \ref{alg:small-nonabelian-smoothness} and \ref{alg:equal-dimension-saturation}. The family dimensions and conjugate containment of strict lifts then identify the
saturated families and their incidence relations, see Algorithm \ref{alg:cross-dimensional-containment}.  These finite-group
calculations use GAP \cite{GAP4}.

The classification for cubic fourfolds also gives a classification for automorphism groups of
cubic threefolds.  Wei--Yu classified the finite groups admitting
faithful actions on smooth cubic threefolds \cite[Theorem~1.1]{LiYuThreefold}.
Here we determine the actions of $\Aut(F)$ and their invariant families
using the cubic suspension
\[
F(x_1,\ldots,x_5)\longmapsto
\widehat F=F(x_1,\ldots,x_5)+x_6^3.
\]
Put $\delta=\diag(1,1,1,1,1,\omega)$.  Restriction to the fixed
hyperplane of $\delta$ gives
\[
C_{\Aut(\widehat F)}(\delta)/\langle\delta\rangle
\cong\Aut(F).
\]
The uniqueness of the maximal additive splitting
\cite{harrison1975higherdegreeforms} ensures that equivalent suspensions come from equivalent cubic threefolds. These arguments lead to Algorithm~\ref{alg:threefold-extraction}.

\begin{thm}\label{thm:threefold}
There are exactly $40$ equivalence classes of saturated families of
smooth cubic threefolds. They are listed in Table~\ref{tab:threefold}. For each family, more precise description of the group and defining equation are given in \S\ref{subsubsection: description of threefolds}.
\end{thm}

There are $7$ A-maximal cubic threefold families; see
Table~\ref{tab:threefold-action-maximal}.  Their automorphism groups are
listed, and their equations are displayed below or given in the cited
classifications.

\begin{rmk}\label{rmk:characteristic-zero}
The classifications in Theorems~\ref{thm:fourfold} and~\ref{thm:threefold},
including the specified automorphism group actions, remain valid over
every algebraically closed field of characteristic zero; see
\S~\ref{subsec:change-of-ground-field}.
\end{rmk}

The next two sections recall the preliminaries and liftability criteria.
We then treat abelian and non-abelian actions and complete the cubic fourfold
classification.  The cubic threefold classification follows by suspension;
the tables and representatives, with references to those already known,
appear at the end.

The complete calculations and results are archived in
\cite{FuWangZhengAutcub4fold2026}. There, \path{gap_classification/} records
the individual calculations, and \path{gap_result/} gives the final groups,
invariant bases, and containment relations in the order of the tables below.
A compact version is provided in the ancillary files in \path{anc/}.

\textbf{Acknowledgements:} We thank Xun Yu and Zigang Zhu for their interest in this work and helpful discussions. The initial mathematical exploration, core ideas and algorithms of this work were developed by the authors. LLM-based AI tools were used at later stages to assist with coding and manuscript polishing. We take full responsibility for the content of this paper and its supplementary files.

\section{Preliminaries}

\subsection{Automorphisms of smooth hypersurfaces}

Let $n,d\geq3$, and let $F\in\CC[x_1,\ldots,x_n]_d$ define the smooth
hypersurface
\[
X_F:=V(F)\subset\PP^{n-1}.
\]
We write $\Aut(X_F)$ for the projective automorphism group of $X_F$ and set
\[
\Aut(F)=\{A\in\GL(n,\CC)\mid A\cdot F=F\}.
\]
From \cite{Matsumura1963OnTA}, if $(n,d)\neq (3,3),(4,4)$, the natural
projection $\pi:\GL(n,\CC)\to\PGL(n,\CC)$ induces an exact sequence of finite groups
\begin{equation}\label{eq:linear-projective}
1\longrightarrow\langle\zeta_d I_n\rangle
\longrightarrow\Aut(F)\overset{\pi}{\longrightarrow}\Aut(X_F)
\longrightarrow 1.
\end{equation}
In particular, classifying $\Aut(F)$ up to conjugacy determines
$\Aut(X_F)\leq\PGL(n,\CC)$.

For any subgroup $G\leq\Aut(X_F)$, define its full inverse image in
$\Aut(F)$ by
\[
\widehat G_F:=\pi^{-1}(G)\cap\Aut(F).
\]
It fits into the exact sequence
\[
1\longrightarrow \langle\zeta_d I_n\rangle \longrightarrow \widehat G_F
\overset{\pi}{\longrightarrow}G
\longrightarrow 1.
\]
We call $\widehat G_F$ the \emph{strict lift} of $G$ with respect to $F$.
A subgroup $H\leq\Aut(F)$ is the strict lift of $\pi(H)$ if and only if
$\mu_dI_n\leq H$; in this case $H=\widehat{\pi(H)}_F$.

For a finite subgroup $H\leq\GL(n,\CC)$, put
\[
W_d(H)=\Sym^d((\CC^n)^*)^H,
\qquad
W_d(H)^{\mathrm{sm}}
=\{F\in W_d(H)\mid X_F\text{ is smooth}\}.
\]
Assume $W_d(H)^{\mathrm{sm}}\ne\varnothing$. This smooth locus is
irreducible, and smooth forms have finite linear automorphism groups.
Since $N_{\GL(n,\CC)}(H)/C_{\GL(n,\CC)}(H)$ is finite, the family of
smooth $H$-invariant forms modulo linear equivalence has dimension
\begin{equation}\label{eq:family-dimension}
m_d(H)=\dim W_d(H)-\dim C_{\GL(n,\CC)}(H)\geq 0.
\end{equation}
For cubic forms we write $W_H=W_3(H)$ and $m(H)=m_3(H)$.

\begin{defn}\label{def:linear-saturation}
The $H$-invariant family is \emph{saturated} if $\Aut(F)=H$ for a
general $F\in W_d(H)^{\mathrm{sm}}$; we then also call $H$ saturated.

Let $K$ be the automorphism group of a general $F\in W_d(H)^{\mathrm{sm}}$.
We call the $K$-invariant family $W_d(K)$ the saturation of the $H$-invariant family $W_d(H)$.
\end{defn}

These notions depend only on the $\GL(n,\CC)$-conjugacy class of $H$. Moreover, we have

\begin{prop}\label{prop:dimension-test}
Assume $W_d(H)^{\mathrm{sm}}\ne\varnothing$.  Then $H$ is saturated if and only if
\[
m_d(L)<m_d(H)
\]
for every finite subgroup $L\leq\GL(n,\CC)$ such that
$H\prec_{\GL(n,\CC)}L$ and
$W_d(L)^{\mathrm{sm}}\ne\varnothing$.
\end{prop}

\begin{proof}
For any such $L$, we may conjugate $H$ so that $H\subsetneq L$.
Then the image of the $L$-family under linear equivalence is contained
in that of the $H$-family, and hence $m_d(L)\leq m_d(H)$.

If $H$ is not saturated, \cite[Proposition~2.5]{FWZ26} gives such
an $L$ for which every smooth $H$-invariant form is linearly equivalent
to a smooth $L$-invariant form. The two family images therefore coincide,
so $m_d(L)=m_d(H)$.

Conversely, suppose $m_d(L)=m_d(H)$ for such an $L$.  The image of the
$L$-family is contained in that of the irreducible $H$-family and has the
same dimension, so it contains a general point of the latter.  Thus a
general smooth $H$-invariant form
is linearly equivalent to an $L$-invariant form. Its linear
automorphism group has order at least $|L|>|H|$, so $H$ is not
saturated.
\end{proof}

\subsection{Additive splittings}

Let $F\in \CC[x_1,\ldots,x_n]_d$, where $d\geq 3$.  An
\emph{additive splitting} of $F$ is an expression which, after a linear
change of coordinates, has the form
\begin{equation}\label{eq:additive-splitting}
    F=F_1+\cdots+F_s,
\end{equation}
where $\{1,\ldots,n\}=I_1\sqcup\cdots\sqcup I_s$
is a partition into nonempty subsets and
$$
    0\neq F_i\in \CC[x_j\mid j\in I_i]_d.
$$
Writing $e_1,\ldots,e_n$ for the coordinate basis, the corresponding coordinate subspaces are $V_i=\operatorname{Span}\{e_j\mid j\in I_i\}$, so that
$\CC^n=V_1\oplus\cdots\oplus V_s$. Thus distinct summands involve
disjoint blocks of coordinates. The splitting is
\emph{maximal} if none of the $F_i$ admits a nontrivial additive splitting.
A form admitting no nontrivial additive splitting is called
\emph{unpartitionable}.

Set
$$
    \Ann_1(F)
    :=
    \left\{
        (a_1,\ldots,a_n)\in\CC^n
        \ \middle|\
        \sum_{i=1}^n a_i\frac{\partial F}{\partial x_i}=0
    \right\}.
$$
The condition $\Ann_1(F)=0$ is equivalent to the linear independence of the
first partial derivatives of $F$.  Equivalently, no linear change of
coordinates makes $F$ independent of one of the variables.  Notice that
$\Ann_1(F)=0$ whenever the hypersurface $V(F)\subset\PP^{n-1}$ is smooth.

The decomposition theorem used below is the following uniqueness statement.

\begin{thm}\label{thm:unique-splitting}
Let $F\in\CC[x_1,\ldots,x_n]_d$, where $d\geq3$, and suppose that
$\Ann_1(F)=0$.  Its maximal additive splitting
$F=F_1+\cdots+F_s$, with coordinate spaces
$\CC^n=V_1\oplus\cdots\oplus V_s$, is unique up to permutation of the
pairs $(F_i,V_i)$. More precisely, if $F=F'_1+\cdots+F'_t$ is any additive splitting,
with coordinate spaces $\CC^n=V'_1\oplus\cdots\oplus V'_t$, then there is a
partition
$\{1,\ldots,s\}=I_1\sqcup\cdots\sqcup I_t$ such that
\[
V'_j=\bigoplus_{i\in I_j}V_i,\qquad
F'_j=\sum_{i\in I_j}F_i
\quad (1\leq j\leq t).
\]
\end{thm}

\begin{proof}
This follows from Harrison's decomposition theory
\cite{harrison1975higherdegreeforms}; see
\cite[Proposition~2.1(1)--(3)]{huang2022centres} for a formulation in terms
of polynomials.  The condition $\Ann_1(F)=0$ is precisely nondegeneracy in
these references. Additive splittings correspond to complete sets of
orthogonal idempotents in the finite-dimensional commutative centre algebra of $F$; see \cite[Equations~(3)--(5)]{huang2022centres}. Its primitive idempotents form a unique complete set.  Their images
give the coordinate subspaces of the maximal splitting.  Every complete set
of orthogonal idempotents is obtained by partitioning this primitive set and
summing within each part.  This proves the final assertion.
\end{proof}

Suppose that $F$ is smooth and that \eqref{eq:additive-splitting} is its
maximal additive splitting.  Then every $F_i$ defines a smooth hypersurface
in its own coordinate block.

In particular, every one-dimensional summand has the form $c x^d$ with
$c\neq0$,  which can be further changed to $x^d$ by
rescaling $x$.  Consequently, every form $F$ of degree $d\geq3$ defining a
smooth hypersurface is
linearly equivalent to
\begin{equation}\label{eq:fermat-splitting}
    x_1^d+\cdots+x_r^d+
    F_0(x_{r+1},\ldots,x_n),
\end{equation}
where the maximal additive splitting of $F_0$ has no one-variable summand.
The integer $r$ and the linear-equivalence class of $F_0$ are uniquely
determined by $F$.  We call $r$ the \emph{Fermat rank} of $F$.  When $r=n$,
we set $F_0=0$.  We call
$x_1^d+\cdots+x_r^d$ the \emph{Fermat part} of $F$, and call $F_0$
\emph{purely non-Fermat}.

The maximal splitting determines $\Aut(F)$.

\begin{prop}\label{prop:aut-splitting}
With the notation of \eqref{eq:fermat-splitting}, there is a natural
isomorphism
\begin{equation}\label{eq:aut-splitting}
    \Aut(F)
    \cong
    \bigl(\mu_d^r\rtimes S_r\bigr)\times\Aut(F_0),
\end{equation}
where $S_r$ acts on $\mu_d^r$ by permuting its factors, and
$\Aut(0)$ is understood to be trivial.
\end{prop}

\begin{proof}
By Theorem~\ref{thm:unique-splitting}, every element of $\Aut(F)$ permutes
the summands of the maximal additive splitting.  Since the dimension of the
corresponding coordinate block is preserved, the one-dimensional summands are
permuted among themselves. Hence both the coordinate space of the Fermat part
and the coordinate space of $F_0$ are invariant.

On the Fermat part, an automorphism is therefore represented by a monomial
matrix.  If its nonzero entries are $\lambda_1,\ldots,\lambda_r$, preservation
of $x_1^d+\cdots+x_r^d$ gives
$\lambda_i^d=1$ for every $i$. Thus its restriction belongs to
$\mu_d^r\rtimes S_r$.  Its restriction to the remaining variables belongs to
$\Aut(F_0)$. Conversely, the two factors in
\eqref{eq:aut-splitting} act independently and preserve $F$.
\end{proof}

\subsection{The global Torelli theorem for cubic fourfolds}

Let $X\subset\PP^5$ be a smooth cubic fourfold, let $h$ be the hyperplane class, and put
\[
\Lambda=H^4(X,\ZZ),\qquad \Lambda_0=(h^2)^\perp\subset\Lambda.
\]
Then
\[
\Lambda\cong I_{21,2},\qquad
\langle h^2\rangle\cong\langle3\rangle,\qquad
\Lambda_0\cong E_8^{\oplus2}\oplus U^{\oplus2}\oplus A_2.
\]
Let $\DD$ be one component of
\[
\PP\{\omega\in\Lambda_{0,\CC}\mid \langle \omega,\omega\rangle=0,
\ \langle\omega,\overline\omega\rangle<0\}.
\]
Let $\Gamma\leq\mathrm O(\Lambda_0)$ be the subgroup preserving $\DD$ and acting trivially on the discriminant group.  Denote by $\calH_2$ and $\calH_6$ the hyperplane arrangements orthogonal to long and short roots, respectively.

The following form of the global Torelli theorem combines the results of
Voisin \cite{voisin1986torelli,voisin2008erratum}, Hassett \cite{hassett2000special}, Looijenga \cite{looi2009period}, and Laza \cite{laza2010period}.

\begin{thm}\label{thm:torelli}
The period map identifies the moduli space $\calM$ of smooth cubic fourfolds with
\[
\Gamma\backslash\bigl(\DD\setminus(\calH_2\cup\calH_6)\bigr).
\]
Moreover, every Hodge isometry of $H^4(X,\ZZ)$ fixing $h^2$ is induced by a unique automorphism of $X$.
\end{thm}

\subsection{Families with symmetry}
\label{subsec:families-prescribed-symmetry}

Let $A\leq\GL(6,\CC)$ be finite, with $\mu_3I_6\leq A$, and let
$\lambda:A\to\CC^\times$ be a character satisfying
$\lambda(cI_6)=c^3$ whenever $cI_6\in A$. Put
\[
V_\lambda=\{F\in\Sym^3((\CC^6)^*)\mid a\cdot F=\lambda(a)F
\text{ for all }a\in A\}
\]
and
\[
N(A,\lambda)=\{g\in\GL(6,\CC)\mid gAg^{-1}=A,
\ \lambda(gag^{-1})=\lambda(a)\text{ for all }a\in A\}.
\]
Let $V_\lambda^{\mathrm{sm}}$ be the locus of forms defining smooth cubic
fourfolds, and assume it is nonempty. We define the family with
specified action $(A,\lambda)$ by the geometric quotient
\begin{equation}\label{eq:YZ-family}
\calF(A,\lambda)
=V_\lambda^{\mathrm{sm}}/N(A,\lambda)
\cong\PP(V_\lambda^{\mathrm{sm}})
\big/\bigl(N(A,\lambda)/(\CC^\times I_6)\bigr).
\end{equation}
Here $\CC^\times I_6$ acts trivially on the projective space.
When $\lambda=1$ (the trivial character), we write
\begin{equation}\label{eq:strict-family-moduli}
\calF_A:=\calF(A,1)
=W_3(A)^{\mathrm{sm}}/N_{\GL(6,\CC)}(A).
\end{equation}
We call the family $\calF_A$ \emph{saturated} if $\Aut(F)=A$
for a general $F\in W_3(A)^{\mathrm{sm}}$.
Let
\[
\mathcal Z_A=\{[X_F]\in\calM:F\in W_3(A)^{\mathrm{sm}}\}
\]
be its image in the moduli space of smooth cubic fourfolds, with the
reduced induced structure. Thus $\calF_A$ parametrizes cubic fourfolds
with the specified group action, whereas $\mathcal Z_A$ parametrizes
their underlying cubic fourfolds. The following result compares $\calF_A$ and $\mathcal Z_A$
\cite[Proposition~2.7]{yu2020moduli}.

\begin{thm}
Forgetting the specified group action gives a finite surjective morphism
\[
\calF_A\longrightarrow\mathcal Z_A,\qquad [F]\longmapsto[X_F].
\]
If $\calF_A$ is saturated, this morphism is generically injective and
identifies $\calF_A$ with the normalization of $\mathcal Z_A$.
\end{thm}

For $\calF_A$, put $G=\pi(A)=A/(\mu_3I_6)$, and let $\chi_{\mathrm{Hdg}}$ be the
character of $G$ on $H^{3,1}(X)$ for $X\in\calF_A$. If
$q=\dim\calF_A$, the $\chi_{\mathrm{Hdg}}$-eigenspace in the
primitive cohomology has signature $(q,2)$ when
$\chi_{\mathrm{Hdg}}=\overline{\chi}_{\mathrm{Hdg}}$, and Hermitian
signature $(q,1)$ otherwise. Let $\DD_A$ be the corresponding
$q$-dimensional type IV domain or complex ball, let
$\Gamma_A$ be its arithmetic monodromy group, and put
\[
\calH_s(A)=\DD_A\cap(\calH_2\cup\calH_6).
\]
We now recall the Torelli theorem with symmetry
\cite[Theorem~1.1]{yu2020moduli}.

\begin{thm}
\label{thm:equivariant-torelli}
The period map induces an algebraic isomorphism
\[
\calF_A\xrightarrow{\sim}
\Gamma_A\backslash\bigl(\DD_A\setminus\calH_s(A)\bigr).
\]
\end{thm}

For a smooth cubic fourfold $X=X_F$, let $\Omega$ be the standard
projective volume form.  Griffiths residue identifies $H^{3,1}(X)$
with $\CC\operatorname{Res}_{X}(\Omega/F^2)$.  Since
$A^*\Omega=\det(A)\Omega$ and $A\cdot F=F$ for $A\in\Aut(F)$, its
action on this line is given by $\det(A)$.  Thus
\begin{equation}\label{eq:symplectic-diagram}
\Auts(F)=\Aut(F)\cap\SL(6,\CC),
\qquad
\Aut(X)/\Auts(X)\cong\mu_{m(X)},
\end{equation}
where $\Auts(F)$ is the strict symplectic lift of $\Auts(X)$ and
$m(X)=[\Aut(X):\Auts(X)]$.

For $G_s=\Auts(X)$, put
$T=\Lambda_0^{G_s}$ and $S=T^\perp\subset\Lambda_0$.
Here $S$ and $T$ are the coinvariant and invariant lattices used in
the lattice calculations. The lattice $T$ contains the transcendental
lattice $T(X)$ of $X$, but need not equal it. The corresponding
connected symplectic family has dimension $20-\rank(S)$
\cite{yu2020moduli}\cite{laza2022automorphisms}. The generic value of
$m(X)$ on such a family is called its \emph{generic index}; for a
positive-dimensional family it is $1$ or $2$ \cite{FWZ26}\cite{FZ26}.

\section{Extensions and liftability}
\label{sec:extension-liftability}

\subsection{Central extensions and liftings}

We recall the cohomological description of central extensions. Let $A$
be an abelian group with trivial $G$-action. A central
extension
\[
1\longrightarrow A\longrightarrow E\longrightarrow G\longrightarrow1
\]
determines a class $[E]\in H^2(G,A)$.  Indeed, if $s\colon G\to E$ is a
set-theoretic section with $s(1)=1$, then
$s(g)s(h)=c(g,h)s(gh)$ for a $2$-cocycle $c\colon G^2\to A$.
Changing the section changes $c$ by a coboundary.  The extension splits if
and only if $[E]=0$.  If $K\leq G$, the restriction of $[E]$ to $H^2(K,A)$ is
the class of the inverse image of $K$ in $E$.

If $K\leq G$ has finite index and $M$ is a $G$-module, restriction and
corestriction satisfy
\[
\operatorname{cor}_K^G\operatorname{res}_K^G(\alpha)
=[G:K]\alpha,
\qquad \alpha\in H^i(G,M).
\]
For these facts and the relevant conventions for $H^2(G,\CC^\times)$,
see \cite[\S2.2]{xiezhenglift}.

A finite subgroup $G\leq\PGL(n,\CC)$ is liftable if there exists a
subgroup $H\leq\GL(n,\CC)$ such that
$\pi|_H\colon H\to G$ is an isomorphism.  In this
case, $H$ is called a \emph{lifting} of $G$.  Choosing a matrix
$A_g$ above every $g\in G$ gives scalars $c(g,h)\in\CC^\times$ such that
$A_gA_h=c(g,h)A_{gh}$.  Their cohomology class
\[
o_G\in H^2(G,\CC^\times)
\]
is independent of the choices, and $G$ is liftable if and only if
$o_G=0$; see \cite[\S3]{xiezhenglift}.

Let $F\in\CC[x_1,\ldots,x_n]_d$ define a smooth hypersurface
$X_F\subset\PP^{n-1}$, and let $G\leq\Aut(X_F)$.  A lifting $H$ is an
\emph{$F$-lifting} if
$H\leq\Aut(F)$.  If such a lifting exists, we say that $G$ is
$F$-liftable.  Put $Z_d=\langle\zeta_dI_n\rangle$.  The invariant
extension from \eqref{eq:linear-projective},
\[
1\longrightarrow Z_d\longrightarrow\widehat G_F
\longrightarrow G\longrightarrow1,
\]
defines a class $e_F(G)\in H^2(G,Z_d)$.  Thus $G$ is $F$-liftable if and
only if this extension splits.  The inclusion $Z_d\hookrightarrow\CC^\times$ sends
$e_F(G)$ to $o_G$.  In particular, $F$-liftability implies liftability.

If $\rho_1,\rho_2\colon G\to\GL(n,\CC)$ are lifting homomorphisms, then
$\rho_2(g)=\chi(g)\rho_1(g)$ for a unique
$\chi\in\Hom(G,\CC^\times)$.  Conversely, twisting a lifting by such a
character gives another lifting. Two $F$-liftings differ by a unique character in $\Hom(G,\mu_d)$. If $\rho$ is a lifting
and $\rho(g)\cdot F=\lambda(g)F$, then $\lambda$ is a character.  An
$F$-lifting exists precisely when $\lambda$ lies in the image of the
$d$-th power map on $\Hom(G,\CC^\times)$.

The next lemma tests whether a character of the scalar subgroup extends
to the whole group.

\begin{lem}\label{lem:central-character-extension}
Let $Z$ be a central subgroup of a finite group $H$, and let
$\lambda_Z\colon Z\to\CC^\times$ be a character.  Then $\lambda_Z$ extends
to a linear character of $H$ if and only if it is trivial on
$Z\cap[H,H]$.
\end{lem}

\begin{proof}
Every linear character of $H$ is trivial on $[H,H]$, which proves the
necessity.  Conversely, suppose that $\lambda_Z$ is trivial on
$Z\cap[H,H]$.  It then defines a character of the subgroup
$Z[H,H]/[H,H]$ of $H_{\mathrm{ab}}$.  A character of a subgroup of a finite
abelian group extends to the whole group because $\CC^\times$ is divisible.
Pulling this extension back to $H$ proves the assertion.
\end{proof}

The central extension gives the following liftability criterion.

\begin{prop}\label{prop:stem-liftability}
Let
\[
1\longrightarrow Z\longrightarrow H\longrightarrow G\longrightarrow1
\]
be a central extension, and let $\rho:H\to\GL(V)$ be a faithful representation such that
$\rho(Z)=\rho(H)\cap(\CC^\times I_V)$.

The induced projective representation of $G$ is liftable if and only if
$Z\cap[H,H]=1$.  In particular, if $Z\cong C_3$, then it is non-liftable if
and only if $Z\leq[H,H]$.

With the trivial $G$-action on $C_3$, the universal coefficient
sequence is
\[
0\longrightarrow\operatorname{Ext}^1(G_{\mathrm{ab}},C_3)
\longrightarrow H^2(G,C_3)
\longrightarrow\operatorname{Hom}(M(G),C_3)
\longrightarrow0,
\]
where $M(G)$ is the Schur multiplier.  The liftable classes are precisely
the embedded $\operatorname{Ext}^1$ part.  If the extension class has nonzero image in $\Hom(M(G),C_3)$,
then $Z\le [H,H]$.
\end{prop}

\begin{proof}
Let $\lambda_Z$ be the scalar character of $Z$.  If
$Z\cap[H,H]=1$, Lemma~\ref{lem:central-character-extension} extends it to a
character $\lambda$ of $H$.  The representation
$h\mapsto\lambda(h)^{-1}\rho(h)$ is trivial on $Z$ and therefore descends to
a linear representation of $G$ with the specialized projectivization.
Conversely, comparison with any linear lift of the projective representation
gives a character of $H$ extending $\lambda_Z$.  Lemma~\ref{lem:central-character-extension} again forces
$Z\cap[H,H]=1$.  Since $Z$ has order three in the cubic case, a nontrivial
intersection is all of $Z$.  As $Z$ is already central, this is exactly the
stem condition for the distinguished kernel.

For the last assertion, the five-term homology sequence of the extension is
\[
M(H)\longrightarrow M(G)\xrightarrow{\tau} Z
\longrightarrow H_{\mathrm{ab}}\longrightarrow G_{\mathrm{ab}}
\longrightarrow0,
\]
where $\tau$ is the connecting map.  Exactness at $Z$ gives
\[
\operatorname{im}(\tau)
=\ker(Z\to H_{\mathrm{ab}})
=Z\cap[H,H].
\]
Under the last map in the
universal coefficient sequence, the extension class maps to $\tau$.
The kernel of the map
$H^2(G,C_3)\to\Hom(M(G),C_3)$
is the image of $\operatorname{Ext}^1(G_{\mathrm{ab}},C_3)$. If $Z\cong C_3$,
every nonzero $\tau$ is surjective.  The result follows.
\end{proof}

Thus the zero class corresponds to $F$-liftability of the invariant
extension, the whole $\operatorname{Ext}^1$ term corresponds to liftability,
and the classes with nonzero multiplier component are non-liftable.  In
particular, a nonzero $\operatorname{Ext}^1$ class may be liftable without
being $F$-liftable.

For an invariant cubic extension with $Z=\langle\omega I_n\rangle$,
$F$-liftability holds exactly when the image of $Z$ in
$H_{\mathrm{ab}}/3H_{\mathrm{ab}}$ is nontrivial.  Indeed, a nonzero image
admits a character $H\to\mu_3$ whose restriction to $Z$ is an isomorphism;
its kernel is a complement to $Z$.  Conversely, a complement gives such
a character since $Z$ is central.  Equivalently,
$Z\cap([H,H]H^3)=1$, where $H^3$ is the subgroup generated by all cubes.

The small central extensions needed below can be read directly from the
groups of orders $9$ and $27$.

\begin{lem}\label{lem:small-central-extensions}
Let
\[
1\longrightarrow Z\longrightarrow E\longrightarrow Q\longrightarrow1
\]
be a central extension with $Z\cong C_3$.
\begin{enumerate}
    \item If $Q\cong C_3$, then the extension splits exactly when
    $E\cong C_3^2$; otherwise $E\cong C_9$.

    \item If $Q\cong C_9$, then the extension splits exactly when
    $E\cong C_9\times C_3$; otherwise $E\cong C_{27}$.

    \item If $Q\cong C_3^2$, then the extension splits exactly when
    $E\cong C_3^3$.  The remaining abelian case is
    $E\cong C_9\times C_3$, and the two non-abelian cases have exponent
    $3$ and $9$, respectively.
\end{enumerate}
Moreover, when $E$ is the invariant extension of a projective
$C_3^2$-action, that action is liftable if and only if $E$ is abelian.
\end{lem}

\begin{proof}
Since $|E|=3|Q|$, the assertions follow from the groups of orders $9$ and
$27$.  If $Q$ is cyclic, then $E$ is abelian, giving $C_3^2,C_9$ for
$Q\cong C_3$ and $C_9\times C_3,C_{27}$ for $Q\cong C_9$; in each pair only
the first extension splits. For $Q\cong C_3^2$, the condition $E/Z\cong C_3^2$
excludes $C_{27}$, gives a split extension for $C_3^3$, and forces $Z$ to be
the subgroup of cubes in $C_9\times C_3$; the two non-abelian groups, of
exponents $3$ and $9$, cannot split.  Finally, $[E,E]\leq Z$, so
Proposition~\ref{prop:stem-liftability} gives liftability exactly when
$[E,E]=1$, equivalently when $E$ is abelian.
\end{proof}

\subsection{Liftability criteria}

We first collect three elementary consequences of the preceding results.

\begin{prop}\label{prop:basic-liftability}
Let $F\in\CC[x_1,\ldots,x_n]_d$ define a smooth hypersurface
$X_F\subset\PP^{n-1}$, and let $G\leq\Aut(X_F)$ be finite.
\begin{enumerate}
    \item Every element $g\in G$ of finite order $m$ has a lifting in
    $\GL(n,\CC)$ of the same order.

    \item If $\gcd(d,m)=1$, then $g$ has a unique $F$-lifting.

    \item If $\gcd(d,|G|)=1$, then $G$ has a unique $F$-lifting.
\end{enumerate}
\end{prop}

\begin{proof}
The first two assertions are \cite[Lemma~4.1 and
Proposition~4.7]{xiezhenglift}.  Existence in the last assertion follows
from \cite[Theorem~1.3]{xiezhenglift}.  It is unique because
$\Hom(G,\mu_d)=1$ when $\gcd(d,|G|)=1$.
\end{proof}

The Sylow criterion of \cite[Theorem~1.1(2) and Theorem~1.3]{xiezhenglift} gives the following
reduction.

\begin{thm}
\label{thm:sylow-liftability}
Let $F\in\CC[x_1,\ldots,x_n]_d$ define a smooth hypersurface
$X_F\subset\PP^{n-1}$, and let $G\leq\Aut(X_F)$ be finite.
\begin{enumerate}
    \item The group $G$ is liftable if and only if, for every prime
    $p\mid\gcd(|G|,d,n)$, a Sylow $p$-subgroup of $G$ is liftable.

    \item The group $G$ is $F$-liftable if and only if, for every prime
    $p\mid\gcd(|G|,d,n)$, a Sylow $p$-subgroup of $G$ is $F$-liftable.
\end{enumerate}
\end{thm}

\begin{cor}\label{cor:coprime-liftability}
Let $F\in\CC[x_1,\ldots,x_n]_d$ define a smooth hypersurface
$X_F\subset\PP^{n-1}$.  If $\gcd(d,n)=1$, then every finite subgroup
$G\leq\Aut(X_F)$ is $F$-liftable.
\end{cor}

For cubic threefolds and fourfolds these criteria simplify. Part~(3) of
the following proposition concerns $\Aut(X_F)$ itself, not an arbitrary
subgroup.

\begin{prop}\label{prop:cubic-liftability}
Let $F$ be a cubic form defining a smooth hypersurface $X_F$.
\begin{enumerate}
    \item If $F$ has five variables, then every finite subgroup of
    $\Aut(X_F)$ is $F$-liftable.

    \item If $F$ has six variables and $G\leq\Aut(X_F)$ is finite, then
    $G$ is liftable (respectively, $F$-liftable) if and only if a Sylow
    $3$-subgroup of $G$ is liftable (respectively, $F$-liftable).

    \item If $F$ has six variables, then $\Aut(X_F)$ is liftable if and
    only if every subgroup isomorphic to $C_3^2$ is liftable.  Moreover,
    $\Aut(X_F)$ is $F$-liftable if and only if every subgroup isomorphic
    to $C_3^2$ is liftable and every element of order $3$ is
    $F$-liftable.
\end{enumerate}
\end{prop}

\begin{proof}
The first assertion follows from
$\gcd(3,5)=1$.  For the second, the only prime dividing
$\gcd(3,6)$ is $3$, so it follows from
Theorem~\ref{thm:sylow-liftability}.  The last assertion is
\cite[Theorem~1.4]{xiezhengfull}.
\end{proof}

\begin{cor}\label{cor:small-preimage-liftability}
Let $F\in\CC[x_1,\ldots,x_6]_3$ define a smooth cubic fourfold $X_F$, put
$G_0=\Aut(X_F)$, and let
$Z=\langle\omega I_6\rangle\leq\Aut(F)$.
\begin{enumerate}
    \item The group $G_0$ is liftable if and only if, for every subgroup
    $G\leq G_0$ with $G\cong C_3^2$, the strict lift $\widehat G_F$ is abelian.
    Equivalently,
    \[
        \widehat G_F\cong C_3^3
        \quad\text{or}\quad
        C_9\times C_3.
    \]

    \item The group $G_0$ is $F$-liftable if and only if
    \[
        \widehat G_F\cong G\times C_3
    \]
    for every subgroup $G\leq G_0$ of order $3$ or $9$.  Explicitly, the
    required strict lifts are
    \[
\begin{array}{c|ccc}
G & C_3 & C_9 & C_3^2 \\ \hline
\widehat G_F & C_3^2 & C_9\times C_3 & C_3^3
\end{array}
\]
\end{enumerate}
\end{cor}

\begin{proof}
By Proposition~\ref{prop:cubic-liftability}, the group $G_0$ is liftable
exactly when each of its subgroups isomorphic to $C_3^2$ is liftable.
Lemma~\ref{lem:small-central-extensions} says that this holds exactly when
the corresponding strict lift is abelian, which proves the first
assertion.

For any $G\leq G_0$, $F$-liftability is equivalent to the splitting of
\[
1\longrightarrow Z\longrightarrow\widehat G_F
\longrightarrow G\longrightarrow1.
\]
If $G_0$ is $F$-liftable, these extensions split for all its subgroups.
Conversely, splitting for the subgroups isomorphic to $C_3$ and $C_3^2$
gives the two conditions in Proposition~\ref{prop:cubic-liftability}(3).
The three splitting types above follow from
Lemma~\ref{lem:small-central-extensions}; the $C_9$ condition is included
to give a uniform statement for all subgroups of order $3$ or $9$.
\end{proof}

\subsection{Subgroups of order three or nine}

We now describe the inverse images of projective subgroups of order three
or nine. The following smoothness observation will also be used later.

\begin{lem}\label{lem:square-monomial}
Let $F\in\CC[x_1,\ldots,x_6]_3$ define a smooth cubic fourfold $X_F$.
Then, for every
$i\in\{1,\ldots,6\}$, there exists $j\in\{1,\ldots,6\}$
such that the monomial $x_i^2x_j$
occurs in $F$ with nonzero coefficient.
\end{lem}
\begin{proof}
This is the case $d=3$, $a=1$, and $b=0$ of
\cite[Lemma~4.1]{Zheng2020OnAA}.  Equivalently, it follows
from \cite[Lemma~2.3]{gonzalez2011automorphisms}; see also
\cite[Lemma~1]{mayanskiy2013abelianautomorphismgroupscubic}.
\end{proof}

The next lemma gives matrix representatives for the obstructions in
Proposition~\ref{prop:cubic-liftability}(3).

\begin{lem}\label{lem:small-3-groups-lifting}
Let $X=X_F\subset\PP^5$ be a smooth cubic fourfold,
and let $G\leq\Aut(X)$.
\begin{enumerate}

    \item If $G\cong C_3$, then $G$ is liftable.  Moreover, $G$ is $F$-liftable
    if and only if its strict lift $\widehat G_F$ is not conjugate in
    $\GL(6,\CC)$ to
    \[
    \left\langle
    \frac19(1,1,4,4,7,7)
    \right\rangle.
    \]

    \item If $G\cong C_9$, then $G$ is $F$-liftable.

    \item If $G\cong C_3^2$ is not liftable, then
    \[
    \widehat G_F
    \sim_{\GL(6,\CC)}
    \left\langle
    (1\;2\;3)(4\;5\;6),
    \frac13(0,1,2,0,1,2)
    \right\rangle
    \]
    or
    \[
    \widehat G_F
    \sim_{\GL(6,\CC)}
    \left\langle
    (1\;2\;3)(4\;5\;6),
    \frac19(1,4,7,1,4,7)
    \right\rangle.
    \]

    \item If $G\cong C_3^2$ is liftable but not $F$-liftable, then
    \[
    \widehat G_F\sim_{\GL(6,\CC)}
    \left\langle
    \frac19(1,1,4,4,7,7),
    \frac13(0,1,0,1,0,1)
    \right\rangle.
    \]
\end{enumerate}
\end{lem}

\begin{proof}

\begin{enumerate}

\item This follows directly from
\cite[Theorem~3.8]{gonzalez2011automorphisms}.  The only action there that
is not $F$-liftable is the family denoted by $\calF_3^7$ in that theorem.

\item Write $G=\langle g\rangle\cong C_9$.  By
Proposition~\ref{prop:basic-liftability}(1), the group $G$ is liftable.

Suppose that $G$ is not $F$-liftable.  Choose
$B\in\widehat G_F$ whose image is $g$.  The order of $B$
cannot be $9$, since otherwise $B$ would generate an $F$-lifting of $G$.
Therefore $B$ has order $27$, and
$B^9=\omega^\varepsilon I_6$ for some $\varepsilon\in\{1,2\}$.  We may
assume that $B$ is diagonal and write
$B=\frac{1}{27}(a_1,\ldots,a_6)$.  Since
$B^9=\omega^\varepsilon I_6$, we have
$a_i\equiv\varepsilon\pmod 3$ for every $i$.

By Lemma~\ref{lem:square-monomial}, for every $i$ there
exists an index $j(i)$ such that $x_i^2x_{j(i)}$ occurs
in $F$.  Since $B\cdot F=F$, we have
$2a_i+a_{j(i)}\equiv0\pmod{27}$.  Iterating the map
$j\colon\{1,\ldots,6\}\to\{1,\ldots,6\}$ produces a directed cycle of
length $1\leq\ell\leq6$.  If $j^{(\ell)}(i)=i$, then
$a_i\equiv(-2)^\ell a_i\pmod{27}$.  Since $3\nmid a_i$, it follows that
$(-2)^\ell\equiv1\pmod{27}$.  This is impossible for
$1\leq\ell\leq6$, since the six residues are
$25,4,19,16,22,$ and $10$.

\item
Since $G$ is not liftable, it is not diagonalizable.
By the classification in
\cite[Corollary~2, Theorems~3 and~4]
{mayanskiy2013abelianautomorphismgroupscubic},
and in particular by the $C_3^2$-rows of Table~3 therein,
after changing generators of $G$ and conjugating in
$\GL(6,\CC)$, we may assume that $G$ is the projective
image of the group $E\leq\GL(6,\CC)$ generated by
\[
P=(1\;2\;3)(4\;5\;6),
\qquad
W=\frac13(0,1,2,0,1,2).
\]
In particular, we have $P^3=W^3=I_6$,
$[P,W]=\omega I_6$, $\operatorname{ord}(E)=27$, and $\pi(E)=G$.

Denote by $\chi_E:E\to \CC^\times$ the character of the $E$-action on
$F$.  Two inverse images in $E$ of a fixed element of $G$ differ by a power
of $\omega I_6$, which acts trivially on cubic forms.  Thus $\chi_E$
factors through $E\twoheadrightarrow G$.  Denote the induced character on
$G$ by $\chi$.

If $\chi$ is trivial, then
$E\leq\widehat G_F$.  Since both $E$ and
$\widehat G_F$ have order $27$, they are equal, and the
conclusion follows.

Suppose that $\chi$ is nontrivial.  Then
$\chi(G)=\chi_E(E)=\langle\omega\rangle$.  Therefore there is a pair
$(u,v)\in\{0,1,2\}^2\setminus\{(0,0)\}$ such that
$\zeta_9^uP,\zeta_9^vW\in\widehat G_F$.

One checks directly that, for any
$(u,v)\in\{0,1,2\}^2\setminus\{(0,0)\}$,
\[
\langle \zeta_9^u P,\zeta_9^v W\rangle\sim_{\GL(6,\CC)} \left\langle
(1\;2\;3)(4\;5\;6),
\frac19(1,4,7,1,4,7)
\right\rangle.
\]
Since both $\langle \zeta_9^u P,\zeta_9^v W\rangle$ and
$\widehat G_F$ have order $27$, the conclusion follows.
\item
Let $\rho\colon G\hookrightarrow\GL(6,\CC)$ be a lift, and let
$\chi\colon G\to\CC^\times$ be the corresponding character of the action
on $F$.  Since $G\cong C_3^2$ is not $F$-liftable, we have
$\chi(G)=\langle\omega\rangle$.

Choose generators $a,b$ of $G$ such that
\[
\langle a\rangle=\ker\chi,
\qquad
\chi(b)=\omega.
\]
Put $A=\rho(a)$ and $B=\rho(b)$.  Then
\[
A^3=B^3=I_6,\quad AB=BA,\quad A\cdot F=F
\text{ and}\quad B\cdot F=\omega F.
\]

By \cite[Theorem~3.8]{gonzalez2011automorphisms}, $B$ is conjugate in
$\GL(6,\CC)$ to the action associated with the family $\calF_3^7$.  Since
$A$ and $B$ can be diagonalized simultaneously, we may assume
$B=\frac{1}{3}(0,0,1,1,2,2)$ and write
$A=\frac{1}{3}(\alpha_1,\ldots,\alpha_6)$.  Applying
\cite[Theorem~3.8]{gonzalez2011automorphisms} to $AB$ and $AB^2$, we find
that the two multisets
\[
\begin{gathered}
\{\alpha_1,\alpha_2,\alpha_3+1,\alpha_4+1,\alpha_5+2,\alpha_6+2\},\\
\{\alpha_1,\alpha_2,\alpha_3+2,\alpha_4+2,\alpha_5+1,\alpha_6+1\}
\end{gathered}
\]
are both equal to $\{0,0,1,1,2,2\}$ in $\ZZ/3\ZZ$.  A direct calculation shows that,
after applying a permutation of the coordinates that fixes $B$, and
replacing $A$ by $A^{-1}$ or $\omega A$ if necessary, we may assume
\[
A=\frac13(0,1,0,1,0,1).
\]
Therefore $\left\langle
\frac19(1,1,4,4,7,7),
\frac13(0,1,0,1,0,1)
\right\rangle\sim_{\GL(6,\CC)}\langle A,\zeta_9^{-1} B\rangle
\leq \widehat G_F$.

Since both $\left\langle
\frac19(1,1,4,4,7,7),
\frac13(0,1,0,1,0,1)
\right\rangle$ and $\widehat G_F$ have order $27$, the conclusion follows.

\end{enumerate}

\end{proof}

\section{Abelian cases}

\subsection{Liftable abelian case}
\label{subsec:liftable-abelian}

\begin{lem}\label{lem:abelian-liftable-diagonal}
Let $G\leq\PGL(n,\CC)$ be a finite abelian group and $D<\PGL(n,\CC)$ be the group consisting of projective diagonal matrices. Then $G$ is liftable if and
only if it is conjugate in $\PGL(n,\CC)$ to a subgroup of $D$.
Equivalently, $G$ admits a lifting which is simultaneously diagonalizable.
\end{lem}

\begin{proof}
Suppose that $G$ is liftable and let $H\leq\GL(n,\CC)$ be a
lifting.  Since $H\cong G$ is finite abelian, its elements are
commuting finite-order matrices.  They are therefore semisimple and can be
simultaneously diagonalized.

Conversely, suppose that $G$ is contained in $D$. Every projective diagonal matrix has a unique representative whose
last diagonal entry is $1$.  Taking these representatives defines a group
homomorphism $G\to\GL(n,\CC)$ lifting the given projective action.
\end{proof}

After a change of coordinates, every liftable abelian projective group is
represented by diagonal matrices; a lift may act on $F$ through a
character rather than fix it.

Let $A$ be a finite abelian group.  A datum
\[
I=(A,\lambda,\lambda_1,\ldots,\lambda_6)
\]
consists of characters $\lambda,\lambda_i:A\to\CC^\times$ such that the
projective representation induced by
\[
\rho_I(a)=\diag(\lambda_1(a),\ldots,\lambda_6(a))
\]
is faithful.  Its semi-invariant cubic space is
$$
W_I=\left\langle
x_1^{e_1}\cdots x_6^{e_6}\ \middle|\
\sum e_i=3,\quad \prod_i\lambda_i^{e_i}=\lambda
\right\rangle.
$$
Thus $\rho_I(a)\cdot F=\lambda(a)F$ for $F\in W_I$, and in general
$\rho_I(A)$ does not preserve $F$ exactly.

Two data $I=(A,\lambda,\lambda_1,\ldots,\lambda_6)$ and
$I'=(A,\lambda',\lambda'_1,\ldots,\lambda'_6)$ are \emph{equivalent} if
there is a character $\xi:A\to\CC^\times$ such that
\[
\lambda'_i=\lambda_i\xi\quad(1\leq i\leq6),
\qquad \lambda'=\lambda\xi^3.
\]
For $B\leq A$, let
\[
I|_B=(B,\lambda|_B,\lambda_1|_B,\ldots,\lambda_6|_B).
\]
We write $[J]\leq[I]$ if $J$ is equivalent to $I|_B$ for some subgroup
$B\leq A$.  Let $\mathcal P_{3,6}$ be the poset of equivalence classes
$[I]$ for which $W_I$ contains a smooth cubic, with this order
\cite[\S2.1]{zheng2021liftableabelian}.

Define the \emph{invariant extension} of $I$ by
$$
\Gamma(I)=
\left\{
c\rho_I(a)\ \middle|\
a\in A,\ c\in\CC^\times,\ c^3\lambda(a)=1
\right\}.
$$
It is a finite diagonal abelian group containing
$Z:=\langle\omega I_6\rangle$, and it fits into an exact sequence
$$
1\longrightarrow Z\longrightarrow\Gamma(I)
\longrightarrow A\longrightarrow1.
$$
This extension need not split, even though the projective $A$-action is
liftable.

\begin{prop}\label{prop:character-to-invariant}
The construction $I\mapsto\Gamma(I)$ has the following properties.
\begin{enumerate}
    \item The space
    $W_3(\Gamma(I))=\Sym^3((\CC^6)^*)^{\Gamma(I)}$ of
    $\Gamma(I)$-invariant cubic forms is equal to $W_I$.

    \item If $I$ and $I'$ are equivalent, then
    $\Gamma(I)=\Gamma(I')$.

    \item If $J$ is obtained by restricting $I$ to a subgroup, then
    $\Gamma(J)\leq\Gamma(I)$.  Conversely, an inclusion
    $\Gamma(I)\leq\Gamma(I')$ induces, after passing to the projective
    quotients, an inequality $[I]\leq[I']$ in
    $\mathcal P_{3,6}$.

    \item The class $[I]$ is maximal in $\mathcal P_{3,6}$ if and only if
    $\Gamma(I)$ is maximal among the invariant extensions coming from
    liftable diagonal abelian actions admitting a smooth invariant cubic.
\end{enumerate}
\end{prop}

\begin{proof}
An element $c\rho_I(a)\in\Gamma(I)$ acts on a cubic monomial $m=x_1^{e_1}\cdots x_6^{e_6}$ by the scalar $c^3\prod_i\lambda_i(a)^{e_i}$.
Since $c^3=\lambda(a)^{-1}$, the monomial is fixed by $\Gamma(I)$ precisely when
$\prod_i\lambda_i^{e_i}=\lambda$.  This proves the first assertion.

If $I'$ is obtained by twisting the lifting by a character
$\xi\colon A\to\CC^\times$, then
$\rho_{I'}=\xi\rho_I$ and $\lambda'=\lambda\xi^3$.  Replacing $c$ by
$c\xi(a)$ in the definition of $\Gamma(I')$ shows that
$\Gamma(I')=\Gamma(I)$.

If $B\leq A$, then $\Gamma(I|_B)\leq\Gamma(I)$ by definition.  Conversely,
an inclusion of invariant extensions gives an inclusion of their
projective quotients.  The two liftings of the smaller projective group
differ by a character, so the corresponding data are equivalent after
restriction.  The maximality statement follows.
\end{proof}

If $G\leq\Aut(X_F)$ is liftable and abelian, choose a diagonal lift $\rho$
and define $\lambda$ by $\rho(g)\cdot F=\lambda(g)F$.  For the corresponding
datum $I$, one has $\widehat G_F=\Gamma(I)$.

A simple cubic is a sum, in disjoint sets of variables, of blocks of the
following types:
\begin{align*}
K_r&=x_1^2x_2+\cdots+x_{r-1}^2x_r+x_r^2x_1,\\
T_r&=x_1^2x_2+\cdots+x_{r-1}^2x_r+x_r^3,\\
Y_{a,b}&=x_1^2x_2+\cdots+x_{a-1}^2x_a
+y_1^2y_2+\cdots+y_{b-1}^2y_b\\
&\hspace{2.3cm}+x_a^2z+y_b^2z+z^2w+w^3+x_ay_bw,
\end{align*}
where $a\geq b\geq1$.

A datum $I$ is called \emph{simple} if $W_I$ contains a simple cubic.  By
\cite[Theorem~4.2(1)--(8)]{zheng2021liftableabelian}, every maximal class
$[I]\in\mathcal P_{3,6}$ is either simple or one of the eight maximal non-simple
classes listed there.

\begin{defn}\label{def:maximal-liftable-abelian-sets}
For a simple cubic $F$ in six variables, put
\[
M(F)=\{D\in(\CC^\times)^6\mid D\cdot F=F\}.
\]
Let $\mathfrak M_0$ be a set of representatives, under conjugacy by
coordinate permutations, of all such $M(F)$ and of the invariant extensions
associated with the eight maximal non-simple cases above.  For
$M,M'\in\mathfrak M_0$, write $M\preceq_{\mathrm{perm}}M'$ if and only if
$P^{-1}MP\leq M'$ for some permutation matrix $P$.
Let $\mathfrak M$ be the set of maximal elements of
$(\mathfrak M_0,\preceq_{\mathrm{perm}})$.
\end{defn}

The set $\mathfrak M$ consists of thirty-two groups: twenty-four of simple
type and eight non-simple.

\begin{prop}\label{prop:maximal-liftable-abelian-groups}
The groups in $\mathfrak M$ are precisely the maximal invariant extensions
arising from liftable abelian actions on smooth cubic fourfolds.  In
particular:
\begin{enumerate}
    \item every $M\in\mathfrak M$ is maximal among such invariant extensions;
    \item for every liftable abelian subgroup
    $G\leq\Aut(X_F)$, the group $\widehat G_F$ satisfies
    $\widehat G_F\preceq_{\GL(6,\CC)}M$ for some $M\in\mathfrak M$.
\end{enumerate}
\end{prop}

\begin{proof}
Let $I$ be a datum whose class is maximal in $\mathcal P_{3,6}$.  If $I$
is non-simple, then $\Gamma(I)$ belongs to $\mathfrak M_0$ by
\cite[Theorem~4.2(1)--(8)]{zheng2021liftableabelian}.

Suppose that $I$ is simple, and choose a smooth simple cubic
$F_0\in W_I$.  Let $M(F_0)$ be the subgroup of diagonal matrices in
$\Aut(F_0)$.  Since
$\Gamma(I)$ fixes every cubic in $W_I$, it is contained in $M(F_0)$.
By Lemma~\ref{lem:abelian-liftable-diagonal}, the projective image of
$M(F_0)$ is liftable, so $M(F_0)$ is itself an invariant extension of the
type considered in Proposition~\ref{prop:character-to-invariant}.
Maximality of $[I]$ therefore gives $\Gamma(I)=M(F_0)$.  Hence every
maximal invariant extension is conjugate to a member of $\mathfrak M_0$,
and therefore to a member of $\mathfrak M$.

Now let $G\leq\Aut(X_F)$ be liftable abelian, with associated datum $I$.
Choose a maximal class $[I_{\max}]$ above $[I]$.  Proposition
\ref{prop:character-to-invariant} gives
$\widehat G_F=\Gamma(I)\leq\Gamma(I_{\max})$, and the latter is conjugate
in $\GL(6,\CC)$ to a member of $\mathfrak M$.  This proves the second
assertion.

Finally, a proper invariant extension of a member of $\mathfrak M$ would be
contained in a maximal member of $\mathfrak M_0$, a contradiction to
Definition~\ref{def:maximal-liftable-abelian-sets}.
\end{proof}

Two finite diagonal groups are linearly conjugate if and only if they are
conjugate by a coordinate permutation, since their natural representations
are determined by the multisets of coordinate characters.

Set
\begin{align*}
\mathscr D=\biggl\{&
\frac13(0,0,0,0,1,2),\
\frac13(0,0,0,0,1,1),\
\frac13(0,0,0,1,1,1),\\
&\frac14(0,0,0,1,2,3),\
\frac16(0,0,1,2,4,5),\
\frac18(0,1,2,3,4,6),\\
&\frac18(0,1,2,4,5,6),\
\frac18(0,1,2,4,6,7),\
\frac1{12}(0,0,3,4,6,9)
\biggr\}.
\end{align*}
Put
\[
\mathscr D^\sharp
=\{zg^a:g\in\mathscr D,\ z\in Z,\ \gcd(a,\ord(g))=1\}.
\]

\begin{lem}\label{lem:dangerous-elements}
Let $F\in\CC[x_1,\ldots,x_6]_3$ define a smooth cubic fourfold.  If
$\Aut(F)$ contains an element conjugate in $\GL(6,\CC)$ to an element of
$\mathscr D^\sharp$, then $\Aut(X_F)$ is non-abelian.
\end{lem}

\begin{proof}
It suffices to treat the representatives in $\mathscr D$.  Indeed,
$\Aut(F)$ contains $Z$, and a primitive power generates the same cyclic
group up to a central scalar.

For $g=\frac13(0,0,0,0,1,1)$, the invariant cubic splits as a cubic in
$x_1,\ldots,x_4$ plus a smooth binary cubic in $x_5,x_6$.  The latter has
a non-abelian projective automorphism group.  For
$g=\frac13(0,0,0,1,1,1)$, the cubic is a sum of two smooth cubics forms in three variables,
and the automorphisms of a smooth plane cubic again give a non-abelian
subgroup; see \cite[Lemma~3.12]{yang2024automorphism}.

For the remaining seven representatives, take $\tau$ and $u$ from the
following table:
$$
\begin{array}{c|c|c|c}
g&\tau&u&\text{monomials exchanged by }\tau\\ \hline
\frac13(0,0,0,0,1,2)&(5\;6)&2&
x_5^3,\ x_6^3\\
\frac14(0,0,0,1,2,3)&(4\;6)&3&
x_4^2x_5,\ x_5x_6^2\\
\frac16(0,0,1,2,4,5)&(3\;6)(4\;5)&5&
x_3^2x_5,\ x_4x_6^2,\ x_4^3,\ x_5^3\\
\frac18(0,1,2,3,4,6)&(2\;4)(3\;6)&3&
x_2^2x_6,\ x_3x_4^2,\ x_3^2x_5,\ x_5x_6^2\\
\frac18(0,1,2,4,5,6)&(2\;5)&5&
x_2^2x_6,\ x_5^2x_6\\
\frac18(0,1,2,4,6,7)&(2\;6)(3\;5)&7&
x_2^2x_5,\ x_3x_6^2,\ x_3^2x_4,\ x_4x_5^2\\
\frac1{12}(0,0,3,4,6,9)&(3\;6)&7&
x_3^2x_5,\ x_5x_6^2
\end{array}
$$
A direct calculation shows that $\tau g\tau^{-1}=g^u$ and that $\tau$
preserves the set of $g$-invariant cubic monomials.  Smoothness, together
with Lemma~\ref{lem:square-monomial}, implies that the coefficients of the monomials in the nontrivial $\tau$-orbits listed above are nonzero.  Comparing these
coefficients, one finds a diagonal matrix $d$ such that
$\varpi=d\tau$ preserves $F$.  Since $d$ commutes with $g$, we have
$\varpi g\varpi^{-1}=g^u$.  In each case the projective classes of $g$ and
$g^u$ are distinct, so $\Aut(X_F)$ is non-abelian.
\end{proof}

\begin{lem}\label{lem:order-72-nonabelian}
Let $X\subset\PP^5$ be a smooth cubic fourfold.  If $\Aut(X)$ contains an
abelian subgroup of order $72$ whose action is liftable, then $\Aut(X)$ is
not abelian.
\end{lem}

\begin{proof}
By Lemma~\ref{lem:abelian-liftable-diagonal}, this action is diagonalizable.
Thus \cite[Corollary~1 and Table~2]
{mayanskiy2013abelianautomorphismgroupscubic} gives three possible normal
forms.  One normal form has two disjoint one-variable cubic
summands.  Exchanging them, after rescaling, does not commute with the
given order-$3$ action.  Each of the other two has a smooth binary cubic
summand. Permutations of roots of the binary cubic is induced by projective transformations, which extend to automorphisms of the full cubic form. These automorphisms do not commute with the involution in the given group. Thus $\Aut(X)$ is non-abelian.
\end{proof}

Suppose that $G=\Aut(X_F)$ is abelian and liftable.  The Yang--Yu--Zhu bound
restricts $|G|$ to divisors of $32$, $48$, or $72$.
Lemma~\ref{lem:order-72-nonabelian} excludes order $72$, while
\cite[Proposition~7.1]{FWZ26} shows that orders $32$ and $48$ occur only
for $G=C_{32}$ and $C_{48}$.  Apart from these two groups, the strict lift
$H=\widehat G_F$ therefore has order dividing $48$, $72$, or $108$.
Lemma~\ref{lem:dangerous-elements} excludes elements conjugate to an element
of $\mathscr D^\sharp$.  The bound $|\Auts(X_F)|\leq4$ also gives
$|H\cap\SL(6,\CC)|\leq12$.

\begin{breakablealgorithm}
\caption{Liftable abelian enumeration}
\label{alg:liftable-abelian-enumeration}
\begin{algorithmic}[1]
\Require The maximal invariant extensions $\mathfrak M$ and the spectral
set $\mathscr D^\sharp$
\Ensure Representatives of the liftable abelian candidates
\State $Z\gets\langle\omega I_6\rangle$
\State $\mathscr E\gets\varnothing$
\ForAll{$M\in\mathfrak M$}
    \State $\mathscr A\gets\varnothing$; $\mathscr B\gets\varnothing$
    \ForAll{$U\leq M/Z$}
        \State $H\gets$ the inverse image of $U$ under $M\to M/Z$
        \State Add $(H,m(H))$ to $\mathscr A$
    \EndFor
    \State Order $\mathscr A$ by decreasing $|H|$, then decreasing $m(H)$
    \ForAll{$(H,m(H))\in\mathscr A$ in this order}
        \If{there is no $(L,m(L))\in\mathscr B$ with $H\subsetneq L$ and $m(H)=m(L)$}
            \State Add $(H,m(H))$ to $\mathscr B$
        \EndIf
    \EndFor
    \ForAll{$(H,m(H))\in\mathscr B$}
        \If{$|H|\mid48$ or $|H|\mid72$ or $|H|\mid108$}
            \State Add $H$ to $\mathscr E$
        \EndIf
    \EndFor
\EndFor
\State Identify groups in $\mathscr E$ up to conjugacy by a coordinate
permutation, choosing one representative of each class
\State $Result\gets\varnothing$
\ForAll{$H\in\mathscr E$}
    \If{$H$ contains no element conjugate to an element of $\mathscr D^\sharp$
    and $|H\cap\SL(6,\CC)|\leq12$}
        \State Add $H$ to $Result$
    \EndIf
\EndFor
\State \Return $Result$
\end{algorithmic}
\end{breakablealgorithm}

\begin{prop}\label{prop:liftable-abelian-completeness}
If $Z\leq H\subsetneq L\leq M$ for some $M\in\mathfrak M$ and $m(H)=m(L)$,
then $H$ is not saturated.  Apart from the strict lifts of $C_{32}$ and
$C_{48}$, Algorithm~\ref{alg:liftable-abelian-enumeration} contains, up to
linear conjugacy, $\Aut(F)$ for every smooth cubic fourfold with abelian
liftable projective automorphism group.
\end{prop}

\begin{proof}
Every $M\in\mathfrak M$ fixes a smooth cubic, so the first assertion follows
from Proposition~\ref{prop:dimension-test}.  For the second, Proposition
\ref{prop:maximal-liftable-abelian-groups} shows that, after conjugacy,
$H=\Aut(F)$ is contained in some $M\in\mathfrak M$. Since $H$ is
saturated, no group of the same family dimension within $M$ strictly
contains it. The other conditions of
Algorithm~\ref{alg:liftable-abelian-enumeration} are necessary and are
therefore also satisfied by $H$.
\end{proof}

Algorithm~\ref{alg:liftable-abelian-enumeration} yields $51$
conjugacy classes; adjoining $C_{32}$ and $C_{48}$ gives $53$ candidates.
Their invariant spaces contain smooth cubics by the Jacobian
criterion. We determine which of these candidates are saturated in \S~\ref{sec:fourfold-completion}.
See also \path{anc/liftable_abelian.txt}.

\subsection{Non-liftable abelian case}

The non-liftable abelian case has only the following two possibilities.

\begin{prop}
    If $X=X_F$ is a smooth cubic fourfold such that $\Aut(X_F)$ is abelian
    and not liftable, then
    \[
    \Aut(F)\sim_{\GL(6,\CC)} \left\langle (1\;2\;3)(4\;5\;6),\frac{1}{9}(1,4,7,1,4,7)\right\rangle
    \]
    or
    \[
    \Aut(F)\sim_{\GL(6,\CC)} \left\langle (1\;2\;3)(4\;5\;6),\frac{1}{9}(1,4,7,1,4,7),\frac{1}{2}(0,0,0,1,1,1)\right\rangle.
    \]
\end{prop}
\begin{proof}
    Let $P=(1\;2\;3)(4\;5\;6)$ and
    $C=\frac19(1,4,7,1,4,7)$ be elements of $\GL(6,\CC)$.

    Put $G=\Aut(X)$ and let $G_3$ be a Sylow $3$-subgroup.  Since $G$ is
    abelian, so is $\Auts(X)$; hence
    $|\Auts(X)|\leq4$ by \cite[Theorem~1.2]{laza2022automorphisms}.  Moreover,
    $|G_3|\leq9$
    \cite[Proposition~5.1]{FWZ26}.  Theorem~\ref{thm:sylow-liftability} and
    Lemma~\ref{lem:small-3-groups-lifting}(3) give $G_3\cong C_3^2$.
    The first possibility in the lemma lies in $\SL(6,\CC)$ and would imply
    $G_3\leq\Auts(X)$, a contradiction to $|\Auts(X)|\leq4$. Hence
    \[
    \widehat{G_3}_F\sim_{\GL(6,\CC)}
    \left\langle (1\;2\;3)(4\;5\;6),\frac19(1,4,7,1,4,7)\right\rangle.
    \]
    Since $\det P=1$ and $\det C=\omega^2$, one has
    $|G_3\cap\Auts(X)|=3$, and hence $\Auts(X)\cong C_3$.  By
    \cite[Theorems~3 and~4]{mayanskiy2013abelianautomorphismgroupscubic},
    one has $G=G_3$ or $G\cong C_2\times G_3$.

    After a linear transformation, we may assume
    $\widehat{G_3}_F=\langle P,C\rangle$.  If $G=G_3$, this is already
    $\Aut(F)$.  Otherwise an $F$-lift $T$ of the involution may be chosen
    with $T^2=I_6$.  Since $G$ is abelian and
    $\ker(\Aut(F)\to G)=\langle\omega I_6\rangle$, there are
    $\varepsilon,\eta\in\langle\omega\rangle$ such
    that
    \[
    TPT^{-1}=\varepsilon P,\qquad TCT^{-1}=\eta C.
    \]
    The equality $T^2=I_6$ gives $\varepsilon^2=\eta^2=1$, and hence
    $\varepsilon=\eta=1$.  Thus $T$ commutes with $P$ and $C$.  Since the
    representation of $\langle P,C\rangle$ is the
    sum of two equivalent irreducible three-dimensional representations, a
    change of basis in its centralizer gives
    $T=\frac12(0,0,0,1,1,1)$, proving the second case.

\end{proof}

The corresponding families have dimensions $2$ and $1$ and both contain
smooth cubics.

\section{Non-abelian cases}

\subsection{Cases with \texorpdfstring{$\rank(S)\geq15$}{rank(S) at least 15}}
\label{sec:large-symplectic}

The rank-$19$ families are classified in \cite[Theorem~1.2 and
\S6]{FWZ26}, and the rank-$20$ cases in
\cite[Theorem~1.2(9), Theorem~1.8, and \S\S4.5, 6.3]{laza2022automorphisms}.
For explicit examples and their automorphism groups, see also
\cite[\S6.1 and Theorems~6.14--6.15]{yang2024automorphism}.
For $15\leq\rank(S)\leq18$, \cite{FZ26} determines the possible abstract
groups $G$ and their actions on the full cohomology lattice $\Lambda$.

\subsubsection{Lattice and equation candidates}

For each connected symplectic family, we use the projective representation
of $G_s$ and its invariant cubic space in \cite{KOIKE202512,KOIKE2026}.  Let
$K^+\leq\GL(6,\CC)$ be the corresponding strict symplectic lift, so that
$K^+/\langle\omega I_6\rangle\cong G_s$.  For each $G$, we then construct a
strict lift $H$ with $K^+\leq H$.  If the component has generic index $2$,
then $H$ is required to contain the strict lift of the automorphism group
of a general member of that symplectic family.
This gives $40$ connected-family candidates and $36$ proper-family
candidates.  Their groups, indices, and dimensions agree with the
lattice-classification table, and they are ordered as in
Table~\ref{tab:fourfold}.
Some equation-side candidates for the proper families are obtained by the
representation enumeration recorded in 
\path{gap_classification/gap_large_group_enumeration/}
\cite{FuWangZhengAutcub4fold2026}.

The index-$2$ $\mathrm{QD}_{16}$ family is saturated: the cubic $X'_{12}$ of
\cite[\S6.1, (12)]{yang2024automorphism} lies in the unique
$\mathrm{QD}_{16}$ component, and
$\Aut(X'_{12})\cong(C_8\times C_2):C_2$.

\begin{prop}\label{prop:extraspecial-saturation}
The four families with symplectic part $3^{1+4}:2$ and non-symplectic
indices $2$, $4$, $6$, and $12$ are saturated.
\end{prop}

\begin{proof}
Write $F=f(x_1,x_2,x_3)+g(x_4,x_5,x_6)$ as in
\cite[\S6.1]{KOIKE202512}.  For a general member of each family, the two
cubic summands in three variables are not linearly equivalent.  The uniqueness assertion in
Theorem~\ref{thm:unique-splitting} therefore implies that every element of
$\Aut(F)$ preserves their coordinate spaces.  The linear automorphism groups
of a general plane cubic form, the $j=1728$ form, and the Fermat form have orders
$54$, $108$, and $162$ \cite[Lemma~3.12]{yang2024automorphism}.  Hence the
four projective automorphism groups have orders
\[
\frac{54^2}{3}=972,\qquad \frac{54\cdot108}{3}=1944,\qquad
\frac{54\cdot162}{3}=2916,\qquad \frac{108\cdot162}{3}=5832,
\]
as required.
\end{proof}

\subsubsection{Smoothness and saturation}

For each of the $76$ cases, a smooth cubic in the family is given in
\path{gap_classification/gap_smoothness/} \cite{FuWangZhengAutcub4fold2026}, as verified by the Jacobian
criterion.  Forty-nine families are saturated by the cited classifications or
the preceding arguments.  The character comparisons in
Algorithm~\ref{alg:equal-dimension-saturation} find no proper inclusion of
any of the remaining $27$ groups in a larger candidate of the same family
dimension.  Their saturatedness and the completeness of the large-family
list are proved in the proof of Theorem~\ref{thm:fourfold}.

\subsection{Cases with \texorpdfstring{$\rank(S)<15$}{rank(S) less than 15}}
\label{sec:small-nonabelian}

The remaining symplectic groups are $C_3$, $C_2^2$, $C_4$, and $S_3$.

\subsubsection{Candidate representations}

Fix a symplectic representation in \cite[\S3]{KOIKE202512} and \cite{KOIKE2026}, and put $Z=\langle\omega I_6\rangle$.  Let
$K^+\leq\GL(6,\CC)$ be the corresponding strict lift of $G_s$, so that
$K^+/Z\cong G_s$.  We seek strict lifts $H$ with
$\ker(\det|_H)=K^+$ and projective quotient among the groups allowed by
Yang--Yu--Zhu \cite{yang2024automorphism}.  The conditions
$G_s\triangleleft G$, $G/G_s$ cyclic, and divisibility by the generic index,
together with Lemma~\ref{lem:G7212-exclusion}, give the following initial list.
The further linear containment condition below is applied after the
representation enumeration.

\begin{center}
\begin{tabular}{c|c|c}
$G_s$ & generic index & $G$ \\ \hline
$C_3$, (3.4) & $1$ & $S_3\times C_3,\ S_3,\ \mathrm{Dic}_3,\ \mathrm{Dic}_3\times C_3,\ C_3:C_8$ \\
$C_3$, (3.3) & $2$ & $S_3\times C_3,\ S_3,\ \mathrm{Dic}_3,\ \mathrm{Dic}_3\times C_3,\ C_3:C_8$ \\
$C_2^2$ & $1$ & $A_4,\ D_8,\ C_8:C_2,\ A_4\times C_2,\ D_8\times C_3$ \\
$C_4$ & $1$ & $Q_8,\ D_8,\ D_8\times C_3,\ C_8:C_2$ \\
$S_3$, (3.6) & $1$ & $S_3\times C_n$, $n=1,2,3,4,6,8,12$ \\
$S_3$, corrigendum & $2$ & $S_3\times C_n$, $n=2,4,6,8,12,24$
\end{tabular}
\end{center}
Here (3.3), (3.4), and (3.6) refer to the corresponding equations in
\cite[\S3]{KOIKE202512}; the last row uses the equations in
\cite{KOIKE2026}.

If the chosen symplectic component has generic index $2$, let
$J^+\leq\GL(6,\CC)$ be the strict lift of its generic automorphism group, so that
$K^+\triangleleft J^+$ and $[J^+:K^+]=2$.  For an extension
$1\to Z\to H_e\xrightarrow{\pi_e}G\to1$, write $\chi_{K^+}$ for the
character of the defining six-dimensional representation of $K^+$.  For
each normal subgroup $Q\triangleleft G$ with $Q\cong G_s$ and each
isomorphism $\varphi:K^+\xrightarrow{\sim}\pi_e^{-1}(Q)$ satisfying
$\varphi(Z)=\ker(\pi_e)$, we consider the degree-six characters $\chi$ of
$H_e$ such that
$\chi|_{\pi_e^{-1}(Q)}=\chi_{K^+}\circ\varphi^{-1}$.  Let $\rho_\chi$ be a representation
affording $\chi$.  Faithful characters with scalar subgroup $Z$ and determinant
kernel $\pi_e^{-1}(Q)$ are taken modulo $\Aut(H_e)$; equality of characters
is equivalent to linear conjugacy.  On a generic-index-two component we
also require $J^+$ to be conjugate to a subgroup of the matrix image.

Write $G_{72,12}$ for the group with GAP ID $[72,12]$.  For the
$S_3$ component given in \cite{KOIKE2026} and $G\cong S_3\times C_{12}$, with GAP ID $[72,27]$, we
enumerate the representations equivalently by determinant-normalized
weights.  Both the split and nonzero extension classes are included, and
the same character equivalence and containment conditions are imposed.
The $S_3\times C_{24}$ action is given by
\cite[Proposition~7.1]{FWZ26}; Lemma~\ref{lem:G7212-exclusion}
excludes $G_{72,12}$.

After collecting the representations, we apply the prime-order restriction to the eigenvalues of elements in $H$ considering the strict lifts of the prime-order automorphisms in \cite[Theorem~3.8]{gonzalez2011automorphisms}. 

\begin{breakablealgorithm}
\caption{Candidate enumeration}
\label{alg:small-nonabelian-enumeration}
\begin{algorithmic}[1]
\Require The pairs $(K^+,G)$ to be enumerated, with
$S_3\times C_{24}$ treated separately
\Ensure Matrix groups $H$ and their invariant spaces
$W_H=\Sym^3((\CC^6)^*)^H$
\State $Z\gets\langle\omega I_6\rangle$; associate $G_s=K^+/Z$ with each $K^+$
\State Exclude pairs for which the generic index does not divide $[G:G_s]$
\ForAll{remaining pairs $(K^+,G)$}
    \State $\mathscr Q(G)\gets\{Q\triangleleft G:Q\cong G_s\}$
    \State $E\gets\operatorname{Ext}^1(G_{\mathrm{ab}},C_3)$
    \ForAll{$e\in E$, including $e=0$}
        \State Construct $1\to Z\to H_e\overset{\pi_e}{\longrightarrow}G\to1$
        \ForAll{$Q\in\mathscr Q(G)$ and isomorphisms
        $\varphi:K^+\xrightarrow{\sim}\pi_e^{-1}(Q)$ with
        $\varphi(Z)=\ker(\pi_e)$}
            \State \parbox[t]{\dimexpr\linewidth-\algorithmicindent-\algorithmicindent-\algorithmicindent\relax}{%
            \raggedright\strut Enumerate degree-six characters $\chi$ of $H_e$ extending\newline
            the fixed character $\chi_{K^+}$ of $K^+$\strut}
            \State \parbox[t]{\dimexpr\linewidth-\algorithmicindent-\algorithmicindent-\algorithmicindent\relax}{%
            \raggedright\strut Select all faithful $\rho_\chi$ with
            $\rho_\chi(H_e)\cap\CC^\times I_6=\rho_\chi(Z)$, and
            $\ker(\det\rho_\chi)=\pi_e^{-1}(Q)$\strut}
        \EndFor
        \State Choose one representative of each $\Aut(H_e)$-orbit of these characters
    \EndFor
    \State Identify conjugate matrix images arising from distinct $e$
    \If{the component has generic index $2$}
        \State Discard $H$ unless $J^+\preceq_{\GL(6,\CC)}H$
    \EndIf
    \State \parbox[t]{\dimexpr\linewidth-\algorithmicindent\relax}{%
    \raggedright\strut Compute and record $W_H$, $c(H)=\dim C_{\GL(6,\CC)}(H)$,
    and $m(H)=\dim W_H-c(H)$\strut}
    \State Discard $H$ if $\dim W_H<c(H)$
\EndFor
\State Add the explicit $S_3\times C_{24}$ action
\State Apply the prime-order restriction of
\cite[Theorem~3.8]{gonzalez2011automorphisms}
\State \Return the remaining groups and their invariant spaces
\end{algorithmic}
\end{breakablealgorithm}

The dimension condition in
Algorithm~\ref{alg:small-nonabelian-enumeration}
is necessary. By \eqref{eq:family-dimension}, $W_H$ contains no smooth
cubic if $\dim W_H<c(H)$.

\subsubsection{Central extensions}

\begin{prop}\label{prop:small-nonabelian-nonliftable-completeness}
For every projective group $G$ listed above, including $S_3\times C_{24}$,
the $3$-primary part of $M(G)$ is trivial.  Consequently, the natural map
$\operatorname{Ext}^1(G_{\mathrm{ab}},C_3)\to H^2(G,C_3)$ is an isomorphism.
\end{prop}

\begin{proof}
Let $P$ be a Sylow $3$-subgroup of $G$.  If $P$ is cyclic,
restriction and corestriction show that $M(G)_{(3)}$ injects into
$M(P)=0$.  Among the groups under consideration, the exceptions are
$S_3\times C_3$, $\mathrm{Dic}_3\times C_3$, $S_3\times C_6$,
$S_3\times C_{12}$, and $S_3\times C_{24}$.  For these five groups, the
direct-product formula, together with
$S_3^{\mathrm{ab}}\cong C_2$ and
$\mathrm{Dic}_3^{\mathrm{ab}}\cong C_4$, gives
$M(G)_{(3)}=0$.  Hence $\Hom(M(G),C_3)=0$, and the universal coefficient
sequence in Proposition~\ref{prop:stem-liftability} identifies
$H^2(G,C_3)$ with $\operatorname{Ext}^1(G_{\mathrm{ab}},C_3)$.
\end{proof}

\subsubsection{Smooth invariant cubics}

Let $F_1,\ldots,F_r$ be a basis of $W_H$, and let $c(H)$ be the
dimension of the centralizer of $H$ in $\GL(6,\CC)$.  The smoothness tests
are applied in the following order.

\begin{breakablealgorithm}
\caption{Smoothness test}
\label{alg:small-nonabelian-smoothness}
\begin{algorithmic}[1]
\Require A strict group $H$, the basis of $W_H$, and $c(H)$
\Ensure \texttt{smooth}, \texttt{singular}, or \texttt{unknown}
\If{$r<c(H)$, or all $F_i$ have a common nonconstant monomial factor}
    \State \Return \texttt{singular}
\EndIf
\If{some variable has exponent at most one in every monomial of every $F_i$}
    \State Verify the corresponding coordinate point and return
    \texttt{singular}
\EndIf
\State Form at most $32$ deterministic coefficient vectors, beginning with
$a_i=i^2+1$
\ForAll{at most three chosen good split primes}
    \State Test the chosen members $F_a=\sum_{i=1}^r a_iF_i$ at this prime
    \If{a member $F_a$ is certified smooth}
        \State Verify it individually, record $a$ and $F_a$, and return
        \texttt{smooth}
    \EndIf
\EndFor
\ForAll{the first at most six chosen coefficient vectors $a$}
    \State Test $F_a$ on the six characteristic-zero charts
    \If{all six charts are empty}
        \State Record $a$ and $F_a$, and return \texttt{smooth}
    \EndIf
\EndFor
\If{the bounded exact search finds a common singular point of all $F_i$}
    \State Record the point and return \texttt{singular}
\EndIf
\If{the exact six-chart common Jacobian locus is nonempty}
    \State Record a nonunit chart ideal and return \texttt{singular}
\EndIf
\State \Return \texttt{unknown} and keep the record
\end{algorithmic}
\end{breakablealgorithm}

At a good prime, If the Jacobian ideal of the tested cubic fourfold defines a zero-dimensional locus, it must be non-singular. Smoothness of this reduction proves smoothness in characteristic zero. 

A projective common zero of the partial derivatives of all basis cubics is
singular for every member of $W_H$. The coordinate-point and common-factor criteria in algorithm \ref{alg:small-nonabelian-smoothness} also show that every member is singular. If the tests neither find a smooth member nor prove that every member is singular, we record the result as unknown.  The
seven systems left unknown are proved singular in
Appendix~\ref{app:singular-families}.

\begin{lem}\label{lem:G7212-exclusion}
There is no smooth cubic fourfold $X$ such that $\Auts(X)\cong C_3$ and
$\Aut(X)\cong G_{72,12}$.
\end{lem}

\begin{proof}
Write
$G=\langle a,b,t\mid a^3=b^3=t^8=1,
[a,b]=[a,t]=1, tbt^{-1}=b^{-1}\rangle
\cong C_3\times(C_3:C_8)$.  Its only normal subgroups of order $3$ are
$\langle a\rangle$ and $\langle b\rangle$; since the quotient by
$\Auts(X)$ is cyclic, $\Auts(X)=\langle b\rangle$.  The Sylow
$3$-subgroup $\langle a,b\rangle$ is liftable.  Otherwise Lemma
\ref{lem:small-3-groups-lifting}(3) gives lifts $A,B,T$ with
$[A,B]\ne1$, while the defining relations give
$T[A,B]T^{-1}=[A,B]^{-1}$, impossible for a nontrivial scalar of order
$3$.  Proposition~\ref{prop:cubic-liftability}(2) therefore lifts $G$.
Let $\rho$ be a linear lift and let $\eta$ be defined by
$\rho(g)\cdot F=\eta(g)F$; the Hodge character is
$\kappa=\det(\rho)\eta^{-2}$.  In Koike's two $C_3$ components
\cite[(3.3), (3.4)]{KOIKE202512}, the multiplicities of the eigenvalues
$1,\omega,\omega^2$ of $b$ are respectively $(4,1,1)$ and $(2,2,2)$.
In the latter case $\rho$ is the sum of two $b$-nontrivial
two-dimensional summands and two $b$-trivial characters.  The former
summands contribute even exponents to $\det(t)$, and Lemma
\ref{lem:square-monomial} forces the two remaining $t$-weights to have the
same parity.  Thus $\kappa$ is not faithful on the $C_8$-factor, a
contradiction.  In the $(4,1,1)$ component, write $W_{0,0}$ for the
two-dimensional summand on which $b$ has eigenvalues
$\omega,\omega^{-1}$ and $t$ interchanges the eigenspaces, and write
$L_{\alpha,\beta}$ for the character
$a\mapsto\omega^\alpha$, $b\mapsto1$, $t\mapsto\zeta_8^\beta$.
After twisting the lift, $\eta=1$ and
$\rho=W_{0,0}\oplus\bigoplus_{i=1}^4L_{\alpha_i,\beta_i}$.  Faithfulness of
$\kappa$ and Lemma~\ref{lem:square-monomial} give
$\sum\alpha_i\ne0\pmod3$, $\sum\beta_i\equiv1\pmod2$, and require
$-2w_i$ to occur for every $w_i=(\alpha_i,\beta_i)$.  Up to
$a\mapsto a^{-1}$ and $t\mapsto t^v$ with $v$ odd, the four weights are
$((1,0),(1,1),(1,4),(1,6))$.  Hence the complete invariant space is
obtained by choosing coordinates $y,z$ on $W_{0,0}$ and $x_i$ on the four
one-dimensional summands in this order:
\[
\left\langle y^3+z^3, x_1^3, x_1x_3^2, x_2^2x_4, x_3x_4^2\right\rangle.
\]
Smoothness forces all five coefficients to be nonzero, and diagonal
rescaling normalizes them.  Thus, its smooth locus represents, after
reordering the variables, the cubic $X_3'$ of
\cite[Proposition~7.1]{FWZ26}.  The same proposition gives
$\Aut(X_3')\cong S_3\times C_{24}$, in which $G_{72,12}$ has index $2$.
This contradicts $G=\Aut(X)$.
\end{proof}

Algorithm \ref{alg:small-nonabelian-enumeration} leaves $182$ representations. The calculation
finds $46$ smooth invariant spaces and $129$ singular ones, and leaves seven
unknown.  After the singularity arguments in Appendix \ref{app:singular-families}, the results are listed as follows.

\begin{center}
\begin{tabular}{c|r|r|r}
$\Auts(X)$ & representations & smooth & singular \\ \hline
$C_3$   & $29$ & $9$  & $20$ \\
$C_2^2$ & $30$ & $9$  & $21$ \\
$C_4$   & $24$ & $5$  & $19$ \\
$S_3$   & $99$ & $23$ & $76$ \\ \hline
total   & $182$ & $46$ & $136$
\end{tabular}
\end{center}

See also \path{anc/small_group_enumeration.txt}.

\begin{prop}\label{prop:small-nonabelian-enumeration-completeness}
Let $X=X_F$ be a smooth cubic fourfold such that $\Aut(X)$ is
non-abelian, and $\Auts(X)$ is isomorphic to $C_3$, $C_2^2$, $C_4$,
or $S_3$.  Then $\Aut(F)$ is conjugate to one of the groups constructed
above, and this conjugacy identifies their invariant cubic spaces.
\end{prop}

\begin{proof}
Let $G=\Aut(X_F)$ and $H=\Aut(F)$.  After conjugacy,
$\ker(\det|_H)=K^+$.  By Propositions
\ref{prop:small-nonabelian-nonliftable-completeness} and
\ref{prop:stem-liftability}, the action is liftable and its invariant
extension lies in $\operatorname{Ext}^1(G_{\mathrm{ab}},C_3)$.  Thus
the extension $1\to Z\to H\to G\to1$ is among those considered by the algorithm.
Its natural character satisfies the faithfulness, scalar, determinant, and
generic-index conditions imposed there.  The prime-order restriction and
the smoothness test are necessary conditions for $W_H^{\mathrm{sm}}\ne
\varnothing$, so neither removes $H$.
\end{proof}

\section{Automorphism groups of smooth cubic fourfolds}
\label{sec:fourfold-completion}

The preceding constructions give $184$ candidates.  The seven invariant
spaces in Appendix~\ref{app:singular-families} contain no smooth form, and
the saturation below reduces the remaining $177$
candidates to $156$ families.

\subsection{Smooth candidates and saturation}
\label{subsec:saturation-verification}

Let $\mathscr C$ be the set of $177$ smooth candidates.  It comprises $76$
large-family, $53$ liftable abelian, $2$ non-liftable abelian, and $46$
small non-abelian records. Of these, $60$ are already known to be saturated
by \cite{he2025cubicfourfoldsorder7automorphism}, \cite{FWZ26} and the results above.
Let $\mathscr C_0$ denote their set.

The strict containment test relies on the following character criterion.

\begin{lem}\label{lem:strict-character-criterion}
Let $H,K\leq\GL(6,\CC)$ be finite subgroups and let $\chi_H,\chi_K$ be the
characters of their defining six-dimensional representations.  Then
$H\preceq_{\GL(6,\CC)}K$ if and only if there is an injective homomorphism
$\varphi:H\to K$ such that $\chi_H=\chi_K\circ\varphi$.
\end{lem}

\begin{proof}
A conjugate inclusion gives such an injection.  Conversely, the given
representation of $H$ and the restricted representation via $\varphi$ have
the same character.  Since complex representations of finite groups are
determined by their characters, they are linearly equivalent; hence
$P^{-1}hP=\varphi(h)$ for some $P\in\GL(6,\CC)$ and every $h\in H$.
\end{proof}

For each $K$-conjugacy class of subgroups $B\leq K$ isomorphic to $H$,
fix an isomorphism $\psi:H\to B$.  Every injection with image $B$ is
obtained by composing $\psi$ with an automorphism of $H$.  Conjugation in
$K$ and inner automorphisms of $H$ do not change the character comparison.
We therefore take automorphism representatives modulo inner automorphisms
and compare traces on every conjugacy class of $H$.  Noncontainment follows
only if an independent necessary condition fails or these comparisons are
complete and none gives equality.

\begin{breakablealgorithm}
\caption{Saturation}
\label{alg:equal-dimension-saturation}
\begin{algorithmic}[1]
\Require The $177$ smooth candidates $\mathscr C$, their dimensions $m(H)$,
and the saturated subset $\mathscr C_0$
\Ensure The retained strict lifts and equal-dimensional containment chains
for the other candidates
\State Order $\mathscr C$ by increasing $m(H)$, then increasing $|H|$
\State $\mathscr R\gets\varnothing$; leave $p(H)$ undefined
\ForAll{$H\in\mathscr C$ in this order}
    \If{$H\preceq_{\GL(6,\CC)}K$ for some $K\in\mathscr R$ with
    $m(K)=m(H)$ and $|K|=|H|$}
        \State $p(H)\gets K$
    \Else
        \State $\mathscr R\gets\mathscr R\cup\{H\}$
    \EndIf
\EndFor
\ForAll{$H\in\mathscr R\setminus\mathscr C_0$, ordered first by increasing
$m(H)$, then by increasing $|H|$}
    \ForAll{$K\in\mathscr R$ with $m(K)=m(H)$ and
    $|H|<|K|$, $|H|\mid|K|$, ordered by increasing $|K|$}
        \If{$H\preceq_{\GL(6,\CC)}K$}
            \State $p(H)\gets K$; proceed to the next $H$
        \EndIf
    \EndFor
\EndFor
\State $\mathscr S\gets\{H\in\mathscr R:p(H)\text{ is undefined}\}$
\State For each $H\in\mathscr C$, iterate $p$ until a member of $\mathscr S$
\State \Return $\mathscr S$ and the resulting containment chains
\end{algorithmic}
\end{breakablealgorithm}

\begin{prop}\label{prop:saturation-output}
Algorithm~\ref{alg:equal-dimension-saturation} returns the $156$ strict
lifts in Table~\ref{tab:fourfold}.

For each of the other $21$ strict lifts $H$,
there is a returned strict lift $K$ such that
$H\prec_{\GL(6,\CC)}K$ and $m(H)=m(K)$.
\end{prop}

\begin{proof}
The character criterion shows that the $177$ candidates are pairwise
nonconjugate.  All comparisons needed for the negative conclusions were
completed.  For precisely $21$ of them, the algorithm finds a proper
conjugate inclusion in a larger candidate of the same family dimension.
Each step of $p$ strictly increases the group order and preserves the
dimension, so iteration terminates at one of the $156$ returned groups.
Transitivity gives the stated inclusion in the terminal group.
Proposition~\ref{prop:dimension-test} shows that the other $21$ candidates
are not saturated.  The saturatedness of the returned groups is proved
in the proof of Theorem~\ref{thm:fourfold}.
\end{proof}

See also \path{anc/saturation.txt}.

Up to linear conjugacy, these groups are generated by the corresponding symplectic lift and at most one element.
Their equivalence with the computed groups is checked in
\path{gap_classification/gap_manuscript_validation/}\cite{FuWangZhengAutcub4fold2026}.

\subsection{Self-conjugacy}
\label{subsec:self-conjugacy}

For $G\leq\GL(n,\CC)$, put
$\overline G=\{\overline g:g\in G\}$, where $\overline g$ is obtained by
the complex conjugation of each entry of $g$.  For a polynomial $F$,
define $\overline F$ by the complex conjugation of its coefficients.

\begin{prop}\label{prop:conjugation}
Let $F$ define a smooth cubic fourfold.  Then
\[
\Aut(\overline F)=\overline{\Aut(F)},\qquad
\Auts(\overline F)=\overline{\Auts(F)}.
\]
Under conjugate markings, the period point $[H^{3,1}(X_{\overline{F}})]$ of $X_{\overline F}$ is the
complex conjugate of the period point of $X_F$ in $\PP(\Lambda_{0}\otimes \CC)$.  In particular, complex
conjugation preserves the family dimension and the non-symplectic index.
\end{prop}

\begin{proof}
The first equality follows by conjugating $A\cdot F=F$, and the second from
$\det(\overline A)=\overline{\det(A)}$.  The assertion about periods follows
by conjugating $\operatorname{Res}(\Omega/F^2)$.
\end{proof}

\begin{defn}
A finite subgroup $H\leq\GL(n,\CC)$, with inclusion representation
$\rho:H\hookrightarrow\GL(n,\CC)$, is \emph{self-conjugate} if there is
$\theta\in\Aut(H)$ such that $\rho\cong\overline\rho\circ\theta$. A family
of smooth $H$-invariant forms is \emph{self-conjugate} if complex
conjugation preserves it up to linear equivalence.
\end{defn}

\begin{prop}\label{prop:self-conjugacy-criterion}
Let $H\leq\GL(6,\CC)$ be a finite strict lift with
$W_3(H)^{\mathrm{sm}}\ne\varnothing$. If
$H\sim_{\GL(6,\CC)}\overline H$, then the family defined by $W_3(H)$ is
self-conjugate.
\end{prop}

\begin{proof}
If $P^{-1}HP=\overline H$, then $F\mapsto P^{-1}\cdot F$ identifies
$W_3(H)$ with $W_3(\overline H)$, preserves smoothness, and descends modulo
linear equivalence.
\end{proof}

Complex conjugation introduces no further fourfold family.

\begin{prop}\label{prop:all-families-self-conjugate}
Every family in Table~\ref{tab:fourfold} is self-conjugate.
\end{prop}

\begin{proof}
The symplectic families are self-conjugate by the explicit equations in
\cite{KOIKE202512,KOIKE2026}. Together with \cite[\S4]{he2025cubicfourfoldsorder7automorphism}, \cite[\S6]{FWZ26},
 and Proposition~\ref{prop:extraspecial-saturation},
this gives $56$ self-conjugate families. For each of the remaining $100$
strict lifts $H$, one checks that $H\sim_{\GL(6,\CC)}\overline H$.
Then, the assertion follows from Proposition~\ref{prop:self-conjugacy-criterion}.
\end{proof}

For the self-conjugacy and liftability calculations, see also
\path{anc/self_conjugacy.txt}.

\begin{rmk}\label{rmk:conjugate-components}
By Proposition~\ref{prop:all-families-self-conjugate} and
Theorems~\ref{thm:torelli} and~\ref{thm:equivariant-torelli}, there is an isometry
of $\Lambda$ fixing $h^2$ that identifies the marked period domains of each
classified family and its complex conjugate.  In the type IV case,
this isometry exchanges the two conjugate components; see
\cite[Proposition~4.8]{yu2020moduli} for the description of the marked
period loci.
\end{rmk}

\begin{rmk}\label{rmk:transpose-conjugate}
Let $G\leq\PGL(6,\CC)$ be the automorphism group or symplectic
automorphism group of a smooth cubic fourfold $X_F$, and put $H=\widehat G_F$.
Since $H$ is conjugate to a unitary group,
$H^t\sim_{\GL(6,\CC)}\overline H$, so conjugacy to the complex-conjugate
group is equivalent to conjugacy to the transpose group.
The left and right actions of $H$ on cubic forms are given by $f(xA)$
and $f(Ax)$, using row and column coordinates, respectively.
Proposition~\ref{prop:all-families-self-conjugate} and its proof therefore
show that these two actions define equivalent families, for both
$\Aut(X_F)$ and $\Auts(X_F)$. The same holds for $G$ acting
projectively.
\end{rmk}

\subsection{The incidence relation among the classified families}
\label{subsec:fourfold-incidence}

Let $\mathscr S=\{H_1,\ldots,H_{156}\}$ be the strict lifts returned by
Algorithm~\ref{alg:equal-dimension-saturation}, numbered as in
Table~\ref{tab:fourfold}. The proof of
Proposition~\ref{prop:saturation-output} establishes their pairwise
nonconjugacy; their saturatedness is proved below.
For $H\in\mathscr S$, let $\mathcal Z_H$ be the moduli image defined in
\S~\ref{subsec:families-prescribed-symmetry}. The \emph{incidence
relation} is inclusion among these images.

Here $\prec$ denotes strict conjugate containment in $\GL(6,\CC)$.
It defines a strict partial order on $\mathscr S$. A family is
\emph{A-maximal} if its strict lift is maximal in this order. This
notion takes the linear action into account, unlike maximality of
abstract automorphism groups in \cite{yang2024automorphism}.

We determine $\prec$ using Lemma~\ref{lem:strict-character-criterion}.
The algorithm records the verified pairs in $\mathscr R$. Its
\emph{transitive reduction} consists of pairs $(H,K)\in\mathscr R$ for
which no $L\in\mathscr S$ satisfies $H\prec L\prec K$.

\begin{breakablealgorithm}
\caption{Incidence relations}
\label{alg:cross-dimensional-containment}
\begin{algorithmic}[1]
\Require The $156$ strict lifts
$\mathscr S=\{H_1,\ldots,H_{156}\}$ and their dimensions $m(H)$
\Ensure The containment relation and its pairs with no intermediate
member of $\mathscr S$
\State $\mathscr S_s\gets\{H\in\mathscr S:|H|\leq2000\}$ and
$\mathscr S_\ell\gets\mathscr S\setminus\mathscr S_s$
\State \parbox[t]{\linewidth}{%
\raggedright\strut Determine $\mathscr R$ using
Lemma~\ref{lem:strict-character-criterion} to test $H\prec K$, with
$H,K\in\mathscr S_s$, $m(H)>m(K)$, and $|H|\mid|K|$\strut}
\State Order $\mathscr S_\ell$ by decreasing $m$, then by increasing
$|H|$, and finally by family number
\State $\mathscr A\gets\mathscr S_s$
\ForAll{$K\in\mathscr S_\ell$ in this order}
    \State $\mathscr U\gets
    \{H\in\mathscr A:m(H)>m(K),\ |H|\mid|K|\}$
    and let $\mathscr Q$ be its maximal elements
    \While{$\mathscr Q\neq\varnothing$}
        \State Choose $H\in\mathscr Q$;
        $\mathscr Q\gets\mathscr Q\setminus\{H\}$;
        test $H\prec K$
        \If{the test is positive}
            \State $\mathscr D\gets
            \{L\in\mathscr U:L\prec H\}$
            \State Compose and verify the conjugating matrices for
            $L\prec H\prec K$ for all $L\in\mathscr D$
            \State \parbox[t]{\dimexpr\linewidth-\algorithmicindent-\algorithmicindent-\algorithmicindent\relax}{%
            \raggedright\strut Add $(H,K)$ and all $(L,K)$ with $L\in\mathscr D$
            to $\mathscr R$; remove $\{H\}\cup\mathscr D$ from
            $\mathscr U$ and $\mathscr Q$\strut}
        \Else
            \State $\mathscr U\gets\mathscr U\setminus\{H\}$;
            add every $L\in\mathscr U$ with
            $L\prec H$ to $\mathscr Q$
        \EndIf
    \EndWhile
    \State $\mathscr A\gets\mathscr A\cup\{K\}$
\EndFor
\State \Return the transitive reduction of $\mathscr R$ and the
groups with no outgoing containment relation
\end{algorithmic}
\end{breakablealgorithm}

\begin{proof}[Correctness of
Algorithm~\ref{alg:cross-dimensional-containment}]
For the saturated groups in $\mathscr S$, $H\prec K$ implies
$|H|\mid|K|$ and $m(H)>m(K)$, by
Proposition~\ref{prop:dimension-test}. If $K\in\mathscr S_s$, then
$|H|<|K|\leq2000$, so the first comparisons cover every such pair.

Proceed by induction through $\mathscr S_\ell$ in the stated order.
Assume that the relations within $\mathscr A$ are known. For the next
group $K$, every $H\prec K$ belongs to $\mathscr A$ by the dimension
inequality and satisfies the defining conditions of $\mathscr U$.
At every step, the relation to $K$ is determined for each group removed
from $\mathscr U$, while $\mathscr Q$ contains every maximal element of
$\mathscr U$. If $H\prec K$, transitivity gives $L\prec K$ for every
$L\in\mathscr U$ with $L\prec H$; these pairs are recorded before the
groups are removed, and no new maximal element appears. If
$H\not\prec K$, only $H$ is removed; any newly maximal element lies
below $H$ and is added to $\mathscr Q$. Thus both statements persist.
Since $\mathscr U$ is finite and decreases at each step, all possible
relations to $K$ are determined. Thus
$\mathscr R=\{(H,K)\in\mathscr S^2:H\prec K\}$. The final step returns
its transitive reduction and maximal elements.
\end{proof}

\begin{prop}\label{prop:fourfold-containment}
Let $H,K\in\mathscr S$ be distinct. Then
$\mathcal Z_K\subsetneq\mathcal Z_H$ if and only if $H\prec K$.

The $21$ A-maximal families are precisely those listed in
Table~\ref{tab:fourfold-action-maximal}. Equivalently, their moduli images
are the minimal elements under inclusion among the classified family images.
\end{prop}

\begin{proof}
If $H\prec K$, a linear change of coordinates makes $H$ a proper
subgroup of $K$. Then $W_3(K)\subseteq W_3(H)$,
and hence $\mathcal Z_K\subseteq\mathcal Z_H$.

Conversely, suppose that $\mathcal Z_K\subseteq\mathcal Z_H$.
Since $K$ is saturated, choose $F\in W_3(K)^{\mathrm{sm}}$
with $\Aut(F)=K$. The inclusion of moduli images means that
$F$ is linearly equivalent to a smooth $H$-invariant form.
Its linear automorphism group is therefore a conjugate
of $K$ containing $H$. Since $H$ and $K$ are not conjugate, this gives
$H\prec K$. If the two moduli images are equal, the same argument
in the opposite direction would give $K\prec H$, a contradiction.
Thus the inclusion of moduli images is proper,
and the equivalence follows.

Algorithm~\ref{alg:cross-dimensional-containment} determines all
conjugate containments, and its maximal elements are precisely the
strict lifts in Table~\ref{tab:fourfold-action-maximal}.
\end{proof}

The strict incidence relation among the $156$ classified families
consists of $1793$ pairs. Of these, $433$ have no intermediate classified
family: for such a pair
$\mathcal Z_K\subsetneq\mathcal Z_H$, there is no $L\in\mathscr S$
with $\mathcal Z_K\subsetneq\mathcal Z_L\subsetneq\mathcal Z_H$.
For the complete incidence data, see also \path{anc/results.txt}.

\Needspace{10\baselineskip}
\subsection{Proof of the classification theorem}
\label{subsec:fourfold-classification-theorem}
\begin{thm}\label{thm:fourfold-classification}
Let $H_1,\ldots,H_{156}\leq\GL(6,\CC)$ be the strict lifts in
Table~\ref{tab:fourfold}.
\begin{enumerate}
\item For every $i$, the space $W_3(H_i)$ contains a smooth cubic form,
and a general $F\in W_3(H_i)^{\mathrm{sm}}$ satisfies $\Aut(F)=H_i$.
\item $\mathcal F_{H_1},\ldots,\mathcal F_{H_{156}}$ are
pairwise inequivalent and exhaust the saturated families of smooth
cubic fourfolds with specified automorphism group actions.
\end{enumerate}

For a general $F\in W_3(H_i)^{\mathrm{sm}}$, we have
\[
\Aut(X_F)=\pi(H_i),\qquad
\Auts(X_F)=\pi\bigl(H_i\cap\SL(6,\CC)\bigr).
\]
The dimension of $\mathcal F_{H_i}$ is
$\dim W_3(H_i)-\dim C_{\GL(6,\CC)}(H_i)$.
These groups and dimensions are given in Table~\ref{tab:fourfold}.
\end{thm}

\begin{proof}
The symplectic actions and their invariant cubic forms are classified in
\cite{laza2022automorphisms}, \cite{KOIKE202512,KOIKE2026}.  The cases
$\rank(S)=20$ and $19$ are determined in
\cite[Theorem~1.8 and Proposition~6.12]{laza2022automorphisms},
\cite[Theorem~1.2 and \S6]{FWZ26}, and \cite[\S6.1.1]{FZ26}.

For $15\leq\rank(S)\leq19$, the lattice actions are classified in
\cite[Theorem~1.1 and Table~1]{FZ26}.  By
\cite[Proposition~3.6]{FZ26}, a lattice-theoretic class of index at least
$3$ determines one family, whereas a class of index $1$ or $2$ determines
at most two; if there are two, they are exchanged by complex conjugation.
Every listed family is self-conjugate by Propositions
\ref{prop:self-conjugacy-criterion} and
\ref{prop:all-families-self-conjugate}.  Hence two listed families with the
same lattice-theoretic class would be equivalent, contrary to the character
comparisons.  The number of listed families agrees with the number of classes
in \cite[Table~1]{FZ26}; thus the listed large families exhaust all
possibilities in these ranks.

When $\rank(S)<15$, the bounds of
\cite{yang2024automorphism} and the enumerations in
\S~\ref{subsec:liftable-abelian}--\ref{sec:small-nonabelian} give all
possible extensions.  Proposition
\ref{prop:small-nonabelian-nonliftable-completeness} excludes any further
non-liftable non-abelian case.  The cited equations and the Jacobian tests,
together with Appendix~\ref{app:singular-families}, leave exactly the $177$
pairwise nonconjugate smooth candidates in $\mathscr C$.  Thus every
saturated family is represented in $\mathscr C$.

The symplectic parts of the large candidates are the saturated symplectic
groups classified above.  We order these candidates first by decreasing
$\rank(S)$ and then, at a fixed rank, by decreasing group order.  Suppose
that $H$ is properly contained in a saturated group $K$.  If the symplectic
subgroup of $K$ properly contains that of $H$, its primitive coinvariant
lattice contains that of $H$; equality of their ranks would make the two
lattices equal, and the description in \cite[\S4.1]{FZ26} and
\cite{laza2022automorphisms} would then make the two symplectic groups equal.
Thus the rank increases.  If the symplectic subgroups are equal, then
$|K|>|H|$.  Hence $K$ occurs earlier in this order.

The candidates in $\mathscr C_0$ are saturated by the results cited above.
Propositions~\ref{prop:saturation-output} and~\ref{prop:dimension-test}
show that the $21$ discarded candidates are not saturated.  Let $H$ be one
of the remaining candidates.  If $H\notin\mathscr C_0$ were not saturated,
let $K\leq\GL(6,\CC)$ be the strict lift defining its saturation as in
Definition~\ref{def:linear-saturation}; equivalently, after conjugacy,
$K=\Aut(F)$ for a general $F\in W_3(H)^{\mathrm{sm}}$.
By \cite[Proposition--Definition~3.5]{FZ26}, the two actions have the same
moduli image, so $H\prec_{\GL(6,\CC)}K$ and $m(H)=m(K)$.  By the
completeness above, $K$ is conjugate to a member of $\mathscr C$ and, in
the large cases, precedes $H$ in the order above.  Since the
$177$ candidates are pairwise nonconjugate, the first stage of
Algorithm~\ref{alg:equal-dimension-saturation} leaves
$\mathscr R=\mathscr C$; its second stage would therefore not retain $H$.
This contradiction proves that all $156$ retained candidates are saturated.

The character criterion gives pairwise inequivalence.  The assertions about
$\Aut(X_F)$, $\Auts(X_F)$, and the family dimensions follow from
\eqref{eq:linear-projective}, \eqref{eq:symplectic-diagram}, and
\eqref{eq:family-dimension}.
\end{proof}

\section{Automorphism groups of smooth cubic threefolds}

\subsection{Cubic suspension}

Let \(F\in\CC[x_1,\ldots,x_5]_3\), and put
\[
 \widehat F=F+x_6^3,\qquad
 \delta=\frac13(0,0,0,0,0,1),\qquad
 V_\delta=\ker(\delta-I_6)\subset\CC^6.
\]
Thus \(V_\delta\cong\CC^5\), with coordinates \(x_1,\ldots,x_5\).

The following lemma compares the two smoothness conditions.

\begin{lem}\label{lem:threefold-suspension-smooth}
The cubic threefold \(V(F)\subset\PP^4\) is smooth if and only if the cubic
fourfold \(V(\widehat F)\subset\PP^5\) is smooth.
\end{lem}

\begin{proof}
The last partial derivative of \(\widehat F\) is \(3x_6^2\).  Hence every
singular point of \(V(\widehat F)\) lies on \(x_6=0\) and is a singular point
of \(V(F)\).  The converse is immediate.
\end{proof}

\Needspace{6\baselineskip}
The following lemma characterizes Fermat summands in any degree.

\begin{lem}\label{lem:threefold-fermat-element}
Let \(P\in\CC[x_1,\ldots,x_n]_d\), where \(d\geq3\), define a smooth
hypersurface, and write its maximal Fermat splitting as
\[
 P\sim_{\mathrm{lin}}x_1^d+\cdots+x_r^d
 +P_0(x_{r+1},\ldots,x_n),
\]
where \(P_0\) is purely non-Fermat.  Then
\(\Aut(P)\cong(\mu_d^r\rtimes S_r)\times\Aut(P_0)\), and
\[
 r=\#\{g\in\Aut(P):\ord(g)=d,\ \tr(g)=n-1+\zeta_d\}.
\]
For \(g\in\Aut(P)\), the conditions \(\ord(g)=d\) and
\(\tr(g)=n-1+\zeta_d\) hold if and only if, after a linear change of
coordinates,
\[
 g=\frac1d(0,\ldots,0,1),\qquad
 P=Q(x_1,\ldots,x_{n-1})+x_n^d,
\]
where \(Q\) is smooth.  These elements are in bijection with the
single-variable Fermat summands in the maximal additive splitting of \(P\); if
they exist, they form one conjugacy class in \(\Aut(P)\).
\end{lem}

\begin{proof}
The decomposition of \(\Aut(P)\) is Proposition~\ref{prop:aut-splitting}.
Since \(g\) has finite order, it is semisimple.  Let
\(\lambda_1,\ldots,\lambda_n\in\mu_d\) be its eigenvalues.
The trace equality gives
\[
 \sum_{i=1}^n(1-\lambda_i)=1-\zeta_d.
\]
For every \(\lambda\in\mu_d\setminus\{1\}\),
\[
 1-\Re(\lambda)\geq1-\cos(2\pi/d),
\]
with equality only for \(\lambda=\zeta_d^{\pm1}\).  Since \(g\) has order
\(d\), taking real parts shows that exactly one eigenvalue is nontrivial;
comparison of imaginary parts shows that it is \(\zeta_d\).  Hence
\(g\sim_{\GL(n,\CC)}\frac1d(0,\ldots,0,1)\).  In an eigenbasis, invariance
under \(g\) gives \(P=Q+c x_n^d\).  Smoothness gives \(c\ne0\) and the
smoothness of \(Q\), and rescaling \(x_n\) gives the asserted form.  The
converse is immediate.

By Theorem~\ref{thm:unique-splitting}, the maximal additive splitting
refines every additive splitting. Thus the \(\zeta_d\)-eigenspace of
\(g\) is one of its one-dimensional Fermat blocks, and the
\(1\)-eigenspace is the sum of all remaining blocks. In the fixed maximal
splitting, \(g\) therefore multiplies exactly one Fermat variable by
\(\zeta_d\) and fixes the others. Conversely, each of these \(r\)
matrices preserves \(P\) and has the required order and trace.
This proves the asserted bijection. Proposition~\ref{prop:aut-splitting}
shows that these summands are permuted transitively, so the elements form
one conjugacy class.
\end{proof}

For a cubic form \(\widehat F\) defining a smooth fourfold, we call a linear
automorphism \(\delta\in\Aut(\widehat F)\) with eigenvalues
\(1,1,1,1,1,\omega\) a \emph{Fermat element}.
For \(\widehat F=F+x_6^3\) and
\(\delta=\diag(1,1,1,1,1,\omega)\), the following proposition expresses
\(\Aut(F)\) in terms of \(C_{\Aut(\widehat F)}(\delta)\).

\begin{prop}\label{prop:threefold-centralizer}
Let \(\widehat F=F+x_6^3\), let
\(\widehat H=\Aut(\widehat F)\), and let
\(C=C_{\widehat H}(\delta)\).  Let
\(\rho_\delta:C\to\GL(V_\delta)\) be the restriction homomorphism.  It gives
an exact
sequence
\[
 1\longrightarrow\langle\delta\rangle\longrightarrow C
 \overset{\rho_\delta}{\longrightarrow}\Aut(F)\longrightarrow1.
\]
In particular,
\[
 C/\langle\delta\rangle\cong\Aut(F).
\]
If \(F\) is purely non-Fermat, then
\[
 \Aut(\widehat F)=\Aut(F)\times\langle\delta\rangle.
\]
\end{prop}

\begin{proof}
Every element of \(C\) preserves the two eigenspaces of \(\delta\), and
therefore has the form \(\diag(B,\lambda)\).  Since it fixes
\(F+x_6^3\), one has \(B\in\Aut(F)\) and \(\lambda^3=1\).  Thus
\(\ker(\rho_\delta)=\langle\delta\rangle\).  Conversely, every
\(B\in\Aut(F)\) extends to \(\diag(B,1)\in C\), so \(\rho_\delta\) is
surjective.  If \(F\) is purely non-Fermat, then \(x_6^3\) is the unique
single-variable Fermat summand of \(\widehat F\), and the last assertion follows
from Proposition~\ref{prop:aut-splitting}.
\end{proof}

The following proposition describes the correspondence induced by cubic suspension.

\begin{prop}\label{prop:threefold-suspension-map}
The assignment
\[
 [F]\longmapsto[F+x_6^3]
\]
is injective on linear-equivalence classes of smooth cubic forms.
It induces a dimension-preserving bijection between the saturated families
of smooth cubic threefolds and the saturated families of smooth cubic
fourfolds whose general linear automorphism group contains a Fermat element.

For such a cubic fourfold family with general linear automorphism group
\(\widehat H\), the corresponding cubic threefold family has general
linear automorphism group
\[
 H=\rho_\delta\bigl(C_{\widehat H}(\delta)\bigr)
\]
for any Fermat element \(\delta\in\widehat H\).  The resulting family is
independent of the choice of \(\delta\).
\end{prop}

\begin{proof}
A linear equivalence between \(F\) and \(F'\) extends to their suspensions.
Conversely, Theorem~\ref{thm:unique-splitting} identifies the maximal
additive splittings of two equivalent suspensions.  Removing any one of
their single-variable Fermat summands leaves forms equivalent to \(F\) and \(F'\),
respectively.  Hence \(F\sim_{\mathrm{lin}}F'\).

Let \(\widehat H=\Aut(\widehat F)\) for a general member \(\widehat F\) of a
saturated fourfold family, and let \(\delta\in\widehat H\) be a Fermat element.  By
Lemma~\ref{lem:threefold-fermat-element}, this member can be written
as \(\widehat F=F+x_6^3\), with
\(\delta=\diag(1,1,1,1,1,\omega)\).
Proposition~\ref{prop:threefold-centralizer} gives
\(\Aut(F)=\rho_\delta(C_{\widehat H}(\delta))=H\), so \(H\) is
saturated.  All Fermat elements of \(\widehat H\) are conjugate by
Lemma~\ref{lem:threefold-fermat-element}; conjugation identifies their
fixed spaces and the restricted centralizer actions.  Thus the resulting
\(H\)-family is independent of \(\delta\).

Write the maximal Fermat splitting of \(F\) as
\[
 F=S_r+F_0,\qquad S_r=x_1^3+\cdots+x_r^3,\qquad 0\le r\le5,
\]
where \(S_0=0\) and \(F_0\in\CC[x_{r+1},\ldots,x_5]_3\) is purely non-Fermat.
Put \(H_0=\Aut(F_0)\) and
\(M_s=\mu_3^s\rtimes\mathfrak S_s\), where \(\mathfrak S_s\) is the symmetric
group on \(s\) letters and \(M_0\) is trivial.  In these variable
blocks, Proposition~\ref{prop:aut-splitting} gives
\[
 H=M_r\times H_0,\qquad \widehat H=M_{r+1}\times H_0.
\]
The independent diagonal \(\mu_3\)-actions imply that no invariant cubic
contains mixed monomials involving a Fermat variable, and the
symmetric-group actions imply that the coefficients of the Fermat cubes
are equal. Consequently,
\[
 W_H=
 \begin{cases}
 W_3(H_0),&r=0,\\
 \CC S_r\oplus W_3(H_0),&r>0,
 \end{cases}
 \qquad
 W_{\widehat H}=\CC(S_r+x_6^3)\oplus W_3(H_0).
\]
When \(F_0=0\), the summand \(W_3(H_0)\) is understood to be zero.
Smoothness forces the coefficient of each Fermat cube to be nonzero.
Rescaling the Fermat variables therefore shows that, up to linear
equivalence, the smooth \(\widehat H\)-invariant cubics are precisely the
suspensions of smooth \(H\)-invariant cubics.

Conversely, let \(H\) define any saturated cubic threefold family and
choose a general member \(F\) with \(\Aut(F)=H\).  By
Lemma~\ref{lem:threefold-suspension-smooth}, \(F+x_6^3\) is smooth.
Its linear automorphism group \(\widehat H=\Aut(F+x_6^3)\) defines a saturated
fourfold family and contains \(\delta\).
Proposition~\ref{prop:threefold-centralizer} recovers \(H\) from
\(\widehat H\), and the invariant-space description above identifies
the corresponding families.  The two constructions are therefore inverse,
which proves the claimed bijection.

Finally, the independent scalar actions on the variable blocks force
commuting matrices to preserve those blocks.  For \(s\ge1\), the
centralizer of \(M_s\) in \(\GL(s,\CC)\) consists of scalar matrices:
commutation with \(\mu_3^s\) forces diagonality, and commutation with
\(\mathfrak S_s\) forces all diagonal entries to be equal.
Thus passing from \(H\) to \(\widehat H\) leaves both the invariant-space
dimension and the centralizer dimension unchanged when \(r>0\), and
increases both by one when \(r=0\).  Formula~\eqref{eq:family-dimension}
gives \(m(H)=m(\widehat H)\).
\end{proof}

\subsection{Extraction from the fourfold classification}

For a Fermat element \(\delta\), one has \(\det(\delta)=\omega\).  The
residue description of \(H^{3,1}\) shows that the non-symplectic index of
its fourfold family is divisible by \(3\).  It is therefore enough to
consider the corresponding families in Table~\ref{tab:fourfold}, of which
there are $63$.

By Lemma~\ref{lem:threefold-fermat-element}, the order-three elements of
trace $5+\omega$, when present, are the Fermat elements and form one
conjugacy class.

\begin{breakablealgorithm}
\caption{Cubic threefold extraction}
\label{alg:threefold-extraction}
\begin{algorithmic}[1]
\Require The linear groups \(\widehat H=\Aut(\widehat F)\) in
Table~\ref{tab:fourfold}
\Ensure The strict groups \(\Aut(F)\), their invariant cubic spaces,
and family dimensions
\State \(\mathscr T\gets\varnothing\)
\ForAll{\(\widehat H\) whose non-symplectic index is divisible by \(3\)}
    \State Determine the conjugacy classes
    \[
      \mathscr D(\widehat H)=
      \{[\delta]\subset\widehat H:
      \ord(\delta)=3,\ \tr(\delta)=5+\omega\}
    \]
    \If{\(\mathscr D(\widehat H)\ne\varnothing\)}
        \State Choose a representative \(\delta\) of a class in
        \(\mathscr D(\widehat H)\), put
        \(V_\delta=\ker(\delta-I_6)\), and set
        \[
          H=\rho_\delta\bigl(C_{\widehat H}(\delta)\bigr)
          \leq\GL(V_\delta)
        \]
        \State Compute \(W_H\) and
        \(m(H)=\dim W_H-\dim C_{\GL(V_\delta)}(H)\)
        \State Add \((H,W_H,m(H))\) to \(\mathscr T\)
    \EndIf
\EndFor
\State \Return \(\mathscr T\)
\end{algorithmic}
\end{breakablealgorithm}

\begin{prop}\label{prop:threefold-extraction}
Algorithm~\ref{alg:threefold-extraction} yields exactly $40$ pairwise
inequivalent saturated families of smooth cubic threefolds, namely those
listed in Table~\ref{tab:threefold}.  Every smooth cubic threefold belongs
to one of these families.
\end{prop}

\begin{proof}
By Proposition~\ref{prop:threefold-suspension-map}, the inverse
construction in Algorithm~\ref{alg:threefold-extraction} gives all
saturated cubic threefold families without repetition.
Lemma~\ref{lem:threefold-fermat-element} identifies the Fermat elements
with the elements satisfying the order and trace conditions in the
algorithm.  Among the $63$ possible strict lifts, exactly $40$ contain
such an element, as one checks from their order-three elements.  Their
restricted centralizer actions give the groups listed in
Table~\ref{tab:threefold}.  Every smooth cubic threefold belongs to the
family associated with its linear automorphism group, which is
saturated.
\end{proof}

The Fermat rank in Table~\ref{tab:threefold} is one less than that of the
source fourfold family.  Its possible values are \(0,1,2,5\), and rank \(5\)
gives the Fermat cubic threefold.  The cases of ranks \(1\) and \(2\) also follow from
the classifications of smooth cubic surfaces and plane cubics
\cite{Hosoh1997}, \cite{Dolgachev2012}, \cite{yang2024automorphism}, \cite{FWZ26}.

\subsection{Incidence relations and A-maximal families}
\label{subsec:threefold-incidence}

The suspension identifies the incidence relations among the corresponding
cubic threefold and cubic fourfold families.

\begin{prop}\label{prop:threefold-containment}
Let $\calF,\calF'$ be saturated families of smooth cubic
threefolds, and let $\widehat{\calF},\widehat{\calF}'$ be the
corresponding cubic fourfold families under suspension. In the inclusions
below, we identify each family with its image in moduli after forgetting
the group action.
Then $\calF\subseteq\calF'$ if and only if
$\widehat{\calF}\subseteq\widehat{\calF}'$.
\end{prop}

\begin{proof}
An inclusion $\calF\subseteq\calF'$ remains an inclusion after
adding the $x_6^3$-term.  Conversely, suppose that
$\widehat{\calF}\subseteq\widehat{\calF}'$ and take any smooth
form $F$ representing a point of $\calF$. Then $F+x_6^3$ is linearly equivalent to $F'+x_6^3$
for some $F'$ in $\calF'$.  Proposition~\ref{prop:threefold-suspension-map}
gives $F\sim_{\mathrm{lin}}F'$, and hence $\calF\subseteq\calF'$.
\end{proof}

The incidence relation is therefore the restriction of the relation in
Proposition~\ref{prop:fourfold-containment} to the $40$ corresponding
cubic fourfold families.  Its strict part consists of $260$ pairs, of
which $83$ have no intermediate classified family.  The same relation
follows from direct tests of conjugate containment for every pair of
corresponding five-dimensional strict lifts satisfying the necessary dimension
and order conditions; see \path{gap_classification/gap_threefold/}\cite{FuWangZhengAutcub4fold2026}.
See also \path{anc/threefold.txt}.
The A-maximal cubic threefold families are
Nos.~$1,2,3,5,24,39,40$ in Table~\ref{tab:threefold}, corresponding to
Nos.~$1,7,8,24,96,154,156$ in Table~\ref{tab:fourfold}, respectively;
they are listed in Table~\ref{tab:threefold-action-maximal}.

The groups in Nos.~$1,2,3,5,24,40$ are the six maximal abstract groups
classified in \cite[Theorem~1.1]{LiYuThreefold}.
The additional A-maximal family, No.~$39$, has group $C_8$.
Although $C_8$ is abstractly a subgroup of $C_{16}$, its strict lift is not
conjugate to a subgroup of the strict lift of family~$40$.

\subsection{Proof of the main theorem}

\begin{proof}[Proof of Theorem~\ref{thm:threefold}]
Proposition~\ref{prop:threefold-extraction} gives precisely the \(40\)
saturated families in Table~\ref{tab:threefold}.  For each
\(H=\Aut(F)\), the exact sequence \eqref{eq:linear-projective} gives
\[
 \Aut(X_F)=H/\langle\omega I_5\rangle.
\]
Every listed family contains a smooth cubic threefold, and the preceding
argument proves that the list of saturated families is complete.
\end{proof}

\section{Classification results}\label{sec:tables}

\subsection{Change of field}
\label{subsec:change-of-ground-field}

We explain Remark~\ref{rmk:characteristic-zero}.
In this subsection, \(d\) and \(n\) denote
the degree and dimension of the hypersurface and \(N=n+2\) the number of
homogeneous coordinates with
 $n\geq2, d\geq3$, and $(n,d)\neq (2,4)$.

\begin{lem}\label{lem:automorphisms-base-change}
Let \(K\subset L\) be algebraically closed fields of characteristic
zero, and let \(F\in K[x_1,\ldots,x_N]_d\) define a smooth
hypersurface \(X_F\subset\PP^{n+1}_K\). Base change induces
isomorphisms
\[
\Aut_K(F)\xrightarrow{\sim}\Aut_L(F),\qquad
\Aut_K(X_F)\xrightarrow{\sim}\Aut_L((X_F)_L).
\]
\end{lem}

\begin{proof}
By \cite[Theorem~1 and~2]{Matsumura1963OnTA}, \(\Aut(X_F)\) is projective linear
and finite.
The exact sequence \eqref{eq:linear-projective} shows that
\(\Aut(F)\) is finite as well. The corresponding
automorphism group schemes are finite over \(K\).
The automorphism group scheme of \((X_F)_L\) is the base change
of that of \(X_F\) \cite[Tag~0DTR]{stacks-project}.
The same holds for the group scheme associated with \(\Aut(F)\),
which is defined by \(F(xA)=F(x)\) in \(\GL_N\).

For any finite \(K\)-scheme \(Z\), the map \(Z(K)\to Z(L)\)
is bijective since both $K$ and $L$ are algebraically closed.
Applying this to the two automorphism group schemes
proves the assertions.
\end{proof}

For a finite subgroup \(H\leq\GL(N,\overline{\mathbb Q})\) and
an algebraically closed extension \(k/\overline{\mathbb Q}\), put
\[
 W_{d,k}(H)=\bigl(k[x_1,\ldots,x_N]_d\bigr)^H.
\]
Since invariance is given by linear equations over
\(\overline{\mathbb Q}\), \[W_{d,k}(H)=W_{d,\overline{\mathbb Q}}(H)
 \otimes_{\overline{\mathbb Q}}k.\]
The smooth locus is also defined over \(\overline{\mathbb Q}\) since the discriminant locus is so. Therefore, the base change preserves the smooth locus.

\begin{lem}
\label{lem:finite-group-action-descent}
Let \(\overline{\mathbb Q}\subset k\) be an extension of algebraically
closed fields. Every finite subgroup \(H\leq\GL(N,k)\) is conjugate
in \(\GL(N,k)\) to a subgroup
\(H_1\leq\GL(N,\overline{\mathbb Q})\).

Consequently, the $k$-family \(\calF_{H,k}\) with its specified
\(H\)-action is linearly equivalent to the base change of
a $\overline{\mathbb{Q}}$-family \(\calF_{H_1}\) with its specified \(H_1\)-action, via such a conjugacy.
\end{lem}

\begin{proof}
Every representation of a finite group over \(k\) descends to
\(\overline{\mathbb Q}\). Choosing a basis over
\(\overline{\mathbb Q}\) gives \(H_1\), the corresponding change
of coordinates identifies the invariant spaces and their smooth loci.
\end{proof}

By Lemma~\ref{lem:finite-group-action-descent}, after a linear
change of coordinates we may assume that the prescribed finite
group \(H\leq\GL(N,k)\) lies in
\(\GL(N,\overline{\mathbb Q})\) in the following proposition.

\begin{prop}\label{prop:classification-base-change}
Let \(n\geq1\), \(d\geq3\), and \((n,d)\neq(1,3),(2,4)\).
Let \(k/\overline{\mathbb Q}\) be an algebraically closed extension.
For families of smooth degree-\(d\) hypersurfaces in \(\PP^{n+1}\)
with finite linear group actions:
\begin{enumerate}
\item The family \(\mathcal F_H\) with specified \(H\)-action is saturated over \(\overline{\mathbb Q}\) if and only if
its base change \((\mathcal F_H)_k\), with the induced \(H\)-action, is saturated over \(k\).
\item Two saturated families over \(\overline{\mathbb Q}\) are
 equivalent if and only if their base changes to \(k\) are equivalent.
\end{enumerate}
\end{prop}

\begin{proof}
For~(1), fix \(H\leq\GL(N,\overline{\mathbb Q})\) and put
\(W=W_{d,\overline{\mathbb Q}}(H)\). Then, \(W_k^{\mathrm{sm}}=W^{\mathrm{sm}}\times_{\overline{\mathbb Q}}k\).
Define
\[
 U_H=\{x\in W^{\mathrm{sm}}:
 \Aut_{\overline{\kappa(x)}}(F_x)=H_{\overline{\kappa(x)}}\},
\]
where \(F_x\) is the form over \(\kappa(x)\) corresponding to \(x\).
Define \(U_{H,k}\) similarly. By \cite[Part~II, Theorem~6.3]{MR1309681}, these loci are open, and we
regard them as open subschemes. The full geometric automorphism
condition is unchanged by field extension, so
\begin{equation}\label{eq:full-group-locus-base-change}
 U_H\times_{\overline{\mathbb Q}}k=U_{H,k}.
\end{equation}
Each locus is nonempty if and only if the corresponding family
is saturated. Since both ground fields
are algebraically closed, equation \eqref{eq:full-group-locus-base-change} gives that 
$U_H(\overline{\mathbb Q})\neq\varnothing$ is equivalent to $U_{H,k}(k)\neq\varnothing$.

For~(2), consider the $\overline{\mathbb{Q}}$-scheme $\mathcal{C}$ defined by conjugating matrices between the two families. Then $\mathcal{C}(\overline{\mathbb{Q}})\neq \varnothing$ is equivalent to $\mathcal{C}(k)\neq \varnothing$.
\end{proof}

Apply the above Proposition \ref{prop:classification-base-change}, we can state and prove Remark \ref{rmk:characteristic-zero} explicitly.

\begin{cor}\label{cor:cubic-classification-characteristic-zero}
Over an arbitrary algebraically closed field $k$ of characteristic zero,
there are exactly $156$ $(40)$ equivalence
classes of saturated $k$-families of smooth cubic fourfolds (cubic threefolds) with specified automorphism group actions.
They are the base changes of the families in
Tables~\ref{tab:fourfold} and~\ref{tab:threefold}.
\end{cor}

\begin{proof}
Choose an embedding \(\overline{\mathbb Q}\hookrightarrow k\).
Apply Proposition~\ref{prop:classification-base-change} to the
extensions \(\overline{\mathbb Q}\subset\CC\) and
\(\overline{\mathbb Q}\subset k\).

Let $\mathcal{S}_{\overline{\mathbb{Q}}}$, $\mathcal{S}_{k}$ and $\mathcal{S}_{\mathbb{C}}$ be sets of all saturated families of smooth cubic fourfolds (cubic threefolds) over $\overline{\mathbb{Q}}$, $k$ and $\mathbb{C}$, respectively.

We consider the following two maps:

\begin{eqnarray*}
&\mathcal{S}_{\overline{\mathbb{Q}}} \longrightarrow \mathcal{S}_{\mathbb{C}},\ \
\mathcal{F}\mapsto \mathcal{F}_{\mathbb{C}},\\
\text{and}\ \ &\mathcal{S}_{\overline{\mathbb{Q}}} \longrightarrow \mathcal{S}_{k},\ \
\mathcal{F}\mapsto \mathcal{F}_{k}.
\end{eqnarray*}

By Proposition \ref{prop:classification-base-change}, the maps are well-defined and injective. By Lemma \ref{lem:finite-group-action-descent}, they are also surjective.
\end{proof}

\subsection{Classification tables}

In the tables, $\calF$ denotes a family and $m=\dim\calF$.
A GAP ID is a SmallGroups identifier; a dash means that no identifier
is available in the library used here.
The symbol $\dagger$ in the dimension column marks an A-maximal family,
with containment understood for strict lifts up to linear conjugacy.

\subsubsection{Cubic fourfolds}

In the table below, $G_s=\Auts(X)$ and $G=\Aut(X)$.  The notation for the
symplectic groups follows \cite[Theorem~1.2]{laza2022automorphisms} and
\cite[\S2]{KOIKE202512}.
The last two columns record liftability and $F$-liftability. By
Corollary~\ref{cor:small-preimage-liftability}, it suffices to examine the
inverse images of representatives of the $G$-conjugacy classes of
$C_3$, $C_9$, and $C_3^2$ subgroups.  The conclusions are also checked
directly using the distinguished scalar subgroup and the commutator and
cube subgroups of each strict lift.  Of the $156$ projective actions,
$134$ are liftable and $131$ are $F$-liftable; see
\path{gap_classification/gap_liftability/}\cite{FuWangZhengAutcub4fold2026} and \path{anc/self_conjugacy.txt}.

{\scriptsize
\setlength{\tabcolsep}{1.5pt}
\begin{longtable}{|c|c|c|c|c|c|c|c|c|c|c|}
\caption{Automorphism groups of smooth cubic fourfolds.}
\label{tab:fourfold}\\
\hline
No. & $\operatorname{rk}S$                 & $G_s$                            & $|G_s|$              & generic index         & $[G:G_s]$ & \multicolumn{1}{c|}{$G$}                             & GAP ID            & $\dim\mathcal F$ & liftable & $F$-liftable \\ \hline
1 & \multirow[t]{8}{*}{$20$} & $C_3^4:A_6$ & $29160$ & $6$ & $6$ & \multicolumn{1}{c|}{$C_3^5:S_6$} & -- & $0^\dagger$ & No & No \\ \cline{1-1} \cline{3-11}
2 &  & \multirow[t]{2}{*}{$A_7$} & \multirow[t]{2}{*}{$2520$} & $1$ & $1$ & \multicolumn{1}{c|}{$A_7$} & -- & $0^\dagger$ & No & No \\ \cline{1-1} \cline{5-11}
3 &  &  &  & $2$ & $2$ & \multicolumn{1}{c|}{$S_7$} & -- & $0^\dagger$ & Yes & Yes \\ \cline{1-1} \cline{3-11}
4 &  & $3^{1+4}:C_2.C_2^2$ & $1944$ & $4$ & $4$ & \multicolumn{1}{c|}{$3^{1+4}:(C_4^2:C_2)$} & -- & $0^\dagger$ & No & No \\ \cline{1-1} \cline{3-11}
5 &  & \multirow[t]{2}{*}{$M_{10}$} & \multirow[t]{2}{*}{$720$} & $1$ & $1$ & \multicolumn{1}{c|}{$M_{10}$} & $(720,765)$ & $0^\dagger$ & No & No \\ \cline{1-1} \cline{5-11}
6 &  &  &  & $1$ & $1$ & \multicolumn{1}{c|}{$M_{10}$} & $(720,765)$ & $0^\dagger$ & No & No \\ \cline{1-1} \cline{3-11}
7 &  & $L_2(11)$ & $660$ & $3$ & $3$ & \multicolumn{1}{c|}{$L_2(11)\times C_3$} & $(1980,57)$ & $0^\dagger$ & Yes & Yes \\ \cline{1-1} \cline{3-11}
8 &  & $A_{3,5}$ & $360$ & $6$ & $6$ & \multicolumn{1}{c|}{$S_{3,5}\times C_3$} & -- & $0^\dagger$ & Yes & Yes \\ \hline
9 & \multirow[t]{12}{*}{$19$} & $3^{1+4}:C_2.C_2$ & $972$ & $2$ & $2$ & \multicolumn{1}{c|}{$3^{1+4}:D_8$} & $(1944,3536)$ & $1$ & No & No \\ \cline{1-1} \cline{3-11}
10 &  & \multirow[t]{2}{*}{$A_6$} & \multirow[t]{2}{*}{$360$} & $1$ & $1$ & \multicolumn{1}{c|}{$A_6$} & $(360,118)$ & $1$ & No & No \\ \cline{1-1} \cline{5-11}
11 &  &  &  & $2$ & $2$ & \multicolumn{1}{c|}{$S_6$} & $(720,763)$ & $1$ & Yes & Yes \\ \cline{1-1} \cline{3-11}
12 &  & \multirow[t]{3}{*}{$L_2(7)$} & \multirow[t]{3}{*}{$168$} & \multirow[t]{3}{*}{$1$} & $1$ & \multicolumn{1}{c|}{$L_2(7)$} & $(168,42)$ & $1$ & Yes & Yes \\ \cline{1-1} \cline{6-11}
13 &  &  &  &  & $2$ & \multicolumn{1}{c|}{$L_2(7):C_2$} & $(336,208)$ & $0^\dagger$ & Yes & Yes \\ \cline{1-1} \cline{6-11}
14 &  &  &  &  & $2$ & \multicolumn{1}{c|}{$L_2(7):C_2$} & $(336,208)$ & $0^\dagger$ & Yes & Yes \\ \cline{1-1} \cline{3-11}
15 &  & \multirow[t]{2}{*}{$S_5$} & \multirow[t]{2}{*}{$120$} & $1$ & $1$ & \multicolumn{1}{c|}{$S_5$} & $(120,34)$ & $1$ & Yes & Yes \\ \cline{1-1} \cline{5-11}
16 &  &  &  & $2$ & $2$ & \multicolumn{1}{c|}{$S_5\times C_2$} & $(240,189)$ & $1$ & Yes & Yes \\ \cline{1-1} \cline{3-11}
17 &  & \multirow[t]{2}{*}{$M_9$} & \multirow[t]{2}{*}{$72$} & \multirow[t]{2}{*}{$1$} & $1$ & \multicolumn{1}{c|}{$M_9$} & $(72,41)$ & $1$ & No & No \\ \cline{1-1} \cline{6-11}
18 &  &  &  &  & $3$ & \multicolumn{1}{c|}{$M_9:C_3$} & $(216,153)$ & $0^\dagger$ & No & No \\ \cline{1-1} \cline{3-11}
19 &  & $N_{72}$ & $72$ & $2$ & $2$ & \multicolumn{1}{c|}{$N_{72}\times C_2$} & $(144,186)$ & $1$ & Yes & Yes \\ \cline{1-1} \cline{3-11}
20 &  & $T_{48}$ & $48$ & $1$ & $1$ & \multicolumn{1}{c|}{$T_{48}$} & $(48,29)$ & $1^\dagger$ & Yes & Yes \\ \hline
21 & \multirow[t]{24}{*}{$18$} & \multirow[t]{4}{*}{$3^{1+4}:C_2$} & \multirow[t]{4}{*}{$486$} & \multirow[t]{4}{*}{$2$} & $2$ & \multicolumn{1}{c|}{$3^{1+4}:C_2^2$} & $(972,812)$ & $2$ & No & No \\ \cline{1-1} \cline{6-11}
22 &  &  &  &  & $4$ & \multicolumn{1}{c|}{$3^{1+4}:(C_4\times C_2)$} & $(1944,3493)$ & $1$ & No & No \\ \cline{1-1} \cline{6-11}
23 &  &  &  &  & $6$ & \multicolumn{2}{c|}{$((C_3\times(C_3^3:C_3)):C_3):(C_2^2)$} & $1$ & No & No \\ \cline{1-1} \cline{6-11}
24 &  &  &  &  & $12$ & \multicolumn{2}{c|}{$((C_3 \times ((C_3^3) : C_3)) : C_3) : (C_4 \times C_2)$} & $0^\dagger$ & No & No \\ \cline{1-1} \cline{3-11}
25 &  & \multirow[t]{4}{*}{$A_{4,3}$} & \multirow[t]{4}{*}{$72$} & \multirow[t]{2}{*}{$1$} & $1$ & \multicolumn{1}{c|}{$A_{4,3}$} & $(72,43)$ & $2$ & No & No \\ \cline{1-1} \cline{6-11}
26 &  &  &  &  & $2$ & \multicolumn{1}{c|}{$A_{4,3}\times C_2$} & $(144,189)$ & $1$ & No & No \\ \cline{1-1} \cline{5-11}
27 &  &  &  & \multirow[t]{2}{*}{$2$} & $2$ & \multicolumn{1}{c|}{$S_{4,3}$} & $(144,183)$ & $2$ & Yes & Yes \\ \cline{1-1} \cline{6-11}
28 &  &  &  &  & $6$ & \multicolumn{1}{c|}{$S_{4,3}\times C_3$} & $(432,745)$ & $1$ & Yes & Yes \\ \cline{1-1} \cline{3-11}
29 &  & \multirow[t]{4}{*}{$A_5$} & \multirow[t]{4}{*}{$60$} & \multirow[t]{2}{*}{$1$} & $1$ & \multicolumn{1}{c|}{$A_5$} & $(60,5)$ & $2$ & Yes & Yes \\ \cline{1-1} \cline{6-11}
30 &  &  &  &  & $3$ & \multicolumn{1}{c|}{$A_5\times C_3$} & $(180,19)$ & $1$ & Yes & Yes \\ \cline{1-1} \cline{5-11}
31 &  &  &  & \multirow[t]{2}{*}{$2$} & $2$ & \multicolumn{1}{c|}{$S_5$} & $(120,34)$ & $2$ & Yes & Yes \\ \cline{1-1} \cline{6-11}
32 &  &  &  &  & $6$ & \multicolumn{1}{c|}{$S_5\times C_3$} & $(360,119)$ & $1$ & Yes & Yes \\ \cline{1-1} \cline{3-11}
33 &  & \multirow[t]{2}{*}{$C_3^2.C_4$} & \multirow[t]{2}{*}{$36$} & $1$ & $1$ & \multicolumn{1}{c|}{$C_3^{2}.C_4$} & $(36,9)$ & $2$ & No & No \\ \cline{1-1} \cline{5-11}
34 &  &  &  & $2$ & $2$ & \multicolumn{1}{c|}{$N_{72}$} & $(72,40)$ & $2$ & Yes & Yes \\ \cline{1-1} \cline{3-11}
35 &  & \multirow[t]{4}{*}{$S_{3,3}$} & \multirow[t]{4}{*}{$36$} & \multirow[t]{2}{*}{$2$} & $2$ & \multicolumn{1}{c|}{$S_{3,3}\times C_2$} & $(72,46)$ & $2$ & Yes & Yes \\ \cline{1-1} \cline{6-11}
36 &  &  &  &  & $6$ & \multicolumn{1}{c|}{$S_{3,3}\times C_6$} & $(216,170)$ & $1$ & Yes & Yes \\ \cline{1-1} \cline{5-11}
37 &  &  &  & \multirow[t]{2}{*}{$2$} & $2$ & \multicolumn{1}{c|}{$N_{72}$} & $(72,40)$ & $2$ & Yes & Yes \\ \cline{1-1} \cline{6-11}
38 &  &  &  &  & $6$ & \multicolumn{1}{c|}{$N_{72}\times C_3$} & $(216,157)$ & $1$ & Yes & Yes \\ \cline{1-1} \cline{3-11}
39 &  & \multirow[t]{3}{*}{$F_{21}$} & \multirow[t]{3}{*}{$21$} & \multirow[t]{3}{*}{$1$} & $1$ & \multicolumn{1}{c|}{$F_{21}$} & $(21,1)$ & $2$ & Yes & Yes \\ \cline{1-1} \cline{6-11}
40 &  &  &  &  & $2$ & \multicolumn{1}{c|}{$C_7:C_6$} & $(42,1)$ & $1$ & Yes & Yes \\ \cline{1-1} \cline{6-11}
41 &  &  &  &  & $6$ & \multicolumn{1}{c|}{$(C_7:C_6)\times C_3$} & $(126,7)$ & $0^\dagger$ & No & No \\ \cline{1-1} \cline{3-11}
42 &  & $\mathrm{Hol}(5)$ & $20$ & $1$ & $1$ & \multicolumn{1}{c|}{$\mathrm{Hol}(5)$} & $(20,3)$ & $2$ & Yes & Yes \\ \cline{1-1} \cline{3-11}
43 &  & \multirow[t]{2}{*}{$\mathrm{QD}_{16}$} & \multirow[t]{2}{*}{$16$} & \multirow[t]{2}{*}{$1$} & $1$ & \multicolumn{1}{c|}{$\mathrm{QD}_{16}$} & $(16,8)$ & $2$ & Yes & Yes \\ \cline{1-1} \cline{6-11}
44 &  &  &  &  & $2$ & \multicolumn{1}{c|}{$(C_8\times C_2): C_2$} & $(32,42)$ & $1^\dagger$ & Yes & Yes \\ \hline
45 & \multirow[t]{7}{*}{$17$} & \multirow[t]{3}{*}{$S_4$} & \multirow[t]{3}{*}{$24$} & \multirow[t]{2}{*}{$1$} & $1$ & \multicolumn{1}{c|}{$S_4$} & $(24,12)$ & $3$ & Yes & Yes \\ \cline{1-1} \cline{6-11}
46 &  &  &  &  & $2$ & \multicolumn{1}{c|}{$S_4\times C_2$} & $(48,48)$ & $2$ & Yes & Yes \\ \cline{1-1} \cline{5-11}
47 &  &  &  & $2$ & $2$ & \multicolumn{1}{c|}{$S_4\times C_2$} & $(48,48)$ & $3$ & Yes & Yes \\ \cline{1-1} \cline{3-11}
48 &  & \multirow[t]{4}{*}{$Q_8$} & \multirow[t]{4}{*}{$8$} & \multirow[t]{4}{*}{$1$} & $1$ & \multicolumn{1}{c|}{$Q_8$} & $(8,4)$ & $3$ & Yes & Yes \\ \cline{1-1} \cline{6-11}
49 &  &  &  &  & $2$ & \multicolumn{1}{c|}{$(C_4\times C_2):C_2$} & $(16,13)$ & $2$ & Yes & Yes \\ \cline{1-1} \cline{6-11}
50 &  &  &  &  & $3$ & \multicolumn{1}{c|}{$\SL(2,3)$} & $(24,3)$ & $1$ & Yes & Yes \\ \cline{1-1} \cline{6-11}
51 &  &  &  &  & $4$ & \multicolumn{1}{c|}{$C_4^2: C_2$} & $(32,11)$ & $1$ & Yes & Yes \\ \hline
52 & \multirow[t]{21}{*}{$16$} & \multirow[t]{6}{*}{$A_{3,3}$} & \multirow[t]{6}{*}{$18$} & \multirow[t]{3}{*}{$1$} & $1$ & \multicolumn{1}{c|}{$A_{3,3}$} & $(18,4)$ & $4$ & No & No \\ \cline{1-1} \cline{6-11}
53 &  &  &  &  & $2$ & \multicolumn{1}{c|}{$A_{3,3}\times C_2$} & $(36,13)$ & $2$ & No & No \\ \cline{1-1} \cline{6-11}
54 &  &  &  &  & $3$ & \multicolumn{1}{c|}{$C_3^2:C_6$} & $(54,5)$ & $1$ & No & No \\ \cline{1-1} \cline{5-11}
55 &  &  &  & \multirow[t]{3}{*}{$2$} & $2$ & \multicolumn{1}{c|}{$S_{3,3}$} & $(36,10)$ & $4$ & Yes & Yes \\ \cline{1-1} \cline{6-11}
56 &  &  &  &  & $6$ & \multicolumn{1}{c|}{$S_{3,3}\times C_3$} & $(108,38)$ & $2$ & Yes & Yes \\ \cline{1-1} \cline{6-11}
57 &  &  &  &  & $6$ & \multicolumn{1}{c|}{$S_{3,3}\times C_3$} & $(108,38)$ & $2$ & Yes & Yes \\ \cline{1-1} \cline{3-11}
58 &  & \multirow[t]{7}{*}{$D_{12}$} & \multirow[t]{7}{*}{$12$} & \multirow[t]{5}{*}{$1$} & $1$ & \multicolumn{1}{c|}{$D_{12}$} & $(12,4)$ & $4$ & Yes & Yes \\ \cline{1-1} \cline{6-11}
59 &  &  &  &  & $2$ & \multicolumn{1}{c|}{$(C_6\times C_2):C_2$} & $(24,8)$ & $2$ & Yes & Yes \\ \cline{1-1} \cline{6-11}
60 &  &  &  &  & $2$ & \multicolumn{1}{c|}{$D_{12}\times C_2$} & $(24,14)$ & $2$ & Yes & Yes \\ \cline{1-1} \cline{6-11}
61 &  &  &  &  & $3$ & \multicolumn{1}{c|}{$S_3\times C_6$} & $(36,12)$ & $2$ & Yes & Yes \\ \cline{1-1} \cline{6-11}
62 &  &  &  &  & $6$ & \multicolumn{1}{c|}{$((C_6\times C_2):C_2)\times C_3$} & $(72,30)$ & $1$ & Yes & Yes \\ \cline{1-1} \cline{5-11}
63 &  &  &  & \multirow[t]{2}{*}{$2$} & $2$ & \multicolumn{1}{c|}{$D_{12}\times C_2$} & $(24,14)$ & $4$ & Yes & Yes \\ \cline{1-1} \cline{6-11}
64 &  &  &  &  & $6$ & \multicolumn{1}{c|}{$S_3\times C_6\times C_2$} & $(72,48)$ & $2$ & Yes & Yes \\ \cline{1-1} \cline{3-11}
65 &  & \multirow[t]{5}{*}{$A_4$} & \multirow[t]{5}{*}{$12$} & \multirow[t]{3}{*}{$1$} & $1$ & \multicolumn{1}{c|}{$A_4$} & $(12,3)$ & $4$ & Yes & Yes \\ \cline{1-1} \cline{6-11}
66 &  &  &  &  & $2$ & \multicolumn{1}{c|}{$A_4\times C_2$} & $(24,13)$ & $3$ & Yes & Yes \\ \cline{1-1} \cline{6-11}
67 &  &  &  &  & $3$ & \multicolumn{1}{c|}{$A_4\times C_3$} & $(36,11)$ & $2$ & Yes & Yes \\ \cline{1-1} \cline{5-11}
68 &  &  &  & \multirow[t]{2}{*}{$2$} & $2$ & \multicolumn{1}{c|}{$S_4$} & $(24,12)$ & $4$ & Yes & Yes \\ \cline{1-1} \cline{6-11}
69 &  &  &  &  & $6$ & \multicolumn{1}{c|}{$S_4\times C_3$} & $(72,42)$ & $2$ & Yes & Yes \\ \cline{1-1} \cline{3-11}
70 &  & \multirow[t]{3}{*}{$D_{10}$} & \multirow[t]{3}{*}{$10$} & \multirow[t]{3}{*}{$1$} & $1$ & \multicolumn{1}{c|}{$D_{10}$} & $(10,1)$ & $4$ & Yes & Yes \\ \cline{1-1} \cline{6-11}
71 &  &  &  &  & $2$ & \multicolumn{1}{c|}{$D_{20}$} & $(20,4)$ & $2$ & Yes & Yes \\ \cline{1-1} \cline{6-11}
72 &  &  &  &  & $3$ & \multicolumn{1}{c|}{$D_{10}\times C_3$} & $(30,2)$ & $2$ & Yes & Yes \\ \hline
73 & \multirow[t]{4}{*}{$15$} & \multirow[t]{4}{*}{$D_8$} & \multirow[t]{4}{*}{$8$} & \multirow[t]{4}{*}{$1$} & $1$ & \multicolumn{1}{c|}{$D_8$} & $(8,3)$ & $5$ & Yes & Yes \\ \cline{1-1} \cline{6-11}
74 &  &  &  &  & $2$ & \multicolumn{1}{c|}{$D_8\times C_2$} & $(16,11)$ & $4$ & Yes & Yes \\ \cline{1-1} \cline{6-11}
75 &  &  &  &  & $2$ & \multicolumn{1}{c|}{$(C_4\times C_2):C_2$} & $(16,13)$ & $3$ & Yes & Yes \\ \cline{1-1} \cline{6-11}
76 &  &  &  &  & $2$ & \multicolumn{1}{c|}{$D_{16}$} & $(16,7)$ & $2$ & Yes & Yes \\ \hline
77 & \multirow[t]{20}{*}{$14$} & \multirow[t]{6}{*}{$C_4$} & \multirow[t]{6}{*}{$4$} & \multirow[t]{6}{*}{$1$} & $1$ & \multicolumn{1}{c|}{$C_{4}$} & $(4,1)$ & $6$ & Yes & Yes \\ \cline{1-1} \cline{6-11}
78 &  &  &  &  & $2$ & \multicolumn{1}{c|}{$D_8$} & $(8,3)$ & $5$ & Yes & Yes \\ \cline{1-1} \cline{6-11}
79 &  &  &  &  & $2$ & \multicolumn{1}{c|}{$C_{4} \times C_{2}$} & $(8,2)$ & $4$ & Yes & Yes \\ \cline{1-1} \cline{6-11}
80 &  &  &  &  & $2$ & \multicolumn{1}{c|}{$D_8$} & $(8,3)$ & $3$ & Yes & Yes \\ \cline{1-1} \cline{6-11}
81 &  &  &  &  & $4$ & \multicolumn{1}{c|}{$C_{4} \times C_{4}$} & $(16,2)$ & $2$ & Yes & Yes \\ \cline{1-1} \cline{6-11}
82 &  &  &  &  & $4$ & \multicolumn{1}{c|}{$C_8:C_2$} & $(16,6)$ & $2$ & Yes & Yes \\ \cline{1-1} \cline{3-11}
83 &  & \multirow[t]{14}{*}{$S_3$} & \multirow[t]{14}{*}{$6$} & \multirow[t]{5}{*}{$1$} & $1$ & \multicolumn{1}{c|}{$S_3$} & $(6,1)$ & $6$ & Yes & Yes \\ \cline{1-1} \cline{6-11}
84 &  &  &  &  & $2$ & \multicolumn{1}{c|}{$D_{12}$} & $(12,4)$ & $3$ & Yes & Yes \\ \cline{1-1} \cline{6-11}
85 &  &  &  &  & $2$ & \multicolumn{1}{c|}{$D_{12}$} & $(12,4)$ & $3$ & Yes & Yes \\ \cline{1-1} \cline{6-11}
86 &  &  &  &  & $3$ & \multicolumn{1}{c|}{$S_3\times C_3$} & $(18,3)$ & $3$ & Yes & Yes \\ \cline{1-1} \cline{6-11}
87 &  &  &  &  & $3$ & \multicolumn{1}{c|}{$S_3\times C_3$} & $(18,3)$ & $2$ & Yes & Yes \\ \cline{1-1} \cline{5-11}
88 &  &  &  & \multirow[t]{9}{*}{$2$} & $2$ & \multicolumn{1}{c|}{$D_{12}$} & $(12,4)$ & $6$ & Yes & Yes \\ \cline{1-1} \cline{6-11}
89 &  &  &  &  & $4$ & \multicolumn{1}{c|}{$S_3\times C_4$} & $(24,5)$ & $3$ & Yes & Yes \\ \cline{1-1} \cline{6-11}
90 &  &  &  &  & $6$ & \multicolumn{1}{c|}{$S_3\times C_6$} & $(36,12)$ & $3$ & Yes & Yes \\ \cline{1-1} \cline{6-11}
91 &  &  &  &  & $6$ & \multicolumn{1}{c|}{$S_3\times C_6$} & $(36,12)$ & $3$ & Yes & Yes \\ \cline{1-1} \cline{6-11}
92 &  &  &  &  & $6$ & \multicolumn{1}{c|}{$S_3\times C_6$} & $(36,12)$ & $2$ & Yes & Yes \\ \cline{1-1} \cline{6-11}
93 &  &  &  &  & $8$ & \multicolumn{1}{c|}{$S_3\times C_8$} & $(48,4)$ & $1$ & Yes & Yes \\ \cline{1-1} \cline{6-11}
94 &  &  &  &  & $12$ & \multicolumn{1}{c|}{$S_3\times C_{12}$} & $(72,27)$ & $1$ & Yes & Yes \\ \cline{1-1} \cline{6-11}
95 &  &  &  &  & $12$ & \multicolumn{1}{c|}{$S_3\times C_{12}$} & $(72,27)$ & $1$ & Yes & Yes \\ \cline{1-1} \cline{6-11}
96 &  &  &  &  & $24$ & \multicolumn{1}{c|}{$S_3\times C_{24}$} & $(144,69)$ & $0^\dagger$ & Yes & Yes \\ \hline
97 & \multirow[t]{22}{*}{$12$} & \multirow[t]{9}{*}{$C_2^2$} & \multirow[t]{9}{*}{$4$} & \multirow[t]{9}{*}{$1$} & $1$ & \multicolumn{1}{c|}{$C_2^2$} & $(4,2)$ & $8$ & Yes & Yes \\ \cline{1-1} \cline{6-11}
98 &  &  &  &  & $2$ & \multicolumn{1}{c|}{$C_2^3$} & $(8,5)$ & $7$ & Yes & Yes \\ \cline{1-1} \cline{6-11}
99 &  &  &  &  & $2$ & \multicolumn{1}{c|}{$D_8$} & $(8,3)$ & $6$ & Yes & Yes \\ \cline{1-1} \cline{6-11}
100 &  &  &  &  & $2$ & \multicolumn{1}{c|}{$C_2^3$} & $(8,5)$ & $5$ & Yes & Yes \\ \cline{1-1} \cline{6-11}
101 &  &  &  &  & $3$ & \multicolumn{1}{c|}{$C_{6} \times C_{2}$} & $(12,5)$ & $4$ & Yes & Yes \\ \cline{1-1} \cline{6-11}
102 &  &  &  &  & $3$ & \multicolumn{1}{c|}{$A_4$} & $(12,3)$ & $3$ & Yes & Yes \\ \cline{1-1} \cline{6-11}
103 &  &  &  &  & $6$ & \multicolumn{1}{c|}{$D_8\times C_3$} & $(24,10)$ & $3$ & Yes & Yes \\ \cline{1-1} \cline{6-11}
104 &  &  &  &  & $6$ & \multicolumn{1}{c|}{$A_4\times C_2$} & $(24,13)$ & $2$ & Yes & Yes \\ \cline{1-1} \cline{6-11}
105 &  &  &  &  & $6$ & \multicolumn{1}{c|}{$A_4\times C_2$} & $(24,13)$ & $2$ & Yes & No \\ \cline{1-1} \cline{3-11}
106 &  & \multirow[t]{13}{*}{$C_3$} & \multirow[t]{13}{*}{$3$} & \multirow[t]{9}{*}{$1$} & $1$ & \multicolumn{1}{c|}{$C_{3}$} & $(3,1)$ & $8$ & Yes & Yes \\ \cline{1-1} \cline{6-11}
107 &  &  &  &  & $2$ & \multicolumn{1}{c|}{$S_3$} & $(6,1)$ & $4$ & Yes & Yes \\ \cline{1-1} \cline{6-11}
108 &  &  &  &  & $2$ & \multicolumn{1}{c|}{$C_{6}$} & $(6,2)$ & $4$ & Yes & Yes \\ \cline{1-1} \cline{6-11}
109 &  &  &  &  & $2$ & \multicolumn{1}{c|}{$C_{6}$} & $(6,2)$ & $4$ & Yes & Yes \\ \cline{1-1} \cline{6-11}
110 &  &  &  &  & $3$ & \multicolumn{1}{c|}{$C_3^2$} & $(9,2)$ & $4$ & Yes & Yes \\ \cline{1-1} \cline{6-11}
111 &  &  &  &  & $3$ & \multicolumn{1}{c|}{$C_3^2$} & $(9,2)$ & $3$ & Yes & Yes \\ \cline{1-1} \cline{6-11}
112 &  &  &  &  & $3$ & \multicolumn{1}{c|}{$C_3^2$} & $(9,2)$ & $2$ & No & No \\ \cline{1-1} \cline{6-11}
113 &  &  &  &  & $6$ & \multicolumn{1}{c|}{$C_{6} \times C_{3}$} & $(18,5)$ & $2$ & Yes & Yes \\ \cline{1-1} \cline{6-11}
114 &  &  &  &  & $6$ & \multicolumn{1}{c|}{$C_{6} \times C_{3}$} & $(18,5)$ & $1$ & No & No \\ \cline{1-1} \cline{5-11}
115 &  &  &  & \multirow[t]{4}{*}{$2$} & $2$ & \multicolumn{1}{c|}{$S_3$} & $(6,1)$ & $8$ & Yes & Yes \\ \cline{1-1} \cline{6-11}
116 &  &  &  &  & $6$ & \multicolumn{1}{c|}{$S_3\times C_3$} & $(18,3)$ & $4$ & Yes & Yes \\ \cline{1-1} \cline{6-11}
117 &  &  &  &  & $6$ & \multicolumn{1}{c|}{$S_3\times C_3$} & $(18,3)$ & $4$ & Yes & Yes \\ \cline{1-1} \cline{6-11}
118 &  &  &  &  & $6$ & \multicolumn{1}{c|}{$S_3\times C_3$} & $(18,3)$ & $3$ & Yes & Yes \\ \hline
119 & \multirow[t]{13}{*}{$8$} & \multirow[t]{13}{*}{$C_2$} & \multirow[t]{13}{*}{$2$} & \multirow[t]{13}{*}{$1$} & $1$ & \multicolumn{1}{c|}{$C_{2}$} & $(2,1)$ & $12$ & Yes & Yes \\ \cline{1-1} \cline{6-11}
120 &  &  &  &  & $2$ & \multicolumn{1}{c|}{$C_2^2$} & $(4,2)$ & $10$ & Yes & Yes \\ \cline{1-1} \cline{6-11}
121 &  &  &  &  & $2$ & \multicolumn{1}{c|}{$C_2^2$} & $(4,2)$ & $8$ & Yes & Yes \\ \cline{1-1} \cline{6-11}
122 &  &  &  &  & $2$ & \multicolumn{1}{c|}{$C_2^2$} & $(4,2)$ & $6$ & Yes & Yes \\ \cline{1-1} \cline{6-11}
123 &  &  &  &  & $3$ & \multicolumn{1}{c|}{$C_{6}$} & $(6,2)$ & $6$ & Yes & Yes \\ \cline{1-1} \cline{6-11}
124 &  &  &  &  & $4$ & \multicolumn{1}{c|}{$C_{4} \times C_{2}$} & $(8,2)$ & $5$ & Yes & Yes \\ \cline{1-1} \cline{6-11}
125 &  &  &  &  & $6$ & \multicolumn{1}{c|}{$C_{6} \times C_{2}$} & $(12,5)$ & $5$ & Yes & Yes \\ \cline{1-1} \cline{6-11}
126 &  &  &  &  & $3$ & \multicolumn{1}{c|}{$C_{6}$} & $(6,2)$ & $4$ & Yes & Yes \\ \cline{1-1} \cline{6-11}
127 &  &  &  &  & $4$ & \multicolumn{1}{c|}{$C_{4} \times C_{2}$} & $(8,2)$ & $4^\dagger$ & Yes & Yes \\ \cline{1-1} \cline{6-11}
128 &  &  &  &  & $6$ & \multicolumn{1}{c|}{$C_{6} \times C_{2}$} & $(12,5)$ & $3$ & Yes & Yes \\ \cline{1-1} \cline{6-11}
129 &  &  &  &  & $8$ & \multicolumn{1}{c|}{$C_{8} \times C_{2}$} & $(16,5)$ & $2$ & Yes & Yes \\ \cline{1-1} \cline{6-11}
130 &  &  &  &  & $12$ & \multicolumn{1}{c|}{$C_{12} \times C_{2}$} & $(24,9)$ & $2$ & Yes & Yes \\ \cline{1-1} \cline{6-11}
131 &  &  &  &  & $12$ & \multicolumn{1}{c|}{$C_{12} \times C_{2}$} & $(24,9)$ & $1$ & Yes & Yes \\ \hline
132 & \multirow[t]{25}{*}{$0$} & \multirow[t]{25}{*}{$C_1$} & \multirow[t]{25}{*}{$1$} & \multirow[t]{25}{*}{$1$} & $1$ & \multicolumn{1}{c|}{$1$} & $(1,1)$ & $20$ & Yes & Yes \\ \cline{1-1} \cline{6-11}
133 &  &  &  &  & $2$ & \multicolumn{1}{c|}{$C_{2}$} & $(2,1)$ & $14$ & Yes & Yes \\ \cline{1-1} \cline{6-11}
134 &  &  &  &  & $2$ & \multicolumn{1}{c|}{$C_{2}$} & $(2,1)$ & $10$ & Yes & Yes \\ \cline{1-1} \cline{6-11}
135 &  &  &  &  & $3$ & \multicolumn{1}{c|}{$C_{3}$} & $(3,1)$ & $10$ & Yes & Yes \\ \cline{1-1} \cline{6-11}
136 &  &  &  &  & $3$ & \multicolumn{1}{c|}{$C_{3}$} & $(3,1)$ & $7$ & Yes & Yes \\ \cline{1-1} \cline{6-11}
137 &  &  &  &  & $4$ & \multicolumn{1}{c|}{$C_{4}$} & $(4,1)$ & $7$ & Yes & Yes \\ \cline{1-1} \cline{6-11}
138 &  &  &  &  & $4$ & \multicolumn{1}{c|}{$C_{4}$} & $(4,1)$ & $7$ & Yes & Yes \\ \cline{1-1} \cline{6-11}
139 &  &  &  &  & $6$ & \multicolumn{1}{c|}{$C_{6}$} & $(6,2)$ & $7$ & Yes & Yes \\ \cline{1-1} \cline{6-11}
140 &  &  &  &  & $3$ & \multicolumn{1}{c|}{$C_{3}$} & $(3,1)$ & $6$ & Yes & No \\ \cline{1-1} \cline{6-11}
141 &  &  &  &  & $6$ & \multicolumn{1}{c|}{$C_{6}$} & $(6,2)$ & $5$ & Yes & Yes \\ \cline{1-1} \cline{6-11}
142 &  &  &  &  & $6$ & \multicolumn{1}{c|}{$C_{6}$} & $(6,2)$ & $4$ & Yes & Yes \\ \cline{1-1} \cline{6-11}
143 &  &  &  &  & $6$ & \multicolumn{1}{c|}{$C_{6}$} & $(6,2)$ & $3$ & Yes & No \\ \cline{1-1} \cline{6-11}
144 &  &  &  &  & $8$ & \multicolumn{1}{c|}{$C_{8}$} & $(8,1)$ & $3$ & Yes & Yes \\ \cline{1-1} \cline{6-11}
145 &  &  &  &  & $8$ & \multicolumn{1}{c|}{$C_{8}$} & $(8,1)$ & $3$ & Yes & Yes \\ \cline{1-1} \cline{6-11}
146 &  &  &  &  & $8$ & \multicolumn{1}{c|}{$C_{8}$} & $(8,1)$ & $3$ & Yes & Yes \\ \cline{1-1} \cline{6-11}
147 &  &  &  &  & $12$ & \multicolumn{1}{c|}{$C_{12}$} & $(12,2)$ & $3$ & Yes & Yes \\ \cline{1-1} \cline{6-11}
148 &  &  &  &  & $12$ & \multicolumn{1}{c|}{$C_{12}$} & $(12,2)$ & $3$ & Yes & Yes \\ \cline{1-1} \cline{6-11}
149 &  &  &  &  & $12$ & \multicolumn{1}{c|}{$C_{12}$} & $(12,2)$ & $2$ & Yes & Yes \\ \cline{1-1} \cline{6-11}
150 &  &  &  &  & $12$ & \multicolumn{1}{c|}{$C_{12}$} & $(12,2)$ & $2$ & Yes & Yes \\ \cline{1-1} \cline{6-11}
151 &  &  &  &  & $16$ & \multicolumn{1}{c|}{$C_{16}$} & $(16,1)$ & $1$ & Yes & Yes \\ \cline{1-1} \cline{6-11}
152 &  &  &  &  & $16$ & \multicolumn{1}{c|}{$C_{16}$} & $(16,1)$ & $1^\dagger$ & Yes & Yes \\ \cline{1-1} \cline{6-11}
153 &  &  &  &  & $24$ & \multicolumn{1}{c|}{$C_{24}$} & $(24,2)$ & $1$ & Yes & Yes \\ \cline{1-1} \cline{6-11}
154 &  &  &  &  & $24$ & \multicolumn{1}{c|}{$C_{24}$} & $(24,2)$ & $1^\dagger$ & Yes & Yes \\ \cline{1-1} \cline{6-11}
155 &  &  &  &  & $32$ & \multicolumn{1}{c|}{$C_{32}$} & $(32,1)$ & $0^\dagger$ & Yes & Yes \\ \cline{1-1} \cline{6-11}
156 &  &  &  &  & $48$ & \multicolumn{1}{c|}{$C_{48}$} & $(48,2)$ & $0^\dagger$ & Yes & Yes \\ \hline
\end{longtable}
}

The A-maximal families of Proposition~\ref{prop:fourfold-containment}
are listed in Table~\ref{tab:fourfold-action-maximal}.

\begin{table}[htbp]
\centering
\caption{A-maximal cubic fourfold families.}
\label{tab:fourfold-action-maximal}
\begin{tabular}{@{}rcccr@{}}
\toprule
No. & $m$ & $G_s$ & $G$ & $|G|$ \\
\midrule
1   & 0 & $C_3^4:A_6$                  & $C_3^5:S_6$                                      & $174\,960$ \\
2   & 0 & $A_7$                        & $A_7$                                            & $2\,520$ \\
3   & 0 & $A_7$                        & $S_7$                                            & $5\,040$ \\
4   & 0 & $3^{1+4}:C_2.C_2^2$          & $3^{1+4}:(C_4^2:C_2)$                           & $7\,776$ \\
5   & 0 & $M_{10}$                     & $M_{10}$                                         & $720$ \\
6   & 0 & $M_{10}$                     & $M_{10}$                                         & $720$ \\
7   & 0 & $L_2(11)$                    & $L_2(11)\times C_3$                             & $1\,980$ \\
8   & 0 & $A_{3,5}$                    & $S_{3,5}\times C_3$                             & $2\,160$ \\
13  & 0 & $L_2(7)$                     & $L_2(7):C_2$                                     & $336$ \\
14  & 0 & $L_2(7)$                     & $L_2(7):C_2$                                     & $336$ \\
18  & 0 & $M_9$                        & $M_9:C_3$                                        & $216$ \\
20  & 1 & $T_{48}$                     & $T_{48}$                                         & $48$ \\
24  & 0 & $3^{1+4}:C_2$                & $((C_3\times(C_3^3:C_3)):C_3):(C_4\times C_2)$ & $5\,832$ \\
41  & 0 & $F_{21}$                     & $(C_7:C_6)\times C_3$                           & $126$ \\
44  & 1 & $\mathrm{QD}_{16}$           & $(C_8\times C_2):C_2$                           & $32$ \\
96  & 0 & $S_3$                        & $S_3\times C_{24}$                              & $144$ \\
127 & 4 & $C_2$                        & $C_4\times C_2$                                 & $8$ \\
152 & 1 & $1$                          & $C_{16}$                                         & $16$ \\
154 & 1 & $1$                          & $C_{24}$                                         & $24$ \\
155 & 0 & $1$                          & $C_{32}$                                         & $32$ \\
156 & 0 & $1$                          & $C_{48}$                                         & $48$ \\
\bottomrule
\end{tabular}
\end{table}

\FloatBarrier

\subsubsection{Cubic threefolds}

In Table~\ref{tab:threefold}, $r$ is the number of single-variable Fermat
summands of a general defining cubic.
By Corollary~\ref{cor:coprime-liftability}, every cubic threefold action is
$F$-liftable, so abstractly
$\Aut(F)\cong C_3\times\Aut(X)$ in every case.

{\scriptsize
\setlength{\tabcolsep}{2.5pt}
\begin{longtable}{|c|c|c|c|c|c|c|}
\caption{Automorphism groups of smooth cubic threefolds.}
\label{tab:threefold}\\
\hline
No. & $\Aut(X)$ & GAP ID & $|\Aut(X)|$ & $|\Aut(F)|$ & $r$ & $\dim\mathcal F$ \\ \hline
1  & $C_3^4:S_5$                         & --          & $9720$ & $29160$ & $5$ & $0^\dagger$ \\ \hline
2  & $L_2(11)$                            & $(660,13)$  & $660$  & $1980$  & $0$ & $0^\dagger$ \\ \hline
3  & $S_5\times C_3$                     & $(360,119)$ & $360$  & $1080$  & $1$ & $0^\dagger$ \\ \hline
4  & $S_3\times((C_3^2:C_3):C_2)$        & $(324,122)$ & $324$  & $972$   & $2$ & $1$ \\ \hline
5  & $S_3\times((C_3^2:C_3):C_4)$        & $(648,541)$ & $648$  & $1944$  & $2$ & $0^\dagger$ \\ \hline
6  & $S_4\times C_3$                     & $(72,42)$   & $72$   & $216$   & $1$ & $1$ \\ \hline
7  & $A_5$                                & $(60,5)$    & $60$   & $180$   & $0$ & $1$ \\ \hline
8  & $S_5$                                & $(120,34)$  & $120$  & $360$   & $0$ & $1$ \\ \hline
9  & $S_3\times C_6$                     & $(36,12)$   & $36$   & $108$   & $1$ & $1$ \\ \hline
10 & $(S_3\times S_3):C_2$               & $(72,40)$   & $72$   & $216$   & $0$ & $1$ \\ \hline
11 & $S_3\times C_3$                     & $(18,3)$    & $18$   & $54$    & $1$ & $2$ \\ \hline
12 & $S_3\times S_3$                     & $(36,10)$   & $36$   & $108$   & $0$ & $2$ \\ \hline
13 & $D_{12}$                             & $(12,4)$    & $12$   & $36$    & $0$ & $2$ \\ \hline
14 & $(C_6\times C_2):C_2$               & $(24,8)$    & $24$   & $72$    & $0$ & $1$ \\ \hline
15 & $C_6\times C_2$                     & $(12,5)$    & $12$   & $36$    & $1$ & $2$ \\ \hline
16 & $A_4$                                & $(12,3)$    & $12$   & $36$    & $0$ & $2$ \\ \hline
17 & $S_4$                                & $(24,12)$   & $24$   & $72$    & $0$ & $2$ \\ \hline
18 & $D_{10}$                             & $(10,1)$    & $10$   & $30$    & $0$ & $2$ \\ \hline
19 & $S_3$                                & $(6,1)$     & $6$    & $18$    & $0$ & $3$ \\ \hline
20 & $C_6$                                & $(6,2)$     & $6$    & $18$    & $1$ & $3$ \\ \hline
21 & $D_{12}$                             & $(12,4)$    & $12$   & $36$    & $0$ & $3$ \\ \hline
22 & $C_{12}$                             & $(12,2)$    & $12$   & $36$    & $1$ & $1$ \\ \hline
23 & $S_3\times C_4$                     & $(24,5)$    & $24$   & $72$    & $0$ & $1$ \\ \hline
24 & $C_{24}$                             & $(24,2)$    & $24$   & $72$    & $1$ & $0^\dagger$ \\ \hline
25 & $C_2^2$                              & $(4,2)$     & $4$    & $12$    & $0$ & $4$ \\ \hline
26 & $D_8$                                & $(8,3)$     & $8$    & $24$    & $0$ & $3$ \\ \hline
27 & $C_3$                                & $(3,1)$     & $3$    & $9$     & $0$ & $4$ \\ \hline
28 & $C_6$                                & $(6,2)$     & $6$    & $18$    & $0$ & $2$ \\ \hline
29 & $C_3$                                & $(3,1)$     & $3$    & $9$     & $1$ & $4$ \\ \hline
30 & $S_3$                                & $(6,1)$     & $6$    & $18$    & $0$ & $4$ \\ \hline
31 & $C_2$                                & $(2,1)$     & $2$    & $6$     & $0$ & $6$ \\ \hline
32 & $C_2^2$                              & $(4,2)$     & $4$    & $12$    & $0$ & $5$ \\ \hline
33 & $C_4\times C_2$                     & $(8,2)$     & $8$    & $24$    & $0$ & $2$ \\ \hline
34 & $1$                                  & $(1,1)$     & $1$    & $3$     & $0$ & $10$ \\ \hline
35 & $C_2$                                & $(2,1)$     & $2$    & $6$     & $0$ & $7$ \\ \hline
36 & $C_4$                                & $(4,1)$     & $4$    & $12$    & $0$ & $3$ \\ \hline
37 & $C_4$                                & $(4,1)$     & $4$    & $12$    & $0$ & $3$ \\ \hline
38 & $C_8$                                & $(8,1)$     & $8$    & $24$    & $0$ & $1$ \\ \hline
39 & $C_8$                                & $(8,1)$     & $8$    & $24$    & $0$ & $1^\dagger$ \\ \hline
40 & $C_{16}$                             & $(16,1)$    & $16$   & $48$    & $0$ & $0^\dagger$ \\ \hline
\end{longtable}
}

The A-maximal families determined in
\S~\ref{subsec:threefold-incidence} are listed in
Table~\ref{tab:threefold-action-maximal}.

\begin{table}[htbp]
\centering
\caption{A-maximal cubic threefold families.}
\label{tab:threefold-action-maximal}
\begin{tabular}{@{}rccr@{}}
\toprule
No. & $m$ & $G$ & $|G|$ \\
\midrule
1 & 0 & $C_3^4:S_5$                           & $9\,720$ \\
2 & 0 & $L_2(11)$                             & $660$ \\
3 & 0 & $S_5\times C_3$                      & $360$ \\
5 & 0 & $S_3\times((C_3^2:C_3):C_4)$         & $648$ \\
24 & 0 & $C_{24}$                              & $24$ \\
39 & 1 & $C_8$                                 & $8$ \\
40 & 0 & $C_{16}$                              & $16$ \\
\bottomrule
\end{tabular}
\end{table}

\FloatBarrier

\subsection{Explicit representatives}

\subsubsection{Cubic fourfolds}
\label{subsubsection: description for fourfolds}

The numbers below refer to the ``No.'' column of
Table~\ref{tab:fourfold}; the families are presented in that order.
In each entry, $H\leq\GL(6,\CC)$ denotes the strict lift of $G$
with respect to a general defining cubic, so
$G\cong H/\langle\omega I_6\rangle$.

Here $G_s$ is the symplectic group, whereas $G$ is the automorphism
group of a general member.  A symplectic family may therefore have
$G\ne G_s$.  Entries whose non-symplectic index equals the generic index
describe the connected symplectic families.  Their proper subfamilies have the same specified
symplectic action and a larger non-symplectic index.

We use Koike's coordinates and the notation fixed above \cite{KOIKE202512,KOIKE2026}.
For each connected symplectic family, $K^+\leq\GL(6,\CC)$ denotes its
strict symplectic lift containing $\langle\omega I_6\rangle$; for the
corresponding group $H$, we have $K^+=H\cap\SL(6,\CC)$. Whenever a basis
$P_1,\ldots,P_r$ of $W_{K^+}$ is displayed, the ensuing proper subfamilies
of that connected family use the same $P_i$. A new displayed basis redefines
the $P_i$, even when the abstract symplectic group is unchanged.
Additional generators are denoted by $T_i$. We write $I_r$,
$J_{r,s}$, $\mathbf 1_r$, and $E_{ij}$ for the identity, all-ones,
all-ones column, and elementary matrices, respectively. A basis may be
specified by its index set.

\par\Needspace{3\baselineskip}\medskip\noindent\textbf{\boldmath 1.\enspace
$(G_s,G)=(C_3^4:A_6,C_3^5:S_6)$.}\label{fam:fourfold-no-1}
This is the Fermat cubic fourfold; see
\cite[Theorem~1.8(1)]{laza2022automorphisms},
\cite[\S6.1, (1)]{yang2024automorphism}, and
\cite[\S8(i), (8.1)]{KOIKE202512}.

\par\Needspace{3\baselineskip}\medskip\noindent\textbf{\boldmath 2.\enspace $(G_s,G)=(A_7,A_7)$.}\label{fam:fourfold-no-2}
An explicit equation and its automorphism
group were given in \cite[Theorem~6.14]{yang2024automorphism}; see also
\cite[\S8(ii), (2.3)]{KOIKE202512} and
\cite[Theorem~1.2(6)]{he2025cubicfourfoldsorder7automorphism}.

\par\Needspace{3\baselineskip}\medskip\noindent\textbf{\boldmath 3.\enspace $(G_s,G)=(A_7,S_7)$.}\label{fam:fourfold-no-3}
This is the Clebsch--Segre cubic fourfold; see \cite[\S6.1, (11)]{yang2024automorphism},
\cite[\S8(ii), (2.1)]{KOIKE202512}, and
\cite[Theorem~1.2(7)]{he2025cubicfourfoldsorder7automorphism}.
Let $K^+$ be the corresponding strict symplectic lift.  Then
$H=\langle K^+,T\rangle$, where $T=(1\,2)$ is a non-symplectic
involution, and $W_H=W_{K^+}$.

\par\Needspace{3\baselineskip}\medskip\noindent\textbf{\boldmath 4.\enspace
$(G_s,G)=(3^{1+4}:C_2.C_2^2,3^{1+4}:(C_4^2:C_2))$.}\label{fam:fourfold-no-4}
See \cite[Theorem~1.8(3)]{laza2022automorphisms}, \cite[\S6.1, (7)]{yang2024automorphism} and
\cite[\S8(iii), (8.2)--(8.3)]{KOIKE202512}.

\par\Needspace{3\baselineskip}\medskip\noindent\mbox{\textbf{\boldmath 5.\enspace $(G_s,G)=(M_{10},M_{10})$, family~1.}}\label{fam:fourfold-no-5}
See
\cite[Theorem~1.8(4)]{laza2022automorphisms},
\cite[\S6.1, (10)]{yang2024automorphism}, and
\cite[\S8(vi), (8.6)]{KOIKE202512},
where it is defined by $F_+$.

\par\Needspace{3\baselineskip}\medskip\noindent\mbox{\textbf{\boldmath 6.\enspace $(G_s,G)=(M_{10},M_{10})$, family~2.}}\label{fam:fourfold-no-6}
See
\cite[Theorem~1.8(4)]{laza2022automorphisms},
\cite[Theorem~6.15]{yang2024automorphism}, and
\cite[\S8(vi), (8.6)]{KOIKE202512},
where it is defined by $F_-$.

\par\Needspace{3\baselineskip}\medskip\noindent\textbf{\boldmath 7.\enspace
$(G_s,G)=(L_2(11),L_2(11)\times C_3)$.}\label{fam:fourfold-no-7}
This is the cubic suspension of the Klein cubic threefold; see \cite[Theorem~1.8(5)]{laza2022automorphisms}, \cite[\S6.1, (6)]{yang2024automorphism} and
\cite[\S8(iv), (8.4)]{KOIKE202512}.

\par\Needspace{3\baselineskip}\medskip\noindent\textbf{\boldmath 8.\enspace
$(G_s,G)=(A_{3,5},S_{3,5}\times C_3)$.}\label{fam:fourfold-no-8}
See \cite[Theorem~1.8(6)]{laza2022automorphisms}, \cite[\S6.1, (4)]{yang2024automorphism} and
\cite[\S8(v), (8.5)]{KOIKE202512}.

\par\Needspace{3\baselineskip}\medskip\noindent\textbf{\boldmath 9.\enspace
$(G_s,G)=(3^{1+4}:C_2.C_2,3^{1+4}:D_8)$.}\label{fam:fourfold-no-9}
See \cite[\S7.1, (7.1)]{KOIKE202512} and
\cite[Proposition~6.1]{FWZ26}.
Let $K^+$ be the corresponding strict symplectic lift.  Then
$H=\langle K^+,T\rangle$, where $T=(1\,2)$ is a non-symplectic
involution, and $W_H=W_{K^+}$.

\par\Needspace{3\baselineskip}\medskip\noindent\textbf{\boldmath 10.\enspace $(G_s,G)=(A_6,A_6)$.}\label{fam:fourfold-no-10}
See \cite[\S2.5, (2.9)]{KOIKE202512} and
\cite[\S6.1.2]{FWZ26}.

\par\Needspace{3\baselineskip}\medskip\noindent\textbf{\boldmath 11.\enspace $(G_s,G)=(A_6,S_6)$.}\label{fam:fourfold-no-11}
See \cite[\S2.5, (2.6)]{KOIKE202512} and
\cite[\S6.1.1]{FWZ26}.
Let $K^+$ be the corresponding strict symplectic lift.  Then
$H=\langle K^+,T\rangle$, where $T=(1\,2)$ is a non-symplectic
involution, and $W_H=W_{K^+}$.

\par\Needspace{3\baselineskip}\medskip\noindent\textbf{\boldmath 12.\enspace $(G_s,G)=(L_2(7),L_2(7))$.}\label{fam:fourfold-no-12}
See \cite[\S7.2, (7.2)]{KOIKE202512} and
\cite[Theorem~1.2(4), \S4.2]{he2025cubicfourfoldsorder7automorphism}.

\par\Needspace{3\baselineskip}\medskip\noindent\textbf{\boldmath 13.\enspace
$(G_s,G)=(L_2(7),L_2(7):C_2)$, family~1.}\label{fam:fourfold-no-13}
$G\text{-ID}=[336,208]$, $H\text{-ID}=[1008,882]$.
See
\cite[\S6.1, (13)]{yang2024automorphism} and
\cite[Theorem~1.2(5), \S4.2]{he2025cubicfourfoldsorder7automorphism},
where it is the first of the two cubics.

\par\Needspace{3\baselineskip}\medskip\noindent\textbf{\boldmath 14.\enspace
$(G_s,G)=(L_2(7),L_2(7):C_2)$, family~2.}\label{fam:fourfold-no-14}
$G\text{-ID}=[336,208]$, $H\text{-ID}=[1008,882]$.
It is the second of the two cubics in
\cite[Theorem~1.2(5), \S4.2]{he2025cubicfourfoldsorder7automorphism}.

\par\Needspace{3\baselineskip}\medskip\noindent\textbf{\boldmath 15.\enspace $(G_s,G)=(S_5,S_5)$.}\label{fam:fourfold-no-15}
See \cite[\S7.3, (7.6)]{KOIKE202512} and
\cite[\S6.2.2]{FWZ26}.

\par\Needspace{3\baselineskip}\medskip\noindent\textbf{\boldmath 16.\enspace
$(G_s,G)=(S_5,S_5\times C_2)$.}\label{fam:fourfold-no-16}
See \cite[\S7.3, (7.3)]{KOIKE202512} and
\cite[\S6.2.1]{FWZ26}.
Let $K^+$ be the corresponding strict symplectic lift.  Then
$H=\langle K^+,T\rangle$, where $T=\diag(1,1,1,1,1,-1)$ is a non-symplectic
involution, and $W_H=W_{K^+}$.

\par\Needspace{3\baselineskip}\medskip\noindent\textbf{\boldmath 17.\enspace $(G_s,G)=(M_9,M_9)$.}\label{fam:fourfold-no-17}
See \cite[\S7.4, (7.9)]{KOIKE202512} and
\cite[\S6.3]{FWZ26}.

\par\Needspace{3\baselineskip}\medskip\noindent\textbf{\boldmath 18.\enspace $(G_s,G)=(M_9,M_9:C_3)$.}\label{fam:fourfold-no-18}
$G\text{-ID}=[216,153]$, $H\text{-ID}=[648,533]$.
See \cite[\S6.1, (15)]{yang2024automorphism},
\cite[\S7.4, (7.9)]{KOIKE202512}, and
\cite[\S6.3]{FWZ26}.

\par\Needspace{3\baselineskip}\medskip\noindent\textbf{\boldmath 19.\enspace
$(G_s,G)=(N_{72},N_{72}\times C_2)$.}\label{fam:fourfold-no-19}
See \cite[\S7.5, (7.10)]{KOIKE202512} and
\cite[\S6.4.1]{FWZ26}.
Let $K^+$ be the corresponding strict symplectic lift.  Then
$H=\langle K^+,T\rangle$, where $T=(1\,2)$ is a non-symplectic
involution, and $W_H=W_{K^+}$.

\par\Needspace{3\baselineskip}\medskip\noindent\textbf{\boldmath 20.\enspace $(G_s,G)=(T_{48},T_{48})$.}\label{fam:fourfold-no-20}
A smooth member with automorphism group
$T_{48}\cong\GL(2,3)$ is given in
\cite[\S6.1, (14)]{yang2024automorphism}.
The family is given in \cite[\S7.6, (7.12)]{KOIKE202512} and
\cite[Theorem~1.2]{FWZ26}.

\par\Needspace{3\baselineskip}\medskip\noindent\textbf{\boldmath 21.\enspace $(G_s,G)=(3^{1+4}:C_2,3^{1+4}:C_2^2)$.}\label{fam:fourfold-no-21}

This is the symplectic family with generic non-symplectic index $2$,
described in \cite[(6.1)]{KOIKE202512}.  Its proper subfamilies are families~22--24.

Let $K^+\subset\GL(6,\CC)$ denote the strict symplectic lift associated with Koike's component (6.1) in \cite{KOIKE202512}.  We use the following basis $P_1,\ldots,P_{4}$ of $W_{K^+}$:
\begin{equation*}\begin{aligned}
(P_1,\ldots,P_{4})&=\bigl({}x_{1}x_{2}x_{3},\ x_{1}^{3}+x_{2}^{3}+x_{3}^{3},\ x_{4}x_{5}x_{6},\ x_{4}^{3}+x_{5}^{3}+x_{6}^{3}\bigr)
\end{aligned}\end{equation*}
Here $H=\langle K^+,T\rangle$, where $T=(1\,2)$ is a
non-symplectic involution.  It fixes $W_{K^+}$ pointwise, so $W_H=W_{K^+}$.
\par\Needspace{3\baselineskip}\medskip\noindent\textbf{\boldmath 22.\enspace $(G_s,G)=(3^{1+4}:C_2,\allowbreak 3^{1+4}:(C_4\times C_2))$.}\label{fam:fourfold-no-22}
$G\text{-ID}=[1944,3493]$, $|H|=5832$.
The non-symplectic index is $4$, and the family dimension is $1$.  The projective action is non-liftable.
\par\Needspace{6\baselineskip}\medskip\noindent Here $H=\langle K^+,T\rangle$, where\par\nopagebreak
\SecEightDenseMatrix{T}{\diag\!\left(
\frac{-1}{\sqrt{3}}\smat{\omega^2&1&\omega\\
\omega^2&\omega^2&\omega^2\\
\omega^2&\omega&1},I_3\right)}

\par\smallskip\noindent A basis of $W_H$ is
$\{-3(\sqrt{3}+1)P_1+P_2,P_3,P_4\}$.

\par\Needspace{3\baselineskip}\medskip\noindent\textbf{\boldmath 23.\enspace $(G_s,G)=(3^{1+4}:C_2,\allowbreak ((C_3\times(C_3^3:C_3)):C_3):(C_2^2))$.}\label{fam:fourfold-no-23}
$|G|=2916$, $|H|=8748$.
The non-symplectic index is $6$, and the family dimension is $1$.  The projective action is non-liftable.
\par\Needspace{4\baselineskip}\medskip\noindent Here $H=\langle K^+,T\rangle$, where\par\nopagebreak
\begin{SecEightMatrixBlock}
\SecEightMatrix{T}{\diag\!\left(1,\smat{0&1\\1&0},\omega^2,1,1\right)}
\end{SecEightMatrixBlock}

\par\smallskip\noindent A basis of $W_H$ is $\{P_{1}, P_{2}, P_{4}\}$.

\par\Needspace{3\baselineskip}\medskip\noindent\textbf{\boldmath 24.\enspace $(G_s,G)=(3^{1+4}:C_2,\allowbreak ((C_3\times(C_3^3:C_3)):C_3):(C_4\times C_2))$.}\label{fam:fourfold-no-24}
$|G|=5832$, $|H|=17496$.
An equation and its automorphism group are given in
\cite[\S6.1, (2)]{yang2024automorphism}.
The non-symplectic index is $12$, and the family dimension is $0$.  The projective action is non-liftable.
\par\Needspace{3\baselineskip}\medskip\noindent\textbf{\boldmath 25.\enspace $(G_s,G)=(A_{4,3},A_{4,3})$.}\label{fam:fourfold-no-25}

This is the symplectic family with generic non-symplectic index $1$,
described in \cite[(6.7)]{KOIKE202512}.
For the generators in \cite[(2.2), (6.6)]{KOIKE202512}, we take
$u_3=b^2aba^{-1}b^{-1}ab^{-2}$.
Its proper subfamily is family~26.

Let $K^+\subset\GL(6,\CC)$ denote the strict symplectic lift associated with Koike's component (6.7) in \cite{KOIKE202512}.  We use the following basis $P_1,\ldots,P_{4}$ of $W_{K^+}$:
\begin{align*}
\PBasisLabel{1}&={}8\,x_{1}x_{3}x_{5}+4\,x_{1}x_{3}x_{6}+4\,x_{1}x_{4}x_{5}+2\,x_{1}x_{4}x_{6}+4\,x_{2}x_{3}x_{5}+2\,x_{2}x_{3}x_{6}+2\,x_{2}x_{4}x_{5}+x_{2}x_{4}x_{6},\\[3pt]
\PBasisLabel{2}&={}2x_{1}^{3}+3x_{1}^{2}x_{2}+48x_{1}x_{3}x_{5}+24x_{1}x_{3}x_{6}+24x_{1}x_{4}x_{5}+9x_{1}x_{4}x_{6}+24x_{2}x_{3}x_{5}\\*[-2pt]
&\quad +9x_{2}x_{3}x_{6}+9x_{2}x_{4}x_{5}+2x_{3}^{3}+3x_{3}^{2}x_{4}+2x_{5}^{3}+3x_{5}^{2}x_{6},\\[3pt]
\PBasisLabel{3}&={}x_{1}x_{2}^{2}-36\,x_{1}x_{3}x_{5}-18\,x_{1}x_{3}x_{6}-18\,x_{1}x_{4}x_{5}-7\,x_{1}x_{4}x_{6}\\*[-2pt]
&\quad -18\,x_{2}x_{3}x_{5}-7\,x_{2}x_{3}x_{6}-7\,x_{2}x_{4}x_{5}+x_{3}x_{4}^{2}+x_{5}x_{6}^{2},\\[3pt]
\PBasisLabel{4}&={}24\,x_{1}x_{3}x_{5}+12\,x_{1}x_{3}x_{6}+12\,x_{1}x_{4}x_{5}+6\,x_{1}x_{4}x_{6}+x_{2}^{3}\\*[-2pt]
&\quad +12\,x_{2}x_{3}x_{5}+6\,x_{2}x_{3}x_{6}+6\,x_{2}x_{4}x_{5}+x_{4}^{3}+x_{6}^{3}.
\end{align*}
\par\Needspace{3\baselineskip}\medskip\noindent\textbf{\boldmath 26.\enspace $(G_s,G)=(A_{4,3},\allowbreak A_{4,3}\times C_2)$.}\label{fam:fourfold-no-26}
$G\text{-ID}=[144,189]$, $H\text{-ID}=[432,538]$.
The non-symplectic index is $2$, and the family dimension is $1$.  The projective action is non-liftable.
\par\Needspace{6\baselineskip}\medskip\noindent Here $H=\langle K^+,T\rangle$, where\par\nopagebreak
\SecEightDenseMatrix{T}{\left(\begin{smallmatrix}0&0&1&0&0&0\\\frac{1}{3}&-\frac{2}{3}&\frac{1}{3}&\frac{1}{3}&\frac{1}{3}&-\frac{2}{3}\\1&0&0&0&0&0\\\frac{1}{3}&\frac{1}{3}&\frac{1}{3}&-\frac{2}{3}&\frac{1}{3}&-\frac{2}{3}\\0&0&0&0&1&0\\\frac{1}{3}&-\frac{2}{3}&\frac{1}{3}&-\frac{2}{3}&\frac{1}{3}&\frac{1}{3}\end{smallmatrix}\right)}

\par\smallskip\noindent A basis of $W_H$ is $\{P_{1}, 2\,P_{2}-P_{4}, 2\,P_{3}+P_{4}\}$.

\par\Needspace{3\baselineskip}\medskip\noindent\textbf{\boldmath 27.\enspace $(G_s,G)=(A_{4,3},S_{4,3})$.}\label{fam:fourfold-no-27}

This is the symplectic family with generic non-symplectic index $2$,
described in \cite[(6.5)]{KOIKE202512}.  Its proper subfamily is family~28.

Use Koike's coordinates
$\boldsymbol{x}=(x_1,x_2,x_3,x_4)$ and
$\boldsymbol{y}=(y_1,y_2,y_3)$ on the hyperplane
$s_1(\boldsymbol{x})+s_1(\boldsymbol{y})=0$, where
$s_j(\boldsymbol{z})=\sum_i z_i^j$.  Then
\begin{equation*}
W_{K^+}=\left\langle
s_1(\boldsymbol{x})^3,
s_1(\boldsymbol{x})s_2(\boldsymbol{x}),
s_1(\boldsymbol{x})s_2(\boldsymbol{y}),
s_3(\boldsymbol{x}),s_3(\boldsymbol{y})
\right\rangle
\end{equation*}
as in \cite[\S6.2, (6.5)]{KOIKE202512}.  We identify this quotient
with $\CC[x_1,\ldots,x_6]$ by setting
$y_1=x_5$, $y_2=x_6$, and
$y_3=-x_1-x_2-x_3-x_4-x_5-x_6$.
Let $K^+$ be the corresponding strict symplectic lift.  Then
$H=\langle K^+,T\rangle$, where $T=(1\,2)$ is a non-symplectic
involution, and $W_H=W_{K^+}$.
\par\Needspace{3\baselineskip}\medskip\noindent\textbf{\boldmath 28.\enspace $(G_s,G)=(A_{4,3},\allowbreak S_{4,3}\times C_3)$.}\label{fam:fourfold-no-28}
$G\text{-ID}=[432,745]$, $H\text{-ID}=[1296,3545]$.
The non-symplectic index is $6$, and the family dimension is $1$.  The projective action is liftable.
\par\Needspace{6\baselineskip}\medskip\noindent Here $H=\langle K^+,T\rangle$, where\par\nopagebreak
\SecEightDenseMatrix{T}{\left(\begin{smallmatrix}
1&0&0&0&\frac{\omega-1}{3}&\frac{\omega^2-\omega}{3}\\
0&1&0&0&\frac{\omega-1}{3}&\frac{\omega^2-\omega}{3}\\
0&0&1&0&\frac{\omega-1}{3}&\frac{\omega^2-\omega}{3}\\
0&0&0&1&\frac{\omega-1}{3}&\frac{\omega^2-\omega}{3}\\
0&0&0&0&\omega&-\omega\\
0&0&0&0&0&-\omega
\end{smallmatrix}\right)}

\par\medskip\noindent In this notation,\par\nopagebreak
\[
W_H=\left\langle
s_1(\boldsymbol{x})^3,
s_1(\boldsymbol{x})s_2(\boldsymbol{x}),
s_3(\boldsymbol{x}),
s_1(\boldsymbol{x})s_2(\boldsymbol{y})+s_3(\boldsymbol{y})
\right\rangle.
\]

\par\Needspace{3\baselineskip}\medskip\noindent\textbf{\boldmath 29.\enspace $(G_s,G)=(A_5,A_5)$.}\label{fam:fourfold-no-29}

This is the symplectic family with generic non-symplectic index $1$,
described in \cite[(6.10)]{KOIKE202512}.  Its proper subfamily is family~30.

Let $K^+\subset\GL(6,\CC)$ denote the strict symplectic lift associated with Koike's component (6.10) in \cite{KOIKE202512}.  We use the following basis $P_1,\ldots,P_{4}$ of $W_{K^+}$:
\begin{align*}
P_{1}&={}x_{1}x_{2}x_{3}+x_{1}x_{2}x_{5}+x_{1}x_{3}x_{6}+x_{1}x_{4}x_{5}+x_{1}x_{4}x_{6}+x_{2}x_{3}x_{4}\\*[-2pt]
&\quad +x_{2}x_{4}x_{6}+x_{2}x_{5}x_{6}+x_{3}x_{4}x_{5}+x_{3}x_{5}x_{6},
\\[2pt]
P_{2}&={}x_{1}x_{2}x_{4}+x_{1}x_{2}x_{6}+x_{1}x_{3}x_{4}+x_{1}x_{3}x_{5}+x_{1}x_{5}x_{6}+x_{2}x_{3}x_{5}\\*[-2pt]
&\quad +x_{2}x_{3}x_{6}+x_{2}x_{4}x_{5}+x_{3}x_{4}x_{6}+x_{4}x_{5}x_{6},
\\[2pt]
P_{3}&={}x_{1}^{2}x_{2}+x_{1}^{2}x_{3}+x_{1}^{2}x_{4}+x_{1}^{2}x_{5}+x_{1}^{2}x_{6}+x_{1}x_{2}^{2}\\*[-2pt]
&\quad +x_{1}x_{3}^{2}+x_{1}x_{4}^{2}+x_{1}x_{5}^{2}+x_{1}x_{6}^{2}+x_{2}^{2}x_{3}+x_{2}^{2}x_{4}\\*[-2pt]
&\quad +x_{2}^{2}x_{5}+x_{2}^{2}x_{6}+x_{2}x_{3}^{2}+x_{2}x_{4}^{2}+x_{2}x_{5}^{2}+x_{2}x_{6}^{2}\\*[-2pt]
&\quad +x_{3}^{2}x_{4}+x_{3}^{2}x_{5}+x_{3}^{2}x_{6}+x_{3}x_{4}^{2}+x_{3}x_{5}^{2}+x_{3}x_{6}^{2}\\*[-2pt]
&\quad +x_{4}^{2}x_{5}+x_{4}^{2}x_{6}+x_{4}x_{5}^{2}+x_{4}x_{6}^{2}+x_{5}^{2}x_{6}+x_{5}x_{6}^{2},\\[2pt]
P_{4}&=x_{1}^{3}+x_{2}^{3}+x_{3}^{3}+x_{4}^{3}+x_{5}^{3}+x_{6}^{3}.
\end{align*}\par
\par\Needspace{3\baselineskip}\medskip\noindent\textbf{\boldmath 30.\enspace $(G_s,G)=(A_5,\allowbreak A_5\times C_3)$.}\label{fam:fourfold-no-30}
$G\text{-ID}=[180,19]$, $H\text{-ID}=[540,88]$.
The non-symplectic index is $3$, and the family dimension is $1$.  The projective action is liftable.
\par\Needspace{6\baselineskip}\medskip\noindent Here $H=\langle K^+,T\rangle$, where\par\nopagebreak
\SecEightDenseMatrix{T}{\left(\begin{smallmatrix}-\frac{1}{6}+\frac{2}{3}\,\omega&-\frac{1}{6}-\frac{1}{3}\,\omega&-\frac{1}{6}-\frac{1}{3}\,\omega&-\frac{1}{6}-\frac{1}{3}\,\omega&-\frac{1}{6}-\frac{1}{3}\,\omega&-\frac{1}{6}-\frac{1}{3}\,\omega\\-\frac{1}{6}-\frac{1}{3}\,\omega&-\frac{1}{6}-\frac{1}{3}\,\omega&-\frac{1}{6}-\frac{1}{3}\,\omega&-\frac{1}{6}+\frac{2}{3}\,\omega&-\frac{1}{6}-\frac{1}{3}\,\omega&-\frac{1}{6}-\frac{1}{3}\,\omega\\-\frac{1}{6}-\frac{1}{3}\,\omega&-\frac{1}{6}-\frac{1}{3}\,\omega&-\frac{1}{6}-\frac{1}{3}\,\omega&-\frac{1}{6}-\frac{1}{3}\,\omega&-\frac{1}{6}+\frac{2}{3}\,\omega&-\frac{1}{6}-\frac{1}{3}\,\omega\\-\frac{1}{6}-\frac{1}{3}\,\omega&-\frac{1}{6}-\frac{1}{3}\,\omega&-\frac{1}{6}+\frac{2}{3}\,\omega&-\frac{1}{6}-\frac{1}{3}\,\omega&-\frac{1}{6}-\frac{1}{3}\,\omega&-\frac{1}{6}-\frac{1}{3}\,\omega\\-\frac{1}{6}-\frac{1}{3}\,\omega&-\frac{1}{6}-\frac{1}{3}\,\omega&-\frac{1}{6}-\frac{1}{3}\,\omega&-\frac{1}{6}-\frac{1}{3}\,\omega&-\frac{1}{6}-\frac{1}{3}\,\omega&-\frac{1}{6}+\frac{2}{3}\,\omega\\-\frac{1}{6}-\frac{1}{3}\,\omega&-\frac{1}{6}+\frac{2}{3}\,\omega&-\frac{1}{6}-\frac{1}{3}\,\omega&-\frac{1}{6}-\frac{1}{3}\,\omega&-\frac{1}{6}-\frac{1}{3}\,\omega&-\frac{1}{6}-\frac{1}{3}\,\omega\end{smallmatrix}\right)}

\par\smallskip\noindent A basis of $W_H$ is $\{3\,P_{1}+P_{4}, 3\,P_{2}+P_{4}, P_{3}-P_{4}\}$.

\par\Needspace{3\baselineskip}\medskip\noindent\textbf{\boldmath 31.\enspace $(G_s,G)=(A_5,S_5)$.}\label{fam:fourfold-no-31}

This is the symplectic family with generic non-symplectic index $2$,
described in \cite[(6.9)]{KOIKE202512}.  Its proper subfamily is family~32.

Let $K^+\subset\GL(6,\CC)$ denote the strict symplectic lift associated with Koike's component (6.9) in \cite{KOIKE202512}.  We use the following basis $P_1,\ldots,P_{7}$ of $W_{K^+}$:
\begin{alignat*}{2}
P_{1}&={}x_{1}x_{2}x_{3}+x_{1}x_{2}x_{4}+x_{1}x_{2}x_{5}+x_{1}x_{3}x_{4}+x_{1}x_{3}x_{5}+x_{1}x_{4}x_{5}\\[-2pt]
&\quad +x_{2}x_{3}x_{4}+x_{2}x_{3}x_{5}+x_{2}x_{4}x_{5}+x_{3}x_{4}x_{5},&&\\[2pt]
P_{2}&={}x_{1}^{2}x_{2}+x_{1}^{2}x_{3}+x_{1}^{2}x_{4}+x_{1}^{2}x_{5}+x_{1}x_{2}^{2}+x_{1}x_{3}^{2}\\[-2pt]
&\quad +x_{1}x_{4}^{2}+x_{1}x_{5}^{2}+x_{2}^{2}x_{3}+x_{2}^{2}x_{4}+x_{2}^{2}x_{5}+x_{2}x_{3}^{2}\\[-2pt]
&\quad +x_{2}x_{4}^{2}+x_{2}x_{5}^{2}+x_{3}^{2}x_{4}+x_{3}^{2}x_{5}+x_{3}x_{4}^{2}+x_{3}x_{5}^{2}\\[-2pt]
&\quad +x_{4}^{2}x_{5}+x_{4}x_{5}^{2},&&\\[2pt]
P_{3}&={}x_{1}x_{2}x_{6}+x_{1}x_{3}x_{6}+x_{1}x_{4}x_{6}+x_{1}x_{5}x_{6}+x_{2}x_{3}x_{6}+x_{2}x_{4}x_{6}\\[-2pt]
&\quad +x_{2}x_{5}x_{6}+x_{3}x_{4}x_{6}+x_{3}x_{5}x_{6}+x_{4}x_{5}x_{6},&&\\[2pt]
P_{4}&=x_{1}^{3}+x_{2}^{3}+x_{3}^{3}+x_{4}^{3}+x_{5}^{3},&&\\
P_{5}&=x_{1}^{2}x_{6}+x_{2}^{2}x_{6}+x_{3}^{2}x_{6}+x_{4}^{2}x_{6}+x_{5}^{2}x_{6},&&\\
P_{6}&=x_{1}x_{6}^{2}+x_{2}x_{6}^{2}+x_{3}x_{6}^{2}+x_{4}x_{6}^{2}+x_{5}x_{6}^{2},&&\\
P_{7}&=x_{6}^{3}.&&
\end{alignat*}
Here $H=\langle K^+,T\rangle$, where $T=(1\,2)$ is a
non-symplectic involution.  It fixes $W_{K^+}$ pointwise, so $W_H=W_{K^+}$.
\par\Needspace{3\baselineskip}\medskip\noindent\textbf{\boldmath 32.\enspace $(G_s,G)=(A_5,\allowbreak S_5\times C_3)$.}\label{fam:fourfold-no-32}
$G\text{-ID}=[360,119]$, $H\text{-ID}=[1080,490]$.
The non-symplectic index is $6$, and the family dimension is $1$.  The projective action is liftable.
\par\Needspace{4\baselineskip}\medskip\noindent Here $H=\langle K^+,T\rangle$, where\par\nopagebreak
\begin{SecEightMatrixBlock}
\SecEightMatrix{T}{(1\,2\,5)(3\,4)\,\diag(1,1,1,1,1,\omega)}
\end{SecEightMatrixBlock}

\par\smallskip\noindent A basis of $W_H$ is $\{P_i\mid i\in\{1, 2, 4, 7\}\}$.

\par\Needspace{3\baselineskip}\medskip\noindent\textbf{\boldmath 33.\enspace $(G_s,G)=(C_3^2.C_4,C_3^2.C_4)$.}\label{fam:fourfold-no-33}
This family is given in \cite[(6.12)]{KOIKE202512}.

\par\Needspace{3\baselineskip}\medskip\noindent\textbf{\boldmath 34.\enspace $(G_s,G)=(C_3^2.C_4,N_{72})$.}\label{fam:fourfold-no-34}
This family is given in \cite[(6.11)]{KOIKE202512}.
Let $K^+$ be the corresponding strict symplectic lift.  Then
$H=\langle K^+,T\rangle$, where $T=(1\,4)(2\,5)(3\,6)$ is a non-symplectic
involution, and $W_H=W_{K^+}$.

\par\Needspace{3\baselineskip}\medskip\noindent\textbf{\boldmath 35.\enspace $(G_s,G)=(S_{3,3},S_{3,3}\times C_2)$.}\label{fam:fourfold-no-35}

This is the symplectic family with generic non-symplectic index $2$,
described in \cite[(6.13)]{KOIKE202512}.  Its proper subfamily is family~36.

Let $K^+\subset\GL(6,\CC)$ denote the strict symplectic lift associated with Koike's component (6.13) in \cite{KOIKE202512}.  We use the following basis $P_1,\ldots,P_{6}$ of $W_{K^+}$:
\begin{equation*}\begin{aligned}
P_{1}&={}x_{1}x_{2}x_{4}+x_{1}x_{2}x_{5}+x_{1}x_{2}x_{6}+x_{1}x_{3}x_{4}+x_{1}x_{3}x_{5}+x_{1}x_{3}x_{6}\\[-2pt]
&\quad +x_{2}x_{3}x_{4}+x_{2}x_{3}x_{5}+x_{2}x_{3}x_{6},
\\[2pt]
P_{2}&={}2x_{1}^{3}-3x_{1}^{2}x_{2}-3x_{1}^{2}x_{3}-3x_{1}x_{2}^{2}\\[-2pt]
&\quad +12x_{1}x_{2}x_{3}-3x_{1}x_{3}^{2}+2x_{2}^{3}-3x_{2}^{2}x_{3}\\[-2pt]
&\quad -3x_{2}x_{3}^{2}+2x_{3}^{3},\\[2pt]
P_{3}&=x_{1}^{2}x_{4}+x_{1}^{2}x_{5}+x_{1}^{2}x_{6}+x_{2}^{2}x_{4}+x_{2}^{2}x_{5}+x_{2}^{2}x_{6}\\[-2pt]
&\quad +x_{3}^{2}x_{4}+x_{3}^{2}x_{5}+x_{3}^{2}x_{6},\\[2pt]
P_{4}&=x_{4}x_{5}x_{6},\\[2pt]
P_{5}&=x_{4}^{2}x_{5}+x_{4}^{2}x_{6}+x_{4}x_{5}^{2}+x_{4}x_{6}^{2}+x_{5}^{2}x_{6}+x_{5}x_{6}^{2},\\[2pt]
P_{6}&=x_{4}^{3}+x_{5}^{3}+x_{6}^{3}.
\end{aligned}\end{equation*}
Here $H=\langle K^+,T\rangle$, where $T=(1\,2)$ is a
non-symplectic involution.  It fixes $W_{K^+}$ pointwise, so $W_H=W_{K^+}$.
\par\Needspace{3\baselineskip}\medskip\noindent\textbf{\boldmath 36.\enspace $(G_s,G)=(S_{3,3},\allowbreak S_{3,3}\times C_6)$.}\label{fam:fourfold-no-36}
$G\text{-ID}=[216,170]$, $H\text{-ID}=[648,746]$.
The non-symplectic index is $6$, and the family dimension is $1$.  The projective action is liftable.
\par\Needspace{6\baselineskip}\medskip\noindent Here $H=\langle K^+,T\rangle$, where\par\nopagebreak
\SecEightDenseMatrix{T}{\diag\!\left(I_3,
\smat{(2+\omega^2)/3&(-1-2\omega^2)/3&(2+\omega^2)/3\\
(-1-2\omega^2)/3&(2+\omega^2)/3&(2+\omega^2)/3\\
(2+\omega^2)/3&(2+\omega^2)/3&(-1-2\omega^2)/3}\right)}

\par\smallskip\noindent A basis of $W_H$ is $\{P_{1}, P_{2}, P_{3}, 6\,P_{4}+P_{6}, P_{5}\}$.

\par\Needspace{3\baselineskip}\medskip\noindent\textbf{\boldmath 37.\enspace $(G_s,G)=(S_{3,3},N_{72})$.}\label{fam:fourfold-no-37}

This is the symplectic family with generic non-symplectic index $2$,
described in \cite{KOIKE2026}.
Its proper subfamily is family~38.

Let $K^+\subset\GL(6,\CC)$ denote the strict symplectic lift associated with the second $S_{3,3}$ component of the corrigendum \cite{KOIKE2026}.  We use the following basis $P_1,\ldots,P_{7}$ of $W_{K^+}$:
\begin{SecEightAlignedBasis}{l@{\qquad}l}
\multicolumn{2}{@{}l@{}}{\PBasisLabel{1}=x_{1}^{3},\qquad \PBasisLabel{2}=x_{1}^{2}x_{2},\qquad \PBasisLabel{3}=x_{1}x_{2}^{2},}\\
\multicolumn{2}{@{}l@{}}{\PBasisLabel{4}=x_{1}x_{3}^{2}+x_{1}x_{3}x_{4}+x_{1}x_{4}^{2}+x_{1}x_{5}^{2}+x_{1}x_{5}x_{6}+x_{1}x_{6}^{2},}\\
\multicolumn{2}{@{}l@{}}{\PBasisLabel{5}=x_{2}^{3},\qquad \PBasisLabel{6}=x_{2}x_{3}^{2}+x_{2}x_{3}x_{4}+x_{2}x_{4}^{2}+x_{2}x_{5}^{2}+x_{2}x_{5}x_{6}+x_{2}x_{6}^{2},}\\
\multicolumn{2}{@{}l@{}}{\PBasisLabel{7}=x_{3}^{2}x_{4}+x_{3}x_{4}^{2}+x_{5}^{2}x_{6}+x_{5}x_{6}^{2}.}
\end{SecEightAlignedBasis}
Here $H=\langle K^+,T\rangle$, where $T=(3\,4)$ is a
non-symplectic involution.  It fixes $W_{K^+}$ pointwise, so $W_H=W_{K^+}$.
\par\Needspace{3\baselineskip}\medskip\noindent\textbf{\boldmath 38.\enspace $(G_s,G)=(S_{3,3},\allowbreak N_{72}\times C_3)$.}\label{fam:fourfold-no-38}
$G\text{-ID}=[216,157]$, $H\text{-ID}=[648,718]$.
The non-symplectic index is $6$, and the family dimension is $1$.  The projective action is liftable.
\par\Needspace{4\baselineskip}\medskip\noindent Here $H=\langle K^+,T\rangle$, where\par\nopagebreak
\begin{SecEightMatrixBlock}
\SecEightMatrix{T}{\diag\!\left(1,\omega,1,1,\smat{0&1\\1&0}\right)}
\end{SecEightMatrixBlock}

\par\smallskip\noindent A basis of $W_H$ is $\{P_i\mid i\in\{1,4,5,7\}\}$.

\par\Needspace{3\baselineskip}\medskip\noindent\textbf{\boldmath 39.\enspace $(G_s,G)=(F_{21},F_{21})$.}\label{fam:fourfold-no-39}
See
\cite[Theorem~1.2(1), \S4.1]{he2025cubicfourfoldsorder7automorphism}.

\par\Needspace{3\baselineskip}\medskip\noindent\textbf{\boldmath 40.\enspace
$(G_s,G)=(F_{21},C_7:C_6)$.}\label{fam:fourfold-no-40}
$G\text{-ID}=[42,1]$, $H\text{-ID}=[126,7]$.
See
\cite[Theorem~1.2(2), \S4.1]{he2025cubicfourfoldsorder7automorphism}.

\par\Needspace{3\baselineskip}\medskip\noindent\textbf{\boldmath 41.\enspace
$(G_s,G)=(F_{21},(C_7:C_6)\times C_3)$.}\label{fam:fourfold-no-41}
$G\text{-ID}=[126,7]$, $H\text{-ID}=[378,8]$.
This is the Klein cubic fourfold; see
\cite[\S6.1, (9)]{yang2024automorphism} and
\cite[Theorem~1.2(3), \S4.1]{he2025cubicfourfoldsorder7automorphism}.

\par\Needspace{3\baselineskip}\medskip\noindent\textbf{\boldmath 42.\enspace $(G_s,G)=(\mathrm{Hol}(5),\mathrm{Hol}(5))$.}\label{fam:fourfold-no-42}
This family is given in \cite[(6.15)]{KOIKE202512}.

\par\Needspace{3\baselineskip}\medskip\noindent\textbf{\boldmath 43.\enspace $(G_s,G)=(\mathrm{QD}_{16},\mathrm{QD}_{16})$.}\label{fam:fourfold-no-43}
This family is given in \cite[(6.3)]{KOIKE202512}.
Here $H=K^+=\langle\omega I_6,a,b\rangle$, where
\[
a=\diag(1,-1,\zeta_4,-\zeta_4,\zeta_8,\zeta_8^3),
\qquad b=(3\,4)(5\,6).
\]
A basis of $W_{K^+}$ is
\begin{alignat*}{3}
\PBasisLabel{1}&=x_1^3,\qquad &\PBasisLabel{2}&=x_1x_2^2,\qquad &\PBasisLabel{3}&=x_2(x_3^2+x_4^2),\\
\PBasisLabel{4}&=x_1x_3x_4,\qquad &\PBasisLabel{5}&=x_4x_5^2+x_3x_6^2,\qquad &\PBasisLabel{6}&=x_2x_5x_6.
\end{alignat*}

\par\Needspace{3\baselineskip}\medskip\noindent\textbf{\boldmath 44.\enspace $(G_s,G)=(\mathrm{QD}_{16},(C_8\times C_2):C_2)$.}\label{fam:fourfold-no-44}
$G\text{-ID}=[32,42]$, $H\text{-ID}=[96,182]$.
It is the one-dimensional subfamily of family~43 with non-symplectic
index $2$.  Here $H=\langle K^+,T\rangle$, where
$T=\diag(1,1,1,1,-1,1)$ is a non-symplectic involution.  A basis of $W_H$ is
$\{P_i\mid 1\le i\le5\}$.
Here $\dim W_H=5$ and $\dim C_{\GL(6,\CC)}(H)=4$.
The member $P_1+\cdots+P_5=0$ is, up to a permutation of variables,
the example with automorphism group $G$ in
\cite[\S6.1, (12)]{yang2024automorphism}.

\par\Needspace{3\baselineskip}\medskip\noindent\textbf{\boldmath 45.\enspace $(G_s,G)=(S_4,S_4)$.}\label{fam:fourfold-no-45}

This is the symplectic family with generic non-symplectic index $1$,
described in \cite[(5.3)]{KOIKE202512}.  Its proper subfamily is family~46.

Let $K^+\subset\GL(6,\CC)$ denote the strict symplectic lift associated with Koike's component (5.3) in \cite{KOIKE202512}.  We use the following basis $P_1,\ldots,P_{6}$ of $W_{K^+}$:
\begin{align*}
\PBasisLabel{1}&=x_{1}x_{2}x_{3}+x_{1}x_{2}x_{4}+x_{1}x_{3}x_{4}+x_{2}x_{3}x_{4},\\[3pt]
\PBasisLabel{2}&={}x_{1}^{2}x_{2}+x_{1}^{2}x_{3}+x_{1}^{2}x_{4}+x_{1}x_{2}^{2}+x_{1}x_{3}^{2}+x_{1}x_{4}^{2}\\*[-2pt]
&\quad +x_{2}^{2}x_{3}+x_{2}^{2}x_{4}+x_{2}x_{3}^{2}+x_{2}x_{4}^{2}+x_{3}^{2}x_{4}+x_{3}x_{4}^{2},\\[3pt]
\PBasisLabel{3}&={}-\omega\,x_{1}x_{2}x_{5}+x_{1}x_{2}x_{6}-x_{1}x_{3}x_{5}+\omega\,x_{1}x_{3}x_{6}-\omega^{2}\,x_{1}x_{4}x_{5}+\omega^{2}\,x_{1}x_{4}x_{6}\\*[-2pt]
&\quad -\omega^{2}\,x_{2}x_{3}x_{5}+\omega^{2}\,x_{2}x_{3}x_{6}-x_{2}x_{4}x_{5}+\omega\,x_{2}x_{4}x_{6}-\omega\,x_{3}x_{4}x_{5}+x_{3}x_{4}x_{6},\\[3pt]
\PBasisLabel{4}&=x_{1}^{3}+x_{2}^{3}+x_{3}^{3}+x_{4}^{3},\\[3pt]
\PBasisLabel{5}&=x_{1}x_{5}x_{6}+x_{2}x_{5}x_{6}+x_{3}x_{5}x_{6}+x_{4}x_{5}x_{6},\\[3pt]
\PBasisLabel{6}&=x_{6}^{3}-x_{5}^{3}.
\end{align*}
\par\Needspace{3\baselineskip}\medskip\noindent\textbf{\boldmath 46.\enspace $(G_s,G)=(S_4,\allowbreak S_4\times C_2)$, family~1.}\label{fam:fourfold-no-46}
$G\text{-ID}=[48,48]$, $H\text{-ID}=[144,188]$.
The non-symplectic index is $2$, and the family dimension is $2$.  The projective action is liftable.
\par\Needspace{6\baselineskip}\medskip\noindent Here $H=\langle K^+,T\rangle$, where\par\nopagebreak
\SecEightDenseMatrix{T}{\diag\!\left(\tfrac12(J_{4,4}-2I_4),I_2\right)}

\par\smallskip\noindent A basis of $W_H$ is $\{3\,P_{1}-P_{4}, P_{2}+P_{4}, P_{3}, P_{5}, P_{6}\}$.

\par\Needspace{3\baselineskip}\medskip\noindent\textbf{\boldmath 47.\enspace $(G_s,G)=(S_4,S_4\times C_2)$, family~2.}\label{fam:fourfold-no-47}

This is the symplectic family with generic non-symplectic index $2$,
described in \cite[(5.2)]{KOIKE202512}.
Put
\[
D=\diag(1,1,1,1,1,-1),\qquad
a=D(1\,2),\qquad b=D(1\,2\,3\,4).
\]
Then $K^+=\langle\omega I_6,a,b\rangle$ and $H=\langle K^+,T\rangle$,
where $T=(1\,2)$ is a non-symplectic involution.
Write $s_j=\sum_{i=1}^4x_i^j$.  A basis of $W_{K^+}=W_H$ is
\begin{alignat*}{3}
\PBasisLabel{1}&=s_1^3,\qquad &\PBasisLabel{2}&=s_1^2x_5,\qquad &\PBasisLabel{3}&=s_1x_5^2,\\
\PBasisLabel{4}&=s_1x_6^2, &\PBasisLabel{5}&=x_5^3, &\PBasisLabel{6}&=x_5x_6^2,\\
\PBasisLabel{7}&=s_1s_2, &\PBasisLabel{8}&=s_2x_5, &\PBasisLabel{9}&=s_3.
\end{alignat*}
Here $\dim W_H=9$ and $\dim C_{\GL(6,\CC)}(H)=6$, so the family dimension is $3$.

\par\Needspace{3\baselineskip}\medskip\noindent\textbf{\boldmath 48.\enspace $(G_s,G)=(Q_8,Q_8)$.}\label{fam:fourfold-no-48}

This is the symplectic family with generic non-symplectic index $1$,
described in \cite[(5.1)]{KOIKE202512}.  Its proper subfamilies are families~49--51.

Let $K^+\subset\GL(6,\CC)$ denote the strict symplectic lift associated with Koike's component (5.1) in \cite{KOIKE202512}.  We use the following basis $P_1,\ldots,P_{8}$ of $W_{K^+}$:
\begin{equation*}\begin{aligned}
(P_1,\ldots,P_{8})&=\bigl({}x_{1}^{3},\ x_{1}x_{2}^{2},\ x_{1}x_{3}^{2},\ x_{1}x_{4}^{2},\ x_{2}x_{3}x_{4},\ x_{2}x_{5}x_{6},\\[-2pt]
&\quad x_{3}(x_{5}^{2}+x_{6}^{2}),\ x_{4}(x_{6}^{2}-x_{5}^{2})\bigr)
\end{aligned}\end{equation*}
\par\Needspace{3\baselineskip}\medskip\noindent\textbf{\boldmath 49.\enspace $(G_s,G)=(Q_8,\allowbreak (C_4\times C_2):C_2)$.}\label{fam:fourfold-no-49}
$G\text{-ID}=[16,13]$, $H\text{-ID}=[48,47]$.
The non-symplectic index is $2$, and the family dimension is $2$.  The projective action is liftable.
\par\Needspace{4\baselineskip}\medskip\noindent Here $H=\langle K^+,T\rangle$, where\par\nopagebreak
\begin{SecEightMatrixBlock}
\SecEightMatrix{T}{\diag\!\left(1,-1,-1,1,\smat{0&\zeta_4\\-\zeta_4&0}\right)}
\end{SecEightMatrixBlock}

\par\smallskip\noindent A basis of $W_H$ is $\{P_i\mid i\in\{1,2,3,4,5,7,8\}\}$.

\par\Needspace{3\baselineskip}\medskip\noindent\textbf{\boldmath 50.\enspace $(G_s,G)=(Q_8,\allowbreak \SL(2,3))$.}\label{fam:fourfold-no-50}
$G\text{-ID}=[24,3]$, $H\text{-ID}=[72,25]$.
The non-symplectic index is $3$, and the family dimension is $1$.  The projective action is liftable.
\par\Needspace{6\baselineskip}\medskip\noindent Here $H=\langle K^+,T\rangle$, where\par\nopagebreak
\SecEightDenseMatrix{T}{(2\,3\,4)\,\diag\!\left(-\zeta_6,1,1,1,
\smat{(-1-\zeta_4)/2&(-1+\zeta_4)/2\\
(1+\zeta_4)/2&(-1+\zeta_4)/2}\right)}

\par\smallskip\noindent A basis of $W_H$ is $\{P_{1}, P_{2}+\omega\,P_{3}-\zeta_{6}\,P_{4}, P_{5}, 2\,P_{6}+\zeta_{4}\,P_{7}-P_{8}\}$.

\par\Needspace{3\baselineskip}\medskip\noindent\textbf{\boldmath 51.\enspace $(G_s,G)=(Q_8,\allowbreak C_4^2: C_2)$.}\label{fam:fourfold-no-51}
$G\text{-ID}=[32,11]$, $H\text{-ID}=[96,54]$.
The non-symplectic index is $4$, and the family dimension is $1$.  The projective action is liftable.
\par\Needspace{4\baselineskip}\medskip\noindent Here $H=\langle K^+,T\rangle$, where\par\nopagebreak
\begin{SecEightMatrixBlock}
\SecEightMatrix{T}{\diag\!\left(1,-1,\zeta_4\smat{0&1\\1&0},
\smat{0&\zeta_4\\1&0}\right)}
\end{SecEightMatrixBlock}

\par\smallskip\noindent A basis of $W_H$ is $\{P_{1}, P_{2}, P_{3}-P_{4}, P_{5}, P_{7}+\zeta_{4}\,P_{8}\}$.

\par\Needspace{3\baselineskip}\medskip\noindent\textbf{\boldmath 52.\enspace $(G_s,G)=(A_{3,3},A_{3,3})$.}\label{fam:fourfold-no-52}

This is the symplectic family with generic non-symplectic index $1$,
described in \cite[(4.11)]{KOIKE202512}.  Its proper subfamilies are families~53 and~54.

Let $K^+\subset\GL(6,\CC)$ denote the strict symplectic lift associated with Koike's component (4.11) in \cite{KOIKE202512}.  We use the following basis $P_1,\ldots,P_{8}$ of $W_{K^+}$:
\begin{equation*}\begin{aligned}
(P_1,\ldots,P_{8})&=\bigl({}x_{1}x_{2}x_{3},\ x_{1}x_{2}x_{6}+x_{1}x_{3}x_{5}+x_{2}x_{3}x_{4},\ x_{1}^{3}+x_{2}^{3}+x_{3}^{3},\\[-2pt]
&\quad x_{1}^{2}x_{4}+x_{2}^{2}x_{5}+x_{3}^{2}x_{6},\ x_{1}x_{5}x_{6}+x_{2}x_{4}x_{6}+x_{3}x_{4}x_{5},\\[-2pt]
&\quad x_{1}x_{4}^{2}+x_{2}x_{5}^{2}+x_{3}x_{6}^{2},\ x_{4}x_{5}x_{6},\ x_{4}^{3}+x_{5}^{3}+x_{6}^{3}\bigr)
\end{aligned}\end{equation*}
\par\Needspace{3\baselineskip}\medskip\noindent\textbf{\boldmath 53.\enspace $(G_s,G)=(A_{3,3},\allowbreak A_{3,3}\times C_2)$.}\label{fam:fourfold-no-53}
$G\text{-ID}=[36,13]$, $H\text{-ID}=[108,28]$.
The non-symplectic index is $2$, and the family dimension is $2$.  The projective action is non-liftable.
\par\Needspace{4\baselineskip}\medskip\noindent Here $H=\langle K^+,T\rangle$, where\par\nopagebreak
\begin{SecEightMatrixBlock}
\SecEightMatrix{T}{\diag(-1,-1,-1,1,1,1)}
\end{SecEightMatrixBlock}

\par\smallskip\noindent A basis of $W_H$ is $\{P_i\mid i\in\{2, 4, 7, 8\}\}$.

\par\Needspace{3\baselineskip}\medskip\noindent\textbf{\boldmath 54.\enspace $(G_s,G)=(A_{3,3},\allowbreak C_3^2:C_6)$.}\label{fam:fourfold-no-54}
$G\text{-ID}=[54,5]$, $H\text{-ID}=[162,10]$.
The non-symplectic index is $3$, and the family dimension is $1$.  The projective action is non-liftable.
\par\Needspace{4\baselineskip}\medskip\noindent Here $H=\langle K^+,T\rangle$, where\par\nopagebreak
\begin{SecEightMatrixBlock}
\SecEightMatrix{T}{\diag(\omega,\omega,1,1,1,\omega^2)}
\end{SecEightMatrixBlock}

\par\smallskip\noindent A basis of $W_H$ is $\{P_{3}, P_{5}, P_{8}\}$.

\par\Needspace{3\baselineskip}\medskip\noindent\textbf{\boldmath 55.\enspace $(G_s,G)=(A_{3,3},S_{3,3})$.}\label{fam:fourfold-no-55}

This is the symplectic family with generic non-symplectic index $2$,
described in \cite[(4.6)]{KOIKE202512}.  Its proper subfamilies, families~56 and~57, have inequivalent linear actions.

Let $K^+\subset\GL(6,\CC)$ denote the strict symplectic lift associated with Koike's component (4.6) in \cite{KOIKE202512}.  We use the following basis $P_1,\ldots,P_{10}$ of $W_{K^+}$:
\begin{align*}
\PBasisLabel{1}&=x_{1}x_{2}x_{3},\\[3pt]
\PBasisLabel{2}&=x_{1}^{2}x_{2}+x_{1}^{2}x_{3}+x_{1}x_{2}^{2}+x_{1}x_{3}^{2}+x_{2}^{2}x_{3}+x_{2}x_{3}^{2},\\[3pt]
\PBasisLabel{3}&={}x_{1}x_{2}x_{4}+x_{1}x_{2}x_{5}+x_{1}x_{2}x_{6}+x_{1}x_{3}x_{4}+x_{1}x_{3}x_{5}+x_{1}x_{3}x_{6}+x_{2}x_{3}x_{4}+x_{2}x_{3}x_{5}+x_{2}x_{3}x_{6},\\[3pt]
\PBasisLabel{4}&=x_{1}^{3}+x_{2}^{3}+x_{3}^{3},\\[3pt]
\PBasisLabel{5}&=x_{1}^{2}x_{4}+x_{1}^{2}x_{5}+x_{1}^{2}x_{6}+x_{2}^{2}x_{4}+x_{2}^{2}x_{5}+x_{2}^{2}x_{6}+x_{3}^{2}x_{4}+x_{3}^{2}x_{5}+x_{3}^{2}x_{6},\\[3pt]
\PBasisLabel{6}&={}x_{1}x_{4}x_{5}+x_{1}x_{4}x_{6}+x_{1}x_{5}x_{6}+x_{2}x_{4}x_{5}+x_{2}x_{4}x_{6}+x_{2}x_{5}x_{6}+x_{3}x_{4}x_{5}+x_{3}x_{4}x_{6}+x_{3}x_{5}x_{6},\\[3pt]
\PBasisLabel{7}&=x_{1}x_{4}^{2}+x_{1}x_{5}^{2}+x_{1}x_{6}^{2}+x_{2}x_{4}^{2}+x_{2}x_{5}^{2}+x_{2}x_{6}^{2}+x_{3}x_{4}^{2}+x_{3}x_{5}^{2}+x_{3}x_{6}^{2},\\[3pt]
\PBasisLabel{8}&=x_{4}x_{5}x_{6},\\[3pt]
\PBasisLabel{9}&=x_{4}^{2}x_{5}+x_{4}^{2}x_{6}+x_{4}x_{5}^{2}+x_{4}x_{6}^{2}+x_{5}^{2}x_{6}+x_{5}x_{6}^{2},\\[3pt]
\PBasisLabel{10}&=x_{4}^{3}+x_{5}^{3}+x_{6}^{3}.
\end{align*}
\par\smallskip\noindent Here $H=\langle K^+,T\rangle$, where $T=(1\,2)$ is a
non-symplectic involution.  It fixes $W_{K^+}$ pointwise, so $W_H=W_{K^+}$.
\par\Needspace{3\baselineskip}\medskip\noindent\textbf{\boldmath 56.\enspace $(G_s,G)=(A_{3,3},\allowbreak S_{3,3}\times C_3)$, family~1.}\label{fam:fourfold-no-56}
$G\text{-ID}=[108,38]$, $H\text{-ID}=[324,165]$.
The non-symplectic index is $6$, and the family dimension is $2$.  The projective action is liftable.
\par\Needspace{6\baselineskip}\medskip\noindent Here $H=\langle K^+,T\rangle$, where\par\nopagebreak
\SecEightDenseMatrix{T}{\diag\!\left(
\smat{(2+\omega^2)/3&(-1-2\omega^2)/3&(2+\omega^2)/3\\
(-1-2\omega^2)/3&(2+\omega^2)/3&(2+\omega^2)/3\\
(2+\omega^2)/3&(2+\omega^2)/3&(-1-2\omega^2)/3},I_3\right)}

\par\medskip\noindent A basis of $W_H$ is\par\nopagebreak
\begin{SecEightBasisBlock}
\SecEightBasisLine{6\,P_{1}+P_{4},}
\SecEightBasisLine{P_{2},}
\SecEightBasisLine{2\,P_{3}+P_{5},}
\SecEightBasisLine{P_{6},}
\SecEightBasisLine{P_{7},}
\SecEightBasisLine{P_{8},}
\SecEightBasisLine{P_{9},}
\SecEightBasisLine{P_{10}.}
\end{SecEightBasisBlock}

\par\Needspace{3\baselineskip}\medskip\noindent\textbf{\boldmath 57.\enspace $(G_s,G)=(A_{3,3},\allowbreak S_{3,3}\times C_3)$, family~2.}\label{fam:fourfold-no-57}
$G\text{-ID}=[108,38]$, $H\text{-ID}=[324,165]$.
The non-symplectic index is $6$, and the family dimension is $2$.  The projective action is liftable.
\par\Needspace{6\baselineskip}\medskip\noindent Here $H=\langle K^+,T\rangle$, where\par\nopagebreak
\SecEightDenseMatrix{T}{\left(\begin{smallmatrix}0&1&0&0&0&0\\1&0&0&0&0&0\\0&0&1&0&0&0\\-\frac{2}{3}-\frac{1}{3}\,\omega^{2}&-\frac{2}{3}-\frac{1}{3}\,\omega^{2}&-\frac{2}{3}-\frac{1}{3}\,\omega^{2}&\frac{1}{3}-\frac{1}{3}\,\omega^{2}&-\frac{2}{3}-\frac{1}{3}\,\omega^{2}&-\frac{2}{3}-\frac{1}{3}\,\omega^{2}\\-\frac{2}{3}-\frac{1}{3}\,\omega^{2}&-\frac{2}{3}-\frac{1}{3}\,\omega^{2}&-\frac{2}{3}-\frac{1}{3}\,\omega^{2}&-\frac{2}{3}-\frac{1}{3}\,\omega^{2}&\frac{1}{3}-\frac{1}{3}\,\omega^{2}&-\frac{2}{3}-\frac{1}{3}\,\omega^{2}\\-\frac{2}{3}-\frac{1}{3}\,\omega^{2}&-\frac{2}{3}-\frac{1}{3}\,\omega^{2}&-\frac{2}{3}-\frac{1}{3}\,\omega^{2}&-\frac{2}{3}-\frac{1}{3}\,\omega^{2}&-\frac{2}{3}-\frac{1}{3}\,\omega^{2}&\frac{1}{3}-\frac{1}{3}\,\omega^{2}\end{smallmatrix}\right)}

\par\medskip\noindent A basis of $W_H$ is\par\nopagebreak
\begin{SecEightBasisBlock}
\SecEightBasisLine{3\,P_{1}-P_{3}+P_{6}-3\,P_{8},}
\SecEightBasisLine{3\,P_{2}-4\,P_{3}-6\,P_{4}+4\,P_{5},}
\SecEightBasisLine{P_{4}-P_{5}+P_{6}-3\,P_{8},}
\SecEightBasisLine{-2\,P_{6}+2\,P_{7}-3\,P_{10},}
\SecEightBasisLine{6\,P_{8}+P_{10},}
\SecEightBasisLine{P_{9}.}
\end{SecEightBasisBlock}

\par\Needspace{3\baselineskip}\medskip\noindent\textbf{\boldmath 58.\enspace $(G_s,G)=(D_{12},D_{12})$.}\label{fam:fourfold-no-58}

This is the symplectic family with generic non-symplectic index $1$,
described in \cite[(4.3)]{KOIKE202512}.  Its proper subfamilies are families~59--62.

Let $K^+\subset\GL(6,\CC)$ denote the strict symplectic lift associated with Koike's component (4.3) in \cite{KOIKE202512}.  We use the following basis $P_1,\ldots,P_{10}$ of $W_{K^+}$:
\begin{equation*}\begin{aligned}
(P_1,\ldots,P_{10})&=\bigl({}x_{1}x_{3}x_{5},\ x_{1}x_{3}x_{6},\ x_{2}x_{4}x_{5},\ x_{2}x_{4}x_{6},\ x_{1}^{3}+x_{3}^{3},\\[-2pt]
&\quad x_{1}x_{2}^{2}+x_{3}x_{4}^{2},\ x_{5}^{3},\ x_{5}^{2}x_{6},\ x_{5}x_{6}^{2},\ x_{6}^{3}\bigr)
\end{aligned}\end{equation*}
\par\Needspace{3\baselineskip}\medskip\noindent\textbf{\boldmath 59.\enspace $(G_s,G)=(D_{12},\allowbreak (C_6\times C_2):C_2)$.}\label{fam:fourfold-no-59}
$G\text{-ID}=[24,8]$, $H\text{-ID}=[72,30]$.
The non-symplectic index is $2$, and the family dimension is $2$.  The projective action is liftable.
\par\Needspace{4\baselineskip}\medskip\noindent Here $H=\langle K^+,T\rangle$, where\par\nopagebreak
\begin{SecEightMatrixBlock}
\SecEightMatrix{T}{\diag(1,1,1,-1,1,1)}
\end{SecEightMatrixBlock}

\par\smallskip\noindent A basis of $W_H$ is $\{P_i\mid i\in\{1,2,5,6,7,8,9,10\}\}$.

\par\Needspace{3\baselineskip}\medskip\noindent\textbf{\boldmath 60.\enspace $(G_s,G)=(D_{12},\allowbreak D_{12}\times C_2)$, family~1.}\label{fam:fourfold-no-60}
$G\text{-ID}=[24,14]$, $H\text{-ID}=[72,48]$.
The non-symplectic index is $2$, and the family dimension is $2$.  The projective action is liftable.
\par\Needspace{4\baselineskip}\medskip\noindent Here $H=\langle K^+,T\rangle$, where\par\nopagebreak
\begin{SecEightMatrixBlock}
\SecEightMatrix{T}{\diag\!\left(\smat{0&0&1&0\\0&0&0&-1\\1&0&0&0\\0&-1&0&0},-1,1\right)}
\end{SecEightMatrixBlock}

\par\smallskip\noindent A basis of $W_H$ is $\{P_i\mid i\in\{2, 4, 5, 6, 8, 10\}\}$.

\par\Needspace{3\baselineskip}\medskip\noindent\textbf{\boldmath 61.\enspace $(G_s,G)=(D_{12},\allowbreak S_3\times C_6)$.}\label{fam:fourfold-no-61}
$G\text{-ID}=[36,12]$, $H\text{-ID}=[108,42]$.
The non-symplectic index is $3$, and the family dimension is $2$.  The projective action is liftable.
\par\Needspace{4\baselineskip}\medskip\noindent Here $H=\langle K^+,T\rangle$, where\par\nopagebreak
\begin{SecEightMatrixBlock}
\SecEightMatrix{T}{\diag(1,1,1,1,1,\omega)}
\end{SecEightMatrixBlock}

\par\smallskip\noindent A basis of $W_H$ is $\{P_i\mid i\in\{1, 3, 5, 6, 7, 10\}\}$.

\par\Needspace{3\baselineskip}\medskip\noindent\textbf{\boldmath 62.\enspace $(G_s,G)=(D_{12},\allowbreak ((C_6\times C_2):C_2)\times C_3)$.}\label{fam:fourfold-no-62}
$G\text{-ID}=[72,30]$, $H\text{-ID}=[216,139]$.
The non-symplectic index is $6$, and the family dimension is $1$.  The projective action is liftable.
\par\Needspace{4\baselineskip}\medskip\noindent Here $H=\langle K^+,T\rangle$, where\par\nopagebreak
\begin{SecEightMatrixBlock}
\SecEightMatrix{T}{\diag(1,-1,1,1,1,\omega)}
\end{SecEightMatrixBlock}

\par\smallskip\noindent A basis of $W_H$ is $\{P_i\mid i\in\{1, 5, 6, 7, 10\}\}$.

\par\Needspace{3\baselineskip}\medskip\noindent\textbf{\boldmath 63.\enspace $(G_s,G)=(D_{12},D_{12}\times C_2)$, family~2.}\label{fam:fourfold-no-63}

This is the symplectic family with generic non-symplectic index $2$,
described in \cite[(4.2)]{KOIKE202512}.  Its proper subfamily is family~64.

Let $K^+\subset\GL(6,\CC)$ denote the strict symplectic lift associated with Koike's component (4.2) in \cite{KOIKE202512}.  We use the following basis $P_1,\ldots,P_{11}$ of $W_{K^+}$:
\begin{equation*}\begin{aligned}
(P_1,\ldots,P_{11})&=\bigl({}x_{1}x_{2}x_{3},\ x_{1}x_{2}x_{4},\ x_{1}^{2}x_{3}+x_{2}^{2}x_{3},\ x_{1}^{2}x_{4}+x_{2}^{2}x_{4},\ x_{3}^{3},\ x_{3}^{2}x_{4},\\[-2pt]
&\quad x_{3}x_{4}^{2},\ x_{3}x_{5}x_{6},\ x_{4}^{3},\ x_{4}x_{5}x_{6},\ x_{5}^{3}+x_{6}^{3}\bigr)
\end{aligned}\end{equation*}
Here $H=\langle K^+,T\rangle$, where $T=(5\,6)$ is a
non-symplectic involution.  It fixes $W_{K^+}$ pointwise, so $W_H=W_{K^+}$.
\par\Needspace{3\baselineskip}\medskip\noindent\textbf{\boldmath 64.\enspace $(G_s,G)=(D_{12},\allowbreak S_3\times C_6\times C_2)$.}\label{fam:fourfold-no-64}
$G\text{-ID}=[72,48]$, $H\text{-ID}=[216,174]$.
The non-symplectic index is $6$, and the family dimension is $2$.  The projective action is liftable.
\par\Needspace{4\baselineskip}\medskip\noindent Here $H=\langle K^+,T\rangle$, where\par\nopagebreak
\begin{SecEightMatrixBlock}
\SecEightMatrix{T}{\diag\!\left(\smat{0&-1\\-1&0},1,1,\omega,\omega\right)}
\end{SecEightMatrixBlock}

\par\smallskip\noindent A basis of $W_H$ is $\{P_i\mid i\in\{1,2,3,4,5,6,7,9,11\}\}$.

\par\Needspace{3\baselineskip}\medskip\noindent\textbf{\boldmath 65.\enspace $(G_s,G)=(A_4,A_4)$.}\label{fam:fourfold-no-65}

This is the symplectic family with generic non-symplectic index $1$,
described in \cite[(4.16)]{KOIKE202512}.  Its proper subfamilies are families~66 and~67.

Let $K^+\subset\GL(6,\CC)$ denote the strict symplectic lift associated with Koike's component (4.16) in \cite{KOIKE202512}.  We use the following basis $P_1,\ldots,P_{8}$ of $W_{K^+}$:
\begin{SecEightAlignedBasis}{l@{\qquad}l}
\multicolumn{2}{@{}l@{}}{\PBasisLabel{1}=x_{1}x_{2}x_{3}+x_{1}x_{2}x_{4}+x_{1}x_{3}x_{4}+x_{2}x_{3}x_{4},}\\
\multicolumn{2}{@{}l@{}}{\PBasisLabel{2}=x_{1}^{2}x_{2}+x_{1}^{2}x_{3}+x_{1}^{2}x_{4}+x_{1}x_{2}^{2}+x_{1}x_{3}^{2}+x_{1}x_{4}^{2}}\\
\multicolumn{2}{@{}l@{}}{\hphantom{\PBasisLabel{1}=}\quad +x_{2}^{2}x_{3}+x_{2}^{2}x_{4}+x_{2}x_{3}^{2}+x_{2}x_{4}^{2}+x_{3}^{2}x_{4}+x_{3}x_{4}^{2},}\\
\multicolumn{2}{@{}l@{}}{\PBasisLabel{3}=x_{1}x_{2}x_{5}+\omega^{2}\,x_{1}x_{3}x_{5}+\omega\,x_{1}x_{4}x_{5}}\\
\multicolumn{2}{@{}l@{}}{\hphantom{\PBasisLabel{1}=}\quad +\omega\,x_{2}x_{3}x_{5}+\omega^{2}\,x_{2}x_{4}x_{5}+x_{3}x_{4}x_{5},}\\
\multicolumn{2}{@{}l@{}}{\PBasisLabel{4}=x_{1}x_{2}x_{6}+\omega\,x_{1}x_{3}x_{6}+\omega^{2}\,x_{1}x_{4}x_{6}+\omega^{2}\,x_{2}x_{3}x_{6}}\\
\multicolumn{2}{@{}l@{}}{\hphantom{\PBasisLabel{1}=}\quad +\omega\,x_{2}x_{4}x_{6}+x_{3}x_{4}x_{6},}\\
\PBasisLabel{5}=x_{1}^{3}+x_{2}^{3}+x_{3}^{3}+x_{4}^{3}, & \PBasisLabel{6}=x_{1}x_{5}x_{6}+x_{2}x_{5}x_{6}+x_{3}x_{5}x_{6}+x_{4}x_{5}x_{6},\\
\PBasisLabel{7}=x_{5}^{3}, & \PBasisLabel{8}=x_{6}^{3}.
\end{SecEightAlignedBasis}
\par\Needspace{3\baselineskip}\medskip\noindent\textbf{\boldmath 66.\enspace $(G_s,G)=(A_4,\allowbreak A_4\times C_2)$.}\label{fam:fourfold-no-66}
$G\text{-ID}=[24,13]$, $H\text{-ID}=[72,47]$.
The non-symplectic index is $2$, and the family dimension is $3$.  The projective action is liftable.
\par\Needspace{6\baselineskip}\medskip\noindent Here $H=\langle K^+,T\rangle$, where\par\nopagebreak
\SecEightDenseMatrix{T}{\diag\!\left(\tfrac12(J_{4,4}-2I_4),I_2\right)}

\par\medskip\noindent A basis of $W_H$ is\par\nopagebreak
\begin{SecEightBasisBlock}
\SecEightBasisLine{3\,P_{1}-P_{5},}
\SecEightBasisLine{P_{2}+P_{5},}
\SecEightBasisLine{P_{3},}
\SecEightBasisLine{P_{4},}
\SecEightBasisLine{P_{6},}
\SecEightBasisLine{P_{7},}
\SecEightBasisLine{P_{8}.}
\end{SecEightBasisBlock}

\par\Needspace{3\baselineskip}\medskip\noindent\textbf{\boldmath 67.\enspace $(G_s,G)=(A_4,\allowbreak A_4\times C_3)$.}\label{fam:fourfold-no-67}
$G\text{-ID}=[36,11]$, $H\text{-ID}=[108,41]$.
The non-symplectic index is $3$, and the family dimension is $2$.  The projective action is liftable.
\par\Needspace{4\baselineskip}\medskip\noindent Here $H=\langle K^+,T\rangle$, where\par\nopagebreak
\begin{SecEightMatrixBlock}
\SecEightMatrix{T}{\diag(1,1,1,1,1,\omega)}
\end{SecEightMatrixBlock}

\par\smallskip\noindent A basis of $W_H$ is $\{P_i\mid i\in\{1,2,3,5,7,8\}\}$.

\par\Needspace{3\baselineskip}\medskip\noindent\textbf{\boldmath 68.\enspace $(G_s,G)=(A_4,S_4)$.}\label{fam:fourfold-no-68}

This is the symplectic family with generic non-symplectic index $2$,
described in \cite[(4.14)]{KOIKE202512}.  Its proper subfamily is family~69.

Let $K^+\subset\GL(6,\CC)$ denote the strict symplectic lift associated with Koike's component (4.14) in \cite{KOIKE202512}.  We use the following basis $P_1,\ldots,P_{14}$ of $W_{K^+}$:
\begin{SecEightAlignedBasis}{l@{\quad}l}
\multicolumn{2}{@{}l@{}}{\PBasisLabel{1}=x_{1}x_{2}x_{3}+x_{1}x_{2}x_{4}+x_{1}x_{3}x_{4}+x_{2}x_{3}x_{4},}\\
\multicolumn{2}{@{}l@{}}{\PBasisLabel{2}=x_{1}^{2}x_{2}+x_{1}^{2}x_{3}+x_{1}^{2}x_{4}+x_{1}x_{2}^{2}+x_{1}x_{3}^{2}+x_{1}x_{4}^{2}}\\
\multicolumn{2}{@{}l@{}}{\hphantom{\PBasisLabel{2}=}\quad +x_{2}^{2}x_{3}+x_{2}^{2}x_{4}+x_{2}x_{3}^{2}+x_{2}x_{4}^{2}+x_{3}^{2}x_{4}+x_{3}x_{4}^{2},}\\
\multicolumn{2}{@{}l@{}}{\PBasisLabel{3}=x_{1}x_{2}x_{5}+x_{1}x_{3}x_{5}+x_{1}x_{4}x_{5}+x_{2}x_{3}x_{5}+x_{2}x_{4}x_{5}+x_{3}x_{4}x_{5},}\\
\multicolumn{2}{@{}l@{}}{\PBasisLabel{4}=x_{1}x_{2}x_{6}+x_{1}x_{3}x_{6}+x_{1}x_{4}x_{6}+x_{2}x_{3}x_{6}+x_{2}x_{4}x_{6}+x_{3}x_{4}x_{6},\quad \PBasisLabel{5}=x_{1}^{3}+x_{2}^{3}+x_{3}^{3}+x_{4}^{3},}\\
\PBasisLabel{6}=x_{1}^{2}x_{5}+x_{2}^{2}x_{5}+x_{3}^{2}x_{5}+x_{4}^{2}x_{5}, & \PBasisLabel{7}=x_{1}^{2}x_{6}+x_{2}^{2}x_{6}+x_{3}^{2}x_{6}+x_{4}^{2}x_{6},\\
\PBasisLabel{8}=x_{1}x_{5}^{2}+x_{2}x_{5}^{2}+x_{3}x_{5}^{2}+x_{4}x_{5}^{2}, & \PBasisLabel{9}=x_{1}x_{5}x_{6}+x_{2}x_{5}x_{6}+x_{3}x_{5}x_{6}+x_{4}x_{5}x_{6},\\
\multicolumn{2}{@{}l@{}}{\PBasisLabel{10}=x_{1}x_{6}^{2}+x_{2}x_{6}^{2}+x_{3}x_{6}^{2}+x_{4}x_{6}^{2},\quad \PBasisLabel{11}=x_{5}^{3},\quad \PBasisLabel{12}=x_{5}^{2}x_{6},\quad \PBasisLabel{13}=x_{5}x_{6}^{2},\quad \PBasisLabel{14}=x_{6}^{3}.}
\end{SecEightAlignedBasis}
Here $H=\langle K^+,T\rangle$, where $T=(1\,2)$ is a
non-symplectic involution.  It fixes $W_{K^+}$ pointwise, so $W_H=W_{K^+}$.
\par\Needspace{3\baselineskip}\medskip\noindent\textbf{\boldmath 69.\enspace $(G_s,G)=(A_4,\allowbreak S_4\times C_3)$.}\label{fam:fourfold-no-69}
$G\text{-ID}=[72,42]$, $H\text{-ID}=[216,163]$.
The non-symplectic index is $6$, and the family dimension is $2$.  The projective action is liftable.
\par\Needspace{4\baselineskip}\medskip\noindent Here $H=\langle K^+,T\rangle$, where\par\nopagebreak
\begin{SecEightMatrixBlock}
\SecEightMatrix{T}{\diag\!\left(\smat{0&1\\1&0},1,1,1,\omega\right)}
\end{SecEightMatrixBlock}

\par\smallskip\noindent A basis of $W_H$ is $\{P_i\mid i\in\{1,2,3,5,6,8,11,14\}\}$.

\par\Needspace{3\baselineskip}\medskip\noindent\textbf{\boldmath 70.\enspace $(G_s,G)=(D_{10},D_{10})$.}\label{fam:fourfold-no-70}

This is the symplectic family with generic non-symplectic index $1$,
described in \cite[(4.1)]{KOIKE202512}.  Its proper subfamilies are families~71 and~72.

Let $K^+\subset\GL(6,\CC)$ denote the strict symplectic lift associated with Koike's component (4.1) in \cite{KOIKE202512}.  We use the following basis $P_1,\ldots,P_{10}$ of $W_{K^+}$:
\begin{equation*}\begin{aligned}
(P_1,\ldots,P_{10})&=\bigl({}x_{1}^{3},\ x_{1}^{2}x_{2},\ x_{1}x_{2}^{2},\ x_{1}x_{3}x_{6},\ x_{1}x_{4}x_{5},\ x_{2}^{3},\\[-2pt]
&\quad x_{2}x_{3}x_{6},\ x_{2}x_{4}x_{5},\ x_{3}^{2}x_{5}+x_{4}x_{6}^{2},\ x_{3}x_{4}^{2}+x_{5}^{2}x_{6}\bigr)
\end{aligned}\end{equation*}
\par\Needspace{3\baselineskip}\medskip\noindent\textbf{\boldmath 71.\enspace $(G_s,G)=(D_{10},\allowbreak D_{20})$.}\label{fam:fourfold-no-71}
$G\text{-ID}=[20,4]$, $H\text{-ID}=[60,10]$.
The non-symplectic index is $2$, and the family dimension is $2$.  The projective action is liftable.
\par\Needspace{4\baselineskip}\medskip\noindent Here $H=\langle K^+,T\rangle$, where\par\nopagebreak
\begin{SecEightMatrixBlock}
\SecEightMatrix{T}{\diag(-1,1,1,1,1,1)}
\end{SecEightMatrixBlock}

\par\smallskip\noindent A basis of $W_H$ is $\{P_i\mid i\in\{2, 6, 7, 8, 9, 10\}\}$.

\par\Needspace{3\baselineskip}\medskip\noindent\textbf{\boldmath 72.\enspace $(G_s,G)=(D_{10},\allowbreak D_{10}\times C_3)$.}\label{fam:fourfold-no-72}
$G\text{-ID}=[30,2]$, $H\text{-ID}=[90,5]$.
The non-symplectic index is $3$, and the family dimension is $2$.  The projective action is liftable.
\par\Needspace{4\baselineskip}\medskip\noindent Here $H=\langle K^+,T\rangle$, where\par\nopagebreak
\begin{SecEightMatrixBlock}
\SecEightMatrix{T}{\diag(1,\omega,1,1,1,1)}
\end{SecEightMatrixBlock}

\par\smallskip\noindent A basis of $W_H$ is $\{P_i\mid i\in\{1, 4, 5, 6, 9, 10\}\}$.

\par\Needspace{3\baselineskip}\medskip\noindent\textbf{\boldmath 73.\enspace $(G_s,G)=(D_8,D_8)$.}\label{fam:fourfold-no-73}

This is the symplectic family with generic non-symplectic index $1$,
described in \cite[(3.7)]{KOIKE202512}.  Its proper subfamilies are families~74--76.

Let $K^+\subset\GL(6,\CC)$ denote the strict symplectic lift associated with Koike's component (3.7) in \cite{KOIKE202512}.  We use the following basis $P_1,\ldots,P_{12}$ of $W_{K^+}$:
\begin{equation*}\begin{aligned}
(P_1,\ldots,P_{12})&=\bigl({}x_{1}^{3},\ x_{1}^{2}x_{2},\ x_{1}x_{2}^{2},\ x_{1}x_{3}^{2},\ x_{1}x_{4}^{2},\ x_{1}x_{5}x_{6},\\[-2pt]
&\quad x_{2}^{3},\ x_{2}x_{3}^{2},\ x_{2}x_{4}^{2},\ x_{2}x_{5}x_{6},\ x_{3}x_{5}^{2}+x_{3}x_{6}^{2},\ x_{4}(x_{6}^{2}-x_{5}^{2})\bigr)
\end{aligned}\end{equation*}
\par\Needspace{3\baselineskip}\medskip\noindent\textbf{\boldmath 74.\enspace $(G_s,G)=(D_8,\allowbreak D_8\times C_2)$.}\label{fam:fourfold-no-74}
$G\text{-ID}=[16,11]$, $H\text{-ID}=[48,45]$.
The non-symplectic index is $2$, and the family dimension is $4$.  The projective action is liftable.
\par\Needspace{4\baselineskip}\medskip\noindent Here $H=\langle K^+,T\rangle$, where\par\nopagebreak
\begin{SecEightMatrixBlock}
\SecEightMatrix{T}{\diag\!\left(1,1,-1,-1,\smat{0&-\zeta_4\\\zeta_4&0}\right)}
\end{SecEightMatrixBlock}

\par\smallskip\noindent A basis of $W_H$ is $\{P_i\mid 1\le i\le 11\}$.
\par\Needspace{3\baselineskip}\medskip\noindent\textbf{\boldmath 75.\enspace $(G_s,G)=(D_8,\allowbreak (C_4\times C_2):C_2)$.}\label{fam:fourfold-no-75}
$G\text{-ID}=[16,13]$, $H\text{-ID}=[48,47]$.
The non-symplectic index is $2$, and the family dimension is $3$.  The projective action is liftable.
\par\Needspace{4\baselineskip}\medskip\noindent Here $H=\langle K^+,T\rangle$, where\par\nopagebreak
\begin{SecEightMatrixBlock}
\SecEightMatrix{T}{\diag(1,1,-1,-1,-\zeta_4,-\zeta_4)}
\end{SecEightMatrixBlock}

\par\smallskip\noindent A basis of $W_H$ is $\{P_i\mid i\in\{1,\allowbreak \ldots,\allowbreak 5,\allowbreak 7,\allowbreak 8,\allowbreak 9,\allowbreak 11,\allowbreak 12\}\}$.

\par\Needspace{3\baselineskip}\medskip\noindent\textbf{\boldmath 76.\enspace $(G_s,G)=(D_8,\allowbreak D_{16})$.}\label{fam:fourfold-no-76}
$G\text{-ID}=[16,7]$, $H\text{-ID}=[48,25]$.
The non-symplectic index is $2$, and the family dimension is $2$.  The projective action is liftable.
\par\Needspace{4\baselineskip}\medskip\noindent Here $H=\langle K^+,T\rangle$, where\par\nopagebreak
\begin{SecEightMatrixBlock}
\SecEightMatrix{T}{\diag\!\left(-1,1,\smat{0&1\\1&0},
\smat{0&\zeta_8\\-\zeta_8^3&0}\right)}
\end{SecEightMatrixBlock}

\par\medskip\noindent A basis of $W_H$ is\par\nopagebreak
\begin{SecEightBasisBlock}
\SecEightBasisLine{P_{2},}
\SecEightBasisLine{P_{4}-P_{5},}
\SecEightBasisLine{P_{7},}
\SecEightBasisLine{P_{8}+P_{9},}
\SecEightBasisLine{P_{10},}
\SecEightBasisLine{P_{11}-\zeta_{4}\,P_{12}.}
\end{SecEightBasisBlock}

\par\Needspace{3\baselineskip}\medskip\noindent\textbf{\boldmath 77.\enspace $(G_s,G)=(C_4,C_4)$.}\label{fam:fourfold-no-77}

This is the symplectic family with generic non-symplectic index $1$,
described in \cite[(3.5)]{KOIKE202512}.  Its proper subfamilies are families~78--82.

Let $K^+\subset\GL(6,\CC)$ denote the strict symplectic lift associated with Koike's $C_4$ component (3.5) in \cite{KOIKE202512}.  We use the following basis $P_1,\ldots,P_{16}$ of $W_{K^+}$:
\begin{equation*}\begin{aligned}
(P_1,\ldots,P_{16})&=\bigl({}x_{1}^{3},\ x_{1}^{2}x_{2},\ x_{1}x_{2}^{2},\ x_{1}x_{3}^{2},\ x_{1}x_{3}x_{4},\ x_{1}x_{4}^{2},\\[-2pt]
&\quad x_{1}x_{5}x_{6},\ x_{2}^{3},\ x_{2}x_{3}^{2},\ x_{2}x_{3}x_{4},\ x_{2}x_{4}^{2},\ x_{2}x_{5}x_{6},\\[-2pt]
&\quad x_{3}x_{5}^{2},\ x_{3}x_{6}^{2},\ x_{4}x_{5}^{2},\ x_{4}x_{6}^{2}\bigr)
\end{aligned}\end{equation*}

\par\Needspace{3\baselineskip}\medskip\noindent\textbf{\boldmath 78.\enspace $(G_s,G)=(C_4,D_8)$, family~1.}\label{fam:fourfold-no-78}
$G\text{-ID}=[8,3]$, $H\text{-ID}=[24,10]$.
The non-symplectic index is $2$, and the family dimension is $5$.  The projective action is liftable.
This is a proper subfamily of family~77.

\par\Needspace{4\baselineskip}\medskip\noindent Here $H=\langle K^+,T\rangle$, where\par\nopagebreak
\begin{SecEightMatrixBlock}
\SecEightMatrix{T}{\diag(1,1,-1,-1,\smat{0&1\\1&0})}
\end{SecEightMatrixBlock}

\par\medskip\noindent A basis of $W_H$ is\par\nopagebreak
\begin{SecEightBasisBlock}
\SecEightBasisLine{P_i\ (1\le i\le 12),}
\SecEightBasisLine{P_{13}-P_{14},}
\SecEightBasisLine{P_{15}-P_{16}.}
\end{SecEightBasisBlock}

\par\Needspace{3\baselineskip}\medskip\noindent\textbf{\boldmath 79.\enspace $(G_s,G)=(C_4,C_4\times C_2)$.}\label{fam:fourfold-no-79}
$G\text{-ID}=[8,2]$, $H\text{-ID}=[24,9]$.
The non-symplectic index is $2$, and the family dimension is $4$.  The projective action is liftable.
This is a proper subfamily of family~77.

\par\Needspace{4\baselineskip}\medskip\noindent Here $H=\langle K^+,T\rangle$, where\par\nopagebreak
\begin{SecEightMatrixBlock}
\SecEightMatrix{T}{\diag(1,1,-1,-1,-\zeta_4,-\zeta_4)}
\end{SecEightMatrixBlock}

\par\smallskip\noindent A basis of $W_H$ is $\{P_i\mid i\in\{1,\allowbreak \ldots,\allowbreak 6,\allowbreak 8,\allowbreak \ldots,\allowbreak 11,\allowbreak 13,\allowbreak \ldots,\allowbreak 16\}\}$.

\par\Needspace{3\baselineskip}\medskip\noindent\textbf{\boldmath 80.\enspace $(G_s,G)=(C_4,D_8)$, family~2.}\label{fam:fourfold-no-80}
$G\text{-ID}=[8,3]$, $H\text{-ID}=[24,10]$.
The non-symplectic index is $2$, and the family dimension is $3$.  The projective action is liftable.
This is a proper subfamily of family~77.

\par\Needspace{4\baselineskip}\medskip\noindent Here $H=\langle K^+,T\rangle$, where\par\nopagebreak
\begin{SecEightMatrixBlock}
\SecEightMatrix{T}{\diag\!\left(1,-1,1,-1,\smat{0&-1\\-1&0}\right)}
\end{SecEightMatrixBlock}

\par\medskip\noindent A basis of $W_H$ is\par\nopagebreak
\begin{SecEightBasisBlock}
\SecEightBasisLine{P_{1},}
\SecEightBasisLine{P_{3},}
\SecEightBasisLine{P_{4},}
\SecEightBasisLine{P_{6},}
\SecEightBasisLine{P_{7},}
\SecEightBasisLine{P_{10},}
\SecEightBasisLine{P_{13}+P_{14},}
\SecEightBasisLine{P_{15}-P_{16}.}
\end{SecEightBasisBlock}

\par\Needspace{3\baselineskip}\medskip\noindent\textbf{\boldmath 81.\enspace $(G_s,G)=(C_4,C_4\times C_4)$.}\label{fam:fourfold-no-81}
$G\text{-ID}=[16,2]$, $H\text{-ID}=[48,20]$.
The non-symplectic index is $4$, and the family dimension is $2$.  The projective action is liftable.
This is a proper subfamily of family~77.

\par\Needspace{4\baselineskip}\medskip\noindent Here $H=\langle K^+,T\rangle$, where\par\nopagebreak
\begin{SecEightMatrixBlock}
\SecEightMatrix{T}{\diag(1,1,-1,1,-\zeta_4,1)}
\end{SecEightMatrixBlock}

\par\smallskip\noindent A basis of $W_H$ is $\{P_i\mid i\in\{1,\allowbreak \ldots,\allowbreak 4,\allowbreak 6,\allowbreak 8,\allowbreak 9,\allowbreak 11,\allowbreak 13,\allowbreak 16\}\}$.

\par\Needspace{3\baselineskip}\medskip\noindent\textbf{\boldmath 82.\enspace $(G_s,G)=(C_4,C_8:C_2)$.}\label{fam:fourfold-no-82}
$G\text{-ID}=[16,6]$, $H\text{-ID}=[48,24]$.
The non-symplectic index is $4$, and the family dimension is $2$.  The projective action is liftable.
This is a proper subfamily of family~77.

\par\Needspace{4\baselineskip}\medskip\noindent Here $H=\langle K^+,T\rangle$, where\par\nopagebreak
\begin{SecEightMatrixBlock}
\SecEightMatrix{T}{\diag\!\left(
\smat{13/17&6/17\\20/17&-13/17},
\smat{\zeta_4&(10/3)\zeta_4\\0&-\zeta_4},
\smat{0&-(4/3)\zeta_8^3\\-(3/4)\zeta_8^3&0}\right)}
\end{SecEightMatrixBlock}

\par\medskip\noindent A basis of $W_H$ is\par\nopagebreak
\begin{SecEightBasisBlock}
\SecEightBasisLine{27\,P_{1}+117\,P_{3}+86\,P_{8},}
\SecEightBasisLine{18\,P_{2}-27\,P_{3}-26\,P_{8},}
\SecEightBasisLine{P_{4}-5\,P_{9},}
\SecEightBasisLine{9\,P_{5}-85\,P_{9}+6\,P_{10},}
\SecEightBasisLine{27\,P_{6}-850\,P_{9}+510\,P_{10}-135\,P_{11},}
\SecEightBasisLine{160\,P_{13}-48\,P_{15}+27\,P_{16},}
\SecEightBasisLine{30\,P_{14}+16\,P_{15}-9\,P_{16}.}
\end{SecEightBasisBlock}

\par\Needspace{3\baselineskip}\medskip\noindent\textbf{\boldmath 83.\enspace $(G_s,G)=(S_3,S_3)$.}\label{fam:fourfold-no-83}

This is the symplectic family with generic non-symplectic index $1$,
described in \cite[(3.6)]{KOIKE202512}.  Its proper subfamilies are families~84--87.

Let $K^+\subset\GL(6,\CC)$ denote the strict symplectic lift associated with Koike's initial $S_3$ component (3.6) in \cite{KOIKE202512}.  We use the following basis $P_1,\ldots,P_{14}$ of $W_{K^+}$:
\begin{SecEightAlignedBasis}{l@{\qquad}l}
\multicolumn{2}{@{}l@{}}{\PBasisLabel{1}=x_{4}\,x_{5}\,x_{6},\qquad \PBasisLabel{2}=x_{5}\,x_{6}^{2}+x_{5}^{2}\,x_{6}+x_{4}\,x_{6}^{2}+x_{4}\,x_{5}^{2}+x_{4}^{2}\,x_{6}+x_{4}^{2}\,x_{5},}\\
\PBasisLabel{3}=x_{6}^{3}+x_{5}^{3}+x_{4}^{3}, & \PBasisLabel{4}=x_{3}\,x_{4}\,x_{5}+x_{2}\,x_{4}\,x_{6}+x_{1}\,x_{5}\,x_{6},\\
\multicolumn{2}{@{}l@{}}{\PBasisLabel{5}=x_{3}\,x_{5}^{2}+x_{3}\,x_{4}^{2}+x_{2}\,x_{6}^{2}+x_{2}\,x_{4}^{2}+x_{1}\,x_{6}^{2}+x_{1}\,x_{5}^{2},}\\
\multicolumn{2}{@{}l@{}}{\PBasisLabel{6}=x_{3}\,x_{5}\,x_{6}+x_{3}\,x_{4}\,x_{6}+x_{2}\,x_{5}\,x_{6}+x_{2}\,x_{4}\,x_{5}+x_{1}\,x_{4}\,x_{6}+x_{1}\,x_{4}\,x_{5},}\\
\PBasisLabel{7}=x_{3}\,x_{6}^{2}+x_{2}\,x_{5}^{2}+x_{1}\,x_{4}^{2}, & \PBasisLabel{8}=x_{2}\,x_{3}\,x_{4}+x_{1}\,x_{3}\,x_{5}+x_{1}\,x_{2}\,x_{6},\\
\multicolumn{2}{@{}l@{}}{\PBasisLabel{9}=x_{2}\,x_{3}\,x_{6}+x_{2}\,x_{3}\,x_{5}+x_{1}\,x_{3}\,x_{6}+x_{1}\,x_{3}\,x_{4}+x_{1}\,x_{2}\,x_{5}+x_{1}\,x_{2}\,x_{4},\qquad \PBasisLabel{10}=x_{1}\,x_{2}\,x_{3},}\\
\multicolumn{2}{@{}l@{}}{\PBasisLabel{11}=x_{3}^{2}\,x_{5}+x_{3}^{2}\,x_{4}+x_{2}^{2}\,x_{6}+x_{2}^{2}\,x_{4}+x_{1}^{2}\,x_{6}+x_{1}^{2}\,x_{5},\qquad \PBasisLabel{12}=x_{3}^{2}\,x_{6}+x_{2}^{2}\,x_{5}+x_{1}^{2}\,x_{4},}\\
\multicolumn{2}{@{}l@{}}{\PBasisLabel{13}=x_{2}\,x_{3}^{2}+x_{2}^{2}\,x_{3}+x_{1}\,x_{3}^{2}+x_{1}\,x_{2}^{2}+x_{1}^{2}\,x_{3}+x_{1}^{2}\,x_{2},\qquad \PBasisLabel{14}=x_{3}^{3}+x_{2}^{3}+x_{1}^{3}.}
\end{SecEightAlignedBasis}

\par\Needspace{3\baselineskip}\medskip\noindent\textbf{\boldmath 84.\enspace $(G_s,G)=(S_3,D_{12})$, family~1.}\label{fam:fourfold-no-84}
$G\text{-ID}=[12,4]$, $H\text{-ID}=[36,12]$.
The non-symplectic index is $2$, and the family dimension is $3$.  The projective action is liftable.
This is a proper subfamily of family~83.

\par\Needspace{4\baselineskip}\medskip\noindent Here $H=\langle K^+,T\rangle$, where\par\nopagebreak
\begin{SecEightMatrixBlock}
\SecEightMatrix{T}{\left(\begin{smallmatrix}-1&0&0&0&0&0\\0&-1&0&0&0&0\\0&0&-1&0&0&0\\2&0&0&1&0&0\\0&2&0&0&1&0\\0&0&2&0&0&1\end{smallmatrix}\right)}
\end{SecEightMatrixBlock}

\par\medskip\noindent A basis of $W_H$ is\par\nopagebreak
\begin{SecEightBasisBlock}
\SecEightBasisLine{P_{1},}
\SecEightBasisLine{P_{2},}
\SecEightBasisLine{P_{3},}
\SecEightBasisLine{-2\,P_{4}+P_{8},}
\SecEightBasisLine{-P_{5}-P_{6}+P_{9},}
\SecEightBasisLine{-2\,P_{6}+P_{11},}
\SecEightBasisLine{-2\,P_{7}+P_{12}.}
\end{SecEightBasisBlock}

\par\Needspace{3\baselineskip}\medskip\noindent\textbf{\boldmath 85.\enspace $(G_s,G)=(S_3,D_{12})$, family~2.}\label{fam:fourfold-no-85}
$G\text{-ID}=[12,4]$, $H\text{-ID}=[36,12]$.
The non-symplectic index is $2$, and the family dimension is $3$.  The projective action is liftable.
This is a proper subfamily of family~83.

\par\Needspace{4\baselineskip}\medskip\noindent Here $H=\langle K^+,T\rangle$, where\par\nopagebreak
\begin{SecEightMatrixBlock}
\SecEightMatrix{T}{\left(\begin{smallmatrix}\frac{1}{3}&-\frac{2}{3}&-\frac{2}{3}&0&0&0\\-\frac{2}{3}&\frac{1}{3}&-\frac{2}{3}&0&0&0\\-\frac{2}{3}&-\frac{2}{3}&\frac{1}{3}&0&0&0\\\frac{2}{3}&\frac{2}{3}&\frac{2}{3}&1&0&0\\\frac{2}{3}&\frac{2}{3}&\frac{2}{3}&0&1&0\\\frac{2}{3}&\frac{2}{3}&\frac{2}{3}&0&0&1\end{smallmatrix}\right)}
\end{SecEightMatrixBlock}

\par\medskip\noindent A basis of $W_H$ is\par\nopagebreak
\begin{SecEightBasisBlock}
\SecEightBasisLine{P_{1},}
\SecEightBasisLine{P_{2},}
\SecEightBasisLine{P_{3},}
\SecEightBasisLine{-2\,P_{4}+P_{6},}
\SecEightBasisLine{-P_{5}+2\,P_{7},}
\SecEightBasisLine{-4\,P_{4}-P_{5}+P_{8}+P_{9},}
\SecEightBasisLine{-2\,P_{8}+P_{11},}
\SecEightBasisLine{P_{8}-P_{9}+P_{12},}
\SecEightBasisLine{12\,P_{10}-3\,P_{13}+2\,P_{14}.}
\end{SecEightBasisBlock}

\par\Needspace{3\baselineskip}\medskip\noindent\textbf{\boldmath 86.\enspace $(G_s,G)=(S_3,S_3\times C_3)$, family~1.}\label{fam:fourfold-no-86}
$G\text{-ID}=[18,3]$, $H\text{-ID}=[54,12]$.
The non-symplectic index is $3$, and the family dimension is $3$.  The projective action is liftable.
This is a proper subfamily of family~83.

\par\Needspace{6\baselineskip}\medskip\noindent Here $H=\langle K^+,T\rangle$, where\par\nopagebreak
\SecEightDenseMatrix{T}{\left(\begin{smallmatrix}1&0&0&0&0&0\\0&1&0&0&0&0\\0&0&1&0&0&0\\-\frac{1}{3}+\frac{1}{3}\,\omega^{2}&-\frac{1}{3}+\frac{1}{3}\,\omega^{2}&-\frac{1}{3}+\frac{1}{3}\,\omega^{2}&\frac{2}{3}+\frac{1}{3}\,\omega^{2}&-\frac{1}{3}+\frac{1}{3}\,\omega^{2}&-\frac{1}{3}+\frac{1}{3}\,\omega^{2}\\-\frac{1}{3}+\frac{1}{3}\,\omega^{2}&-\frac{1}{3}+\frac{1}{3}\,\omega^{2}&-\frac{1}{3}+\frac{1}{3}\,\omega^{2}&-\frac{1}{3}+\frac{1}{3}\,\omega^{2}&\frac{2}{3}+\frac{1}{3}\,\omega^{2}&-\frac{1}{3}+\frac{1}{3}\,\omega^{2}\\-\frac{1}{3}+\frac{1}{3}\,\omega^{2}&-\frac{1}{3}+\frac{1}{3}\,\omega^{2}&-\frac{1}{3}+\frac{1}{3}\,\omega^{2}&-\frac{1}{3}+\frac{1}{3}\,\omega^{2}&-\frac{1}{3}+\frac{1}{3}\,\omega^{2}&\frac{2}{3}+\frac{1}{3}\,\omega^{2}\end{smallmatrix}\right)}

\par\medskip\noindent A basis of $W_H$ is\par\nopagebreak
\begin{SecEightBasisBlock}
\SecEightBasisLine{P_{2},}
\SecEightBasisLine{6\,P_{1}+P_{3},}
\SecEightBasisLine{6\,P_{1}-2\,P_{4}+P_{5},}
\SecEightBasisLine{3\,P_{1}+P_{4}-P_{6}+P_{7},}
\SecEightBasisLine{2\,P_{4}-P_{6}-2\,P_{8}+P_{9},}
\SecEightBasisLine{-P_{1}+P_{4}-P_{8}+P_{10},}
\SecEightBasisLine{4\,P_{4}-2\,P_{6}-P_{11}+2\,P_{12},}
\SecEightBasisLine{-6\,P_{1}+6\,P_{4}-4\,P_{8}-P_{11}+P_{13},}
\SecEightBasisLine{-6\,P_{1}+6\,P_{4}-3\,P_{11}+2\,P_{14}.}
\end{SecEightBasisBlock}

\par\Needspace{3\baselineskip}\medskip\noindent\textbf{\boldmath 87.\enspace $(G_s,G)=(S_3,S_3\times C_3)$, family~2.}\label{fam:fourfold-no-87}
$G\text{-ID}=[18,3]$, $H\text{-ID}=[54,12]$.
The non-symplectic index is $3$, and the family dimension is $2$.  The projective action is liftable.
This is a proper subfamily of family~83.

\par\Needspace{6\baselineskip}\medskip\noindent Here $H=\langle K^+,T\rangle$, where\par\nopagebreak
\SecEightDenseMatrix{T}{\left(\begin{smallmatrix}1&0&0&0&0&0\\0&1&0&0&0&0\\0&0&1&0&0&0\\-\frac{5}{3}-\frac{1}{3}\,\omega^{2}&\frac{1}{3}+\frac{2}{3}\,\omega^{2}&\frac{1}{3}+\frac{2}{3}\,\omega^{2}&-\frac{2}{3}-\frac{1}{3}\,\omega^{2}&\frac{1}{3}+\frac{2}{3}\,\omega^{2}&\frac{1}{3}+\frac{2}{3}\,\omega^{2}\\\frac{1}{3}+\frac{2}{3}\,\omega^{2}&-\frac{5}{3}-\frac{1}{3}\,\omega^{2}&\frac{1}{3}+\frac{2}{3}\,\omega^{2}&\frac{1}{3}+\frac{2}{3}\,\omega^{2}&-\frac{2}{3}-\frac{1}{3}\,\omega^{2}&\frac{1}{3}+\frac{2}{3}\,\omega^{2}\\\frac{1}{3}+\frac{2}{3}\,\omega^{2}&\frac{1}{3}+\frac{2}{3}\,\omega^{2}&-\frac{5}{3}-\frac{1}{3}\,\omega^{2}&\frac{1}{3}+\frac{2}{3}\,\omega^{2}&\frac{1}{3}+\frac{2}{3}\,\omega^{2}&-\frac{2}{3}-\frac{1}{3}\,\omega^{2}\end{smallmatrix}\right)}

\par\medskip\noindent A basis of $W_H$ is\par\nopagebreak
\begin{SecEightBasisBlock}
\SecEightBasisLine{P_{2},}
\SecEightBasisLine{6\,P_{1}+P_{3},}
\SecEightBasisLine{18\,P_{1}-2\,P_{4}-P_{5}+P_{6}+2\,P_{7},}
\SecEightBasisLine{-P_{1}+P_{4}-P_{8}+P_{10},}
\SecEightBasisLine{P_{5}+2\,P_{6}-2\,P_{9}-P_{11}+P_{13},}
\SecEightBasisLine{-42\,P_{1}+6\,P_{4}+3\,P_{5}-3\,P_{6}-6\,P_{12}+2\,P_{14}.}
\end{SecEightBasisBlock}

\par\Needspace{3\baselineskip}\medskip\noindent\textbf{\boldmath 88.\enspace $(G_s,G)=(S_3,D_{12})$, family~3.}\label{fam:fourfold-no-88}

This is the symplectic family with generic non-symplectic index $2$,
described in \cite{KOIKE2026}.  Its proper subfamilies are families~89--96.

Let $K^+\subset\GL(6,\CC)$ denote the strict symplectic lift associated with the additional $S_3$ component (0.1) in the corrigendum \cite{KOIKE2026}.  Put
\[
s=x_1+x_2+x_3,\qquad q=x_1^2+x_2^2+x_3^2,\qquad c=x_1^3+x_2^3+x_3^3.
\]
The following is a basis of $W_{K^+}$:
\begin{SecEightAlignedBasis}{l@{\qquad}l@{\qquad}l@{\qquad}l}
\PBasisLabel{1}=s^3,  & \PBasisLabel{2}=sq,  & \PBasisLabel{3}=c,  & \PBasisLabel{4}=s^2x_4,\\
\PBasisLabel{5}=s^2x_5, & \PBasisLabel{6}=qx_4, & \PBasisLabel{7}=qx_5, & \PBasisLabel{8}=sx_4^2,\\
\PBasisLabel{9}=sx_4x_5, & \PBasisLabel{10}=sx_5^2, & \PBasisLabel{11}=x_4^3, & \PBasisLabel{12}=x_4^2x_5,\\
\PBasisLabel{13}=x_4x_5^2, & \PBasisLabel{14}=x_5^3, & \PBasisLabel{15}=sx_6^2, & \PBasisLabel{16}=x_4x_6^2,\\
\PBasisLabel{17}=x_5x_6^2. & & &
\end{SecEightAlignedBasis}
Here $H=\langle K^+,T\rangle$, where $T=\diag(1,1,1,1,1,-1)$ is a
non-symplectic involution.  It fixes $W_{K^+}$ pointwise, so $W_H=W_{K^+}$.

\par\Needspace{3\baselineskip}\medskip\noindent\textbf{\boldmath 89.\enspace $(G_s,G)=(S_3,S_3\times C_4)$.}\label{fam:fourfold-no-89}
$G\text{-ID}=[24,5]$, $H\text{-ID}=[72,27]$.
The non-symplectic index is $4$, and the family dimension is $3$.  The projective action is liftable.
This is a proper subfamily of family~88.

\par\Needspace{4\baselineskip}\medskip\noindent Here $H=\langle K^+,T\rangle$, where\par\nopagebreak
\begin{SecEightMatrixBlock}
\SecEightMatrix{T}{\diag(1,1,1,1,-1,-\zeta_4)}
\end{SecEightMatrixBlock}

\par\smallskip\noindent A basis of $W_H$ is $\{P_i\mid i\in\{1,\allowbreak \ldots,\allowbreak 4,\allowbreak 6,\allowbreak 8,\allowbreak 10,\allowbreak 11,\allowbreak 13,\allowbreak 17\}\}$.

\par\Needspace{3\baselineskip}\medskip\noindent\textbf{\boldmath 90.\enspace $(G_s,G)=(S_3,S_3\times C_6)$, family~1.}\label{fam:fourfold-no-90}
$G\text{-ID}=[36,12]$, $H\text{-ID}=[108,42]$.
The non-symplectic index is $6$, and the family dimension is $3$.  The projective action is liftable.
This is a proper subfamily of family~88.

\par\Needspace{6\baselineskip}\medskip\noindent Here $H=\langle K^+,T\rangle$, where\par\nopagebreak
\SecEightDenseMatrix{T}{\diag\!\left(
\smat{(1+2\omega^2)/3&(1-\omega^2)/3&(1-\omega^2)/3\\
(1-\omega^2)/3&(1+2\omega^2)/3&(1-\omega^2)/3\\
(1-\omega^2)/3&(1-\omega^2)/3&(1+2\omega^2)/3},1,1,-1\right)}

\par\medskip\noindent A basis of $W_H$ is\par\nopagebreak
\begin{SecEightBasisBlock}
\SecEightBasisLine{P_{1},}
\SecEightBasisLine{P_{2}-P_{3},}
\SecEightBasisLine{P_{4},}
\SecEightBasisLine{P_{5},}
\SecEightBasisLine{P_{8},}
\SecEightBasisLine{P_{9},}
\SecEightBasisLine{P_{10},}
\SecEightBasisLine{P_{11},}
\SecEightBasisLine{P_{12},}
\SecEightBasisLine{P_{13},}
\SecEightBasisLine{P_{14},}
\SecEightBasisLine{P_{15},}
\SecEightBasisLine{P_{16},}
\SecEightBasisLine{P_{17}.}
\end{SecEightBasisBlock}

\par\Needspace{3\baselineskip}\medskip\noindent\textbf{\boldmath 91.\enspace $(G_s,G)=(S_3,S_3\times C_6)$, family~2.}\label{fam:fourfold-no-91}
$G\text{-ID}=[36,12]$, $H\text{-ID}=[108,42]$.
The non-symplectic index is $6$, and the family dimension is $3$.  The projective action is liftable.
This is a proper subfamily of family~88.

\par\Needspace{6\baselineskip}\medskip\noindent Here $H=\langle K^+,T\rangle$, where\par\nopagebreak
\SecEightDenseMatrix{T}{\diag\!\left(
\smat{\frac23+\frac13\omega^2&-\frac13+\frac13\omega^2&-\frac13+\frac13\omega^2&0&0\\
-\frac13+\frac13\omega^2&\frac23+\frac13\omega^2&-\frac13+\frac13\omega^2&0&0\\
-\frac13+\frac13\omega^2&-\frac13+\frac13\omega^2&\frac23+\frac13\omega^2&0&0\\
1-\omega^2&1-\omega^2&1-\omega^2&1&0\\
-\frac32+\frac32\omega^2&-\frac32+\frac32\omega^2&-\frac32+\frac32\omega^2&0&1},-1\right)}

\par\medskip\noindent A basis of $W_H$ is\par\nopagebreak
\begin{SecEightBasisBlock}
\SecEightBasisLine{2\,P_{1}-9\,P_{2}+9\,P_{3},}
\SecEightBasisLine{2\,P_{2}-2\,P_{3}-6\,P_{4}+9\,P_{5}+6\,P_{6}-9\,P_{7}+12\,P_{8}-36\,P_{9}+27\,P_{10},}
\SecEightBasisLine{3\,P_{6}-P_{4},}
\SecEightBasisLine{-P_{5}+3\,P_{7},}
\SecEightBasisLine{P_{11},}
\SecEightBasisLine{P_{12},}
\SecEightBasisLine{P_{13},}
\SecEightBasisLine{P_{14},}
\SecEightBasisLine{P_{16},}
\SecEightBasisLine{P_{17}.}
\end{SecEightBasisBlock}

\par\Needspace{3\baselineskip}\medskip\noindent\textbf{\boldmath 92.\enspace $(G_s,G)=(S_3,S_3\times C_6)$, family~3.}\label{fam:fourfold-no-92}
$G\text{-ID}=[36,12]$, $H\text{-ID}=[108,42]$.
The non-symplectic index is $6$, and the family dimension is $2$.  The projective action is liftable.
This is a proper subfamily of family~88.

\par\Needspace{6\baselineskip}\medskip\noindent Here $H=\langle K^+,T\rangle$, where\par\nopagebreak
\SecEightDenseMatrix{T}{\diag\!\left(
\smat{(-1+\omega^2)/3&(-1-2\omega^2)/3&(-1-2\omega^2)/3\\
(-1-2\omega^2)/3&(-1+\omega^2)/3&(-1-2\omega^2)/3\\
(-1-2\omega^2)/3&(-1-2\omega^2)/3&(-1+\omega^2)/3},\omega^2,1,-1\right)}

\par\medskip\noindent A basis of $W_H$ is\par\nopagebreak
\begin{SecEightBasisBlock}
\SecEightBasisLine{P_{1},}
\SecEightBasisLine{P_{3}-P_{2},}
\SecEightBasisLine{-P_{4}+3\,P_{6},}
\SecEightBasisLine{P_{9},}
\SecEightBasisLine{P_{11},}
\SecEightBasisLine{P_{14},}
\SecEightBasisLine{P_{17}.}
\end{SecEightBasisBlock}

\par\Needspace{3\baselineskip}\medskip\noindent\textbf{\boldmath 93.\enspace $(G_s,G)=(S_3,S_3\times C_8)$.}\label{fam:fourfold-no-93}
$G\text{-ID}=[48,4]$, $H\text{-ID}=[144,69]$.
The non-symplectic index is $8$, and the family dimension is $1$.  The projective action is liftable.
This is a proper subfamily of family~88.

\par\Needspace{4\baselineskip}\medskip\noindent Here $H=\langle K^+,T\rangle$, where\par\nopagebreak
\begin{SecEightMatrixBlock}
\SecEightMatrix{T}{\diag(1,1,1,-1,-\zeta_4,-\zeta_8)}
\end{SecEightMatrixBlock}

\par\smallskip\noindent A basis of $W_H$ is $\{P_i\mid i\in\{1,2,3,8,13,17\}\}$.

\par\Needspace{3\baselineskip}\medskip\noindent\textbf{\boldmath 94.\enspace $(G_s,G)=(S_3,S_3\times C_{12})$, family~1.}\label{fam:fourfold-no-94}
$G\text{-ID}=[72,27]$, $H\text{-ID}=[216,136]$.
The non-symplectic index is $12$, and the family dimension is $1$.  The projective action is liftable.
This is a proper subfamily of family~88.

\par\Needspace{4\baselineskip}\medskip\noindent Here $H=\langle K^+,T\rangle$, where\par\nopagebreak
\begin{SecEightMatrixBlock}
\SecEightMatrix{T}{\diag\!\left(
\smat{(-\omega+\omega^2)/3&(-\omega-2\omega^2)/3&(-\omega-2\omega^2)/3\\
(-\omega-2\omega^2)/3&(-\omega+\omega^2)/3&(-\omega-2\omega^2)/3\\
(-\omega-2\omega^2)/3&(-\omega-2\omega^2)/3&(-\omega+\omega^2)/3},1,-1,\zeta_4\right)}
\end{SecEightMatrixBlock}

\par\medskip\noindent A basis of $W_H$ is\par\nopagebreak
\begin{SecEightBasisBlock}
\SecEightBasisLine{P_{1},}
\SecEightBasisLine{P_{2}-P_{3},}
\SecEightBasisLine{P_{4},}
\SecEightBasisLine{P_{8},}
\SecEightBasisLine{P_{10},}
\SecEightBasisLine{P_{11},}
\SecEightBasisLine{P_{13},}
\SecEightBasisLine{P_{17}.}
\end{SecEightBasisBlock}

\par\Needspace{3\baselineskip}\medskip\noindent\textbf{\boldmath 95.\enspace $(G_s,G)=(S_3,S_3\times C_{12})$, family~2.}\label{fam:fourfold-no-95}
$G\text{-ID}=[72,27]$, $H\text{-ID}=[216,136]$.
The non-symplectic index is $12$, and the family dimension is $1$.  The projective action is liftable.
This is a proper subfamily of family~88.

\par\Needspace{4\baselineskip}\medskip\noindent Here $H=\langle K^+,T\rangle$, where\par\nopagebreak
\begin{SecEightMatrixBlock}
\SecEightMatrix{T}{\diag\!\left(
\smat{(1-\omega^2)/3&(-2-\omega^2)/3&(-2-\omega^2)/3\\
(-2-\omega^2)/3&(1-\omega^2)/3&(-2-\omega^2)/3\\
(-2-\omega^2)/3&(-2-\omega^2)/3&(1-\omega^2)/3},-1,1,\zeta_4\right)}
\end{SecEightMatrixBlock}

\par\medskip\noindent A basis of $W_H$ is\par\nopagebreak
\begin{SecEightBasisBlock}
\SecEightBasisLine{P_{1},}
\SecEightBasisLine{P_{2}-P_{3},}
\SecEightBasisLine{-P_{5}+3\,P_{7},}
\SecEightBasisLine{P_{12},}
\SecEightBasisLine{P_{14},}
\SecEightBasisLine{P_{16}.}
\end{SecEightBasisBlock}

\par\Needspace{3\baselineskip}\medskip\noindent\textbf{\boldmath 96.\enspace $(G_s,G)=(S_3,S_3\times C_{24})$.}\label{fam:fourfold-no-96}
$G\text{-ID}=[144,69]$, $H\text{-ID}=[432,464]$.
See
\cite[\S6.1, (3)]{yang2024automorphism} and
\cite[Proposition~7.1]{FWZ26}.

\par\Needspace{3\baselineskip}\medskip\noindent\textbf{\boldmath 97.\enspace $(G_s,G)=(C_2^2,C_2^2)$.}\label{fam:fourfold-no-97}

This is the symplectic family with generic non-symplectic index $1$,
described in \cite[(3.2)]{KOIKE202512}.  Its proper subfamilies are families~98--105.

Let $K^+\subset\GL(6,\CC)$ denote the strict symplectic lift associated with Koike's $C_2^2$ component (3.2) in \cite{KOIKE202512}.  We use the following basis $P_1,\ldots,P_{20}$ of $W_{K^+}$:
\begin{equation*}\begin{aligned}
(P_1,\ldots,P_{20})&=\bigl({}x_{1}^{3},\ x_{1}^{2}x_{2},\ x_{1}^{2}x_{3},\ x_{1}x_{2}^{2},\ x_{1}x_{2}x_{3},\ x_{1}x_{3}^{2},\\[-2pt]
&\quad x_{1}x_{4}^{2},\ x_{1}x_{5}^{2},\ x_{1}x_{6}^{2},\ x_{2}^{3},\ x_{2}^{2}x_{3},\ x_{2}x_{3}^{2},\\[-2pt]
&\quad x_{2}x_{4}^{2},\ x_{2}x_{5}^{2},\ x_{2}x_{6}^{2},\ x_{3}^{3},\ x_{3}x_{4}^{2},\ x_{3}x_{5}^{2},\\[-2pt]
&\quad x_{3}x_{6}^{2},\ x_{4}x_{5}x_{6}\bigr)
\end{aligned}\end{equation*}

\par\Needspace{3\baselineskip}\medskip\noindent\textbf{\boldmath 98.\enspace $(G_s,G)=(C_2^2,C_2^3)$, family~1.}\label{fam:fourfold-no-98}
$G\text{-ID}=[8,5]$, $H\text{-ID}=[24,15]$.
The non-symplectic index is $2$, and the family dimension is $7$.  The projective action is liftable.
This is a proper subfamily of family~97.

\par\Needspace{4\baselineskip}\medskip\noindent Here $H=\langle K^+,T\rangle$, where\par\nopagebreak
\begin{SecEightMatrixBlock}
\SecEightMatrix{T}{\diag(1,1,1,1,-1,1)}
\end{SecEightMatrixBlock}

\par\smallskip\noindent A basis of $W_H$ is $\{P_i\mid 1\le i\le 19\}$.

\par\Needspace{3\baselineskip}\medskip\noindent\textbf{\boldmath 99.\enspace $(G_s,G)=(C_2^2,D_8)$.}\label{fam:fourfold-no-99}
$G\text{-ID}=[8,3]$, $H\text{-ID}=[24,10]$.
The non-symplectic index is $2$, and the family dimension is $6$.  The projective action is liftable.
This is a proper subfamily of family~97.

\par\Needspace{4\baselineskip}\medskip\noindent Here $H=\langle K^+,T\rangle$, where\par\nopagebreak
\begin{SecEightMatrixBlock}
\SecEightMatrix{T}{(5\,6)}
\end{SecEightMatrixBlock}

\par\medskip\noindent A basis of $W_H$ is\par\nopagebreak
\begin{SecEightBasisBlock}
\SecEightBasisLine{P_i\ (i\in\{1,2,3,4,5,6,7,10,11,12,13,16,17,20\}),}
\SecEightBasisLine{P_{8}+P_{9},}
\SecEightBasisLine{P_{14}+P_{15},}
\SecEightBasisLine{P_{18}+P_{19}.}
\end{SecEightBasisBlock}

\par\Needspace{3\baselineskip}\medskip\noindent\textbf{\boldmath 100.\enspace $(G_s,G)=(C_2^2,C_2^3)$, family~2.}\label{fam:fourfold-no-100}
$G\text{-ID}=[8,5]$, $H\text{-ID}=[24,15]$.
The non-symplectic index is $2$, and the family dimension is $5$.  The projective action is liftable.
This is a proper subfamily of family~97.

\par\Needspace{4\baselineskip}\medskip\noindent Here $H=\langle K^+,T\rangle$, where\par\nopagebreak
\begin{SecEightMatrixBlock}
\SecEightMatrix{T}{\diag(1,1,-1,1,1,1)}
\end{SecEightMatrixBlock}

\par\smallskip\noindent A basis of $W_H$ is $\{P_i\mid i\in\{1,\allowbreak 2,\allowbreak 4,\allowbreak 6,\allowbreak \ldots,\allowbreak 10,\allowbreak 12,\allowbreak \ldots,\allowbreak 15,\allowbreak 20\}\}$.

\par\Needspace{3\baselineskip}\medskip\noindent\textbf{\boldmath 101.\enspace $(G_s,G)=(C_2^2,C_6\times C_2)$.}\label{fam:fourfold-no-101}
$G\text{-ID}=[12,5]$, $H\text{-ID}=[36,14]$.
The non-symplectic index is $3$, and the family dimension is $4$.  The projective action is liftable.
This is a proper subfamily of family~97.

\par\Needspace{4\baselineskip}\medskip\noindent Here $H=\langle K^+,T\rangle$, where\par\nopagebreak
\begin{SecEightMatrixBlock}
\SecEightMatrix{T}{\diag(\omega,1,1,1,1,1)}
\end{SecEightMatrixBlock}

\par\smallskip\noindent A basis of $W_H$ is $\{P_i\mid i\in\{1,\allowbreak 10,\allowbreak \ldots,\allowbreak 20\}\}$.

\par\Needspace{3\baselineskip}\medskip\noindent\textbf{\boldmath 102.\enspace $(G_s,G)=(C_2^2,A_4)$.}\label{fam:fourfold-no-102}
$G\text{-ID}=[12,3]$, $H\text{-ID}=[36,11]$.
The non-symplectic index is $3$, and the family dimension is $3$.  The projective action is liftable.
This is a proper subfamily of family~97.

\par\Needspace{4\baselineskip}\medskip\noindent Here $H=\langle K^+,T\rangle$, where\par\nopagebreak
\begin{SecEightMatrixBlock}
\SecEightMatrix{T}{\diag\!\left(\omega,\omega^2,\omega^2,
\smat{0&0&-1\\1&0&0\\0&-1&0}\right)}
\end{SecEightMatrixBlock}

\par\medskip\noindent A basis of $W_H$ is\par\nopagebreak
\begin{SecEightBasisBlock}
\SecEightBasisLine{P_{1},}
\SecEightBasisLine{P_{7}+\omega\,P_{8}+\omega^{2}\,P_{9},}
\SecEightBasisLine{P_{10},}
\SecEightBasisLine{P_{11},}
\SecEightBasisLine{P_{12},}
\SecEightBasisLine{P_{13}+\omega^{2}\,P_{14}+\omega\,P_{15},}
\SecEightBasisLine{P_{16},}
\SecEightBasisLine{P_{17}+\omega^{2}\,P_{18}+\omega\,P_{19},}
\SecEightBasisLine{P_{20}.}
\end{SecEightBasisBlock}

\par\Needspace{3\baselineskip}\medskip\noindent\textbf{\boldmath 103.\enspace $(G_s,G)=(C_2^2,D_8\times C_3)$.}\label{fam:fourfold-no-103}
$G\text{-ID}=[24,10]$, $H\text{-ID}=[72,37]$.
The non-symplectic index is $6$, and the family dimension is $3$.  The projective action is liftable.
This is a proper subfamily of family~97.

\par\Needspace{4\baselineskip}\medskip\noindent Here $H=\langle K^+,T\rangle$, where\par\nopagebreak
\begin{SecEightMatrixBlock}
\SecEightMatrix{T}{\diag\!\left(\omega^2,1,1,1,\smat{0&-1\\-1&0}\right)}
\end{SecEightMatrixBlock}

\par\medskip\noindent A basis of $W_H$ is\par\nopagebreak
\begin{SecEightBasisBlock}
\SecEightBasisLine{P_i\ (i\in\{1,10,11,12,13,16,17,20\}),}
\SecEightBasisLine{P_{14}+P_{15},}
\SecEightBasisLine{P_{18}+P_{19}.}
\end{SecEightBasisBlock}

\par\Needspace{3\baselineskip}\medskip\noindent\textbf{\boldmath 104.\enspace $(G_s,G)=(C_2^2,A_4\times C_2)$, family~1.}\label{fam:fourfold-no-104}
$G\text{-ID}=[24,13]$, $H\text{-ID}=[72,47]$.
The non-symplectic index is $6$, and the family dimension is $2$.  The projective action is liftable.
This is a proper subfamily of family~97.

\par\Needspace{4\baselineskip}\medskip\noindent Here $H=\langle K^+,T\rangle$, where\par\nopagebreak
\begin{SecEightMatrixBlock}
\SecEightMatrix{T}{\diag\!\left(1,\omega^2,-1,
\smat{0&0&1\\1&0&0\\0&1&0}\right)}
\end{SecEightMatrixBlock}

\par\medskip\noindent A basis of $W_H$ is\par\nopagebreak
\begin{SecEightBasisBlock}
\SecEightBasisLine{P_{1},}
\SecEightBasisLine{P_{6},}
\SecEightBasisLine{P_{7}+P_{8}+P_{9},}
\SecEightBasisLine{P_{10},}
\SecEightBasisLine{P_{13}+\omega^{2}\,P_{14}+\omega\,P_{15},}
\SecEightBasisLine{P_{20}.}
\end{SecEightBasisBlock}

\par\Needspace{3\baselineskip}\medskip\noindent\textbf{\boldmath 105.\enspace $(G_s,G)=(C_2^2,A_4\times C_2)$, family~2.}\label{fam:fourfold-no-105}
$G\text{-ID}=[24,13]$, $H\text{-ID}=[72,16]$.
The non-symplectic index is $6$, and the family dimension is $2$.  The projective action is liftable.
This is a proper subfamily of family~97.

\par\Needspace{4\baselineskip}\medskip\noindent Here $H=\langle K^+,T\rangle$, where\par\nopagebreak
\begin{SecEightMatrixBlock}
\SecEightMatrix{T}{\diag\!\left(\zeta_{9}^{2},\zeta_{9}^{5},
-\zeta_{9}^{2}-\zeta_{9}^{5},\smat{0&-1&0\\0&0&-\omega^{2}\\-1&0&0}\right)}
\end{SecEightMatrixBlock}

\par\medskip\noindent A basis of $W_H$ is\par\nopagebreak
\begin{SecEightBasisBlock}
\SecEightBasisLine{P_{2},}
\SecEightBasisLine{P_{6},}
\SecEightBasisLine{\zeta_{9}^{7}\,P_{7}+\zeta_{9}^{5}\,P_{8}+P_{9},}
\SecEightBasisLine{P_{11},}
\SecEightBasisLine{\zeta_{9}^{4}\,P_{13}+\zeta_{9}^{8}\,P_{14}+P_{15},}
\SecEightBasisLine{\zeta_{9}\,P_{17}+\zeta_{9}^{2}\,P_{18}+P_{19}.}
\end{SecEightBasisBlock}

\par\Needspace{3\baselineskip}\medskip\noindent\textbf{\boldmath 106.\enspace $(G_s,G)=(C_3,C_3)$.}\label{fam:fourfold-no-106}

This is the symplectic family with generic non-symplectic index $1$,
described in \cite[(3.4)]{KOIKE202512}.  Its proper subfamilies are families~107--114.

Let $K^+\subset\GL(6,\CC)$ denote the strict symplectic lift associated with Koike's $C_3$ component (3.4) in \cite{KOIKE202512}.  We use the following basis $P_1,\ldots,P_{20}$ of $W_{K^+}$:
\begin{equation*}\begin{aligned}
(P_1,\ldots,P_{20})&=\bigl({}x_{1}^{3},\ x_{1}^{2}x_{2},\ x_{1}x_{2}^{2},\ x_{1}x_{3}x_{5},\ x_{1}x_{3}x_{6},\ x_{1}x_{4}x_{5},\\[-2pt]
&\quad x_{1}x_{4}x_{6},\ x_{2}^{3},\ x_{2}x_{3}x_{5},\ x_{2}x_{3}x_{6},\ x_{2}x_{4}x_{5},\ x_{2}x_{4}x_{6},\\[-2pt]
&\quad x_{3}^{3},\ x_{3}^{2}x_{4},\ x_{3}x_{4}^{2},\ x_{4}^{3},\\[-2pt]
&\quad x_{5}^{3},\ x_{5}^{2}x_{6},\ x_{5}x_{6}^{2},\ x_{6}^{3}\bigr)
\end{aligned}\end{equation*}

\par\Needspace{3\baselineskip}\medskip\noindent\textbf{\boldmath 107.\enspace $(G_s,G)=(C_3,S_3)$, family~1.}\label{fam:fourfold-no-107}
$G\text{-ID}=[6,1]$, $H\text{-ID}=[18,3]$.
The non-symplectic index is $2$, and the family dimension is $4$.  The projective action is liftable.
This is a proper subfamily of family~106.

\par\Needspace{4\baselineskip}\medskip\noindent Here $H=\langle K^+,T\rangle$, where\par\nopagebreak
\begin{SecEightMatrixBlock}
\SecEightMatrix{T}{\diag\!\left(1,-1,\smat{0&0&1&0\\0&0&0&1\\1&0&0&0\\0&1&0&0}\right)}
\end{SecEightMatrixBlock}

\par\medskip\noindent A basis of $W_H$ is\par\nopagebreak
\begin{SecEightBasisBlock}
\SecEightBasisLine{P_{1},}
\SecEightBasisLine{P_{3},}
\SecEightBasisLine{P_{4},}
\SecEightBasisLine{P_{5}+P_{6},}
\SecEightBasisLine{P_{7},}
\SecEightBasisLine{P_{11}-P_{10},}
\SecEightBasisLine{P_{13}+P_{17},}
\SecEightBasisLine{P_{14}+P_{18},}
\SecEightBasisLine{P_{15}+P_{19},}
\SecEightBasisLine{P_{16}+P_{20}.}
\end{SecEightBasisBlock}

\par\Needspace{3\baselineskip}\medskip\noindent\textbf{\boldmath 108.\enspace $(G_s,G)=(C_3,C_6)$, family~1.}\label{fam:fourfold-no-108}
$G\text{-ID}=[6,2]$, $H\text{-ID}=[18,5]$.
The non-symplectic index is $2$, and the family dimension is $4$.  The projective action is liftable.
This is a proper subfamily of family~106.

\par\Needspace{4\baselineskip}\medskip\noindent Here $H=\langle K^+,T\rangle$, where\par\nopagebreak
\begin{SecEightMatrixBlock}
\SecEightMatrix{T}{\diag(1,1,1,-1,1,1)}
\end{SecEightMatrixBlock}

\par\smallskip\noindent A basis of $W_H$ is $\{P_i\mid i\in\{1,\allowbreak \ldots,\allowbreak 5,\allowbreak 8,\allowbreak 9,\allowbreak 10,\allowbreak 13,\allowbreak 15,\allowbreak 17,\allowbreak \ldots,\allowbreak 20\}\}$.

\par\Needspace{3\baselineskip}\medskip\noindent\textbf{\boldmath 109.\enspace $(G_s,G)=(C_3,C_6)$, family~2.}\label{fam:fourfold-no-109}
$G\text{-ID}=[6,2]$, $H\text{-ID}=[18,5]$.
The non-symplectic index is $2$, and the family dimension is $4$.  The projective action is liftable.
This is a proper subfamily of family~106.

\par\Needspace{4\baselineskip}\medskip\noindent Here $H=\langle K^+,T\rangle$, where\par\nopagebreak
\begin{SecEightMatrixBlock}
\SecEightMatrix{T}{\diag(-1,1,-1,1,-1,1)}
\end{SecEightMatrixBlock}

\par\smallskip\noindent A basis of $W_H$ is $\{P_i\mid i\in\{2,\allowbreak 5,\allowbreak 6,\allowbreak 8,\allowbreak 9,\allowbreak 12,\allowbreak 14,\allowbreak 16,\allowbreak 18,\allowbreak 20\}\}$.

\par\Needspace{3\baselineskip}\medskip\noindent\textbf{\boldmath 110.\enspace $(G_s,G)=(C_3,C_3^2)$, family~1.}\label{fam:fourfold-no-110}
$G\text{-ID}=[9,2]$, $H\text{-ID}=[27,5]$.
The non-symplectic index is $3$, and the family dimension is $4$.  The projective action is liftable.
This is a proper subfamily of family~106.

\par\Needspace{4\baselineskip}\medskip\noindent Here $H=\langle K^+,T\rangle$, where\par\nopagebreak
\begin{SecEightMatrixBlock}
\SecEightMatrix{T}{\diag(1,\omega,1,1,1,1)}
\end{SecEightMatrixBlock}

\par\smallskip\noindent A basis of $W_H$ is $\{P_i\mid i\in\{1,\allowbreak 4,\allowbreak \ldots,\allowbreak 8,\allowbreak 13,\allowbreak \ldots,\allowbreak 20\}\}$.

\par\Needspace{3\baselineskip}\medskip\noindent\textbf{\boldmath 111.\enspace $(G_s,G)=(C_3,C_3^2)$, family~2.}\label{fam:fourfold-no-111}
$G\text{-ID}=[9,2]$, $H\text{-ID}=[27,5]$.
The non-symplectic index is $3$, and the family dimension is $3$.  The projective action is liftable.
This is a proper subfamily of family~106.

\par\Needspace{4\baselineskip}\medskip\noindent Here $H=\langle K^+,T\rangle$, where\par\nopagebreak
\begin{SecEightMatrixBlock}
\SecEightMatrix{T}{\diag(1,\omega,\omega^2,1,\omega,1)}
\end{SecEightMatrixBlock}

\par\smallskip\noindent A basis of $W_H$ is $\{P_i\mid i\in\{1, 4, 7, 8, 10, 13, 16, 17, 20\}\}$.

\par\Needspace{3\baselineskip}\medskip\noindent\textbf{\boldmath 112.\enspace $(G_s,G)=(C_3,C_3^2)$, family~3.}\label{fam:fourfold-no-112}
$G\text{-ID}=[9,2]$, $H\text{-ID}=[27,4]$.
The non-symplectic index is $3$, and the family dimension is $2$.  The projective action is non-liftable.
This is a proper subfamily of family~106.

\par\Needspace{4\baselineskip}\medskip\noindent Here $H=\langle K^+,T\rangle$, where\par\nopagebreak
\begin{SecEightMatrixBlock}
\SecEightMatrix{T}{\diag(-\zeta_9,-\zeta_9,\zeta_9,\zeta_9,-\zeta_9,-\zeta_9)(1\,6\,4)(2\,5\,3)}
\end{SecEightMatrixBlock}

\par\medskip\noindent A basis of $W_H$ is\par\nopagebreak
\begin{SecEightBasisBlock}
\SecEightBasisLine{P_{1}+\omega\,P_{16}-\omega^{2}\,P_{20},}
\SecEightBasisLine{P_{2}+\omega\,P_{15}-\omega^{2}\,P_{19},}
\SecEightBasisLine{P_{3}+\omega\,P_{14}-\omega^{2}\,P_{18},}
\SecEightBasisLine{P_{4}+\omega^{2}\,P_{10}+\omega\,P_{11},}
\SecEightBasisLine{P_{5}+\omega\,P_{6}+\omega^{2}\,P_{12},}
\SecEightBasisLine{P_{8}+\omega\,P_{13}-\omega^{2}\,P_{17}.}
\end{SecEightBasisBlock}

\par\Needspace{3\baselineskip}\medskip\noindent\textbf{\boldmath 113.\enspace $(G_s,G)=(C_3,C_6\times C_3)$, family~1.}\label{fam:fourfold-no-113}
$G\text{-ID}=[18,5]$, $H\text{-ID}=[54,15]$.
The non-symplectic index is $6$, and the family dimension is $2$.  The projective action is liftable.
This is a proper subfamily of family~106.

\par\Needspace{4\baselineskip}\medskip\noindent Here $H=\langle K^+,T\rangle$, where\par\nopagebreak
\begin{SecEightMatrixBlock}
\SecEightMatrix{T}{\diag(-1,1,1,1,1,\omega)}
\end{SecEightMatrixBlock}

\par\smallskip\noindent A basis of $W_H$ is $\{P_i\mid i\in\{2,\allowbreak 8,\allowbreak 9,\allowbreak 11,\allowbreak 13,\allowbreak \ldots,\allowbreak 17,\allowbreak 20\}\}$.

\par\Needspace{3\baselineskip}\medskip\noindent\textbf{\boldmath 114.\enspace $(G_s,G)=(C_3,C_6\times C_3)$, family~2.}\label{fam:fourfold-no-114}
$G\text{-ID}=[18,5]$, $H\text{-ID}=[54,11]$.
The non-symplectic index is $6$, and the family dimension is $1$.  The projective action is non-liftable.
This is a proper subfamily of family~106.

\par\Needspace{4\baselineskip}\medskip\noindent Here $H=\langle K^+,T\rangle$, where\par\nopagebreak
\begin{SecEightMatrixBlock}
\SecEightMatrix{T}{\zeta_9^2\diag(-1,1,-\omega^2,\omega^2,\omega,-\omega)
(1\,6\,4)(2\,5\,3)}
\end{SecEightMatrixBlock}

\par\smallskip\noindent A basis of $W_H$ is $\{P_{1}+\omega^{2}\,P_{16}-\omega\,P_{20}, P_{3}+\omega^{2}\,P_{14}-\omega\,P_{18}, P_{4}+\omega\,P_{10}+\omega^{2}\,P_{11}\}$.

\par\Needspace{3\baselineskip}\medskip\noindent\textbf{\boldmath 115.\enspace $(G_s,G)=(C_3,S_3)$, family~2.}\label{fam:fourfold-no-115}

This is the symplectic family with generic non-symplectic index $2$,
described in \cite[(3.3)]{KOIKE202512}.  Its proper subfamilies are families~116--118.

Let $K^+\subset\GL(6,\CC)$ denote the strict symplectic lift associated with Koike's second $C_3$ component (3.3) in \cite{KOIKE202512}.  We use the following basis $P_1,\ldots,P_{26}$ of $W_{K^+}$:
\begin{equation*}\begin{aligned}
(P_1,\ldots,P_{26})&=\bigl({}x_{1}^{3},\ x_{1}^{2}x_{2},\ x_{1}^{2}x_{3},\ x_{1}^{2}x_{4},\ x_{1}x_{2}^{2},\ x_{1}x_{2}x_{3},\\[-2pt]
&\quad x_{1}x_{2}x_{4},\ x_{1}x_{3}^{2},\ x_{1}x_{3}x_{4},\ x_{1}x_{4}^{2},\ x_{1}x_{5}x_{6},\ x_{2}^{3},\\[-2pt]
&\quad x_{2}^{2}x_{3},\ x_{2}^{2}x_{4},\ x_{2}x_{3}^{2},\ x_{2}x_{3}x_{4},\ x_{2}x_{4}^{2},\ x_{2}x_{5}x_{6},\\[-2pt]
&\quad x_{3}^{3},\ x_{3}^{2}x_{4},\ x_{3}x_{4}^{2},\ x_{3}x_{5}x_{6},\ x_{4}^{3},\\[-2pt]
&\quad x_{4}x_{5}x_{6},\ x_{5}^{3},\ x_{6}^{3}\bigr)
\end{aligned}\end{equation*}
After rescaling $x_5$ and $x_6$, the coefficients of $P_{25}=x_5^3$
and $P_{26}=x_6^3$ may be taken equal: both are nonzero for a smooth
member.  In these coordinates,
\[
H=\langle K^+,T\rangle,\qquad T=(5\,6),
\]
where $T$ is a non-symplectic involution.  A basis of $W_H$ is
\[
P_1,\ldots,P_{24},\quad P_{25}+P_{26}.
\]

\par\Needspace{3\baselineskip}\medskip\noindent\textbf{\boldmath 116.\enspace $(G_s,G)=(C_3,S_3\times C_3)$, family~1.}\label{fam:fourfold-no-116}
$G\text{-ID}=[18,3]$, $H\text{-ID}=[54,12]$.
The non-symplectic index is $6$, and the family dimension is $4$.  The projective action is liftable.
This is a proper subfamily of family~115.

\par\Needspace{4\baselineskip}\medskip\noindent Here $H=\langle K^+,T\rangle$, where\par\nopagebreak
\begin{SecEightMatrixBlock}
\SecEightMatrix{T}{\diag\!\left(1,1,1,1,\smat{0&1\\\omega&0}\right)}
\end{SecEightMatrixBlock}

\par\noindent A basis of $W_H$ is\par\nopagebreak
\begin{SecEightBasisBlock}
\SecEightBasisLine{P_i\ (i\in\{1,2,3,4,5,6,7,8,9,10,12,13,14,15,16,17,19,20,21,23\}),}
\SecEightBasisLine{P_{25}+P_{26}.}
\end{SecEightBasisBlock}

\par\Needspace{3\baselineskip}\medskip\noindent\textbf{\boldmath 117.\enspace $(G_s,G)=(C_3,S_3\times C_3)$, family~2.}\label{fam:fourfold-no-117}
$G\text{-ID}=[18,3]$, $H\text{-ID}=[54,12]$.
The non-symplectic index is $6$, and the family dimension is $4$.  The projective action is liftable.
This is a proper subfamily of family~115.

\par\Needspace{4\baselineskip}\medskip\noindent Here $H=\langle K^+,T\rangle$, where\par\nopagebreak
\begin{SecEightMatrixBlock}
\SecEightMatrix{T}{\diag\!\left(\omega^{2},1,1,1,\smat{0&1\\1&0}\right)}
\end{SecEightMatrixBlock}

\par\noindent A basis of $W_H$ is\par\nopagebreak
\begin{SecEightBasisBlock}
\SecEightBasisLine{P_i\ (i\in\{1,12,13,14,15,16,17,18,19,20,21,22,23,24\}),}
\SecEightBasisLine{P_{25}+P_{26}.}
\end{SecEightBasisBlock}

\par\Needspace{3\baselineskip}\medskip\noindent\textbf{\boldmath 118.\enspace $(G_s,G)=(C_3,S_3\times C_3)$, family~3.}\label{fam:fourfold-no-118}
$G\text{-ID}=[18,3]$, $H\text{-ID}=[54,12]$.
The non-symplectic index is $6$, and the family dimension is $3$.  The projective action is liftable.
This is a proper subfamily of family~115.

\par\Needspace{4\baselineskip}\medskip\noindent Here $H=\langle K^+,T\rangle$, where\par\nopagebreak
\begin{SecEightMatrixBlock}
\SecEightMatrix{T}{\diag\!\left(\omega^{2},1,\omega,\omega,\smat{0&1\\1&0}\right)}
\end{SecEightMatrixBlock}

\par\noindent A basis of $W_H$ is\par\nopagebreak
\begin{SecEightBasisBlock}
\SecEightBasisLine{P_i\ (i\in\{1,6,7,12,18,19,20,21,23\}),}
\SecEightBasisLine{P_{25}+P_{26}.}
\end{SecEightBasisBlock}

\par\Needspace{3\baselineskip}\medskip\noindent\textbf{\boldmath 119.\enspace $(G_s,G)=(C_2,C_2)$.}\label{fam:fourfold-no-119}

This is the symplectic family with generic non-symplectic index $1$,
described in \cite[(3.1)]{KOIKE202512}.  Its proper subfamilies are families~120--131.

Let $K^+\subset\GL(6,\CC)$ denote the strict symplectic lift associated with Koike's $C_2$ component (3.1) in \cite{KOIKE202512}.  We use the following basis $P_1,\ldots,P_{32}$ of $W_{K^+}$:
\begin{equation*}\begin{aligned}
(P_1,\ldots,P_{32})&=\bigl({}x_{1}^{3},\ x_{1}^{2}x_{2},\ x_{1}^{2}x_{3},\ x_{1}^{2}x_{4},\ x_{1}x_{2}^{2},\ x_{1}x_{2}x_{3},\\[-2pt]
&\quad x_{1}x_{2}x_{4},\ x_{1}x_{3}^{2},\ x_{1}x_{3}x_{4},\ x_{1}x_{4}^{2},\ x_{1}x_{5}^{2},\ x_{1}x_{5}x_{6},\\[-2pt]
&\quad x_{1}x_{6}^{2},\ x_{2}^{3},\ x_{2}^{2}x_{3},\ x_{2}^{2}x_{4},\ x_{2}x_{3}^{2},\ x_{2}x_{3}x_{4},\\[-2pt]
&\quad x_{2}x_{4}^{2},\ x_{2}x_{5}^{2},\ x_{2}x_{5}x_{6},\ x_{2}x_{6}^{2},\ x_{3}^{3},\ x_{3}^{2}x_{4},\\[-2pt]
&\quad x_{3}x_{4}^{2},\ x_{3}x_{5}^{2},\ x_{3}x_{5}x_{6},\ x_{3}x_{6}^{2},\ x_{4}^{3},\ x_{4}x_{5}^{2},\\[-2pt]
&\quad x_{4}x_{5}x_{6},\ x_{4}x_{6}^{2}\bigr)
\end{aligned}\end{equation*}

\par\Needspace{3\baselineskip}\medskip\noindent\textbf{\boldmath 120.\enspace $(G_s,G)=(C_2,C_2^2)$, family~1.}\label{fam:fourfold-no-120}
$G\text{-ID}=[4,2]$, $H\text{-ID}=[12,5]$.
The non-symplectic index is $2$, and the family dimension is $10$.  The projective action is liftable.
This is a proper subfamily of family~119.

\par\Needspace{4\baselineskip}\medskip\noindent Here $H=\langle K^+,T\rangle$, where\par\nopagebreak
\begin{SecEightMatrixBlock}
\SecEightMatrix{T}{\diag(1,1,1,1,1,-1)}
\end{SecEightMatrixBlock}

\par\smallskip\noindent A basis of $W_H$ is $\{P_i\mid i\in\{1,\ldots,32\}\setminus\{12,21,27,31\}\}$.

\par\Needspace{3\baselineskip}\medskip\noindent\textbf{\boldmath 121.\enspace $(G_s,G)=(C_2,C_2^2)$, family~2.}\label{fam:fourfold-no-121}
$G\text{-ID}=[4,2]$, $H\text{-ID}=[12,5]$.
The non-symplectic index is $2$, and the family dimension is $8$.  The projective action is liftable.
This is a proper subfamily of family~119.

\par\Needspace{4\baselineskip}\medskip\noindent Here $H=\langle K^+,T\rangle$, where\par\nopagebreak
\begin{SecEightMatrixBlock}
\SecEightMatrix{T}{\diag(1,1,1,-1,1,1)}
\end{SecEightMatrixBlock}

\par\smallskip\noindent A basis of $W_H$ is $\{P_i\mid i\in\{1,\ldots,32\}\setminus\{4,7,9,16,18,24,29,\ldots,32\}\}$.

\par\Needspace{3\baselineskip}\medskip\noindent\textbf{\boldmath 122.\enspace $(G_s,G)=(C_2,C_2^2)$, family~3.}\label{fam:fourfold-no-122}
$G\text{-ID}=[4,2]$, $H\text{-ID}=[12,5]$.
The non-symplectic index is $2$, and the family dimension is $6$.  The projective action is liftable.
This is a proper subfamily of family~119.

\par\Needspace{4\baselineskip}\medskip\noindent Here $H=\langle K^+,T\rangle$, where\par\nopagebreak
\begin{SecEightMatrixBlock}
\SecEightMatrix{T}{\diag(1,1,-1,-1,1,-1)}
\end{SecEightMatrixBlock}

\par\smallskip\noindent A basis of $W_H$ is $\{P_i\mid i\in\{1,\allowbreak 2,\allowbreak 5,\allowbreak 8,\allowbreak \ldots,\allowbreak 11,\allowbreak 13,\allowbreak 14,\allowbreak 17,\allowbreak \ldots,\allowbreak 20,\allowbreak 22,\allowbreak 27,\allowbreak 31\}\}$.

\par\Needspace{3\baselineskip}\medskip\noindent\textbf{\boldmath 123.\enspace $(G_s,G)=(C_2,C_6)$, family~1.}\label{fam:fourfold-no-123}
$G\text{-ID}=[6,2]$, $H\text{-ID}=[18,5]$.
The non-symplectic index is $3$, and the family dimension is $6$.  The projective action is liftable.
This is a proper subfamily of family~119.

\par\Needspace{4\baselineskip}\medskip\noindent Here $H=\langle K^+,T\rangle$, where\par\nopagebreak
\begin{SecEightMatrixBlock}
\SecEightMatrix{T}{\diag(1,1,1,\omega,1,1)}
\end{SecEightMatrixBlock}

\par\smallskip\noindent A basis of $W_H$ is $\{P_i\mid i\in\{1,\allowbreak 2,\allowbreak 3,\allowbreak 5,\allowbreak 6,\allowbreak 8,\allowbreak 11,\allowbreak \ldots,\allowbreak 15,\allowbreak 17,\allowbreak 20,\allowbreak \ldots,\allowbreak 23,\allowbreak 26,\allowbreak \ldots,\allowbreak 29\}\}$.

\par\Needspace{3\baselineskip}\medskip\noindent\textbf{\boldmath 124.\enspace $(G_s,G)=(C_2,C_4\times C_2)$, family~1.}\label{fam:fourfold-no-124}
$G\text{-ID}=[8,2]$, $H\text{-ID}=[24,9]$.
The non-symplectic index is $4$, and the family dimension is $5$.  The projective action is liftable.
This is a proper subfamily of family~119.

\par\Needspace{4\baselineskip}\medskip\noindent Here $H=\langle K^+,T\rangle$, where\par\nopagebreak
\begin{SecEightMatrixBlock}
\SecEightMatrix{T}{\diag(1,1,1,-1,1,-\zeta_4)}
\end{SecEightMatrixBlock}

\par\smallskip\noindent A basis of $W_H$ is $\{P_i\mid i\in\{1,\allowbreak 2,\allowbreak 3,\allowbreak 5,\allowbreak 6,\allowbreak 8,\allowbreak 10,\allowbreak 11,\allowbreak 14,\allowbreak 15,\allowbreak 17,\allowbreak 19,\allowbreak 20,\allowbreak 23,\allowbreak 25,\allowbreak 26,\allowbreak 32\}\}$.

\par\Needspace{3\baselineskip}\medskip\noindent\textbf{\boldmath 125.\enspace $(G_s,G)=(C_2,C_6\times C_2)$, family~1.}\label{fam:fourfold-no-125}
$G\text{-ID}=[12,5]$, $H\text{-ID}=[36,14]$.
The non-symplectic index is $6$, and the family dimension is $5$.  The projective action is liftable.
This is a proper subfamily of family~119.

\par\Needspace{4\baselineskip}\medskip\noindent Here $H=\langle K^+,T\rangle$, where\par\nopagebreak
\begin{SecEightMatrixBlock}
\SecEightMatrix{T}{\diag(1,1,1,\omega,-1,1)}
\end{SecEightMatrixBlock}

\par\smallskip\noindent A basis of $W_H$ is $\{P_i\mid i\in\{1,\allowbreak 2,\allowbreak 3,\allowbreak 5,\allowbreak 6,\allowbreak 8,\allowbreak 11,\allowbreak 13,\allowbreak 14,\allowbreak 15,\allowbreak 17,\allowbreak 20,\allowbreak 22,\allowbreak 23,\allowbreak 26,\allowbreak 28,\allowbreak 29\}\}$.

\par\Needspace{3\baselineskip}\medskip\noindent\textbf{\boldmath 126.\enspace $(G_s,G)=(C_2,C_6)$, family~2.}\label{fam:fourfold-no-126}
$G\text{-ID}=[6,2]$, $H\text{-ID}=[18,5]$.
The non-symplectic index is $3$, and the family dimension is $4$.  The projective action is liftable.
This is a proper subfamily of family~119.

\par\Needspace{4\baselineskip}\medskip\noindent Here $H=\langle K^+,T\rangle$, where\par\nopagebreak
\begin{SecEightMatrixBlock}
\SecEightMatrix{T}{\diag(\omega^2,\omega,1,1,1,\omega)}
\end{SecEightMatrixBlock}

\par\smallskip\noindent A basis of $W_H$ is $\{P_i\mid i\in\{1,\allowbreak 6,\allowbreak 7,\allowbreak 12,\allowbreak 14,\allowbreak 22,\allowbreak \ldots,\allowbreak 26,\allowbreak 29,\allowbreak 30\}\}$.

\par\Needspace{3\baselineskip}\medskip\noindent\textbf{\boldmath 127.\enspace $(G_s,G)=(C_2,C_4\times C_2)$, family~2.}\label{fam:fourfold-no-127}
$G\text{-ID}=[8,2]$, $H\text{-ID}=[24,9]$.
See also \cite[Theorem~4.2(3)]{zheng2021liftableabelian}.
The non-symplectic index is $4$, and the family dimension is $4$.  The projective action is liftable.
This is a proper subfamily of family~119.

\par\Needspace{4\baselineskip}\medskip\noindent Here $H=\langle K^+,T\rangle$, where\par\nopagebreak
\begin{SecEightMatrixBlock}
\SecEightMatrix{T}{\diag(\zeta_4,-1,1,1,1,-1)}
\end{SecEightMatrixBlock}

\par\smallskip\noindent A basis of $W_H$ is $\{P_i\mid i\in\{2,\allowbreak 15,\allowbreak 16,\allowbreak 21,\allowbreak 23,\allowbreak \ldots,\allowbreak 26,\allowbreak 28,\allowbreak 29,\allowbreak 30,\allowbreak 32\}\}$.

\par\Needspace{3\baselineskip}\medskip\noindent\textbf{\boldmath 128.\enspace $(G_s,G)=(C_2,C_6\times C_2)$, family~2.}\label{fam:fourfold-no-128}
$G\text{-ID}=[12,5]$, $H\text{-ID}=[36,14]$.
The non-symplectic index is $6$, and the family dimension is $3$.  The projective action is liftable.
This is a proper subfamily of family~119.

\par\Needspace{4\baselineskip}\medskip\noindent Here $H=\langle K^+,T\rangle$, where\par\nopagebreak
\begin{SecEightMatrixBlock}
\SecEightMatrix{T}{\diag(1,1,\omega,\omega^2,-1,\omega)}
\end{SecEightMatrixBlock}

\par\smallskip\noindent A basis of $W_H$ is $\{P_i\mid i\in\{1,\allowbreak 2,\allowbreak 5,\allowbreak 9,\allowbreak 11,\allowbreak 14,\allowbreak 18,\allowbreak 20,\allowbreak 23,\allowbreak 28,\allowbreak 29\}\}$.

\par\Needspace{3\baselineskip}\medskip\noindent\textbf{\boldmath 129.\enspace $(G_s,G)=(C_2,C_8\times C_2)$.}\label{fam:fourfold-no-129}
$G\text{-ID}=[16,5]$, $H\text{-ID}=[48,23]$.
The non-symplectic index is $8$, and the family dimension is $2$.  The projective action is liftable.
This is a proper subfamily of family~119.

\par\Needspace{4\baselineskip}\medskip\noindent Here $H=\langle K^+,T\rangle$, where\par\nopagebreak
\begin{SecEightMatrixBlock}
\SecEightMatrix{T}{\diag(-\zeta_4,-1,1,1,\zeta_8,1)}
\end{SecEightMatrixBlock}

\par\smallskip\noindent A basis of $W_H$ is $\{P_i\mid i\in\{2,\allowbreak 11,\allowbreak 15,\allowbreak 16,\allowbreak 23,\allowbreak 24,\allowbreak 25,\allowbreak 28,\allowbreak 29,\allowbreak 32\}\}$.

\par\Needspace{3\baselineskip}\medskip\noindent\textbf{\boldmath 130.\enspace $(G_s,G)=(C_2,C_{12}\times C_2)$, family~1.}\label{fam:fourfold-no-130}
$G\text{-ID}=[24,9]$, $H\text{-ID}=[72,36]$.
The non-symplectic index is $12$, and the family dimension is $2$.  The projective action is liftable.
This is a proper subfamily of family~119.

\par\Needspace{4\baselineskip}\medskip\noindent Here $H=\langle K^+,T\rangle$, where\par\nopagebreak
\begin{SecEightMatrixBlock}
\SecEightMatrix{T}{\diag(-1,1,1,\omega,\zeta_4,1)}
\end{SecEightMatrixBlock}

\par\smallskip\noindent A basis of $W_H$ is $\{P_i\mid i\in\{2,\allowbreak 3,\allowbreak 11,\allowbreak 14,\allowbreak 15,\allowbreak 17,\allowbreak 22,\allowbreak 23,\allowbreak 28,\allowbreak 29\}\}$.

\par\Needspace{3\baselineskip}\medskip\noindent\textbf{\boldmath 131.\enspace $(G_s,G)=(C_2,C_{12}\times C_2)$, family~2.}\label{fam:fourfold-no-131}
$G\text{-ID}=[24,9]$, $H\text{-ID}=[72,36]$.
The non-symplectic index is $12$, and the family dimension is $1$.  The projective action is liftable.
This is a proper subfamily of family~119.

\par\Needspace{4\baselineskip}\medskip\noindent Here $H=\langle K^+,T\rangle$, where\par\nopagebreak
\begin{SecEightMatrixBlock}
\SecEightMatrix{T}{\diag(\zeta_6,\omega^2,1,\omega,\zeta_{12}^{11},1)}
\end{SecEightMatrixBlock}

\par\smallskip\noindent A basis of $W_H$ is $\{P_i\mid i\in\{2,\allowbreak 11,\allowbreak 14,\allowbreak 18,\allowbreak 23,\allowbreak 28,\allowbreak 29\}\}$.

\par\Needspace{3\baselineskip}\medskip\noindent\textbf{\boldmath 132.\enspace $(G_s,G)=(1,1)$.}\label{fam:fourfold-no-132}
This is the symplectic family with $G_s=1$ and generic non-symplectic
index $1$.  Its proper subfamilies with $G_s=1$ are families~133--156.
We take $K^+=\langle\omega I_6\rangle$ and order
the $56$ cubic monomials as follows:
\begin{equation*}\begin{aligned}
(P_1,\ldots,P_{56})&=\bigl({}  x_{1}^{3},\ x_{1}^{2}\,x_{2},\ x_{1}^{2}\,x_{3},\ x_{1}^{2}\,x_{4},\ x_{1}^{2}\,x_{5},\ x_{1}^{2}\,x_{6},\ x_{1}\,x_{2}^{2},\\[-2pt]
&\quad   x_{1}\,x_{2}\,x_{3},\ x_{1}\,x_{2}\,x_{4},\ x_{1}\,x_{2}\,x_{5},\ x_{1}\,x_{2}\,x_{6},\ x_{1}\,x_{3}^{2},\ x_{1}\,x_{3}\,x_{4},\ x_{1}\,x_{3}\,x_{5},\\[-2pt]
&\quad   x_{1}\,x_{3}\,x_{6},\ x_{1}\,x_{4}^{2},\ x_{1}\,x_{4}\,x_{5},\ x_{1}\,x_{4}\,x_{6},\ x_{1}\,x_{5}^{2},\ x_{1}\,x_{5}\,x_{6},\ x_{1}\,x_{6}^{2},\\[-2pt]
&\quad   x_{2}^{3},\ x_{2}^{2}\,x_{3},\ x_{2}^{2}\,x_{4},\ x_{2}^{2}\,x_{5},\ x_{2}^{2}\,x_{6},\ x_{2}\,x_{3}^{2},\ x_{2}\,x_{3}\,x_{4},\\[-2pt]
&\quad   x_{2}\,x_{3}\,x_{5},\ x_{2}\,x_{3}\,x_{6},\ x_{2}\,x_{4}^{2},\ x_{2}\,x_{4}\,x_{5},\ x_{2}\,x_{4}\,x_{6},\ x_{2}\,x_{5}^{2},\ x_{2}\,x_{5}\,x_{6},\\[-2pt]
&\quad   x_{2}\,x_{6}^{2},\ x_{3}^{3},\ x_{3}^{2}\,x_{4},\ x_{3}^{2}\,x_{5},\ x_{3}^{2}\,x_{6},\ x_{3}\,x_{4}^{2},\ x_{3}\,x_{4}\,x_{5},\\[-2pt]
&\quad   x_{3}\,x_{4}\,x_{6},\ x_{3}\,x_{5}^{2},\ x_{3}\,x_{5}\,x_{6},\ x_{3}\,x_{6}^{2},\ x_{4}^{3},\ x_{4}^{2}\,x_{5},\ x_{4}^{2}\,x_{6},\\[-2pt]
&\quad   x_{4}\,x_{5}^{2},\ x_{4}\,x_{5}\,x_{6},\ x_{4}\,x_{6}^{2},\ x_{5}^{3},\ x_{5}^{2}\,x_{6},\ x_{5}\,x_{6}^{2},\ x_{6}^{3}\bigr).
\end{aligned}\end{equation*}
\par\Needspace{3\baselineskip}\medskip\noindent\textbf{\boldmath 133.\enspace $(G_s,G)=(1,C_2)$, family~1.}\label{fam:fourfold-no-133}
$G\text{-ID}=[2,1]$, $H\text{-ID}=[6,2]$.
The non-symplectic index is $2$, and the family dimension is $14$.  The projective action is liftable.
This is a proper subfamily of family~132.

\par\Needspace{4\baselineskip}\medskip\noindent With $K^+=\langle\omega I_6\rangle$, $H=\langle K^+,T\rangle$, where\par\nopagebreak
\begin{SecEightMatrixBlock}
\SecEightMatrix{T}{\diag(-1,1,1,1,1,1)}
\end{SecEightMatrixBlock}

\par\smallskip\noindent A basis of $W_H$ is $\{P_i\mid 2\le i\le6\text{ or }22\le i\le56\}$.

\par\Needspace{3\baselineskip}\medskip\noindent\textbf{\boldmath 134.\enspace $(G_s,G)=(1,C_2)$, family~2.}\label{fam:fourfold-no-134}
$G\text{-ID}=[2,1]$, $H\text{-ID}=[6,2]$.
The non-symplectic index is $2$, and the family dimension is $10$.  The projective action is liftable.
This is a proper subfamily of family~132.

\par\Needspace{4\baselineskip}\medskip\noindent With $K^+=\langle\omega I_6\rangle$, $H=\langle K^+,T\rangle$, where\par\nopagebreak
\begin{SecEightMatrixBlock}
\SecEightMatrix{T}{\diag(1,1,1,-1,-1,-1)}
\end{SecEightMatrixBlock}

\par\smallskip\noindent A basis of $W_H$ is $\{P_i\mid i\in\{1,\allowbreak 2,\allowbreak 3,\allowbreak 7,\allowbreak 8,\allowbreak 12,\allowbreak 16,\allowbreak \ldots,\allowbreak 23,\allowbreak 27,\allowbreak 31,\allowbreak \ldots,\allowbreak 37,\allowbreak 41,\allowbreak \ldots,\allowbreak 46\}\}$.

\par\Needspace{3\baselineskip}\medskip\noindent\textbf{\boldmath 135.\enspace $(G_s,G)=(1,C_3)$, family~1.}\label{fam:fourfold-no-135}
$G\text{-ID}=[3,1]$, $H\text{-ID}=[9,2]$.
The non-symplectic index is $3$, and the family dimension is $10$.  The projective action is liftable.
This is a proper subfamily of family~132.

\par\Needspace{4\baselineskip}\medskip\noindent With $K^+=\langle\omega I_6\rangle$, $H=\langle K^+,T\rangle$, where\par\nopagebreak
\begin{SecEightMatrixBlock}
\SecEightMatrix{T}{\diag(1,1,1,\omega,1,1)}
\end{SecEightMatrixBlock}

\par\smallskip\noindent A basis of $W_H$ is $\{P_i\mid i\in\{1,\allowbreak \ldots,\allowbreak 56\}\setminus \{4,\allowbreak 9,\allowbreak 13,\allowbreak 16,\allowbreak \ldots,\allowbreak 18,\allowbreak 24,\allowbreak 28,\allowbreak 31,\allowbreak \ldots,\allowbreak 33,\allowbreak 38,\allowbreak 41,\allowbreak \ldots,\allowbreak 43,\allowbreak 48,\allowbreak \ldots,\allowbreak 52\}\}$.

\par\Needspace{3\baselineskip}\medskip\noindent\textbf{\boldmath 136.\enspace $(G_s,G)=(1,C_3)$, family~2.}\label{fam:fourfold-no-136}
$G\text{-ID}=[3,1]$, $H\text{-ID}=[9,2]$.
The non-symplectic index is $3$, and the family dimension is $7$.  The projective action is liftable.
This is a proper subfamily of family~132.

\par\Needspace{4\baselineskip}\medskip\noindent With $K^+=\langle\omega I_6\rangle$, $H=\langle K^+,T\rangle$, where\par\nopagebreak
\begin{SecEightMatrixBlock}
\SecEightMatrix{T}{\diag(\omega^2,1,\omega,\omega,1,1)}
\end{SecEightMatrixBlock}

\par\smallskip\noindent A basis of $W_H$ is $\{P_i\mid i\in\{1,\allowbreak 8,\allowbreak 9,\allowbreak 14,\allowbreak 15,\allowbreak 17,\allowbreak 18,\allowbreak 22,\allowbreak 25,\allowbreak 26,\allowbreak 34,\allowbreak \ldots,\allowbreak 38,\allowbreak 41,\allowbreak 47,\allowbreak 53,\allowbreak \ldots,\allowbreak 56\}\}$.

\par\Needspace{3\baselineskip}\medskip\noindent\textbf{\boldmath 137.\enspace $(G_s,G)=(1,C_4)$, family~1.}\label{fam:fourfold-no-137}
$G\text{-ID}=[4,1]$, $H\text{-ID}=[12,2]$.
The non-symplectic index is $4$, and the family dimension is $7$.  The projective action is liftable.
This is a proper subfamily of family~132.

\par\Needspace{4\baselineskip}\medskip\noindent With $K^+=\langle\omega I_6\rangle$, $H=\langle K^+,T\rangle$, where\par\nopagebreak
\begin{SecEightMatrixBlock}
\SecEightMatrix{T}{\diag(-\zeta_4,-1,1,1,1,1)}
\end{SecEightMatrixBlock}

\par\smallskip\noindent A basis of $W_H$ is $\{P_i\mid i\in\{2,23,\ldots,26,37,\ldots,56\}\}$.

\par\Needspace{3\baselineskip}\medskip\noindent\textbf{\boldmath 138.\enspace $(G_s,G)=(1,C_4)$, family~2.}\label{fam:fourfold-no-138}
$G\text{-ID}=[4,1]$, $H\text{-ID}=[12,2]$.
The non-symplectic index is $4$, and the family dimension is $7$.  The projective action is liftable.
This is a proper subfamily of family~132.

\par\Needspace{4\baselineskip}\medskip\noindent With $K^+=\langle\omega I_6\rangle$, $H=\langle K^+,T\rangle$, where\par\nopagebreak
\begin{SecEightMatrixBlock}
\SecEightMatrix{T}{\diag(\zeta_4,-1,1,-1,1,1)}
\end{SecEightMatrixBlock}

\par\smallskip\noindent A basis of $W_H$ is\par\nopagebreak
\begin{SecEightBasisBlock}
\SecEightBasisLine{\{P_i\mid i\in\{2,4,23,25,26,28,32,33,37,39,40,41,44,45,46,48,49,53,\ldots,56\}\}.}
\end{SecEightBasisBlock}

\par\Needspace{3\baselineskip}\medskip\noindent\textbf{\boldmath 139.\enspace $(G_s,G)=(1,C_6)$, family~1.}\label{fam:fourfold-no-139}
$G\text{-ID}=[6,2]$, $H\text{-ID}=[18,5]$.
The non-symplectic index is $6$, and the family dimension is $7$.  The projective action is liftable.
This is a proper subfamily of family~132.

\par\Needspace{4\baselineskip}\medskip\noindent With $K^+=\langle\omega I_6\rangle$, $H=\langle K^+,T\rangle$, where\par\nopagebreak
\begin{SecEightMatrixBlock}
\SecEightMatrix{T}{\diag(-1,1,1,1,1,\omega^2)}
\end{SecEightMatrixBlock}

\par\smallskip\noindent A basis of $W_H$ is $\{P_i\mid i\in\{2,\allowbreak \ldots,\allowbreak 5,\allowbreak 22,\allowbreak \ldots,\allowbreak 25,\allowbreak 27,\allowbreak 28,\allowbreak 29,\allowbreak 31,\allowbreak 32,\allowbreak 34,\allowbreak 37,\allowbreak 38,\allowbreak 39,\allowbreak 41,\allowbreak 42,\allowbreak 44,\allowbreak 47,\allowbreak 48,\allowbreak 50,\allowbreak 53,\allowbreak 56\}\}$.

\par\Needspace{3\baselineskip}\medskip\noindent\textbf{\boldmath 140.\enspace $(G_s,G)=(1,C_3)$, family~3.}\label{fam:fourfold-no-140}
$G\text{-ID}=[3,1]$, $H\text{-ID}=[9,1]$.
The non-symplectic index is $3$, and the family dimension is $6$.  The projective action is liftable.
This is a proper subfamily of family~132.

\par\Needspace{4\baselineskip}\medskip\noindent With $K^+=\langle\omega I_6\rangle$, $H=\langle K^+,T\rangle$, where\par\nopagebreak
\begin{SecEightMatrixBlock}
\SecEightMatrix{T}{\diag(\zeta_9^7,\zeta_9^4,\zeta_9,\zeta_9,\zeta_9^7,\zeta_9^4)}
\end{SecEightMatrixBlock}

\par\smallskip\noindent A basis of $W_H$ is $\{P_i\mid i\in\{2,\allowbreak 6,\allowbreak 10,\allowbreak 12,\allowbreak 13,\allowbreak 16,\allowbreak 20,\allowbreak 23,\allowbreak 24,\allowbreak 30,\allowbreak 33,\allowbreak 34,\allowbreak 39,\allowbreak 42,\allowbreak 46,\allowbreak 48,\allowbreak 52,\allowbreak 54\}\}$.

\par\Needspace{3\baselineskip}\medskip\noindent\textbf{\boldmath 141.\enspace $(G_s,G)=(1,C_6)$, family~2.}\label{fam:fourfold-no-141}
$G\text{-ID}=[6,2]$, $H\text{-ID}=[18,5]$.
The non-symplectic index is $6$, and the family dimension is $5$.  The projective action is liftable.
This is a proper subfamily of family~132.

\par\Needspace{4\baselineskip}\medskip\noindent With $K^+=\langle\omega I_6\rangle$, $H=\langle K^+,T\rangle$, where\par\nopagebreak
\begin{SecEightMatrixBlock}
\SecEightMatrix{T}{\diag(-1,1,1,\omega^2,\omega,\omega)}
\end{SecEightMatrixBlock}

\par\smallskip\noindent A basis of $W_H$ is $\{P_i\mid i\in\{2,\allowbreak 3,\allowbreak 22,\allowbreak 23,\allowbreak 27,\allowbreak 32,\allowbreak 33,\allowbreak 37,\allowbreak 42,\allowbreak 43,\allowbreak 47,\allowbreak 53,\allowbreak \ldots,\allowbreak 56\}\}$.

\par\Needspace{3\baselineskip}\medskip\noindent\textbf{\boldmath 142.\enspace $(G_s,G)=(1,C_6)$, family~3.}\label{fam:fourfold-no-142}
$G\text{-ID}=[6,2]$, $H\text{-ID}=[18,5]$.
The non-symplectic index is $6$, and the family dimension is $4$.  The projective action is liftable.
This is a proper subfamily of family~132.

\par\Needspace{4\baselineskip}\medskip\noindent With $K^+=\langle\omega I_6\rangle$, $H=\langle K^+,T\rangle$, where\par\nopagebreak
\begin{SecEightMatrixBlock}
\SecEightMatrix{T}{\diag(\zeta_6,\omega^2,1,1,\omega,1)}
\end{SecEightMatrixBlock}

\par\smallskip\noindent A basis of $W_H$ is $\{P_i\mid i\in\{2,\allowbreak 22,\allowbreak 29,\allowbreak 32,\allowbreak 35,\allowbreak 37,\allowbreak 38,\allowbreak 40,\allowbreak 41,\allowbreak 43,\allowbreak 46,\allowbreak 47,\allowbreak 49,\allowbreak 52,\allowbreak 53,\allowbreak 56\}\}$.

\par\Needspace{3\baselineskip}\medskip\noindent\textbf{\boldmath 143.\enspace $(G_s,G)=(1,C_6)$, family~4.}\label{fam:fourfold-no-143}
$G\text{-ID}=[6,2]$, $H\text{-ID}=[18,2]$.
See also \cite[Theorem~4.2(1)]{zheng2021liftableabelian}.
The non-symplectic index is $6$, and the family dimension is $3$.  The projective action is liftable.
This is a proper subfamily of family~132.

\par\Needspace{4\baselineskip}\medskip\noindent With $K^+=\langle\omega I_6\rangle$, $H=\langle K^+,T\rangle$, where\par\nopagebreak
\begin{SecEightMatrixBlock}
\SecEightMatrix{T}{\diag(\zeta_{18}^4,\zeta_{18}^{10},\zeta_{18}^{16},\zeta_{18}^7,\zeta_{18}^{13},\zeta_{18})}
\end{SecEightMatrixBlock}

\par\smallskip\noindent A basis of $W_H$ is $\{P_i\mid i\in\{2, 12, 16, 20, 23, 33, 34, 42, 46\}\}$.

\par\Needspace{3\baselineskip}\medskip\noindent\textbf{\boldmath 144.\enspace $(G_s,G)=(1,C_8)$, family~1.}\label{fam:fourfold-no-144}
$G\text{-ID}=[8,1]$, $H\text{-ID}=[24,2]$.
The non-symplectic index is $8$, and the family dimension is $3$.  The projective action is liftable.
This is a proper subfamily of family~132.

\par\Needspace{4\baselineskip}\medskip\noindent With $K^+=\langle\omega I_6\rangle$, $H=\langle K^+,T\rangle$, where\par\nopagebreak
\begin{SecEightMatrixBlock}
\SecEightMatrix{T}{\diag(\zeta_8^3,\zeta_4,-1,1,1,1)}
\end{SecEightMatrixBlock}

\par\smallskip\noindent A basis of $W_H$ is $\{P_i\mid i\in\{2,23,38,39,40,47,\ldots,56\}\}$.

\par\Needspace{3\baselineskip}\medskip\noindent\textbf{\boldmath 145.\enspace $(G_s,G)=(1,C_8)$, family~2.}\label{fam:fourfold-no-145}
$G\text{-ID}=[8,1]$, $H\text{-ID}=[24,2]$.
The non-symplectic index is $8$, and the family dimension is $3$.  The projective action is liftable.
This is a proper subfamily of family~132.

\par\Needspace{4\baselineskip}\medskip\noindent With $K^+=\langle\omega I_6\rangle$, $H=\langle K^+,T\rangle$, where\par\nopagebreak
\begin{SecEightMatrixBlock}
\SecEightMatrix{T}{\diag(\zeta_8,-\zeta_4,-1,-1,1,1)}
\end{SecEightMatrixBlock}

\par\smallskip\noindent A basis of $W_H$ is $\{P_i\mid i\in\{2,\allowbreak 23,\allowbreak 24,\allowbreak 39,\allowbreak 40,\allowbreak 42,\allowbreak 43,\allowbreak 48,\allowbreak 49,\allowbreak 53,\allowbreak \ldots,\allowbreak 56\}\}$.

\par\Needspace{3\baselineskip}\medskip\noindent\textbf{\boldmath 146.\enspace $(G_s,G)=(1,C_8)$, family~3.}\label{fam:fourfold-no-146}
$G\text{-ID}=[8,1]$, $H\text{-ID}=[24,2]$.
The non-symplectic index is $8$, and the family dimension is $3$.  The projective action is liftable.
This is a proper subfamily of family~132.

\par\Needspace{4\baselineskip}\medskip\noindent With $K^+=\langle\omega I_6\rangle$, $H=\langle K^+,T\rangle$, where\par\nopagebreak
\begin{SecEightMatrixBlock}
\SecEightMatrix{T}{\diag(\zeta_8^7,\zeta_4,-\zeta_4,-1,1,1)}
\end{SecEightMatrixBlock}

\par\smallskip\noindent A basis of $W_H$ is $\{P_i\mid i\in\{2,\allowbreak 24,\allowbreak 29,\allowbreak 30,\allowbreak 38,\allowbreak 48,\allowbreak 49,\allowbreak 53,\allowbreak \ldots,\allowbreak 56\}\}$.

\par\Needspace{3\baselineskip}\medskip\noindent\textbf{\boldmath 147.\enspace $(G_s,G)=(1,C_{12})$, family~1.}\label{fam:fourfold-no-147}
$G\text{-ID}=[12,2]$, $H\text{-ID}=[36,8]$.
The non-symplectic index is $12$, and the family dimension is $3$.  The projective action is liftable.
This is a proper subfamily of family~132.

\par\Needspace{4\baselineskip}\medskip\noindent With $K^+=\langle\omega I_6\rangle$, $H=\langle K^+,T\rangle$, where\par\nopagebreak
\begin{SecEightMatrixBlock}
\SecEightMatrix{T}{\diag(\zeta_4,-1,1,1,1,\omega)}
\end{SecEightMatrixBlock}

\par\smallskip\noindent A basis of $W_H$ is $\{P_i\mid i\in\{2,\allowbreak 23,\allowbreak 24,\allowbreak 25,\allowbreak 37,\allowbreak 38,\allowbreak 39,\allowbreak 41,\allowbreak 42,\allowbreak 44,\allowbreak 47,\allowbreak 48,\allowbreak 50,\allowbreak 53,\allowbreak 56\}\}$.

\par\Needspace{3\baselineskip}\medskip\noindent\textbf{\boldmath 148.\enspace $(G_s,G)=(1,C_{12})$, family~2.}\label{fam:fourfold-no-148}
$G\text{-ID}=[12,2]$, $H\text{-ID}=[36,8]$.
The non-symplectic index is $12$, and the family dimension is $3$.  The projective action is liftable.
This is a proper subfamily of family~132.

\par\Needspace{4\baselineskip}\medskip\noindent With $K^+=\langle\omega I_6\rangle$, $H=\langle K^+,T\rangle$, where\par\nopagebreak
\begin{SecEightMatrixBlock}
\SecEightMatrix{T}{\diag(\zeta_4,-1,1,-1,1,\omega)}
\end{SecEightMatrixBlock}

\par\smallskip\noindent A basis of $W_H$ is $\{P_i\mid i\in\{2,\allowbreak 4,\allowbreak 23,\allowbreak 25,\allowbreak 28,\allowbreak 32,\allowbreak 37,\allowbreak 39,\allowbreak 41,\allowbreak 44,\allowbreak 48,\allowbreak 53,\allowbreak 56\}\}$.

\par\Needspace{3\baselineskip}\medskip\noindent\textbf{\boldmath 149.\enspace $(G_s,G)=(1,C_{12})$, family~3.}\label{fam:fourfold-no-149}
$G\text{-ID}=[12,2]$, $H\text{-ID}=[36,8]$.
The non-symplectic index is $12$, and the family dimension is $2$.  The projective action is liftable.
This is a proper subfamily of family~132.

\par\Needspace{4\baselineskip}\medskip\noindent With $K^+=\langle\omega I_6\rangle$, $H=\langle K^+,T\rangle$, where\par\nopagebreak
\begin{SecEightMatrixBlock}
\SecEightMatrix{T}{\diag(\zeta_{12}^{11},\zeta_6,\omega^2,\omega,1,1)}
\end{SecEightMatrixBlock}

\par\smallskip\noindent A basis of $W_H$ is $\{P_i\mid i\in\{2,\allowbreak 23,\allowbreak 37,\allowbreak 42,\allowbreak 43,\allowbreak 47,\allowbreak 53,\allowbreak \ldots,\allowbreak 56\}\}$.

\par\Needspace{3\baselineskip}\medskip\noindent\textbf{\boldmath 150.\enspace $(G_s,G)=(1,C_{12})$, family~4.}\label{fam:fourfold-no-150}
$G\text{-ID}=[12,2]$, $H\text{-ID}=[36,8]$.
The non-symplectic index is $12$, and the family dimension is $2$.  The projective action is liftable.
This is a proper subfamily of family~132.

\par\Needspace{4\baselineskip}\medskip\noindent With $K^+=\langle\omega I_6\rangle$, $H=\langle K^+,T\rangle$, where\par\nopagebreak
\begin{SecEightMatrixBlock}
\SecEightMatrix{T}{\diag(\zeta_4,-1,1,-\omega,\omega,\omega^2)}
\end{SecEightMatrixBlock}

\par\smallskip\noindent A basis of $W_H$ is $\{P_i\mid i\in\{2, 23, 33, 37, 45, 48, 53, 56\}\}$.

\par\Needspace{3\baselineskip}\medskip\noindent\textbf{\boldmath 151.\enspace $(G_s,G)=(1,C_{16})$, family~1.}\label{fam:fourfold-no-151}
$G\text{-ID}=[16,1]$, $H\text{-ID}=[48,2]$.
The non-symplectic index is $16$, and the family dimension is $1$.  The projective action is liftable.
This is a proper subfamily of family~132.

\par\Needspace{4\baselineskip}\medskip\noindent With $K^+=\langle\omega I_6\rangle$, $H=\langle K^+,T\rangle$, where\par\nopagebreak
\begin{SecEightMatrixBlock}
\SecEightMatrix{T}{\diag(\zeta_{16}^{11},-\zeta_8,-\zeta_4,-1,1,1)}
\end{SecEightMatrixBlock}

\par\smallskip\noindent A basis of $W_H$ is $\{P_i\mid i\in\{2, 23, 38, 48, 49, 53, 54, 55, 56\}\}$.

\par\Needspace{3\baselineskip}\medskip\noindent\textbf{\boldmath 152.\enspace $(G_s,G)=(1,C_{16})$, family~2.}\label{fam:fourfold-no-152}
$G\text{-ID}=[16,1]$, $H\text{-ID}=[48,2]$.
See also \cite[Proposition~3.3 and Theorem~4.2]{zheng2021liftableabelian}.
The non-symplectic index is $16$, and the family dimension is $1$.  The projective action is liftable.
This is a proper subfamily of family~132.

\par\Needspace{4\baselineskip}\medskip\noindent With $K^+=\langle\omega I_6\rangle$, $H=\langle K^+,T\rangle$, where\par\nopagebreak
\begin{SecEightMatrixBlock}
\SecEightMatrix{T}{\diag(\zeta_{16},-\zeta_8^3,\zeta_4,-\zeta_4,-1,1)}
\end{SecEightMatrixBlock}

\par\smallskip\noindent A basis of $W_H$ is $\{P_i\mid i\in\{2,23,39,43,48,54,56\}\}$.

\par\Needspace{3\baselineskip}\medskip\noindent\textbf{\boldmath 153.\enspace $(G_s,G)=(1,C_{24})$, family~1.}\label{fam:fourfold-no-153}
$G\text{-ID}=[24,2]$, $H\text{-ID}=[72,14]$.
The non-symplectic index is $24$, and the family dimension is $1$.  The projective action is liftable.
This is a proper subfamily of family~132.

\par\Needspace{4\baselineskip}\medskip\noindent With $K^+=\langle\omega I_6\rangle$, $H=\langle K^+,T\rangle$, where\par\nopagebreak
\begin{SecEightMatrixBlock}
\SecEightMatrix{T}{\diag(\zeta_8,-\zeta_4,-1,1,1,\omega)}
\end{SecEightMatrixBlock}

\par\smallskip\noindent A basis of $W_H$ is $\{P_i\mid i\in\{2, 23, 38, 39, 47, 48, 50, 53, 56\}\}$.

\par\Needspace{3\baselineskip}\medskip\noindent\textbf{\boldmath 154.\enspace $(G_s,G)=(1,C_{24})$, family~2.}\label{fam:fourfold-no-154}
$G\text{-ID}=[24,2]$, $H\text{-ID}=[72,14]$.
See also \cite[Proposition~3.3 and Theorem~4.2]{zheng2021liftableabelian}.
The non-symplectic index is $24$, and the family dimension is $1$.  The projective action is liftable.
This is a proper subfamily of family~132.

\par\Needspace{4\baselineskip}\medskip\noindent With $K^+=\langle\omega I_6\rangle$, $H=\langle K^+,T\rangle$, where\par\nopagebreak
\begin{SecEightMatrixBlock}
\SecEightMatrix{T}{\diag(-\zeta_8,-\zeta_4,\zeta_4,-1,1,\omega^2)}
\end{SecEightMatrixBlock}

\par\smallskip\noindent A basis of $W_H$ is $\{P_i\mid i\in\{2,24,29,38,48,53,56\}\}$.

\par\Needspace{3\baselineskip}\medskip\noindent\textbf{\boldmath 155.\enspace $(G_s,G)=(1,C_{32})$.}\label{fam:fourfold-no-155}
$G\text{-ID}=[32,1]$, $H\text{-ID}=[96,2]$.
See
\cite[\S2.2 and Theorem~4.2]{zheng2021liftableabelian},
\cite[\S6.1, (8)]{yang2024automorphism}, and
\cite[Proposition~7.1]{FWZ26}.

\par\Needspace{3\baselineskip}\medskip\noindent\textbf{\boldmath 156.\enspace $(G_s,G)=(1,C_{48})$.}\label{fam:fourfold-no-156}
$G\text{-ID}=[48,2]$, $H\text{-ID}=[144,30]$.
See
\cite[\S2.2 and Theorem~4.2]{zheng2021liftableabelian},
\cite[\S6.1, (5)]{yang2024automorphism}, and
\cite[Proposition~7.1]{FWZ26}.

\subsubsection{Cubic threefolds}
\label{subsubsection: description of threefolds}

The entries are numbered as in Table~\ref{tab:threefold} and ordered by
their source cubic fourfold families in Table~\ref{tab:fourfold}.  Each
source family contains the cubic suspensions of the corresponding threefold
family.  For a general defining cubic, $H\leq\GL(5,\CC)$ denotes the
strict lift of $G$, so $G\cong H/\langle\omega I_5\rangle$.  Where
generators and cubic bases are displayed, we use convenient coordinates
on the fixed five-space; their agreement with the extraction is verified
in \path{gap_classification/gap_threefold/}. See also \path{anc/threefold.txt}

\par\Needspace{3\baselineskip}\medskip\noindent\textbf{\boldmath 1.\enspace $G=C_3^4:S_5$.}\label{fam:threefold-no-1}
This is the Fermat cubic threefold; see \cite[Example~3.1(1)]{LiYuThreefold}.
Source cubic fourfold family: No.~1.

\par\Needspace{3\baselineskip}\medskip\noindent\textbf{\boldmath 2.\enspace $G=L_2(11)$.}\label{fam:threefold-no-2}
This is the Klein cubic threefold; see \cite[Example~3.1(5)]{LiYuThreefold}.
Source cubic fourfold family: No.~7.

\par\Needspace{3\baselineskip}\medskip\noindent\textbf{\boldmath 3.\enspace $G=S_5\times C_3$.}\label{fam:threefold-no-3}
$G\text{-ID}=[360,119]$, $H\text{-ID}=[1080,490]$.
This is the suspension of the Clebsch diagonal cubic surface; see \cite[Example~3.1(6)]{LiYuThreefold}.
Source cubic fourfold family: No.~8.  Here
$H\cong S_5\times C_3^2$.
The family dimension is $0$.
\par\smallskip\noindent Generators of $H$ are\par\nopagebreak
\noindent \(g_1=(2\,3)\).  The remaining generators are
\begin{familydisplay}
\[
g_{2}=\diag\!\left(1,
\smat{\omega^2&0&0&-\omega^2\\0&0&0&-\omega^2\\0&\omega^2&0&-\omega^2\\0&0&\omega^2&-\omega^2}\right)
\qquad
g_{3}=(2\,4\,5\,3)
\qquad g_{4}=\omega I_5.
\]
\end{familydisplay}
\par\smallskip\noindent A basis of $W_H$ is\par\nopagebreak
\begin{familydisplay}
\begin{alignat*}{2}
P_{1}&=x_{1}^{3},\\[2pt]
P_{2}&=x_{2}^{2}x_{3}+x_{2}^{2}x_{4}+x_{2}^{2}x_{5}+x_{2}x_{3}^{2}+2x_{2}x_{3}x_{4}+2x_{2}x_{3}x_{5}+x_{2}x_{4}^{2}+2x_{2}x_{4}x_{5}\\
&\quad +x_{2}x_{5}^{2}+x_{3}^{2}x_{4}+x_{3}^{2}x_{5}+x_{3}x_{4}^{2}+2x_{3}x_{4}x_{5}+x_{3}x_{5}^{2}+x_{4}^{2}x_{5}+x_{4}x_{5}^{2}.\\[2pt]
\end{alignat*}
\end{familydisplay}

\par\Needspace{3\baselineskip}\medskip\noindent\textbf{\boldmath 4.\enspace $G=S_3\times((C_3^2:C_3):C_2)$.}\label{fam:threefold-no-4}
$G\text{-ID}=[324,122]$, $H\text{-ID}=[972,791]$.
Source cubic fourfold family: No.~23.  Here $H\cong C_3\times(S_3\times((C_3^2:C_3):C_2))$.
The family dimension is $1$.
\par\smallskip\noindent Generators of $H$ are\par\nopagebreak
\begin{familydisplay}
\begin{align*}
g_1&=\diag(1,\omega^2,\omega,1,1),
&
g_2&=(1\,3)\diag(1,1,1,1,\omega),\\
g_3&=\diag\!\left(
\smat{0&1&0\\0&0&\omega\\\omega^2&0&0},
\smat{0&\omega\\\omega^2&0}
\right),
&
g_4&=\omega I_5.
\end{align*}
\end{familydisplay}
\par\smallskip\noindent A basis of $W_H$ is\par\nopagebreak
\begin{familydisplay}
\begin{SecEightAlignedBasis}{l@{\qquad}l@{\qquad}l}
\PBasisLabel{1}=x_{1}x_{2}x_{3}, &
\PBasisLabel{2}=x_{1}^{3}+x_{2}^{3}+x_{3}^{3}, &
\PBasisLabel{3}=x_{4}^{3}+x_{5}^{3}.
\end{SecEightAlignedBasis}
\end{familydisplay}

\par\Needspace{3\baselineskip}\medskip\noindent\textbf{\boldmath 5.\enspace $G=S_3\times((C_3^2:C_3):C_4)$.}\label{fam:threefold-no-5}
$G\text{-ID}=[648,541]$, $H\text{-ID}=[1944,3483]$.
See \cite[Example~3.1(2)]{LiYuThreefold}.
Source cubic fourfold family: No.~24.  Here $H\cong C_3\times(S_3\times((C_3^2:C_3):C_4))$.
The family dimension is $0$.
\par\smallskip\noindent Generators of $H$ are\par\nopagebreak
\begin{familydisplay}
\begin{align*}
&g_1=(1\,2)\diag(1,1,1,1,\omega^2)\qquad
g_2=\diag\!\left(\frac{1}{\sqrt{3}}\smat{1&1&1\\1&\omega&\omega^2\\1&\omega^2&\omega},I_2\right),\\
&g_3=\diag(1,1,1,\omega,1)\qquad g_4=\diag(1,1,1,\omega,\omega),\\
&g_5=(4\,5)\qquad g_6=\omega I_5.
\end{align*}
\end{familydisplay}
\par\smallskip\noindent A basis of $W_H$ is\par\nopagebreak
\begin{familydisplay}
\begin{alignat*}{2}
P_{1}&=x_{1}^{3}+x_{2}^{3}+x_{3}^{3}+3(\sqrt{3}-1)x_{1}x_{2}x_{3},\\[2pt]
P_{2}&=x_{4}^{3}+x_{5}^{3}.\\[2pt]
\end{alignat*}
\end{familydisplay}

\par\Needspace{3\baselineskip}\medskip\noindent\textbf{\boldmath 6.\enspace $G=S_4\times C_3$.}\label{fam:threefold-no-6}
$G\text{-ID}=[72,42]$, $H\text{-ID}=[216,163]$.
Source cubic fourfold family: No.~28.  Here
$H\cong S_4\times C_3^2$.
The family dimension is $1$.
\par\smallskip\noindent Generators of $H$ are\par\nopagebreak
\begin{familydisplay}
\begin{align*}
&g_{1}=\diag(1,1,\omega,\omega^2,\omega^2)(1\,5\,3)
\qquad g_{2}=\diag\!\left(1,\omega,\smat{0&1\\1&0},1\right)\\
\end{align*}
\end{familydisplay}
\noindent We also include $g_{3}=\omega I_5$.
\par\smallskip\noindent A basis of $W_H$ is\par\nopagebreak
\begin{familydisplay}
\begin{alignat*}{2}
P_{1}&=x_{2}^{3},\\[2pt]
P_{2}&=x_{1}^{3}+x_{3}^{3}+x_{4}^{3}+x_{5}^{3},\\[2pt]
P_{3}&=x_{1}x_{3}x_{4}+x_{1}x_{3}x_{5}+x_{1}x_{4}x_{5}+\omega^2x_{3}x_{4}x_{5},\\[2pt]
P_{4}&=x_{1}^{2}x_{3}+x_{1}^{2}x_{4}+x_{1}^{2}x_{5}+\omega^2x_{1}x_{3}^{2}+\omega^2x_{1}x_{4}^{2}+\omega^2x_{1}x_{5}^{2}\\
&\quad +\omega\,x_{3}^{2}x_{4}+\omega\,x_{3}^{2}x_{5}+\omega\,x_{3}x_{4}^{2}+\omega\,x_{3}x_{5}^{2}\\
&\quad +\omega\,x_{4}^{2}x_{5}+\omega\,x_{4}x_{5}^{2}.\\[2pt]
\end{alignat*}
\end{familydisplay}

\par\Needspace{3\baselineskip}\medskip\noindent\textbf{\boldmath 7.\enspace $G=A_5$.}\label{fam:threefold-no-7}
$G\text{-ID}=[60,5]$, $H\text{-ID}=[180,19]$.
Source cubic fourfold family: No.~30.  Here
$H\cong\GL(2,4)$.
The family dimension is $1$.
\par\smallskip\noindent Generators of $H$ are\par\nopagebreak
\begin{familydisplay}
\begin{align*}
g_{1}&=\diag(1,1,\omega,1,\omega^2)(1\,5\,4\,3\,2),
&
g_{2}&=\diag(1,1,\omega,1,1)(1\,2\,5\,3\,4),\\
g_{3}&=\diag(1,\omega^2,\omega^2,1,\omega)(1\,5\,2),
&
g_{4}&=\omega I_5.
\end{align*}
\end{familydisplay}
\par\smallskip\noindent A basis of $W_H$ is\par\nopagebreak
\begin{familydisplay}
\begin{alignat*}{2}
P_{1}&=x_{1}^{3}+x_{2}^{3}+x_{3}^{3}+x_{4}^{3}+x_{5}^{3},\\[2pt]
P_{2}&=x_{1}x_{2}x_{3}+\omega^2x_{1}x_{2}x_{4}+\omega^2x_{1}x_{2}x_{5}+x_{1}x_{3}x_{4}+\omega^2x_{1}x_{3}x_{5}\\
&\quad +x_{1}x_{4}x_{5}+\omega\,x_{2}x_{3}x_{4}+\omega^2x_{2}x_{3}x_{5}+\omega^2x_{2}x_{4}x_{5}+\omega\,x_{3}x_{4}x_{5}.\\[2pt]
\end{alignat*}
\end{familydisplay}

\par\Needspace{3\baselineskip}\medskip\noindent\textbf{\boldmath 8.\enspace $G=S_5$.}\label{fam:threefold-no-8}
$G\text{-ID}=[120,34]$, $H\text{-ID}=[360,119]$.
Source cubic fourfold family: No.~32.  Here
$H\cong S_5\times C_3$.
The family dimension is $1$.
\par\smallskip\noindent Generators of $H$ are\par\nopagebreak
\begin{familydisplay}
\begin{align*}
&g_{1}=(1\,2\,5)(3\,4)\qquad g_{2}=(2\,5\,4)\\
&g_{3}=(1\,4\,2)\qquad g_{4}=\omega I_5\\
\end{align*}
\end{familydisplay}
\par\smallskip\noindent A basis of $W_H$ is\par\nopagebreak
\begin{familydisplay}
\begin{alignat*}{2}
P_{1}&=x_{1}x_{2}x_{3}+x_{1}x_{2}x_{4}+x_{1}x_{2}x_{5}+x_{1}x_{3}x_{4}+x_{1}x_{3}x_{5}+x_{1}x_{4}x_{5}\\
&\quad +x_{2}x_{3}x_{4}+x_{2}x_{3}x_{5}+x_{2}x_{4}x_{5}+x_{3}x_{4}x_{5},\\[2pt]
P_{2}&=x_{1}^{2}x_{2}+x_{1}^{2}x_{3}+x_{1}^{2}x_{4}+x_{1}^{2}x_{5}+x_{1}x_{2}^{2}+x_{1}x_{3}^{2}+x_{1}x_{4}^{2}\\
&\quad +x_{1}x_{5}^{2}+x_{2}^{2}x_{3}+x_{2}^{2}x_{4}+x_{2}^{2}x_{5}+x_{2}x_{3}^{2}+x_{2}x_{4}^{2}+x_{2}x_{5}^{2}\\
&\quad +x_{3}^{2}x_{4}+x_{3}^{2}x_{5}+x_{3}x_{4}^{2}+x_{3}x_{5}^{2}+x_{4}^{2}x_{5}+x_{4}x_{5}^{2},\\[2pt]
P_{3}&=x_{1}^{3}+x_{2}^{3}+x_{3}^{3}+x_{4}^{3}+x_{5}^{3}.\\[2pt]
\end{alignat*}
\end{familydisplay}

\par\Needspace{3\baselineskip}\medskip\noindent\textbf{\boldmath 9.\enspace $G=S_3\times C_6$.}\label{fam:threefold-no-9}
$G\text{-ID}=[36,12]$, $H\text{-ID}=[108,42]$.
Source cubic fourfold family: No.~36.  Here
$H\cong S_3\times C_6\times C_3$.
The family dimension is $1$.
\par\smallskip\noindent Generators of $H$ are\par\nopagebreak
\begin{familydisplay}
\begin{align*}
g_{1}&=\diag(1,1,\omega^2,\omega,\omega^2),
&
g_{2}&=\diag(\omega^2,1,1,\omega,1),\\
g_{3}&=\diag(1,\omega,1,1,1),
&
g_{4}&=(1\,4),\\
g_{5}&=(3\,5),
&
g_{6}&=\omega I_5.
\end{align*}
\end{familydisplay}
\par\smallskip\noindent A basis of $W_H$ is\par\nopagebreak
\begin{familydisplay}
\begin{alignat*}{2}
P_{1}&=x_{1}x_{3}x_{4}+x_{1}x_{4}x_{5},\\[2pt]
P_{2}&=x_{1}^{3}+x_{4}^{3},\\[2pt]
P_{3}&=x_{3}^{3}-x_{3}^{2}x_{5}-x_{3}x_{5}^{2}+x_{5}^{3},\\[2pt]
P_{4}&=x_{2}^{3},\\[2pt]
P_{5}&=x_{3}^{2}x_{5}+x_{3}x_{5}^{2}.\\[2pt]
\end{alignat*}
\end{familydisplay}

\par\Needspace{3\baselineskip}\medskip\noindent\textbf{\boldmath 10.\enspace $G=(S_3\times S_3):C_2$.}\label{fam:threefold-no-10}
$G\text{-ID}=[72,40]$, $H\text{-ID}=[216,157]$.
Source cubic fourfold family: No.~38.  Here
$H\cong((S_3\times S_3):C_2)\times C_3$.
The family dimension is $1$.
\par\smallskip\noindent Generators of $H$ are\par\nopagebreak
\begin{familydisplay}
\begin{align*}
&g_{1}=\diag(1,\omega,1,1,\omega)(1\,5)(3\,4)
\qquad g_{2}=(3\,5)\\
\end{align*}
\end{familydisplay}
\noindent We also include $g_{3}=\omega I_5$.
\par\smallskip\noindent A basis of $W_H$ is\par\nopagebreak
\begin{familydisplay}
\begin{SecEightAlignedBasis}{l@{\qquad}l}
\PBasisLabel{1}=x_{2}^{3}, & \PBasisLabel{2}=x_{1}x_{2}x_{4}+\omega^2x_{2}x_{3}x_{5},\\
\multicolumn{2}{@{}l@{}}{\PBasisLabel{3}=x_{1}^{3}+x_{3}^{3}+x_{4}^{3}+x_{5}^{3}.}
\end{SecEightAlignedBasis}
\end{familydisplay}

\par\Needspace{3\baselineskip}\medskip\noindent\textbf{\boldmath 11.\enspace $G=S_3\times C_3$.}\label{fam:threefold-no-11}
$G\text{-ID}=[18,3]$, $H\text{-ID}=[54,12]$.
Source cubic fourfold family: No.~56.  Here
$H\cong S_3\times C_3^2$.
The family dimension is $2$.
\par\smallskip\noindent Generators of $H$ are\par\nopagebreak
\begin{familydisplay}
\begin{align*}
&g_{1}=\diag\!\left(\smat{0&-\omega\\1&-\omega^2},I_3\right)
\qquad g_{2}=(3\,4)\\
&g_{3}=(3\,5\,4)\qquad
g_{4}=\diag\!\left(\smat{0&-1\\\omega^2&-\omega},\omega^2I_3\right)\\
\end{align*}
\end{familydisplay}
\noindent We also include $g_{5}=\omega I_5$.
\par\smallskip\noindent A basis of $W_H$ is\par\nopagebreak
\begin{familydisplay}
\begin{alignat*}{2}
P_{1}&=x_{1}^{2}x_{2}+\omega\,x_{1}x_{2}^{2},\\[2pt]
P_{2}&=x_{1}^{3}-3\,\omega^2x_{1}x_{2}^{2}-x_{2}^{3},\\[2pt]
P_{3}&=(x_{1}-\omega^2x_{2})^2(x_{3}+x_{4}+x_{5}),\\[2pt]
P_{4}&=(x_{1}-\omega^2x_{2})(x_{3}x_{4}+x_{3}x_{5}+x_{4}x_{5}),\\[2pt]
P_{5}&=(x_{1}-\omega^2x_{2})(x_{3}^{2}+x_{4}^{2}+x_{5}^{2}),\\[2pt]
P_{6}&=x_{3}x_{4}x_{5},\\[2pt]
P_{7}&=x_{3}^{2}x_{4}+x_{3}^{2}x_{5}+x_{3}x_{4}^{2}+x_{3}x_{5}^{2}+x_{4}^{2}x_{5}+x_{4}x_{5}^{2},\\[2pt]
P_{8}&=x_{3}^{3}+x_{4}^{3}+x_{5}^{3}.\\[2pt]
\end{alignat*}
\end{familydisplay}

\par\Needspace{3\baselineskip}\medskip\noindent\textbf{\boldmath 12.\enspace $G=S_3\times S_3$.}\label{fam:threefold-no-12}
$G\text{-ID}=[36,10]$, $H\text{-ID}=[108,38]$.
Source cubic fourfold family: No.~57.  Here
$H\cong S_3^2\times C_3$.
The family dimension is $2$.
\par\smallskip\noindent Generators of $H$ are\par\nopagebreak
\begin{familydisplay}
\begin{align*}
&g_{1}=\diag\!\left(I_3,\smat{-1&0\\-1&1}\right)
\qquad g_{2}=\diag\!\left(I_3,\smat{-1&1\\-1&0}\right)\\
&g_{3}=(1\,2)\qquad g_{4}=(1\,2\,3)\qquad g_{5}=\omega I_5\\
\end{align*}
\end{familydisplay}
\par\smallskip\noindent A basis of $W_H$ is\par\nopagebreak
\begin{familydisplay}
\begin{alignat*}{2}
P_{1}&=x_{1}x_{2}x_{3},\\[2pt]
P_{2}&=x_{1}^{2}x_{2}+x_{1}^{2}x_{3}+x_{1}x_{2}^{2}+x_{1}x_{3}^{2}+x_{2}^{2}x_{3}+x_{2}x_{3}^{2},\\[2pt]
P_{3}&=x_{1}^{3}+x_{2}^{3}+x_{3}^{3},\\[2pt]
P_{4}&=(x_{1}+x_{2}+x_{3})(x_{4}^{2}+x_{4}x_{5}+x_{5}^{2}),\\[2pt]
P_{5}&=x_{4}^{2}x_{5}+x_{4}x_{5}^{2}.\\[2pt]
\end{alignat*}
\end{familydisplay}

\par\Needspace{3\baselineskip}\medskip\noindent\textbf{\boldmath 13.\enspace $G=D_{12}$, family~1.}\label{fam:threefold-no-13}
$G\text{-ID}=[12,4]$, $H\text{-ID}=[36,12]$.
Source cubic fourfold family: No.~61.  Here
$H\cong S_3\times C_6$.
The family dimension is $2$.
\par\smallskip\noindent Generators of $H$ are\par\nopagebreak
\begin{familydisplay}
\begin{align*}
&g_{1}=\diag\!\left(\smat{0&0&\omega^2&0\\0&0&0&-\omega^2\\\omega&0&0&0\\0&-\omega&0&0},1\right)
\qquad g_{2}=\diag(1,-1,1,-1,1)\\
&g_{3}=\diag(\omega,\omega,\omega^2,\omega^2,1)
\qquad g_{4}=\omega I_5\\
\end{align*}
\end{familydisplay}
\par\smallskip\noindent A basis of $W_H$ is\par\nopagebreak
\begin{familydisplay}
\begin{SecEightAlignedBasis}{l@{\qquad}l}
\PBasisLabel{1}=x_{1}x_{3}x_{5}, & \PBasisLabel{2}=x_{2}x_{4}x_{5},\\
\PBasisLabel{3}=x_{1}^{3}+x_{3}^{3}, & \PBasisLabel{4}=x_{1}x_{2}^{2}+x_{3}x_{4}^{2},\\
\multicolumn{2}{@{}l@{}}{\PBasisLabel{5}=x_{5}^{3}.}
\end{SecEightAlignedBasis}
\end{familydisplay}

\par\Needspace{3\baselineskip}\medskip\noindent\textbf{\boldmath 14.\enspace $G=(C_6\times C_2):C_2$.}\label{fam:threefold-no-14}
$G\text{-ID}=[24,8]$, $H\text{-ID}=[72,30]$.
Source cubic fourfold family: No.~62.  Here $H\cong((C_6\times C_2):C_2)\times C_3$.
The family dimension is $1$.
\par\smallskip\noindent Generators of $H$ are\par\nopagebreak
\begin{familydisplay}
\begin{align*}
&g_{1}=\diag(\omega^2,\omega^2,\omega,\omega,1)
\qquad g_{2}=\omega I_5\\
&g_{3}=(1\,3)(2\,4)\qquad
g_{4}=\diag(1,-1,1,-1,1)(1\,3)(2\,4)\\
&g_{5}=\diag(1,-1,1,1,1)\\
\end{align*}
\end{familydisplay}
\par\smallskip\noindent A basis of $W_H$ is\par\nopagebreak
\begin{familydisplay}
\begin{SecEightAlignedBasis}{l@{\qquad}l}
\PBasisLabel{1}=x_{1}x_{3}x_{5}, & \PBasisLabel{2}=x_{1}^{3}+x_{3}^{3},\\
\PBasisLabel{3}=x_{1}x_{2}^{2}+x_{3}x_{4}^{2}, & \PBasisLabel{4}=x_{5}^{3}.
\end{SecEightAlignedBasis}
\end{familydisplay}

\par\Needspace{3\baselineskip}\medskip\noindent\textbf{\boldmath 15.\enspace $G=C_6\times C_2$.}\label{fam:threefold-no-15}
$G\text{-ID}=[12,5]$, $H\text{-ID}=[36,14]$.
Source cubic fourfold family: No.~64.  Here
$H\cong C_6^2$.
The family dimension is $2$.
\par\smallskip\noindent Generators of $H$ are\par\nopagebreak
\begin{familydisplay}
\begin{align*}
&g_{1}=\diag(1,1,1,1,\omega)\qquad
g_{2}=\diag(-1,-1,1,1,1)(1\,2)\\
&g_{3}=(1\,2)\qquad g_{4}=\diag(\omega^2,\omega^2,\omega^2,\omega^2,1)\\
\end{align*}
\end{familydisplay}
\noindent We also include $g_{5}=\omega I_5$.
\par\smallskip\noindent A basis of $W_H$ is\par\nopagebreak
\begin{familydisplay}
\begin{SecEightAlignedBasis}{l@{\qquad}l@{\qquad}l}
\PBasisLabel{1}=x_{1}x_{2}x_{3}, & \PBasisLabel{2}=x_{1}x_{2}x_{4}, & \PBasisLabel{3}=x_{1}^{2}x_{3}+x_{2}^{2}x_{3},\\
\PBasisLabel{4}=x_{1}^{2}x_{4}+x_{2}^{2}x_{4}, & \PBasisLabel{5}=x_{3}^{3}, & \PBasisLabel{6}=x_{3}^{2}x_{4},\\
\PBasisLabel{7}=x_{3}x_{4}^{2}, & \PBasisLabel{8}=x_{4}^{3}, & \PBasisLabel{9}=x_{5}^{3}.
\end{SecEightAlignedBasis}
\end{familydisplay}

\par\Needspace{3\baselineskip}\medskip\noindent\textbf{\boldmath 16.\enspace $G=A_4$.}\label{fam:threefold-no-16}
$G\text{-ID}=[12,3]$, $H\text{-ID}=[36,11]$.
Source cubic fourfold family: No.~67.  Here
$H\cong A_4\times C_3$.
The family dimension is $2$.
\par\smallskip\noindent Generators of $H$ are\par\nopagebreak
\begin{familydisplay}
\begin{align*}
&g_{1}=(1\,2\,3)\,\diag(1,1,1,1,\omega)\qquad g_{2}=(1\,3)(2\,4)\qquad g_{3}=\omega I_5\\
\end{align*}
\end{familydisplay}
\par\smallskip\noindent A basis of $W_H$ is\par\nopagebreak
\begin{familydisplay}
\begin{SecEightAlignedBasis}{l@{\qquad}l}
\multicolumn{2}{@{}l@{}}{\PBasisLabel{1}=x_{1}x_{2}x_{3}+x_{1}x_{2}x_{4}+x_{1}x_{3}x_{4}+x_{2}x_{3}x_{4},}\\
\multicolumn{2}{@{}l@{}}{\PBasisLabel{2}=x_{1}^{2}x_{2}+x_{1}^{2}x_{3}+x_{1}^{2}x_{4}+x_{1}x_{2}^{2}+x_{1}x_{3}^{2}+x_{1}x_{4}^{2}}\\
\multicolumn{2}{@{}l@{}}{\hphantom{\PBasisLabel{1}=}\quad +x_{2}^{2}x_{3}+x_{2}^{2}x_{4}+x_{2}x_{3}^{2}+x_{2}x_{4}^{2}+x_{3}^{2}x_{4}+x_{3}x_{4}^{2},}\\
\multicolumn{2}{@{}l@{}}{\PBasisLabel{3}=x_{1}x_{2}x_{5}+\omega\,x_{1}x_{3}x_{5}+\omega^2x_{1}x_{4}x_{5}+\omega^2x_{2}x_{3}x_{5}+\omega\,x_{2}x_{4}x_{5}+x_{3}x_{4}x_{5},}\\
\PBasisLabel{4}=x_{1}^{3}+x_{2}^{3}+x_{3}^{3}+x_{4}^{3}, & \PBasisLabel{5}=x_{5}^{3}.
\end{SecEightAlignedBasis}
\end{familydisplay}

\par\Needspace{3\baselineskip}\medskip\noindent\textbf{\boldmath 17.\enspace $G=S_4$.}\label{fam:threefold-no-17}
$G\text{-ID}=[24,12]$, $H\text{-ID}=[72,42]$.
Source cubic fourfold family: No.~69.  Here
$H\cong S_4\times C_3$.
The family dimension is $2$.
\par\smallskip\noindent Generators of $H$ are\par\nopagebreak
\begin{familydisplay}
\begin{align*}
&g_{1}=(1\,3\,2)\qquad g_{2}=\left(\begin{smallmatrix}0 & \omega^2 & -\omega^2 & 0 & 0 \\ 0 & 0 & -\omega^2 & 0 & 0 \\ \omega^2 & 0 & -\omega^2 & 0 & 0 \\ 0 & 0 & -12/5\,\omega^2 & \omega^2 & 0 \\ 0 & 0 & -2\,\omega^2 & 0 & \omega^2\end{smallmatrix}\right)\\
\end{align*}
\end{familydisplay}
\noindent We also include $g_{3}=\omega I_5$.
\par\smallskip\noindent A basis of $W_H$ is\par\nopagebreak
\begin{familydisplay}
\begin{alignat*}{2}
P_{1}&=5x_{1}^{2}x_{2}+5x_{1}^{2}x_{3}+5x_{1}x_{2}^{2}+10x_{1}x_{2}x_{3}+12x_{1}x_{2}x_{4}\\
&\quad +10x_{1}x_{2}x_{5}+5x_{1}x_{3}^{2}+12x_{1}x_{3}x_{4}+10x_{1}x_{3}x_{5}+5x_{2}^{2}x_{3}\\
&\quad +5x_{2}x_{3}^{2}+12x_{2}x_{3}x_{4}+10x_{2}x_{3}x_{5},\\[2pt]
P_{2}&=5x_{1}^{2}x_{4}+5x_{1}x_{2}x_{4}+5x_{1}x_{3}x_{4}+12x_{1}x_{4}^{2}+10x_{1}x_{4}x_{5}\\
&\quad +5x_{2}^{2}x_{4}+5x_{2}x_{3}x_{4}+12x_{2}x_{4}^{2}+10x_{2}x_{4}x_{5}+5x_{3}^{2}x_{4}\\
&\quad +12x_{3}x_{4}^{2}+10x_{3}x_{4}x_{5},\\[2pt]
P_{3}&=5x_{1}^{2}x_{5}+5x_{1}x_{2}x_{5}+5x_{1}x_{3}x_{5}+12x_{1}x_{4}x_{5}+10x_{1}x_{5}^{2}\\
&\quad +5x_{2}^{2}x_{5}+5x_{2}x_{3}x_{5}+12x_{2}x_{4}x_{5}+10x_{2}x_{5}^{2}+5x_{3}^{2}x_{5}\\
&\quad +12x_{3}x_{4}x_{5}+10x_{3}x_{5}^{2},\\[2pt]
P_{4}&=\mathmakebox[\widthof{$x_{4}x_{5}^{2}$}][l]{x_{4}^{3}},\qquad P_{5}=x_{4}^{2}x_{5},\\[2pt]
P_{6}&=x_{4}x_{5}^{2},\qquad P_{7}=x_{5}^{3}.\\[2pt]
\end{alignat*}
\end{familydisplay}

\par\Needspace{3\baselineskip}\medskip\noindent\textbf{\boldmath 18.\enspace $G=D_{10}$.}\label{fam:threefold-no-18}
$G\text{-ID}=[10,1]$, $H\text{-ID}=[30,2]$.
Source cubic fourfold family: No.~72.  Here
$H\cong D_{10}\times C_3$.
The family dimension is $2$.
\par\smallskip\noindent Generators of $H$ are\par\nopagebreak
\begin{familydisplay}
\begin{align*}
&g_{1}=(2\,5)(3\,4)\qquad
g_{2}=\diag(1,\zeta_5,\zeta_5^2,\zeta_5^3,\zeta_5^4)\qquad g_{3}=\omega I_5\\
\end{align*}
\end{familydisplay}
\par\smallskip\noindent A basis of $W_H$ is\par\nopagebreak
\begin{familydisplay}
\begin{SecEightAlignedBasis}{l@{\qquad}l}
\PBasisLabel{1}=x_{1}^{3}, & \PBasisLabel{2}=x_{1}x_{2}x_{5},\\
\PBasisLabel{3}=x_{1}x_{3}x_{4}, & \PBasisLabel{4}=x_{2}^{2}x_{4}+x_{3}x_{5}^{2},\\
\multicolumn{2}{@{}l@{}}{\PBasisLabel{5}=x_{2}x_{3}^{2}+x_{4}^{2}x_{5}.}
\end{SecEightAlignedBasis}
\end{familydisplay}

\par\Needspace{3\baselineskip}\medskip\noindent\textbf{\boldmath 19.\enspace $G=S_3$, family~1.}\label{fam:threefold-no-19}
$G\text{-ID}=[6,1]$, $H\text{-ID}=[18,3]$.
Source cubic fourfold family: No.~86.  Here
$H\cong S_3\times C_3$.
The family dimension is $3$.
\par\smallskip\noindent Generators of $H$ are\par\nopagebreak
\begin{familydisplay}
\begin{align*}
&g_{1}=\diag\!\left(\smat{0&1\\1&0},1,\smat{-1&0\\-1&1}\right)\\
&g_{2}=\diag\!\left(\smat{0&0&1\\1&0&0\\0&1&0},
\smat{0&-1\\1&-1}\right)\qquad g_{3}=\omega I_5\\
\end{align*}
\end{familydisplay}
\par\smallskip\noindent A basis of $W_H$ is\par\nopagebreak
\begin{familydisplay}
\begin{alignat*}{2}
P_{1}&=x_{1}x_{2}x_{3},\\[2pt]
P_{2}&=x_{1}^{2}x_{2}+x_{1}^{2}x_{3}+x_{1}x_{2}^{2}+x_{1}x_{3}^{2}+x_{2}^{2}x_{3}+x_{2}x_{3}^{2},\\[2pt]
P_{3}&=x_{1}x_{2}x_{5}+x_{1}x_{3}x_{4}-x_{2}x_{3}x_{4}-x_{2}x_{3}x_{5},\\[2pt]
P_{4}&=x_{1}^{3}+x_{2}^{3}+x_{3}^{3},\\[2pt]
P_{5}&=x_{1}^{2}x_{4}+x_{1}^{2}x_{5}-x_{2}^{2}x_{4}-x_{3}^{2}x_{5},\\[2pt]
P_{6}&=x_{1}x_{4}x_{5}-x_{2}x_{4}x_{5}-x_{2}x_{5}^{2}-x_{3}x_{4}^{2}-x_{3}x_{4}x_{5},\\[2pt]
P_{7}&=x_{1}x_{4}^{2}+x_{1}x_{5}^{2}+x_{2}x_{4}^{2}+2x_{2}x_{4}x_{5}+2x_{2}x_{5}^{2}+2x_{3}x_{4}^{2}+2x_{3}x_{4}x_{5}+x_{3}x_{5}^{2},\\[2pt]
P_{8}&=x_{4}^{2}x_{5}+x_{4}x_{5}^{2}.\\[2pt]
\end{alignat*}
\end{familydisplay}

\par\Needspace{3\baselineskip}\medskip\noindent\textbf{\boldmath 20.\enspace $G=C_6$, family~1.}\label{fam:threefold-no-20}
$G\text{-ID}=[6,2]$, $H\text{-ID}=[18,5]$.
Source cubic fourfold family: No.~90.  Here
$H\cong C_6\times C_3$.
The family dimension is $3$.
\par\smallskip\noindent Generators of $H$ are\par\nopagebreak
\begin{familydisplay}
\begin{align*}
&g_{1}=\diag(1,\omega,1,1,1)
\qquad g_{2}=\diag(1,1,1,1,-1)\qquad g_{3}=\omega I_5\\
\end{align*}
\end{familydisplay}
\par\smallskip\noindent A basis of $W_H$ is\par\nopagebreak
\begin{familydisplay}
\begin{SecEightAlignedBasis}{l@{\qquad}l@{\qquad}l@{\qquad}l}
\PBasisLabel{1}=x_{1}^{3}, & \PBasisLabel{2}=x_{2}^{3}, & \PBasisLabel{3}=x_{1}^{2}x_{3}, & \PBasisLabel{4}=x_{1}^{2}x_{4},\\
\PBasisLabel{5}=x_{1}x_{3}^{2}, & \PBasisLabel{6}=x_{1}x_{3}x_{4}, & \PBasisLabel{7}=x_{1}x_{4}^{2}, & \PBasisLabel{8}=x_{1}x_{5}^{2},\\
\PBasisLabel{9}=x_{3}^{3}, & \PBasisLabel{10}=x_{3}^{2}x_{4}, & \PBasisLabel{11}=x_{3}x_{4}^{2}, & \PBasisLabel{12}=x_{3}x_{5}^{2},\\
\PBasisLabel{13}=x_{4}^{3}, & \PBasisLabel{14}=x_{4}x_{5}^{2}.
\end{SecEightAlignedBasis}
\end{familydisplay}

\par\Needspace{3\baselineskip}\medskip\noindent\textbf{\boldmath 21.\enspace $G=D_{12}$, family~2.}\label{fam:threefold-no-21}
$G\text{-ID}=[12,4]$, $H\text{-ID}=[36,12]$.
Source cubic fourfold family: No.~91.  Here
$H\cong S_3\times C_6$.
The family dimension is $3$.
\par\smallskip\noindent Generators of $H$ are\par\nopagebreak
\begin{familydisplay}
\begin{align*}
&g_{1}=(1\,2)
\qquad g_{2}=\diag(\omega,\omega^2,1,1,1)\\
&g_{3}=\diag(1,1,1,1,-1)
\qquad g_{4}=\omega I_5\\
\end{align*}
\end{familydisplay}
\par\smallskip\noindent A basis of $W_H$ is\par\nopagebreak
\begin{familydisplay}
\begin{SecEightAlignedBasis}{l@{\qquad}l@{\qquad}l}
\PBasisLabel{1}=x_{1}^{3}+x_{2}^{3}, & \PBasisLabel{2}=x_{1}x_{2}x_{3}, & \PBasisLabel{3}=x_{1}x_{2}x_{4},\\
\PBasisLabel{4}=x_{3}^{3}, & \PBasisLabel{5}=x_{3}^{2}x_{4}, & \PBasisLabel{6}=x_{3}x_{4}^{2},\\
\PBasisLabel{7}=x_{3}x_{5}^{2}, & \PBasisLabel{8}=x_{4}^{3}, & \PBasisLabel{9}=x_{4}x_{5}^{2}.
\end{SecEightAlignedBasis}
\end{familydisplay}

\par\Needspace{3\baselineskip}\medskip\noindent\textbf{\boldmath 22.\enspace $G=C_{12}$.}\label{fam:threefold-no-22}
$G\text{-ID}=[12,2]$, $H\text{-ID}=[36,8]$.
Source cubic fourfold family: No.~94.  Here
$H\cong C_{12}\times C_3$.
The family dimension is $1$.
\par\smallskip\noindent Generators of $H$ are\par\nopagebreak
\begin{familydisplay}
\begin{align*}
&g_{1}=\diag\!\left(\smat{0&-\omega\\1&-\omega^2},1,-1,\zeta_4\right)
\qquad g_{2}=\diag\!\left(\smat{-\omega^2&\omega\\-1&0},\omega,\omega,\omega\right)\\
\end{align*}
\end{familydisplay}
\noindent We also include $g_{3}=\omega I_5$.
\par\smallskip\noindent A basis of $W_H$ is\par\nopagebreak
\begin{familydisplay}
\begin{SecEightAlignedBasis}{l@{\qquad}l}
\multicolumn{2}{@{}l@{}}{\PBasisLabel{1}=x_{1}^{2}x_{2}+\omega x_{1}x_{2}^{2},}\\
\multicolumn{2}{@{}l@{}}{\PBasisLabel{2}=x_{1}^{3}-3\omega^{2}x_{1}x_{2}^{2}-x_{2}^{3},}\\
\multicolumn{2}{@{}l@{}}{\PBasisLabel{3}=x_{1}^{2}x_{3}-2\omega^{2}x_{1}x_{2}x_{3}+\omega x_{2}^{2}x_{3},}\\
\PBasisLabel{4}=x_{1}x_{3}^{2}-\omega^{2}x_{2}x_{3}^{2}, & \PBasisLabel{5}=x_{1}x_{4}^{2}-\omega^{2}x_{2}x_{4}^{2},\\
\PBasisLabel{6}=x_{3}^{3}, & \PBasisLabel{7}=x_{3}x_{4}^{2},\\
\multicolumn{2}{@{}l@{}}{\PBasisLabel{8}=x_{4}x_{5}^{2}.}
\end{SecEightAlignedBasis}
\end{familydisplay}

\par\Needspace{3\baselineskip}\medskip\noindent\textbf{\boldmath 23.\enspace $G=S_3\times C_4$.}\label{fam:threefold-no-23}
$G\text{-ID}=[24,5]$, $H\text{-ID}=[72,27]$.
Source cubic fourfold family: No.~95.  Here
$H\cong S_3\times C_{12}$.
The family dimension is $1$.
\par\smallskip\noindent Generators of $H$ are\par\nopagebreak
\begin{familydisplay}
\begin{align*}
&g_{1}=\diag\!\left(\smat{-1&0\\-1&1},1,1,-1\right)
\qquad g_{2}=\diag\!\left(\smat{-1&1\\-1&0},I_3\right)\\
&g_{3}=\diag(\omega^2,\omega^2,\zeta_6,\omega^2,\zeta_{12}^{11})\\
\end{align*}
\end{familydisplay}
\noindent We also include $g_{4}=\omega I_5$.
\par\smallskip\noindent A basis of $W_H$ is\par\nopagebreak
\begin{familydisplay}
\begin{SecEightAlignedBasis}{l@{\qquad}l}
\PBasisLabel{1}=x_{1}^{2}x_{2}+x_{1}x_{2}^{2}, & \PBasisLabel{2}=x_{1}^{2}x_{4}+x_{1}x_{2}x_{4}+x_{2}^{2}x_{4},\\
\PBasisLabel{3}=x_{3}^{2}x_{4}, & \PBasisLabel{4}=x_{3}x_{5}^{2},\\
\multicolumn{2}{@{}l@{}}{\PBasisLabel{5}=x_{4}^{3}.}
\end{SecEightAlignedBasis}
\end{familydisplay}

\par\Needspace{3\baselineskip}\medskip\noindent\textbf{\boldmath 24.\enspace $G=C_{24}$.}\label{fam:threefold-no-24}
$G\text{-ID}=[24,2]$, $H\text{-ID}=[72,14]$.
See \cite[Example~3.1(3)]{LiYuThreefold}.
Source cubic fourfold family: No.~96.  Here
$H\cong C_{24}\times C_3$.
The family dimension is $0$.
\par\smallskip\noindent Generators of $H$ are\par\nopagebreak
\begin{familydisplay}
\begin{align*}
&g_{1}=\diag\!\left(\smat{0&-1\\\omega^2&-\omega},-\omega^2,-\zeta_{12}^{11},\zeta_{24}^{19}\right)
\qquad g_{2}=\diag\!\left(\smat{0&-\omega\\1&-\omega^2},I_3\right)\\
\end{align*}
\end{familydisplay}
\noindent We also include $g_{3}=\omega I_5$.
\par\smallskip\noindent A basis of $W_H$ is\par\nopagebreak
\begin{familydisplay}
\begin{SecEightAlignedBasis}{l@{\quad}l@{\quad}l}
\PBasisLabel{1}=x_{1}^{2}x_{2}+\omega\,x_{1}x_{2}^{2}, & \multicolumn{2}{@{}l@{}}{\PBasisLabel{2}=x_{1}^{3}-3\,\omega^2x_{1}x_{2}^{2}-x_{2}^{3},}\\
\PBasisLabel{3}=x_{1}x_{3}^{2}-\omega^2x_{2}x_{3}^{2}, & \PBasisLabel{4}=x_{3}x_{4}^{2}, & \PBasisLabel{5}=x_{4}x_{5}^{2}.
\end{SecEightAlignedBasis}
\end{familydisplay}

\par\Needspace{3\baselineskip}\medskip\noindent\textbf{\boldmath 25.\enspace $G=C_2^2$, family~1.}\label{fam:threefold-no-25}
$G\text{-ID}=[4,2]$, $H\text{-ID}=[12,5]$.
Source cubic fourfold family: No.~101.  Here
$H\cong C_6\times C_2$.
The family dimension is $4$.
\par\smallskip\noindent Generators of $H$ are\par\nopagebreak
\begin{familydisplay}
\begin{align*}
&g_{1}=\omega I_5
\qquad g_{2}=\diag(-1,1,-1,1,1)\\
&g_{3}=\diag(1,-1,-1,1,1)\\
\end{align*}
\end{familydisplay}
\par\smallskip\noindent A basis of $W_H$ is\par\nopagebreak
\begin{familydisplay}
\begin{SecEightAlignedBasis}{l@{\quad}l@{\quad}l@{\quad}l@{\quad}l}
\PBasisLabel{1}=x_{1}^{2}x_{4}, & \PBasisLabel{2}=x_{1}^{2}x_{5}, & \PBasisLabel{3}=x_{1}x_{2}x_{3}, & \PBasisLabel{4}=x_{2}^{2}x_{4}, & \PBasisLabel{5}=x_{2}^{2}x_{5},\\
\PBasisLabel{6}=x_{3}^{2}x_{4}, & \PBasisLabel{7}=x_{3}^{2}x_{5}, & \PBasisLabel{8}=x_{4}^{3}, & \PBasisLabel{9}=x_{4}^{2}x_{5}, & \PBasisLabel{10}=x_{4}x_{5}^{2},\\
\multicolumn{5}{@{}l@{}}{\PBasisLabel{11}=x_{5}^{3}.}
\end{SecEightAlignedBasis}
\end{familydisplay}

\par\Needspace{3\baselineskip}\medskip\noindent\textbf{\boldmath 26.\enspace $G=D_8$.}\label{fam:threefold-no-26}
$G\text{-ID}=[8,3]$, $H\text{-ID}=[24,10]$.
Source cubic fourfold family: No.~103.  Here
$H\cong D_8\times C_3$.
The family dimension is $3$.
\par\smallskip\noindent Generators of $H$ are\par\nopagebreak
\begin{familydisplay}
\begin{align*}
&g_{1}=\omega I_5
\qquad g_{2}=\diag(1,1,-1,-1,1)\\
&g_{3}=\diag(1,1,-1,1,-1)\qquad
g_{4}=\diag(1,1,1,-1,-1)(4\,5)\\
\end{align*}
\end{familydisplay}
\par\smallskip\noindent A basis of $W_H$ is\par\nopagebreak
\begin{familydisplay}
\begin{SecEightAlignedBasis}{l@{\qquad}l@{\qquad}l}
\PBasisLabel{1}=x_{1}^{3}, & \PBasisLabel{2}=x_{1}^{2}x_{2}, & \PBasisLabel{3}=x_{1}x_{2}^{2},\\
\PBasisLabel{4}=x_{1}x_{3}^{2}, & \PBasisLabel{5}=x_{1}x_{4}^{2}+x_{1}x_{5}^{2}, & \PBasisLabel{6}=x_{2}^{3},\\
\PBasisLabel{7}=x_{2}x_{3}^{2}, & \PBasisLabel{8}=x_{2}x_{4}^{2}+x_{2}x_{5}^{2}, & \PBasisLabel{9}=x_{3}x_{4}x_{5}.
\end{SecEightAlignedBasis}
\end{familydisplay}

\par\Needspace{3\baselineskip}\medskip\noindent\textbf{\boldmath 27.\enspace $G=C_3$, family~1.}\label{fam:threefold-no-27}
$G\text{-ID}=[3,1]$, $H\text{-ID}=[9,2]$.
Source cubic fourfold family: No.~110.  Here
$H\cong C_3^2$.
The family dimension is $4$.
\par\smallskip\noindent Generators of $H$ are\par\nopagebreak
\begin{familydisplay}
\begin{align*}
&g_{1}=\diag(\omega^2,1,\omega,\omega^2,\omega)
\qquad g_{2}=\omega I_5\\
\end{align*}
\end{familydisplay}
\par\smallskip\noindent A basis of $W_H$ is\par\nopagebreak
\begin{familydisplay}
\begin{SecEightAlignedBasis}{l@{\quad}l@{\quad}l@{\quad}l@{\quad}l}
\PBasisLabel{1}=x_{1}^{3}, & \PBasisLabel{2}=x_{1}^{2}x_{4}, & \PBasisLabel{3}=x_{1}x_{2}x_{3}, & \PBasisLabel{4}=x_{1}x_{2}x_{5}, & \PBasisLabel{5}=x_{1}x_{4}^{2},\\
\PBasisLabel{6}=x_{2}^{3}, & \PBasisLabel{7}=x_{2}x_{3}x_{4}, & \PBasisLabel{8}=x_{2}x_{4}x_{5}, & \PBasisLabel{9}=x_{3}^{3}, & \PBasisLabel{10}=x_{3}^{2}x_{5},\\
\PBasisLabel{11}=x_{3}x_{5}^{2}, & \PBasisLabel{12}=x_{4}^{3}, & \multicolumn{3}{@{}l@{}}{\PBasisLabel{13}=x_{5}^{3}.}
\end{SecEightAlignedBasis}
\end{familydisplay}

\par\Needspace{3\baselineskip}\medskip\noindent\textbf{\boldmath 28.\enspace $G=C_6$, family~2.}\label{fam:threefold-no-28}
$G\text{-ID}=[6,2]$, $H\text{-ID}=[18,5]$.
Source cubic fourfold family: No.~113.  Here
$H\cong C_6\times C_3$.
The family dimension is $2$.
\par\smallskip\noindent Generators of $H$ are\par\nopagebreak
\begin{familydisplay}
\begin{align*}
&g_{1}=\diag(\omega,\omega,1,1,\omega^2)
\qquad g_{2}=\omega I_5\\
&g_{3}=\diag(-1,1,1,1,1)\\
\end{align*}
\end{familydisplay}
\par\smallskip\noindent A basis of $W_H$ is\par\nopagebreak
\begin{familydisplay}
\begin{SecEightAlignedBasis}{l@{\quad}l@{\quad}l@{\quad}l@{\quad}l}
\PBasisLabel{1}=x_{1}^{2}x_{2}, & \PBasisLabel{2}=x_{2}^{3}, & \PBasisLabel{3}=x_{2}x_{3}x_{5}, & \PBasisLabel{4}=x_{2}x_{4}x_{5}, & \PBasisLabel{5}=x_{3}^{3},\\
\PBasisLabel{6}=x_{3}^{2}x_{4}, & \PBasisLabel{7}=x_{3}x_{4}^{2}, & \PBasisLabel{8}=x_{4}^{3}, & \multicolumn{2}{@{}l@{}}{\PBasisLabel{9}=x_{5}^{3}.}
\end{SecEightAlignedBasis}
\end{familydisplay}

\par\Needspace{3\baselineskip}\medskip\noindent\textbf{\boldmath 29.\enspace $G=C_3$, family~2.}\label{fam:threefold-no-29}
$G\text{-ID}=[3,1]$, $H\text{-ID}=[9,2]$.
Source cubic fourfold family: No.~116.  Here
$H\cong C_3^2$.
The family dimension is $4$.
\par\smallskip\noindent Generators of $H$ are\par\nopagebreak
\begin{familydisplay}
\begin{align*}
&g_{1}=\diag(1,1,1,1,\omega)
\qquad g_{2}=\diag(\omega^2,\omega^2,\omega^2,\omega^2,1)\\
\end{align*}
\end{familydisplay}
\noindent We also include $g_{3}=\omega I_5$.
\par\smallskip\noindent A basis of $W_H$ is\par\nopagebreak
\begin{familydisplay}
\begin{SecEightAlignedBasis}{l@{\quad}l@{\quad}l@{\quad}l@{\quad}l}
\PBasisLabel{1}=x_{1}^{3}, & \PBasisLabel{2}=x_{1}^{2}x_{2}, & \PBasisLabel{3}=x_{1}^{2}x_{3}, & \PBasisLabel{4}=x_{1}^{2}x_{4}, & \PBasisLabel{5}=x_{1}x_{2}^{2},\\
\PBasisLabel{6}=x_{1}x_{2}x_{3}, & \PBasisLabel{7}=x_{1}x_{2}x_{4}, & \PBasisLabel{8}=x_{1}x_{3}^{2}, & \PBasisLabel{9}=x_{1}x_{3}x_{4}, & \PBasisLabel{10}=x_{1}x_{4}^{2},\\
\PBasisLabel{11}=x_{2}^{3}, & \PBasisLabel{12}=x_{2}^{2}x_{3}, & \PBasisLabel{13}=x_{2}^{2}x_{4}, & \PBasisLabel{14}=x_{2}x_{3}^{2}, & \PBasisLabel{15}=x_{2}x_{3}x_{4},\\
\PBasisLabel{16}=x_{2}x_{4}^{2}, & \PBasisLabel{17}=x_{3}^{3}, & \PBasisLabel{18}=x_{3}^{2}x_{4}, & \PBasisLabel{19}=x_{3}x_{4}^{2}, & \PBasisLabel{20}=x_{4}^{3},\\
\multicolumn{5}{@{}l@{}}{\PBasisLabel{21}=x_{5}^{3}.}
\end{SecEightAlignedBasis}
\end{familydisplay}

\par\Needspace{3\baselineskip}\medskip\noindent\textbf{\boldmath 30.\enspace $G=S_3$, family~2.}\label{fam:threefold-no-30}
$G\text{-ID}=[6,1]$, $H\text{-ID}=[18,3]$.
Source cubic fourfold family: No.~117.  Here
$H\cong S_3\times C_3$.
The family dimension is $4$.
\par\smallskip\noindent Generators of $H$ are\par\nopagebreak
\begin{familydisplay}
\begin{align*}
&g_{1}=\diag\!\left(I_3,\smat{0&-\tfrac35\omega^2\\-\tfrac53\omega&0}\right)\\
&g_{2}=\diag(1,1,1,\omega^2,\omega)\qquad g_{3}=\omega I_5\\
\end{align*}
\end{familydisplay}
\par\smallskip\noindent A basis of $W_H$ is\par\nopagebreak
\begin{familydisplay}
\begin{SecEightAlignedBasis}{l@{\qquad}l@{\qquad}l}
\PBasisLabel{1}=x_{1}^{3}, & \PBasisLabel{2}=x_{1}^{2}x_{2}, & \PBasisLabel{3}=x_{1}^{2}x_{3},\\
\PBasisLabel{4}=x_{1}x_{2}^{2}, & \PBasisLabel{5}=x_{1}x_{2}x_{3}, & \PBasisLabel{6}=x_{1}x_{3}^{2},\\
\PBasisLabel{7}=x_{1}x_{4}x_{5}, & \PBasisLabel{8}=x_{2}^{3}, & \PBasisLabel{9}=x_{2}^{2}x_{3},\\
\PBasisLabel{10}=x_{2}x_{3}^{2}, & \PBasisLabel{11}=x_{2}x_{4}x_{5}, & \PBasisLabel{12}=x_{3}^{3},\\
\PBasisLabel{13}=x_{3}x_{4}x_{5}, & \PBasisLabel{14}=27x_{4}^{3}-125x_{5}^{3}.
\end{SecEightAlignedBasis}
\end{familydisplay}

\par\Needspace{3\baselineskip}\medskip\noindent\textbf{\boldmath 31.\enspace $G=C_2$, family~1.}\label{fam:threefold-no-31}
$G\text{-ID}=[2,1]$, $H\text{-ID}=[6,2]$.
Source cubic fourfold family: No.~123.  Here
$H\cong C_6$.
The family dimension is $6$.
\par\smallskip\noindent Generators of $H$ are\par\nopagebreak
\begin{familydisplay}
\begin{align*}
&g_{1}=\omega I_5
\qquad g_{2}=\diag(-1,1,-1,1,1)\\
\end{align*}
\end{familydisplay}
\par\smallskip\noindent A basis of $W_H$ is\par\nopagebreak
\begin{familydisplay}
\begin{SecEightAlignedBasis}{l@{\quad}l@{\quad}l@{\quad}l@{\quad}l}
\PBasisLabel{1}=x_{1}^{2}x_{2}, & \PBasisLabel{2}=x_{1}^{2}x_{4}, & \PBasisLabel{3}=x_{1}^{2}x_{5}, & \PBasisLabel{4}=x_{1}x_{2}x_{3}, & \PBasisLabel{5}=x_{1}x_{3}x_{4},\\
\PBasisLabel{6}=x_{1}x_{3}x_{5}, & \PBasisLabel{7}=x_{2}^{3}, & \PBasisLabel{8}=x_{2}^{2}x_{4}, & \PBasisLabel{9}=x_{2}^{2}x_{5}, & \PBasisLabel{10}=x_{2}x_{3}^{2},\\
\PBasisLabel{11}=x_{2}x_{4}^{2}, & \PBasisLabel{12}=x_{2}x_{4}x_{5}, & \PBasisLabel{13}=x_{2}x_{5}^{2}, & \PBasisLabel{14}=x_{3}^{2}x_{4}, & \PBasisLabel{15}=x_{3}^{2}x_{5},\\
\PBasisLabel{16}=x_{4}^{3}, & \PBasisLabel{17}=x_{4}^{2}x_{5}, & \PBasisLabel{18}=x_{4}x_{5}^{2}, & \multicolumn{2}{@{}l@{}}{\PBasisLabel{19}=x_{5}^{3}.}
\end{SecEightAlignedBasis}
\end{familydisplay}

\par\Needspace{3\baselineskip}\medskip\noindent\textbf{\boldmath 32.\enspace $G=C_2^2$, family~2.}\label{fam:threefold-no-32}
$G\text{-ID}=[4,2]$, $H\text{-ID}=[12,5]$.
Source cubic fourfold family: No.~125.  Here
$H\cong C_6\times C_2$.
The family dimension is $5$.
\par\smallskip\noindent Generators of $H$ are\par\nopagebreak
\begin{familydisplay}
\begin{align*}
&g_{1}=\omega I_5
\qquad g_{2}=\diag(-1,1,1,1,1)\\
&g_{3}=\diag(1,1,1,-1,1)\\
\end{align*}
\end{familydisplay}
\par\smallskip\noindent A basis of $W_H$ is\par\nopagebreak
\begin{familydisplay}
\begin{SecEightAlignedBasis}{l@{\quad}l@{\quad}l@{\quad}l@{\quad}l}
\PBasisLabel{1}=x_{1}^{2}x_{2}, & \PBasisLabel{2}=x_{1}^{2}x_{3}, & \PBasisLabel{3}=x_{1}^{2}x_{5}, & \PBasisLabel{4}=x_{2}^{3}, & \PBasisLabel{5}=x_{2}^{2}x_{3},\\
\PBasisLabel{6}=x_{2}^{2}x_{5}, & \PBasisLabel{7}=x_{2}x_{3}^{2}, & \PBasisLabel{8}=x_{2}x_{3}x_{5}, & \PBasisLabel{9}=x_{2}x_{4}^{2}, & \PBasisLabel{10}=x_{2}x_{5}^{2},\\
\PBasisLabel{11}=x_{3}^{3}, & \PBasisLabel{12}=x_{3}^{2}x_{5}, & \PBasisLabel{13}=x_{3}x_{4}^{2}, & \PBasisLabel{14}=x_{3}x_{5}^{2}, & \PBasisLabel{15}=x_{4}^{2}x_{5},\\
\multicolumn{5}{@{}l@{}}{\PBasisLabel{16}=x_{5}^{3}.}
\end{SecEightAlignedBasis}
\end{familydisplay}

\par\Needspace{3\baselineskip}\medskip\noindent\textbf{\boldmath 33.\enspace $G=C_4\times C_2$.}\label{fam:threefold-no-33}
$G\text{-ID}=[8,2]$, $H\text{-ID}=[24,9]$.
Source cubic fourfold family: No.~130.  Here
$H\cong C_{12}\times C_2$.
The family dimension is $2$.
\par\smallskip\noindent Generators of $H$ are\par\nopagebreak
\begin{familydisplay}
\begin{align*}
&g_{1}=\diag(-\zeta_4,-1,1,1,1)
\qquad g_{2}=\omega I_5\\
&g_{3}=\diag(1,1,1,-1,1)\\
\end{align*}
\end{familydisplay}
\par\smallskip\noindent A basis of $W_H$ is\par\nopagebreak
\begin{familydisplay}
\begin{SecEightAlignedBasis}{l@{\quad}l@{\quad}l@{\quad}l@{\quad}l}
\PBasisLabel{1}=x_{1}^{2}x_{2}, & \PBasisLabel{2}=x_{2}^{2}x_{3}, & \PBasisLabel{3}=x_{2}^{2}x_{5}, & \PBasisLabel{4}=x_{3}^{3}, & \PBasisLabel{5}=x_{3}^{2}x_{5},\\
\PBasisLabel{6}=x_{3}x_{4}^{2}, & \PBasisLabel{7}=x_{3}x_{5}^{2}, & \PBasisLabel{8}=x_{4}^{2}x_{5}, & \multicolumn{2}{@{}l@{}}{\PBasisLabel{9}=x_{5}^{3}.}
\end{SecEightAlignedBasis}
\end{familydisplay}

\par\Needspace{3\baselineskip}\medskip\noindent\textbf{\boldmath 34.\enspace $G=1$.}\label{fam:threefold-no-34}
$G\text{-ID}=[1,1]$, $H\text{-ID}=[3,1]$.
Source cubic fourfold family: No.~135.  Here
$H\cong C_3$.
The family dimension is $10$.
\par\smallskip\noindent The group $H$ is generated by $g_1=\omega I_5$.
\par\smallskip\noindent A basis of $W_H$ is\par\nopagebreak
\begin{familydisplay}
\begin{SecEightAlignedBasis}{l@{\quad}l@{\quad}l@{\quad}l@{\quad}l}
\PBasisLabel{1}=x_{1}^{3}, & \PBasisLabel{2}=x_{1}^{2}x_{2}, & \PBasisLabel{3}=x_{1}^{2}x_{3}, & \PBasisLabel{4}=x_{1}^{2}x_{4}, & \PBasisLabel{5}=x_{1}^{2}x_{5},\\
\PBasisLabel{6}=x_{1}x_{2}^{2}, & \PBasisLabel{7}=x_{1}x_{2}x_{3}, & \PBasisLabel{8}=x_{1}x_{2}x_{4}, & \PBasisLabel{9}=x_{1}x_{2}x_{5}, & \PBasisLabel{10}=x_{1}x_{3}^{2},\\
\PBasisLabel{11}=x_{1}x_{3}x_{4}, & \PBasisLabel{12}=x_{1}x_{3}x_{5}, & \PBasisLabel{13}=x_{1}x_{4}^{2}, & \PBasisLabel{14}=x_{1}x_{4}x_{5}, & \PBasisLabel{15}=x_{1}x_{5}^{2},\\
\PBasisLabel{16}=x_{2}^{3}, & \PBasisLabel{17}=x_{2}^{2}x_{3}, & \PBasisLabel{18}=x_{2}^{2}x_{4}, & \PBasisLabel{19}=x_{2}^{2}x_{5}, & \PBasisLabel{20}=x_{2}x_{3}^{2},\\
\PBasisLabel{21}=x_{2}x_{3}x_{4}, & \PBasisLabel{22}=x_{2}x_{3}x_{5}, & \PBasisLabel{23}=x_{2}x_{4}^{2}, & \PBasisLabel{24}=x_{2}x_{4}x_{5}, & \PBasisLabel{25}=x_{2}x_{5}^{2},\\
\PBasisLabel{26}=x_{3}^{3}, & \PBasisLabel{27}=x_{3}^{2}x_{4}, & \PBasisLabel{28}=x_{3}^{2}x_{5}, & \PBasisLabel{29}=x_{3}x_{4}^{2}, & \PBasisLabel{30}=x_{3}x_{4}x_{5},\\
\PBasisLabel{31}=x_{3}x_{5}^{2}, & \PBasisLabel{32}=x_{4}^{3}, & \PBasisLabel{33}=x_{4}^{2}x_{5}, & \PBasisLabel{34}=x_{4}x_{5}^{2}, & \PBasisLabel{35}=x_{5}^{3}.
\end{SecEightAlignedBasis}
\end{familydisplay}

\par\Needspace{3\baselineskip}\medskip\noindent\textbf{\boldmath 35.\enspace $G=C_2$, family~2.}\label{fam:threefold-no-35}
$G\text{-ID}=[2,1]$, $H\text{-ID}=[6,2]$.
Source cubic fourfold family: No.~139.  Here
$H\cong C_6$.
The family dimension is $7$.
\par\smallskip\noindent Generators of $H$ are\par\nopagebreak
\begin{familydisplay}
\begin{align*}
&g_{1}=\omega I_5
\qquad g_{2}=\diag(-1,1,1,1,1)\\
\end{align*}
\end{familydisplay}
\par\smallskip\noindent A basis of $W_H$ is\par\nopagebreak
\begin{familydisplay}
\begin{SecEightAlignedBasis}{l@{\quad}l@{\quad}l@{\quad}l@{\quad}l}
\PBasisLabel{1}=x_{1}^{2}x_{2}, & \PBasisLabel{2}=x_{1}^{2}x_{3}, & \PBasisLabel{3}=x_{1}^{2}x_{4}, & \PBasisLabel{4}=x_{1}^{2}x_{5}, & \PBasisLabel{5}=x_{2}^{3},\\
\PBasisLabel{6}=x_{2}^{2}x_{3}, & \PBasisLabel{7}=x_{2}^{2}x_{4}, & \PBasisLabel{8}=x_{2}^{2}x_{5}, & \PBasisLabel{9}=x_{2}x_{3}^{2}, & \PBasisLabel{10}=x_{2}x_{3}x_{4},\\
\PBasisLabel{11}=x_{2}x_{3}x_{5}, & \PBasisLabel{12}=x_{2}x_{4}^{2}, & \PBasisLabel{13}=x_{2}x_{4}x_{5}, & \PBasisLabel{14}=x_{2}x_{5}^{2}, & \PBasisLabel{15}=x_{3}^{3},\\
\PBasisLabel{16}=x_{3}^{2}x_{4}, & \PBasisLabel{17}=x_{3}^{2}x_{5}, & \PBasisLabel{18}=x_{3}x_{4}^{2}, & \PBasisLabel{19}=x_{3}x_{4}x_{5}, & \PBasisLabel{20}=x_{3}x_{5}^{2},\\
\PBasisLabel{21}=x_{4}^{3}, & \PBasisLabel{22}=x_{4}^{2}x_{5}, & \PBasisLabel{23}=x_{4}x_{5}^{2}, & \multicolumn{2}{@{}l@{}}{\PBasisLabel{24}=x_{5}^{3}.}
\end{SecEightAlignedBasis}
\end{familydisplay}

\par\Needspace{3\baselineskip}\medskip\noindent\textbf{\boldmath 36.\enspace $G=C_4$, family~1.}\label{fam:threefold-no-36}
$G\text{-ID}=[4,1]$, $H\text{-ID}=[12,2]$.
Source cubic fourfold family: No.~147.  Here
$H\cong C_{12}$.
The family dimension is $3$.
\par\smallskip\noindent Generators of $H$ are\par\nopagebreak
\begin{familydisplay}
\begin{align*}
&g_{1}=\diag(-\zeta_4,-1,1,1,1)
\qquad g_{2}=\omega I_5\\
\end{align*}
\end{familydisplay}
\par\smallskip\noindent A basis of $W_H$ is\par\nopagebreak
\begin{familydisplay}
\begin{SecEightAlignedBasis}{l@{\quad}l@{\quad}l@{\quad}l@{\quad}l}
\PBasisLabel{1}=x_{1}^{2}x_{2}, & \PBasisLabel{2}=x_{2}^{2}x_{3}, & \PBasisLabel{3}=x_{2}^{2}x_{4}, & \PBasisLabel{4}=x_{2}^{2}x_{5}, & \PBasisLabel{5}=x_{3}^{3},\\
\PBasisLabel{6}=x_{3}^{2}x_{4}, & \PBasisLabel{7}=x_{3}^{2}x_{5}, & \PBasisLabel{8}=x_{3}x_{4}^{2}, & \PBasisLabel{9}=x_{3}x_{4}x_{5}, & \PBasisLabel{10}=x_{3}x_{5}^{2},\\
\PBasisLabel{11}=x_{4}^{3}, & \PBasisLabel{12}=x_{4}^{2}x_{5}, & \PBasisLabel{13}=x_{4}x_{5}^{2}, & \multicolumn{2}{@{}l@{}}{\PBasisLabel{14}=x_{5}^{3}.}
\end{SecEightAlignedBasis}
\end{familydisplay}

\par\Needspace{3\baselineskip}\medskip\noindent\textbf{\boldmath 37.\enspace $G=C_4$, family~2.}\label{fam:threefold-no-37}
$G\text{-ID}=[4,1]$, $H\text{-ID}=[12,2]$.
Source cubic fourfold family: No.~148.  Here
$H\cong C_{12}$.
The family dimension is $3$.
\par\smallskip\noindent Generators of $H$ are\par\nopagebreak
\begin{familydisplay}
\begin{align*}
&g_{1}=\diag(-\zeta_4,-1,1,-1,1)
\qquad g_{2}=\omega I_5\\
\end{align*}
\end{familydisplay}
\par\smallskip\noindent A basis of $W_H$ is\par\nopagebreak
\begin{familydisplay}
\begin{SecEightAlignedBasis}{l@{\quad}l@{\quad}l@{\quad}l@{\quad}l}
\PBasisLabel{1}=x_{1}^{2}x_{2}, & \PBasisLabel{2}=x_{1}^{2}x_{4}, & \PBasisLabel{3}=x_{2}^{2}x_{3}, & \PBasisLabel{4}=x_{2}^{2}x_{5}, & \PBasisLabel{5}=x_{2}x_{3}x_{4},\\
\PBasisLabel{6}=x_{2}x_{4}x_{5}, & \PBasisLabel{7}=x_{3}^{3}, & \PBasisLabel{8}=x_{3}^{2}x_{5}, & \PBasisLabel{9}=x_{3}x_{4}^{2}, & \PBasisLabel{10}=x_{3}x_{5}^{2},\\
\PBasisLabel{11}=x_{4}^{2}x_{5}, & \multicolumn{4}{@{}l@{}}{\PBasisLabel{12}=x_{5}^{3}.}
\end{SecEightAlignedBasis}
\end{familydisplay}

\par\Needspace{3\baselineskip}\medskip\noindent\textbf{\boldmath 38.\enspace $G=C_8$, family~1.}\label{fam:threefold-no-38}
$G\text{-ID}=[8,1]$, $H\text{-ID}=[24,2]$.
Source cubic fourfold family: No.~153.  Here
$H\cong C_{24}$.
The family dimension is $1$.
\par\smallskip\noindent Generators of $H$ are\par\nopagebreak
\begin{familydisplay}
\begin{align*}
&g_{1}=\diag(\zeta_8^3,\zeta_4,-1,1,1)
\qquad g_{2}=\omega I_5\\
\end{align*}
\end{familydisplay}
\par\smallskip\noindent A basis of $W_H$ is\par\nopagebreak
\begin{familydisplay}
\begin{SecEightAlignedBasis}{l@{\quad}l@{\quad}l@{\quad}l@{\quad}l}
\PBasisLabel{1}=x_{1}^{2}x_{2}, & \PBasisLabel{2}=x_{2}^{2}x_{3}, & \PBasisLabel{3}=x_{3}^{2}x_{4}, & \PBasisLabel{4}=x_{3}^{2}x_{5}, & \PBasisLabel{5}=x_{4}^{3},\\
\PBasisLabel{6}=x_{4}^{2}x_{5}, & \PBasisLabel{7}=x_{4}x_{5}^{2}, & \multicolumn{3}{@{}l@{}}{\PBasisLabel{8}=x_{5}^{3}.}
\end{SecEightAlignedBasis}
\end{familydisplay}

\par\Needspace{3\baselineskip}\medskip\noindent\textbf{\boldmath 39.\enspace $G=C_8$, family~2.}\label{fam:threefold-no-39}
$G\text{-ID}=[8,1]$, $H\text{-ID}=[24,2]$.
Source cubic fourfold family: No.~154.  Here
$H\cong C_{24}$.
The family dimension is $1$.
\par\smallskip\noindent Generators of $H$ are\par\nopagebreak
\begin{familydisplay}
\begin{align*}
&g_{1}=\diag(\zeta_8^7,\zeta_4,-\zeta_4,-1,1)
\qquad g_{2}=\omega I_5\\
\end{align*}
\end{familydisplay}
\par\smallskip\noindent A basis of $W_H$ is\par\nopagebreak
\begin{familydisplay}
\begin{SecEightAlignedBasis}{l@{\quad}l@{\quad}l@{\quad}l@{\quad}l}
\PBasisLabel{1}=x_{1}^{2}x_{2}, & \PBasisLabel{2}=x_{2}^{2}x_{4}, & \PBasisLabel{3}=x_{2}x_{3}x_{5}, & \PBasisLabel{4}=x_{3}^{2}x_{4}, & \PBasisLabel{5}=x_{4}^{2}x_{5},\\
\multicolumn{5}{@{}l@{}}{\PBasisLabel{6}=x_{5}^{3}.}
\end{SecEightAlignedBasis}
\end{familydisplay}

\par\Needspace{3\baselineskip}\medskip\noindent\textbf{\boldmath 40.\enspace $G=C_{16}$.}\label{fam:threefold-no-40}
$G\text{-ID}=[16,1]$, $H\text{-ID}=[48,2]$.
See \cite[Example~3.1(4)]{LiYuThreefold}.
Source cubic fourfold family: No.~156.  Here
$H\cong C_{48}$.
The family dimension is $0$.
\par\smallskip\noindent Generators of $H$ are\par\nopagebreak
\begin{familydisplay}
\begin{align*}
&g_{1}=\diag(\zeta_{16}^{13},\zeta_8^3,\zeta_4,-1,1)
\qquad g_{2}=\omega I_5\\
\end{align*}
\end{familydisplay}
\par\smallskip\noindent A basis of $W_H$ is\par\nopagebreak
\begin{familydisplay}
\begin{SecEightAlignedBasis}{l@{\quad}l@{\quad}l@{\quad}l@{\quad}l}
\PBasisLabel{1}=x_{1}^{2}x_{2}, & \PBasisLabel{2}=x_{2}^{2}x_{3}, & \PBasisLabel{3}=x_{3}^{2}x_{4}, & \PBasisLabel{4}=x_{4}^{2}x_{5}, & \PBasisLabel{5}=x_{5}^{3}.
\end{SecEightAlignedBasis}
\end{familydisplay}

\appendix
\section{Singular families in the small-group enumeration}
\label{app:singular-families}

This appendix proves that seven candidate invariant cubic spaces consist
entirely of singular cubics.  Thus none defines a family in
Table~\ref{tab:fourfold}.

The dimensions below refer to the vector spaces $W_H$.  The GAP IDs
of the projective and linear groups, respectively, are
\[
\begin{array}{c|c|c|c|c}
 &G_s&\operatorname{Id}(\pi(H))&\operatorname{Id}(H)&\dim W_H\\ \hline
1&C_2^2&[8,3]&[24,10]&9\\
2&C_2^2&[16,6]&[48,24]&6\\
3&C_2^2&[24,10]&[72,37]&6\\
4&C_4&[8,4]&[24,11]&8\\
5&C_4&[24,10]&[72,37]&8\\
6&S_3&[24,5]&[72,27]&8\\
7&S_3&[72,27]&[216,136]&7
\end{array}
\]
All matrix and polynomial identities below are exact over
$\mathbb Q(\zeta_{12})$.

\subsection{The first $C_2^2$ family}

Let $H$ be generated by $\omega I_6$ and
\[
 A_1=\frac12(0,0,0,1,1,0),\qquad
 A_3=\frac12(0,0,0,0,1,1),
\]
and
\[
 A_2=\diag(B_1,1,C_1),
\]
where
\[
 B_1=
 \begin{pmatrix}
 -443/523&-90/523&-180/523\\
 -2512/1569&419/523&1884/523\\
 -32/1569&12/523&-499/523
 \end{pmatrix},\qquad
 C_1=
 \begin{pmatrix}
 0&(2/7)\omega^2\\
 (7/2)\omega&0
 \end{pmatrix}.
\]

Set
\[
 L=-15x_1+157x_2+2x_3,\qquad
 I=(x_4,x_5,x_6).
\]
Solving the invariance equations gives a basis $f_1,\ldots,f_9$ of
$W_H$ such that $f_1,\ldots,f_5\in I^2$ and
\[
 f_i=LQ_i\qquad (6\leq i\leq9),
\]
where $Q_i\in\CC[x_1,x_2,x_3]_2$.
Consider the line
\[
 \ell=\{L=x_4=x_5=x_6=0\}\subset\PP^5.
\]
For $F=\sum_{i=1}^9a_if_i$, both $F$ and the gradients of
$a_1f_1+\cdots+a_5f_5$ vanish on $\ell$.  On this line,
\[
 \nabla F=\left(\sum_{i=6}^9a_iQ_i\right)\nabla L.
\]
The coefficient in parentheses is a homogeneous quadratic on
$\ell\cong\PP^1$.  It is either zero or has a zero over $\CC$.  Hence every
$F\in W_H$ has a projective singular point.

\subsection{The second $C_2^2$ family}

Let $H$ be generated by $\omega I_6$ and
\begin{align*}
 A_1&=\diag(B_2,C_2,-\zeta _4),\\
 A_2&=\frac12(0,0,0,0,1,1),\\
 A_3&=\diag(B_3,\zeta _4,\zeta _4,-1),\\
 A_4&=\frac12(0,0,0,1,1,0),
\end{align*}
where
\[
 B_2=
 \begin{pmatrix}
 \frac{337-780\zeta _4}{523}&\frac{-12-78\zeta _4}{523}&
   \frac{-453+273\zeta _4}{523}\\
 \frac{-3992+1480\zeta _4}{1569}&\frac{1109+148\zeta _4}{1569}&
   \frac{6170-518\zeta _4}{1569}\\
 \frac{1748-1780\zeta _4}{1569}&\frac{214-178\zeta _4}{1569}&
   \frac{-2120+623\zeta _4}{1569}
 \end{pmatrix},
\]
\[
 C_2=\begin{pmatrix}0&(2/7)\zeta _4\\7/2&0\end{pmatrix},
 \qquad
 B_3=
 \begin{pmatrix}
 -1037/523&-156/523&546/523\\
 2960/1569&1865/1569&-1036/1569\\
 -3560/1569&-356/1569&2815/1569
 \end{pmatrix}.
\]

Put
\[
 L=99x_1-304x_2+196x_3,\qquad
 p=(-20,-2,7).
\]
The invariant space has a basis $f_1,\ldots,f_6$ with
\[
 f_1=x_4x_5x_6,\qquad f_2=Lx_6^2,\qquad
 f_3\in(x_4,x_5)^2,\qquad
 f_4,f_5,f_6\in\CC[x_1,x_2,x_3].
\]
Direct evaluation gives $f_i(p)=0$ for $4\leq i\leq6$ and
\[
 \nabla(a_4f_4+a_5f_5+a_6f_6)(p)=\lambda\nabla L
\]
for some $\lambda\in\CC$.  Write $F=\sum_{i=1}^6a_if_i$.
If $a_2=0$, then $[0:0:0:0:0:1]$ is singular on $F=0$.  If
$a_2\neq0$, choose $w\in\CC$ such that $w^2=-\lambda/a_2$.  The point
\[
 [-20:-2:7:0:0:w]
\]
is singular on $F=0$. Thus every cubic in $W_H$ is singular.

\subsection{The third $C_2^2$ family}

Let $H$ be generated by $\omega I_6$ and
\begin{align*}
 A_1&=\frac12(0,0,0,1,1,0),\\
 A_2&=\diag(B_3,-1,C_3),\\
 A_3&=\diag(B_4,\omega^2,\omega^2,\omega^2),\\
 A_4&=\frac12(0,0,0,0,1,1),
\end{align*}
where $B_3$ is as above,
\[
 C_3=\begin{pmatrix}0&-7/2\\-2/7&0\end{pmatrix},
\]
and
\[
 B_4=\begin{pmatrix}
 \frac{40\omega+483\omega^2}{523}&
 \frac{-45\omega+45\omega^2}{523}&
 \frac{-90\omega+90\omega^2}{523}\\
 \frac{-1256\omega+1256\omega^2}{1569}&
 \frac{471\omega+52\omega^2}{523}&
 \frac{942\omega-942\omega^2}{523}\\
 \frac{-16\omega+16\omega^2}{1569}&
 \frac{6\omega-6\omega^2}{523}&
 \frac{12\omega+511\omega^2}{523}
 \end{pmatrix}.
\]

Choose a basis $f_1,\ldots,f_6$ of $W_H$ so that $f_1,f_2,f_3$ are
linear in $x_1,x_2,x_3$ and quadratic in $x_4,x_5,x_6$, while
$f_4,f_5,f_6\in\CC[x_1,x_2,x_3]$.  On the plane
\[
 \Pi=\{x_1=x_2=x_3=0\},
\]
all the basis forms vanish.  For $F=\sum_{i=1}^6a_if_i$, its only possibly
nonzero partial derivatives on $\Pi$ are
\[
 \begin{pmatrix}F_{x_1}\\F_{x_2}\\F_{x_3}\end{pmatrix}
 =M_a\begin{pmatrix}x_4^2\\x_5^2\\x_6^2\end{pmatrix},
\]
where the rows of $M_a$ are
\begin{align*}
 r_1={}&(a_3,a_2,a_1),\\
 r_2={}&\left(-\frac{304}{99}a_3,
 -\frac{2383}{1287}a_2-\frac{25627}{1716}a_1,
 -\frac{2092}{21021}a_2-\frac{2383}{1287}a_1\right),\\
 r_3={}&\left(\frac{196}{99}a_3,
 \frac{3527}{2574}a_2+\frac{25627}{3432}a_1,
 \frac{1046}{21021}a_2+\frac{3527}{2574}a_1\right).
\end{align*}
These rows satisfy the identity
\[
 -\frac89r_1+r_2+2r_3=0.
\]
Consequently $M_a$ has a nonzero kernel vector $(z_4,z_5,z_6)$.  Since
$\CC$ is algebraically closed, choose $x_i$ with $x_i^2=z_i$ for
$4\leq i\leq6$.  The resulting point
$[0:0:0:x_4:x_5:x_6]$ is singular on $F=0$.

\subsection{The first $C_4$ family}

Let $H$ be generated by $\omega I_6$ and
\begin{align*}
 A_1&=\frac14(0,0,2,2,1,3),\\
 A_2&=\diag(B_5,-1,-1,C_4),\\
 A_3&=\frac12(0,0,0,0,1,1),
\end{align*}
where
\[
 B_5=\begin{pmatrix}-25/101&72/101\\133/101&25/101\end{pmatrix},
 \qquad
 C_4=\begin{pmatrix}0&(3/4)\zeta _4\\(4/3)\zeta _4&0\end{pmatrix}.
\]
The invariant space is spanned by
\begin{align*}
 f_1&=16x_4x_6^2+9x_4x_5^2,&
 f_2&=16x_3x_6^2+9x_3x_5^2,\\
 f_3&=(18x_1-19x_2)x_5x_6,&
 f_4&=(4x_1+7x_2)x_4^2,\\
 f_5&=(4x_1+7x_2)x_3x_4,&
 f_6&=(4x_1+7x_2)x_3^2,\\
 f_7&=175x_2^3+604x_1x_2^2+288x_1^2x_2,&
 f_8&=833x_2^3-700x_1x_2^2+384x_1^3.
\end{align*}
For $F=\sum_{i=1}^8a_if_i$, restrict to the line
\[
 [x_1:x_2:x_3:x_4:x_5:x_6]=[0:0:s:t:0:0].
\]
The value of $F$ is zero, and its gradient is
\[
 (a_4t^2+a_5st+a_6s^2)(4,7,0,0,0,0).
\]
The binary quadratic has a projective zero.  Every member of the family is
therefore singular.

\subsection{The second $C_4$ family}

Let $H$ be generated by $\omega I_6$ and
\begin{align*}
 A_1&=\diag(1,1,-1,-1,C_5),\\
 A_2&=\diag(1,1,1,1,C_6),\\
 A_3&=\diag(B_6,\omega^2,\omega^2,\omega^2,\omega^2),\\
 A_4&=\frac12(0,0,0,0,1,1),
\end{align*}
where
\[
 C_5=\begin{pmatrix}0&3/4\\4/3&0\end{pmatrix},\qquad
 C_6=\begin{pmatrix}0&-(3/4)\zeta _4\\(4/3)\zeta _4&0\end{pmatrix},
\]
and
\[
 B_6=\begin{pmatrix}
 \frac{38\omega+63\omega^2}{101}&
 \frac{36\omega-36\omega^2}{101}\\
 \frac{133\omega-133\omega^2}{202}&
 \frac{63\omega+38\omega^2}{101}
 \end{pmatrix}.
\]
The invariant space is spanned by
\begin{align*}
 f_1&=-16x_4x_6^2+9x_4x_5^2,&
 f_2&=-16x_3x_6^2+9x_3x_5^2,\\
 f_3&=(18x_1-19x_2)x_5x_6,&
 f_4&=(18x_1-19x_2)x_4^2,\\
 f_5&=(18x_1-19x_2)x_3x_4,&
 f_6&=(18x_1-19x_2)x_3^2,\\
 f_7&=3019x_2^3+2700x_1x_2^2+3888x_1^2x_2,&
 f_8&=3325x_2^3+14364x_1x_2^2+2592x_1^3.
\end{align*}
On the line $[0:0:s:t:0:0]$, the value of
$F=\sum_{i=1}^8a_if_i$ is zero and
\[
 \nabla F=(a_4t^2+a_5st+a_6s^2)(18,-19,0,0,0,0).
\]
Again the binary quadratic has a projective zero, so every invariant cubic
is singular.

\subsection{The first $S_3$ family}

Let $H$ be generated by $\omega I_6$, $(1\,2)$, $(1\,2\,3)$,
$\frac12(0,0,0,0,0,1)$, and $\diag(D_1,-\zeta _4)$, where
\[
 D_1=\begin{pmatrix}
 1717/1995&-278/1995&-278/1995&263/133&-2104/1995\\
 -278/1995&1717/1995&-278/1995&263/133&-2104/1995\\
 -278/1995&-278/1995&1717/1995&263/133&-2104/1995\\
 -312/665&-312/665&-312/665&-367/133&624/665\\
 -724/665&-724/665&-724/665&-543/133&783/665
 \end{pmatrix}.
\]

Take a basis $f_1,\ldots,f_8$ of $W_H$ in which
$f_1,f_2\in(x_6^2)$.  On the $S_3$-fixed plane, write
\[
 (x_1,x_2,x_3,x_4,x_5,x_6)=(t,t,t,u,v,0)
\]
and put $L=-263t+78u+181v$.  Direct substitution gives
\[
 f_i(t,t,t,u,v,0)=LQ_i(t,u,v)\qquad(3\leq i\leq8),
\]
where $Q_i\in\CC[t,u,v]_2$.
For $F=\sum_{i=1}^8a_if_i$, put
$Q=\sum_{i=3}^8a_iQ_i$.  On the line in the fixed plane defined by
$L=0$, symmetry in $x_1,x_2,x_3$ gives
\[
 \nabla F=\frac{Q}{3}\,(-263,-263,-263,234,543,0).
\]
The restriction of $Q$ to this line is a binary quadratic.  It has a
projective zero, so every member of the family is singular.

\subsection{The second $S_3$ family}

Let
\[
 Q_{12}=(1\,2)\frac12(0,0,0,0,0,1).
\]
Put
\[
 a=\frac{-\omega+\omega^2}{3},\qquad
 b=\frac{-\omega-2\omega^2}{3},\qquad
 D_2=\begin{pmatrix}a&b&b\\b&a&b\\b&b&a\end{pmatrix}.
\]
Let $H$ be generated by $\omega I_6$, $Q_{12}$, $(1\,3\,2)$, and
\[
 \diag(D_2,-1,-1,-\zeta _4).
\]
The invariant space is spanned by
\begin{align*}
 f_1&=x_5x_6^2,& f_2&=x_4x_6^2,\\
 f_3&=(x_1+x_2+x_3)x_5^2,&
 f_4&=(x_1+x_2+x_3)x_4x_5,\\
 f_5&=(x_1+x_2+x_3)x_4^2,\\
 f_6&=x_1^2x_2+x_1^2x_3+x_2^2x_1+x_2^2x_3+x_3^2x_1+x_3^2x_2,\\
 f_7&=x_1^3+x_2^3+x_3^3+6x_1x_2x_3.
\end{align*}
For $F=\sum_{i=1}^7a_if_i$, consider points
\[
 [x_1:x_2:x_3:x_4:x_5:x_6]=[0:0:0:u:v:0].
\]
At such a point $F=0$, and all partial derivatives are zero except possibly
the first three.  These three derivatives are equal to
\[
 a_5u^2+a_4uv+a_3v^2.
\]
This binary quadratic has a projective zero.  Hence every cubic in $W_H$ is
singular.

\bibliography{reference}

@article {Matsumura1963OnTA,
    AUTHOR = {Matsumura, H. and Monsky, P.},
     TITLE = {On the automorphisms of hypersurfaces},
   JOURNAL = {J. Math. Kyoto Univ.},
  FJOURNAL = {Journal of Mathematics of Kyoto University},
    VOLUME = {3},
      YEAR = {1963/64},
     PAGES = {347--361},
      ISSN = {0023-608X},
   MRCLASS = {14.20},
       DOI = {10.1215/kjm/1250524785},
       URL = {https://doi.org/10.1215/kjm/1250524785},
}

@article {voisin1986torelli,
    AUTHOR = {Voisin, C.},
     TITLE = {Th\'eor\`eme de {T}orelli pour les cubiques de {$\mathbb P^5$}},
   JOURNAL = {Invent. Math.},
  FJOURNAL = {Inventiones Mathematicae},
    VOLUME = {86},
      YEAR = {1986},
    NUMBER = {3},
     PAGES = {577--601},
      ISSN = {0020-9910,1432-1297},
   MRCLASS = {14J30 (14C30)},
MRREVIEWER = {Daniel\ Huybrechts},
       DOI = {10.1007/BF01389270},
       URL = {https://doi.org/10.1007/BF01389270},
}

@article {voisin2008erratum,
    AUTHOR = {Voisin, C.},
     TITLE = {Erratum: ``{A} {T}orelli theorem for cubics in {$\mathbb P^5$}''},
   JOURNAL = {Invent. Math.},
  FJOURNAL = {Inventiones Mathematicae},
    VOLUME = {172},
      YEAR = {2008},
    NUMBER = {2},
     PAGES = {455--458},
      ISSN = {0020-9910,1432-1297},
   MRCLASS = {14J30},
       DOI = {10.1007/s00222-008-0116-z},
       URL = {https://doi.org/10.1007/s00222-008-0116-z},
}

@article {hassett2000special,
    AUTHOR = {Hassett, B.},
     TITLE = {Special cubic fourfolds},
   JOURNAL = {Compositio Math.},
  FJOURNAL = {Compositio Mathematica},
    VOLUME = {120},
      YEAR = {2000},
    NUMBER = {1},
     PAGES = {1--23},
      ISSN = {0010-437X},
       DOI = {10.1023/A:1001706324425},
       URL = {https://doi.org/10.1023/A:1001706324425},
}

@article {laza2010period,
    AUTHOR = {Laza, R.},
     TITLE = {The moduli space of cubic fourfolds via the period map},
   JOURNAL = {Ann. of Math. (2)},
  FJOURNAL = {Annals of Mathematics. Second Series},
    VOLUME = {172},
      YEAR = {2010},
    NUMBER = {1},
     PAGES = {673--711},
      ISSN = {0003-486X,1939-8980},
   MRCLASS = {14J10 (14D22 14J45 14L24)},
MRREVIEWER = {Keiji\ Oguiso},
       DOI = {10.4007/annals.2010.172.673},
       URL = {https://doi.org/10.4007/annals.2010.172.673},
}

@article {looi2009period,
    AUTHOR = {Looijenga, E.},
     TITLE = {The period map for cubic fourfolds},
   JOURNAL = {Invent. Math.},
  FJOURNAL = {Inventiones Mathematicae},
    VOLUME = {177},
      YEAR = {2009},
    NUMBER = {1},
     PAGES = {213--233},
      ISSN = {0020-9910,1432-1297},
   MRCLASS = {14J10 (14D07 14D20 14J45)},
MRREVIEWER = {Valery\ Alexeev},
       DOI = {10.1007/s00222-009-0178-6},
       URL = {https://doi.org/10.1007/s00222-009-0178-6},
}

@article {gonzalez2011automorphisms,
    AUTHOR = {Gonz\'alez-Aguilera, V. and Liendo, A.},
     TITLE = {Automorphisms of prime order of smooth cubic {$n$}-folds},
   JOURNAL = {Arch. Math. (Basel)},
  FJOURNAL = {Archiv der Mathematik},
    VOLUME = {97},
      YEAR = {2011},
    NUMBER = {1},
     PAGES = {25--37},
      ISSN = {0003-889X,1420-8938},
   MRCLASS = {14J50 (14J45)},
MRREVIEWER = {Ewa\ Szelachowska},
       DOI = {10.1007/s00013-011-0247-0},
       URL = {https://doi.org/10.1007/s00013-011-0247-0},
}

@article {laza2022automorphisms,
    AUTHOR = {Laza, R. and Zheng, Z.},
     TITLE = {Automorphisms and periods of cubic fourfolds},
   JOURNAL = {Math. Z.},
  FJOURNAL = {Mathematische Zeitschrift},
    VOLUME = {300},
      YEAR = {2022},
    NUMBER = {2},
     PAGES = {1455--1507},
      ISSN = {0025-5874,1432-1823},
   MRCLASS = {14J50 (14C30 14D20 14J10 14J28)},
       DOI = {10.1007/s00209-021-02810-x},
       URL = {https://doi.org/10.1007/s00209-021-02810-x},
}

@article {LiYuThreefold,
    AUTHOR = {Wei, L. and Yu, X.},
     TITLE = {Automorphism groups of smooth cubic threefolds},
   JOURNAL = {J. Math. Soc. Japan},
  FJOURNAL = {Journal of the Mathematical Society of Japan},
    VOLUME = {72},
      YEAR = {2020},
    NUMBER = {4},
     PAGES = {1327--1343},
      ISSN = {0025-5645,1881-1167},
   MRCLASS = {14J30 (14L30)},
       DOI = {10.2969/jmsj/83088308},
       URL = {https://doi.org/10.2969/jmsj/83088308},
}

@article {yu2020moduli,
    AUTHOR = {Yu, C. and Zheng, Z. },
     TITLE = {Moduli spaces of symmetric cubic fourfolds and locally symmetric varieties},
   JOURNAL = {Algebra Number Theory},
  FJOURNAL = {Algebra \& Number Theory},
    VOLUME = {14},
      YEAR = {2020},
    NUMBER = {10},
     PAGES = {2647--2683},
      ISSN = {1937-0652,1944-7833},
   MRCLASS = {14J10 (11F55 14D20 14J28 14J45 14L30)},
       DOI = {10.2140/ant.2020.14.2647},
       URL = {https://doi.org/10.2140/ant.2020.14.2647},
}

@article {Zheng2020OnAA,
    AUTHOR = {Zheng, Z. },
     TITLE = {On abelian automorphism groups of hypersurfaces},
   JOURNAL = {Israel J. Math.},
  FJOURNAL = {Israel Journal of Mathematics},
    VOLUME = {247},
      YEAR = {2022},
    NUMBER = {1},
     PAGES = {479--498},
      ISSN = {0021-2172,1565-8511},
   MRCLASS = {14J50 (14J40 14J45 14L30)},
       DOI = {10.1007/s11856-021-2275-1},
       URL = {https://doi.org/10.1007/s11856-021-2275-1},
}

@article {yang2024automorphism,
    AUTHOR = {Yang, S. and Yu, X. and Zhu, Z. },
     TITLE = {Automorphism groups of cubic fivefolds and fourfolds},
   JOURNAL = {J. Lond. Math. Soc. (2)},
  FJOURNAL = {Journal of the London Mathematical Society. Second Series},
    VOLUME = {110},
      YEAR = {2024},
    NUMBER = {4},
     PAGES = {Paper No. e12997, 35},
      ISSN = {0024-6107,1469-7750},
   MRCLASS = {14J50 (14J30 14J45 14L30)},
       DOI = {10.1112/jlms.12997},
       URL = {https://doi.org/10.1112/jlms.12997},
}

@article {KOIKE202512,
    AUTHOR = {Koike, K. },
     TITLE = {Cubic fourfolds with symplectic automorphisms},
   JOURNAL = {J. Algebra},
  FJOURNAL = {Journal of Algebra},
    VOLUME = {680},
      YEAR = {2025},
     PAGES = {12--57},
      ISSN = {0021-8693,1090-266X},
   MRCLASS = {14J50 (14L30)},
       DOI = {10.1016/j.jalgebra.2025.04.037},
       URL = {https://doi.org/10.1016/j.jalgebra.2025.04.037},
}

@article {KOIKE2026,
    AUTHOR = {Koike, K. },
     TITLE = {Corrigendum to ``{C}ubic fourfolds with symplectic automorphisms''},
   JOURNAL = {J. Algebra},
  FJOURNAL = {Journal of Algebra},
    VOLUME = {705},
      YEAR = {2026},
     PAGES = {358--362},
      ISSN = {0021-8693,1090-266X},
       DOI = {10.1016/j.jalgebra.2026.05.021},
       URL = {https://doi.org/10.1016/j.jalgebra.2026.05.021},
}

@misc {mayanskiy2013abelianautomorphismgroupscubic,
    AUTHOR = {Mayanskiy, E. },
     TITLE = {Abelian automorphism groups of cubic fourfolds},
      YEAR = {2013},
      NOTE = {arXiv:1308.5150v2},
       URL = {https://arxiv.org/abs/1308.5150},
}

@misc {zheng2021liftableabelian,
    AUTHOR = {Peng, T.  and Zheng, Z. },
     TITLE = {Abelian automorphism groups of quartic surfaces and cubic fourfolds},
      YEAR = {2021},
      NOTE = {arXiv:2106.14214v2},
       URL = {https://arxiv.org/abs/2106.14214},
}

@misc {FWZ26,
    AUTHOR = {Fu, J.  and Wang, S.  and Zheng, Z. },
     TITLE = {Non-symplectic indices of automorphism groups of smooth cubic fourfolds},
      YEAR = {2026},
      NOTE = {arXiv:2606.11754v1},
       URL = {https://arxiv.org/abs/2606.11754},
}

@misc {FZ26,
    AUTHOR = {Fu, J.  and Zheng, Z. },
     TITLE = {Automorphism groups of smooth cubic fourfolds through lattice theory},
      YEAR = {2026},
      NOTE = {arXiv:2609.06683v1},
       URL = {https://arxiv.org/abs/2609.06683},
}

@misc {xiezhenglift,
    AUTHOR = {Xie, B.  and Zheng, Z. },
     TITLE = {Sylow criteria for liftability of automorphism groups of smooth hypersurfaces},
      YEAR = {2026},
      NOTE = {arXiv:2607.23465v1},
       URL = {https://arxiv.org/abs/2607.23465},
}

@misc {xiezhengfull,
    AUTHOR = {Xie, B.  and Zheng, Z. },
     TITLE = {Small-subgroup criteria for liftability of automorphism groups of smooth hypersurfaces},
      YEAR = {2026},
      NOTE = {arXiv:2609.15613v1},
       URL = {https://arxiv.org/abs/2609.15613},
}

@incollection {MR1309681,
    AUTHOR = {Popov, V. L. and Vinberg, E. B.},
     TITLE = {Invariant theory},
 BOOKTITLE = {Algebraic geometry. {IV}},
    SERIES = {Encyclopaedia of Mathematical Sciences},
    VOLUME = {55},
      YEAR = {1994},
     PAGES = {123--278},
 PUBLISHER = {Springer-Verlag},
   ADDRESS = {Berlin},
       DOI = {10.1007/978-3-662-03073-8_2},
       URL = {https://doi.org/10.1007/978-3-662-03073-8_2},
}

@misc {stacks-project,
    AUTHOR = {{The Stacks Project Authors}},
     TITLE = {The {Stacks Project}},
      YEAR = {2026},
       URL = {https://stacks.math.columbia.edu},
}

@article {harrison1975higherdegreeforms,
    AUTHOR = {Harrison, D. K.},
     TITLE = {A {G}rothendieck ring of higher degree forms},
   JOURNAL = {J. Algebra},
  FJOURNAL = {Journal of Algebra},
    VOLUME = {35},
      YEAR = {1975},
     PAGES = {123--138},
      ISSN = {0021-8693},
       DOI = {10.1016/0021-8693(75)90039-3},
       URL = {https://doi.org/10.1016/0021-8693(75)90039-3},
}

@article {huang2022centres,
    AUTHOR = {Huang, H.-L. and Lu, H. and Ye, Y. and Zhang, C.},
     TITLE = {On centres and direct sum decompositions of higher degree forms},
   JOURNAL = {Linear Multilinear Algebra},
  FJOURNAL = {Linear and Multilinear Algebra},
    VOLUME = {70},
      YEAR = {2022},
    NUMBER = {22},
     PAGES = {7290--7306},
       DOI = {10.1080/03081087.2021.1985057},
       URL = {https://doi.org/10.1080/03081087.2021.1985057},
}

@article {Hosoh1997,
    AUTHOR = {Hosoh, T.},
     TITLE = {Automorphism groups of cubic surfaces},
   JOURNAL = {J. Algebra},
  FJOURNAL = {Journal of Algebra},
    VOLUME = {192},
      YEAR = {1997},
    NUMBER = {2},
     PAGES = {651--677},
       DOI = {10.1006/jabr.1996.6968},
       URL = {https://doi.org/10.1006/jabr.1996.6968},
}

@book {Dolgachev2012,
    AUTHOR = {Dolgachev, I. V.},
     TITLE = {Classical algebraic geometry: a modern view},
 PUBLISHER = {Cambridge University Press},
   ADDRESS = {Cambridge},
      YEAR = {2012},
     PAGES = {xii+639},
      ISBN = {978-1-107-01765-8},
       DOI = {10.1017/CBO9781139084437},
       URL = {https://doi.org/10.1017/CBO9781139084437},
}

@misc {he2025cubicfourfoldsorder7automorphism,
    AUTHOR = {He, X. and Li, Y. and Wang, S. and Zheng, Z.},
     TITLE = {Cubic fourfolds with an order-$7$ automorphism},
      YEAR = {2025},
      NOTE = {arXiv:2509.26359v1},
       URL = {https://arxiv.org/abs/2509.26359},
}

@article {Brandhorst_Hofmann_2023,
    AUTHOR = {Brandhorst, S. and Hofmann, T.},
     TITLE = {Finite subgroups of automorphisms of {K3} surfaces},
   JOURNAL = {Forum Math. Sigma},
  FJOURNAL = {Forum of Mathematics, Sigma},
    VOLUME = {11},
      YEAR = {2023},
     PAGES = {e54},
       DOI = {10.1017/fms.2023.50},
       URL = {https://doi.org/10.1017/fms.2023.50},
}

@manual {GAP4,
       KEY = {GAP},
ORGANIZATION = {The GAP Group},
     TITLE = {{GAP} -- Groups, Algorithms, and Programming, Version 4.15.1},
      YEAR = {2025},
       URL = {https://www.gap-system.org},
}

@manual {OSCAR,
       KEY = {OSCAR},
ORGANIZATION = {The OSCAR Team},
     TITLE = {{OSCAR} -- Open Source Computer Algebra Research System, Version 1.7.3},
      YEAR = {2026},
       URL = {https://www.oscar-system.org},
}

@misc{FuWangZhengAutcub4fold2026,
  author       = {Fu, J. and Wang, S. and Zheng, Z.},
  title        = {Computational materials for automorphism groups of smooth cubic threefolds and fourfolds},
  year         = {2026},
  howpublished = {Zenodo, \url{https://doi.org/10.5281/zenodo.22799306}},
  note         = {Version v1.0.1},
  doi          = {10.5281/zenodo.22799306},
  url          = {https://zenodo.org/records/22799306}
}

\end{document}